\documentclass[12pt]{smfart}
\usepackage{smfenum}
\usepackage[latin1]{inputenc}
\usepackage[french]{babel}

\usepackage{amsfonts}
\usepackage{amssymb}
\usepackage[all]{xypic}
\usepackage{bbm}
\usepackage{hyperref}
\usepackage[mathscr]{eucal}

\numberwithin{equation}{subsection}
\newtheorem{theorem}{Th\'eor\`eme}
\newtheorem{lemme}[equation]{Lemme}
\newtheorem{proposition}[equation]{Proposition}
\newtheorem{definition}[equation]{D\'efinition}
\newtheorem{theoreme}[equation]{Th\'eor\`eme}
\newtheorem{corollaire}[equation]{Corollaire}

\newtheorem{apptheoreme}{Th\'eor\`eme}

\newtheorem{appproposition}[apptheoreme]{Proposition}

\theoremstyle{remark}
\newtheorem{exemple}[equation]{Exemple}
\newtheorem{hypothese}[equation]{Hypoth\`ese}
\newtheorem{remarque}[equation]{Remarque}
\newtheorem{numero}[equation]{}

\newcommand\hfld[2]{\ \smash{\mathop{\hbox to 7mm{\rightarrowfill}}
     \limits^{\scriptstyle#1}_{\scriptstyle#2}}\ }
\newcommand\hflg[2]{\smash{\mathop{\hbox to 7mm{\leftarrowfill}}
     \limits^{\scriptstyle#1}_{\scriptstyle#2}}}

\newcommand\ogg{\leavevmode\raise.3ex\hbox{$\scriptscriptstyle\langle\!\langle$}\,}
\newcommand\fgg{\leavevmode\raise.3ex\hbox{$\scriptscriptstyle\,\rangle\!\rangle$}}
\newcommand\rta{\rightarrow}

\newcommand\calA{{\mathscr A}}
\newcommand\calP{{\mathscr P}}
\newcommand\calB{{\mathscr B}}
\newcommand\calR{{\mathscr R}}
\newcommand\calL{{\mathcal L}}
\newcommand\calM{{\mathscr M}}
\newcommand\calO{{\mathscr O}}

\newcommand\calH{{\mathcal H}}

\newcommand\calU{{\mathscr U}}
\newcommand\calW{{\mathscr W}}

\newcommand\calQ{{\mathscr G}}
\newcommand\calG{{\mathscr G}}

\newcommand\rmT{{\rm T}}
\newcommand\frakg{{\mathfrak g}}
\newcommand\frakt{{\mathfrak t}}

\newcommand\frakc{{\mathfrak c}}

\newcommand\build[3]{\mathrel{\mathop{\kern
0pt#1}\limits_{\textstyle #2}^{\textstyle #3}}}

\newcommand\bfO{{\mathbf O}}

\newcommand\bfx{{\mathbf x}}
\newcommand\bfB{{\mathbf B}}
\newcommand\CC{{\mathbbm C}}
\newcommand\GG{{\mathbbm G}}
\newcommand\FF{{\mathbbm F}}

\newcommand\ZZ{{\mathbbm Z}}
\newcommand\NN{{\mathbbm N}}
\newcommand\QQ{{\mathbbm Q}}
\newcommand\HH{{\mathbbm H}}
\newcommand\bbP{{\mathbbm P}}
\newcommand\bbA{{\mathbb A}}
\newcommand\bbH{{\mathbbm H}}
\newcommand\bbB{{\mathbbm B}}
\newcommand\bbU{{\mathbbm U}}

\newcommand\bbW{{\mathbbm W}}

\newcommand\bbT{{\mathbbm T}}
\newcommand\bbG{{\mathbbm G}}
\newcommand\bbx{{\mathbbm x}}
\newcommand\bbg{{\mathbbm g}}
\newcommand\bbt{{\mathbbm t}}
\newcommand\bbc{{\mathbbm c}}
\newcommand\rmH{{\rm H}}
\newcommand\rmR{{\rm R}}
\def\rmT{{\rm T}}
\newcommand\rmD{{\rm D}}
\newcommand\bbX{{\mathbbm X}}

\newcommand\frakh{\mathfrak h}

\newcommand\GL{{\rm GL}}

\newcommand\Gal{{\rm Gal}}

\newcommand\pr{{\rm pr}}

\newcommand\tr{{\rm tr}}

\newcommand\Spec{{\rm Spec}}

\newcommand\SL{{\rm SL}}
\newcommand\PGL{{\rm PGL}}
\newcommand\Sp{{\rm Sp}}

\newcommand\Hom{{\rm Hom}}
\newcommand\Aut{{\rm Aut}}
\newcommand\Pic{{\rm Pic}}

\newcommand\Lie{{\rm Lie}}

\newcommand\ad{{\rm ad}}

\newcommand\cl{{\rm cl}}

\newcommand\Out{{\rm Out}}
\newcommand\ovl[1]{{\overline #1}}

\newcommand\hookr\hookrightarrow
\newcommand\hookl\hookleftarrow

\newcommand\inv{{\rm inv}}
\newcommand\isom{\,\smash{\mathop{\hbox to 5mm{\rightarrowfill}}
    \limits^\sim}\,}
\newcommand\cf{{\it cf.\ }}

\newcommand\rs{{\rm rs}}

\newcommand\reg{{\rm reg}}
\newcommand\ani{{\rm ani}}

\newcommand\Ql{{\overline\QQ_\ell}}

\newcommand\RHom{{\rm RHom}}
\newcommand\resultant{\mathfrak R}
\newcommand\discrim{\mathfrak D}

\numberwithin{equation}{subsection}

\begin{document}
\title{Le lemme fondamental pour les alg\`ebres de Lie}
\author{Ng\^o Bao Ch\^au}
\address{\newline
Institute for Advanced Study, Einstein Drive, Princeton NJ 08540, USA.
D\'epartement de Math\'ematiques, Universit\'e Paris-Sud, 91405 Orsay,
France.} \email {{ngo@ias.edu} et {Bao-Chau.Ngo@math.u-psud.fr} }
\date{}
\maketitle
\let\languagename\relax


\section*[Intro]{Introduction}

Dans cet article, nous proposons une d\'emonstration pour des conjectures
de Langlands, Shelstad et Waldspurger plus connues sous le nom de lemme
fondamental pour les alg\`ebres de Lie et lemme fondamental
non standard. On se reporte \`a \ref{LS} et \`a \ref{Waldspurger} pour
plus de pr\'ecisions dans les \'enonc\'es suivants.

\begin{theorem}
Soient $k$ un corps fini \`a $q$ \'el\'ements, $\calO$ un anneau de
valuation discr\`ete complet de corps r\'esiduel $k$ et $F$ son
corps des fractions. Soit $G$ un sch\'ema en groupes r\'eductifs au-dessus de
$\calO$ dont l'ordre du groupe de Weyl n'est pas divisible par la
caract\'eristique de $k$. Soient $(\kappa,\rho_\kappa)$ une donn\'ee
endoscopique de $G$ au-dessus de $\calO$ et $H$ le sch\'ema en groupes
endoscopiques associ\'e.

On a l'\'egalit\'e entre la $\kappa$-int\'egrale orbitale et l'int\'egrale orbitale stable
$$\Delta_G(a){\bf O}_a^\kappa(1_{\frakg},{\rm d}t)=\Delta_H(a_H){\bf
SO}_{a_H}(1_{\frakh},{\rm d}t)
$$
associ\'ees aux classes de conjugaison stable semi-simples
r\'eguli\`eres $a$ et $a_h$  de  $\frakg(F)$ et $\frakh(F)$ qui se correspondent, aux fonctions caract\'eristiques $1_{\frakg}$ et $1_\frakh$ des compacts $\frakg(\calO)$ et $\frakh(\calO)$
dans $\frakg(F)$ et  $\frakh(F)$ et o\`u on a not\'e
$$\Delta_G(a)=q^{-{\rm val}(\discrim_G(a))/ 2} \ { et}\
\Delta_H(a_H)=q^{-{\rm val}(\discrim_H(a_H))/ 2}
$$
$\discrim_G$ et $\discrim_H$ \'etant les fonctions discriminant de $G$
et de $H$.
\end{theorem}

\begin{theorem}
Soient $G_1,G_2$ deux sch\'emas en groupes r\'eductifs sur $\calO$ ayant
des donn\'ees radicielles isog\`enes dont l'ordre du groupe de Weyl n'est pas
divisible par la caract\'eristique de $k$. Alors, on a l'\'egalit\'e suivante
entre les int\'egrales orbitales stables
$${\bf SO}_{a_1}(1_{\frakg_1},{\rm d}t)= {\bf SO}_{a_2}(1_{\frakg_2},{\rm d}t)$$
associ\'ees aux classes de conjugaison stable semi-simples
r\'eguli\`eres $a_1$ et $a_2$  de $\frakg_1(F)$ et $\frakg_2(F)$ qui se correspondent
et aux fonctions caract\'eristiques $1_{\frakg_1}$ et $1_{\frakg_2}$ des compacts
$\frakg_1(\calO)$ et $\frakg_2(\calO)$ dans $\frakg_1(F)$ et $\frakg_2(F)$.
\end{theorem}

Nous d\'emontrons ces th\'eor\`emes dans le cas d'\'egale
caract\'eristique. D'apr\`es Waldspurger, le cas d'in\'egales
caract\'eristiques s'en d\'eduit \cf \cite{W}.

Les applications principales du lemme fondamental se trouvent dans la
r\'ealisation de certains cas particuliers du principe de fonctorialit\'e de
Lang\-lands via la comparaison de formules des traces et dans la
construction de repr\'esentations galoisiennes attach\'ees aux formes
automorphes par le biais du calcul de cohomologie des vari\'et\'es de
Shimura. On se r\'ef\`ere aux travaux d'Arthur \cite{A} pour les applications
\`a la comparaison de formules des traces et \`a l'article de Kottwitz
\cite{K-Shim} ainsi qu'au livre en pr\'eparation \'edit\'e par Harris pour les
applications aux vari\'et\'es de Shimura.

\subsubsection*{Cas connus et r\'eductions}
Le lemme fondamental a \'et\'e \'etabli dans un grand nombre de cas
particuliers. Son analogue archim\'edien a \'et\'e enti\`erement r\'esolu par
Shelstad dans \cite{Sh}. Ce cas a incit\'e Lang\-lands et Shelstad \`a
formuler leur conjecture pour un corps non-archim\'edien. Le cas du groupe
$\SL(2)$ a \'et\'e trait\'e par Labesse et Langlands dans \cite{LL}. Le cas du
groupe unitaire \`a trois variables a \'et\'e r\'esolu par Rogawski dans
\cite{Rog}. Les cas assimil\'es aux $\Sp(4)$ et $\GL(4)$ tordu ont \'et\'e
r\'esolus par Hales, Schr\"{o}der et Weissauer par des calculs explicites \cf
\cite{Hales-Sp(4)}, \cite{Sch} et \cite{We}. R\'ecemment, Whitehouse a
poursuivi ces calculs pour d\'emontrer le lemme fondamental pond\'er\'e
tordu dans ce cas \cf \cite{Wh}.

Le lemme fondamental pour le changement de base stable a \'et\'e \'etabli
par Clozel \cite{C} et Labesse \cite{Lab} \`a partir du cas de l'unit\'e de
l'alg\`ebre de Hecke d\'emontr\'e par Kottwitz \cite{K-unit}. Auparavant, le
cas $\GL(2)$ a \'et\'e \'etabli par Langlands \cite{L-GL(2)} et le cas $\GL(3)$
par Kottwitz \cite{K-GL(3)}.

Un autre cas important est le cas $\SL(n)$ r\'esolu par Waldspurger dans
\cite{W-SL}. Le cas $\SL(3)$ avec un tore elliptique a \'et\'e \'etabli
auparavant par Kottwitz \cf \cite{K-SL(3)} et le cas $\SL(n)$ avec un tore
elliptique par Kazhdan \cf \cite{Ka}.

R\'ecemment, avec Laumon, nous avons d\'emontr\'e le lemme fondamental
pour les alg\`ebres de Lie des groupes unitaires dans le cas d'\'egale
caract\'eristique. La m\'ethode que nous utilisons est g\'eom\'etrique et ne
s'applique qu'aux corps locaux d'\'egale caract\'eristique. Comme nous avons
d\'ej\`a mentionn\'e, le cas d'in\'egales caract\'eristiques s'en d\'eduit gr\^ace aux travaux
de Waldspurger  \cite{W}. Le changement de caract\'eristiques a aussi \'et\'e \'etabli par
Cluckers et Loeser \cite{CL2} dans un cadre plus g\'en\'eral mais moins pr\'ecis sur la borne
de la caract\'eristique r\'esiduelle en utilisant la logique.

Dans une s\'erie de travaux comprenant notamment \cite{W-Trans} et
\cite{W-tordu}, Waldspurger a d\'emontr\'e que le lemme fondamental pour
les groupes ainsi que le transfert se d\'eduit du lemme fondamental ordinaire
pour les alg\`ebres de Lie. De m\^eme, le lemme fondamental tordu se d\'eduit
de la conjonction du lemme ordinaire pour les alg\`ebres de Lie et de ce
qu'il appelle le lemme non standard. Dans la suite de cet article, on se restreint
au lemme fondamental ordinaire
pour les alg\`ebres de Lie et sa variante non standard sur un
corps local d'\'egale caract\'eristique.

\subsubsection*{Approche g\'eom\'etrique locale}
Kazhdan et Lusztig ont introduit dans \cite{KL} les fibres de Springer affines
qui sont des incarnations g\'eom\'etriques des int\'egrales orbitales. Ce
travail fournit des renseignements de base sur la g\'eom\'etrie
des fibres de Springer affines que nous rappellerons dans le chapitre
\ref{section : Fibres de Springer affines} du pr\'esent article.

Dans l'annexe \`a \cite{KL}, Bernstein et Kazhdan ont construit une fibre de
Springer affine pour le groupe $\Sp(6)$ dont le nombre de points n'est pas
un polyn\^ome en $q$. En fait, le motif associ\'e \`a cette fibre de Springer
affine contient le motif d'une courbe hyperelliptique. Cet exemple
sugg\`ere qu'il est peu probable qu'on puisse obtenir une formule explicite
pour les int\'egrales orbitales.

L'interpr\'etation des $\kappa$-int\'egrales orbitales en termes des
quotients des fibres de Springer affines a \'et\'e \'etablie dans l'article de
Goresky, Kottwitz et MacPherson \cite{GKM}. Ils ont aussi introduit dans
\cite{GKM} l'usage de la cohomologie \'equivariante dans l'\'etude des fibres
de Springer affines. Cette strat\'egie est tr\`es adapt\'ee au cas particulier
des \'el\'ements qui appartiennent \`a un tore non ramifi\'e car on dispose
dans ce cas d'une action d'un gros tore sur la fibre de Springer affine n'ayant
que des points fixes isol\'es. Ils ont aussi d\'ecouvert une relation
remarquable entre la cohomologie \'equivariante d'une fibre de Springer
affine pour $G$ et la fibre correspondante pour le groupe endoscopique
$H$ dans ce cadre non ramifi\'e. Cette relation d\'epend en plus d'une conjecture
de puret\'e de la cohomologie de ces fibres de Springer affines. Cette
conjecture a \'et\'e v\'erifi\'ee pour les \'el\'ements ayant des valuations
radicielles \'egales dans \cite{GKM-equivaluation}.

Dans \cite{L} et \cite{L-LFU}, Laumon a introduit une m\'ethode de
d\'eformation des fibres de Springer dans le cas des groupes unitaires
fond\'ee sur la th\'eorie des d\'eformation des courbes planes. Sa strat\'egie
consiste \`a introduire une courbe plane de genre g\'eom\'etrique nul
ayant une singularit\'e prescrite par la situation locale et ensuite \`a d\'eformer
cette courbe plane. Il a aussi remarqu\'e que dans le cas unitaire, il existe
un tore de dimension un agissant sur les fibres de Springer affine de $U(n)$
associ\'ees \`a une classe stable provenant de $U(n_1)\times U(n_2)$ dont la
vari\'et\'e des points fixes est la fibre de Springer affine correspondant de
$U(n_1)\times U(n_2)$. En calculant la cohomologie \'equivariante en
famille, il a pu d\'emontrer le lemme fondamental
dans le cas unitaire mais pour les \'el\'ements \'eventuellement ramifi\'es.
Sa d\'emonstration d\'epend de nouveau de la conjecture de puret\'e des
fibres de Springer affines formul\'ee par  Goresky, Kottwitz et MacPherson.

Cette conjecture de puret\'e m'avait sembl\'e l'obstacle principal de
l'approche g\'eom\'etrique locale. Elle joue un r\^ole essentiel pour faire
d\'eg\'en\'erer certaines suites spectrales et permet d'appliquer le
th\'eor\`eme de localisation d'Atiyah-Borel-Segal.

Il existe en fait un autre obstacle au moins aussi s\'erieux \`a l'usage de la
cohomologie \'equivariante pour les fibres de Springer affines g\'en\'erales. Pour des
\'el\'ements tr\`es ramifi\'es des groupes autres qu'unitaires, il n'y a pas
d'action torique sur la fibre de Springer affine correspondante. Dans ce cas,
m\^eme si on dispose de la conjecture de puret\'e, il n'est pas clair que les
strat\'egies de \cite{GKM} et \cite{L-LFU} peuvent s'appliquer.

\subsubsection*{Approche g\'eom\'etrique globale} Dans \cite{H}, Hitchin a
d\'emontr\'e que le fibr\'e cotangent de l'espace de module des fibr\'es
stables sur une surface de Riemann compacte est un syst\`eme hamiltonien
compl\`etement int\'egrable. Pour cela, il interpr\`ete ce fibr\'e
cotangent comme l'espace de module des fibr\'es principaux munis d'un champ de Higgs.
Les hamiltoniens sont alors donn\'es par les coefficients du polyn\^ome
caract\'eristique du champ de Higgs. Il d\'efinit ainsi la fameuse fibration de
Hitchin $f:\calM\rightarrow\calA$ o\`u $\calM$ est le fibr\'e cotangent
ci-dessus, o\`u $\calA$ est l'espace affine classifiant les polyn\^omes
caract\'eristiques \`a coefficients dans l'espace des sections globales de puissances convenables du fibr\'e canonique de $X$ et o\`u $f$ est un morphisme dont la fibre g\'en\'erique
est essentiellement une vari\'et\'e ab\'elienne.

De notre point de vue, les fibr\'es de Hitchin sont des analogues globaux des
fibres de Springer affines. Il est par ailleurs important dans notre approche
de prendre les fibr\'es de Higgs non \`a valeurs dans le fibr\'e canonique
mais \`a valeur dans un fibr\'e inversible arbitraire de degr\'e assez grand. Dans
cette g\'en\'eralit\'e, l'espace de module des fibr\'es de Higgs $\calM$ n'est
plus muni d'une forme symplectique mais dispose toujours d'une fibration de
Hitchin $f:\calM\rightarrow \calA$.

Nous avons observ\'e dans \cite{N} qu'un comptage formel de points de
$\calM$ \`a coefficients dans un corps fini donne une expression quasiment
identique au c\^ot\'e g\'eom\'etrique de la formule des traces pour
l'alg\`ebre de Lie. Nous nous sommes propos\'e dans \cite{N}
d'interpr\'eter le processus de stabilisation de la formule des traces de
Langlands et Kottwitz en terme de cohomologie $\ell$-adique de la
fibration de Hitchin. Cette interpr\'etation conduit \`a une formulation d'une variante globale
du lemme fondamental en termes de cohomologie $\ell$-adique de la
fibration de Hitchin \cf \cite{Dat} et \cite{N-ICM}. L'interpr\'etation
g\'eom\'etrique du processus de stabilisation avec la fibration de Hitchin
ainsi que la formulation de cette variante globale du lemme fondamental est
plus complexe que l'analogue local avec les fibres de Springer affines. En
contrepartie, on dispose d'une g\'eom\'etrie plus riche.

L'interpr\'etation g\'eom\'etrique du processus de stabilisation est fond\'ee
sur l'action d'un champ de Picard $\calP\rightarrow \calA$ sur $\calM$ qui
est en quelque sorte le champ des sym\'etries naturelles de la fibration de
Hitchin. Le comptage de points avec l'aide de $\calP$ fait dans le
paragraphe 9 de \cite{N} est repris de fa{\c c}on plus syst\'ematique dans
le dernier chapitre \ref{section : comptage} du pr\'esent article. La
construction de $\calP$ est fond\'ee sur celle du centralisateur r\'egulier
que nous rappelons dans le chapitre \ref{section : centralisateur regulier}.
L'observation qui joue un r\^ole cl\'e dans le comptage est que pour tout
$a\in\calM(k)$, le quotient $[\calM_a/\calP_a]$ s'exprime en un produit de
quotients locaux des fibres de Springer affines $\calM_v(a)$ par leurs
groupes de sym\'etries naturelles $\calP_v(J_a)$ \cf \ref{produit} et
\cite[4.6]{N}. Cette formule de produit appara\^it  en filigrane
tout le long de l'article.

Le champ alg\'ebrique $\calM$ n'est pas de type fini. Il existe n\'eanmoins
un ouvert anisotrope $\calA^\ani$ de $\calA$, construit dans le chapitre
\ref{section : anisotrope}, au-dessus duquel $f^\ani:\calM^\ani \rightarrow
\calA^\ani$ est un morphisme propre. On sait par ailleurs que $\calM^\ani$
est lisse sur le corps de base $k$. D'apr\`es Deligne \cite{Weil2}, l'image
directe d\'eriv\'ee $f^\ani_* \Ql$ est un complexe pur c'est-\`a-dire que les
faisceaux pervers de cohomologie
$$K^n=\,^p\rmH^n(f^\ani_* \Ql)$$
sont des faisceaux pervers purs. D'apr\`es le th\'eor\`eme de
d\'ecomposition, ceux-ci deviennent semi-simples apr\`es le changement de
base \`a $\calA^\ani\otimes_k \bar k$. En vue de d\'emontrer le lemme fondamental,
il est essentiel de comprendre les facteurs g\'eom\'etriquement simples dans cette
d\'ecomposition.

L'action de $\calP$ sur $\calM$ induit une action de $\calP^\ani$ sur les
faisceaux pervers $K^n$. Ceci permet de d\'ecomposer $K^{n}$ en une
somme directe
$$K^n=\bigoplus_{[\kappa]} K^n_{[\kappa]}$$
o\`u $[\kappa]$ parcourt l'ensemble des classes de conjugaison
semi-simples dans le groupe dual $\hat G$ et seul un nombre fini de
$[\kappa]$ contribue un facteur $K^n_{[\kappa]}$ non nul. On note
$K^n_{\rm st}$ le facteur direct correspondant \`a $\kappa=1$. L'apparition
du groupe dual ici r\'esulte du calcul du faisceau des composantes connexes
des fibres de $\calP$ \`a la Tate-Nakayama qui a
\'et\'e entam\'ee dans \cite[6]{N} et est compl\'et\'ee dans les paragraphes
\ref{subsection : pi_0(P_a)} et \ref{subsection : description pi_0(P)} du
pr\'esent article. Dans \cite{N} et \cite{N-ICM}, nous avons montr\'e que cette d\'ecomposition correspond exactement \`a la d\'ecomposition endoscopique du c\^ot\'e g\'eom\'etrique de la formule des traces qui a \'et\'e \'etablie par Langlands et Kottwitz \cf \cite{Langlands} et \cite{K-EST}.

Fixons un \'el\'ement $\kappa$ dans la classe de conjugaison $[\kappa]$. Il
n'y a qu'un nombre fini de groupes endoscopiques non ramifi\'es $H$
associ\'es \`a $\kappa$. Si $H$ est l'un de ceux-ci, on a un morphisme
$\calA_H \rightarrow \calA$ qui au-dessus de $\calA^\ani$ est un morphisme
fini non ramifi\'e \cf \cite[7.2]{N}. En gros, l'interpr\'etation g\'eom\'etrique
de la variante globale du lemme fondamental consiste en la comparaison
entre le facteur $K^n_{[\kappa]}$ et le facteur stable $K^n_{H,\rm st}$
dans la d\'ecomposition de la cohomologie de la fibration de Hitchin des
groupes endoscopiques $H$ associ\'es \`a $\kappa$. On se reporte \`a \ref{stabilisation sur tilde
A} pour un \'enonc\'e pr\'ecis de la stabilisation g\'eom\'etrique. Le lemme
fondamental de Langlands-Shelstad est une cons\'equence de
\ref{stabilisation sur tilde A}.

Dans \cite{N}, on a d\'emontr\'e que $K^n_{[\kappa]}$ est support\'e par la
r\'eunion des images des morphismes $\calA_H\rightarrow \calA$. Comme on a vu dans \cite{N} et \cite{LN}, cet \'enonc\'e peut remplacer la conjecture de puret\'e de Goresky, Kottwitz
et MacPherson dans le contexte de la fibration de Hitchin. Par ailleurs, il est
tr\`es tentant de conjecturer que tous les facteurs simples de
$K^n_{[\kappa]}$ ont comme support l'image de l'un des morphismes
$\calA_H\rightarrow \calA$. Remarquons que le th\'eor\`eme de
stabilisation g\'eom\'etrique se d\'eduit de cette conjecture du support car en effet
il n'est pas difficile de l'\'etablir sur un ouvert dense de $\calA_H$. Sur un ouvert assez petit, on peut utiliser un calcul de cohomologie \'equivariante comme dans \cite{LN} ou un comptage de points comme dans cet article pour \'etablir la stabilisation.

Dans le cas unitaire, Laumon et moi avons utilis\'e la suite exacte d'Atiyah-Borel-Segal en
cohomologie \'equivariante pour d\'emontrer une variante ad hoc de
l'\'enonc\'e de support ci-dessus. Mes tentatives ult\'erieures de g\'en\'eraliser cette
m\'ethode \'equivariante \`a d'autres groupes se sont heurt\'ees \`a
l'absence de l'action torique dans certaines fibres de Hitchin associ\'ees aux
\'el\'ements tr\`es ramifi\'es.

Dans ce travail, on d\'emontre pour tous les groupes une forme faible de la
conjecture du support \ref{support faible}. En fait, on d\'emontre un th\'eor\`eme de support g\'en\'eral \ref{support general kappa} qui implique dans la situation particuli\`ere de la fibration de Hitchin l'\'enonc\'e \ref{support faible}. Pour d\'emontrer la conjecture de support dans la forme forte, il reste \`a v\'erifier une condition de $\delta$-r\'egularit\'e \ref{delta regulier}. Pour la fibration de Hitchin, cette $\delta$-r\'egularit\'e a \'et\'e v\'erifi\'ee en caract\'eristique z\'ero mais pas en caract\'eristique $p$. En caract\'eristique $p$, on d\'emontre un \'enonc\'e plus faible \ref{codimension} qui suffit n\'eanmoins pour d\'eduire le lemme fondamental de \ref{support faible}. En renversant l'argument local-global, le lemme fondamental implique le th\'eor\`eme de stabilisation g\'eom\'etrique \ref{stabilisation sur tilde A}.

Passons maintenant en revue l'organisation de l'article. Le premier chapitre contient l'\'enonc\'e du lemme fondamental pour les alg\`ebres de Lie ainsi que sa variante non standard. Le coeur de l'article est l'avant-dernier chapitre \ref{support} o\`u on d\'emontre le th\'eor\`eme du support. Le dernier chapitre \ref{section : comptage} contient l'argument de comptage qui permet de d\'eduire le lemme fondamental \`a partir du th\'eor\`eme du support. Les autres chapitres sont de nature pr\'eparatoire. Le chapitre \ref{section : Fibres de Springer affines} contient une \'etude de la g\'eom\'etrie des fibres de Springer affines. Le chapitre suivant \ref{section : Fibration de Hitchin} contient une \'etude parall\`ele de la g\'eom\'etrie de la fibreation de Hitchin. Notre outil favori dans ces \'etudes est l'action des sym\'etries naturelles fond\'ee sur la construction du centralisateur r\'egulier rappel\'ee dans le chapitre \ref{section : centralisateur regulier}. Dans le chapitre \ref{section : stratification}, on d\'efinit diverses stratifications de la base de Hitchin. Dans le chapitre suivant \ref{section : anisotrope}, on d\'efinit l'ouvert anisotrope de la base de Hitchin au-dessus du quel la fibration a toutes les bonnes propri\'et\'es pour qu'on puisse formuler le th\'eor\`eme de stabilisation g\'eom\'etrique \ref{stabilisation sur tilde A} qui sera d\'emontr\'e dans les deux derniers chapitres \ref{support} et \ref{section : comptage}. Dans l'appendice \ref{appendice GM}, nous rappelons l'argument de
comptage de dimension et dualit\'e de Poincar\'e du \`a Goresky et MacPherson. Cet argument est utilis\'e dans la d\'emonstration du th\'eor\`eme du support. Une table des mati\`eres d\'etaill\'ee est ins\'er\'ee \`a la fin de l'article. Le lecteur pourra trouver au d\'ebut de chaque chapitre un r\'esum\'e de son contenu et des notations utilis\'ees.

\begin{small}
\subsubsection*{Remerciements}
Sans l'aide et l'encouragement des math\'ematiciens nomm\'es ci-dessous, ce
programme n'aurait probablement pas abouti et n'aurait probablement
m\^eme pas eu lieu. Je tiens \`a leur exprimer toute ma reconnaissance. R.
Kottwitz et G. Laumon qui m'ont appris la th\'eorie de l'endoscopie et la g\'eom\'etrie alg\'ebrique, n'ont jamais cess\'e de m'aider avec beaucoup de g\'en\'erosit\'e. P. Deligne et V. Drinfeld ont
relu attentivement certaines parties du manuscrit. Leurs
nombreux commentaires m'ont permis de corriger quelques erreurs et am\'eliorer certains arguments. L'argument de dualit\'e de Poincar\'e et de
comptage de dimension que m'a expliqu\'e Goresky a jou\'e un r\^ole
catalyseur de cet article. Il est \'evident que la lecture de l'article de
Hitchin \cite{H} a jou\'e un r\^ole dans la conception de ce programme. Il
en a \'et\'e de m\^eme des articles de Faltings \cite{Fa}, de Donagi et Gaitsgory
\cite{DG} et de Rapoport \cite{Ra}. Les conversations que j'ai eues avec M.
Harris sur le lemme non standard ont renforc\'e ma conviction sur la
conjecture du support. M. Raynaud a eu la gentillesse de r\'epondre \`a
certaines de mes question techniques. Je voudrais remercier J. Arthur,
J.-P. Labesse, L. Lafforgue , R. Langlands, C. Moeglin, H. Saito et J.-L.
Waldspurger de m'avoir encourag\'e dans ce long marche \`a la poursuite du lemme. Je dis
un merci chaleureux aux math\'ematiciens qui ont particip\'e activement
aux s\'eminaires sur l'endoscopie et le lemme fondamental que j'ai
contribu\'e \`a organiser \`a Paris-Nord et Bures au printemps 2003 et \`a
Princeton aux automnes 2006 et 2007 parmi lesquels P.-H. Chaudouard, J.-F.
Dat, L. Fargues, A. Genestier, A. Ichino, V. Lafforgue, S. Morel, Nguyen Chu
Gia Vuong, Ngo Dac Tuan, S.W. Shin, D. Whitehouse et Zhiwei Yun.

J'exprime ma gratitude \`a l'I.H.E.S. \`a Bures-sur-Yvettes pour un s\'ejour tr\`es agr\'eable en 2003 o\`u ce projet a \'et\'e concu. Il a \'et\'e men\'e \`a son terme durant mes s\'ejours en automne 2006 et pendant l'ann\'ee universitaire 2007-2008 \`a l'Institute for Advanced Study \`a Princeton qui m'a offert des conditions de travail id\'eales. Pendant mes s\'ejours \`a Princeton, j'ai b\'en\'efici\'e des soutiens financiers de l'AMIAS en 2006, de la fondation Charles Simonyi et ainsi de la NSF \`a travers le contrat DMS-0635607 en 2007-2008.
\end{small}

\section{Conjectures de Langlands-Shelstad et Waldspurger}

Dans ce chapitre, on rappelle l'\'enonc\'e du lemme fondamental pour les alg\`ebres de Lie conjectur\'e par Langlands et Shelstad \cf \ref{LS} ainsi que la variante non standard conjectur\'ee par Waldspurger \cf \ref{Waldspurger}. Dans le dernier paragraphe de ce chapitre, nous rappelons o\`u intervient le lemme fondamental de Langlands-Shelstad dans le processus de stabilisation du c\^ot\'e g\'eom\'etrique la formule des traces. Notre \'etude de la fibration de Hitchin qui sera faite ult\'erieurement dans cet article est directement inspir\'e de ce processus de stabilisation. Nous nous effor\c cons de maintenir l'exposition de ce chapitre dans un langage aussi \'el\'ementaire que possible.

\subsection{Morphisme de Chevalley}
\label{subsection : Chevalley}

Soient $k$ un corps et $\bbG$ un groupe r\'eductif lisse connexe et d\'eploy\'e sur $k$. On fixera un tore maximal d\'eploy\'e $\bbT$ et un sous-groupe de Borel $\bbB$ contenant $\bbT$. Soient $N_G(\bbT)$ le normalisateur de $\bbT$ et $\bbW=N_G(\bbT)/\bbT$ le groupe de Weyl. On supposera que la
caract\'eristique du corps $k$ ne divise pas l'ordre du
groupe de Weyl. On notera $\bbg$ l'alg\`ebre de Lie de $\bbG$ et $k[\bbg]$
l'alg\`ebre des fonctions r\'eguli\`eres sur $\bbg$. De m\^eme, on notera
$\bbt=\Spec(k[\bbt])$ l'alg\`ebre de Lie de $\bbT$. Soit $r$ le rang de
$\bbG$, qui par d\'efinition est la dimension du tore maximal $\bbT$.

Le groupe $\bbG$ agit sur son alg\`ebre de Lie par l'action adjointe. L'alg\`ebre des fonctions $\bbG$-invariantes sur $\bbg$ s'identifie par la restriction de $\bbg$ \`a $\bbt$ \`a l'alg\`ebre des fonctions
$\bbW$-invariantes sur $\bbt$. De plus, d'apr\`es Shephard, Todd \cite{S-T} et Chevalley \cite{Cheval} celle-ci est une alg\`ebre de polyn\^omes \cf \cite[5.5]{Bourbaki}.

\begin{theoreme}\label{Chevalley}
Par restriction de $\bbg$ \`a $\bbt$, on a un isomorphisme d'alg\`ebres
$k[\bbg]^\bbG =k[\bbt]^\bbW$. De plus, il existe des fonctions
homog\`enes $a_1,\ldots,a_r\in k[\bbg]$ de degr\'es $e_1,\ldots,e_r$ telles
que $k[\bbg]^\bbG$ soit l'alg\`ebre des polyn\^omes de variables
$a_1,\ldots,a_r$.
\end{theoreme}

Les entiers $e_1,\ldots,e_r$ rang\'es dans l'ordre croissants sont canoniquement d\'efinis. De plus, les entiers $e_{1}-1,\ldots,e_{r}-1$ sont des exposants du syst\`eme de racines d'apr\`es Kostant \cite{Kos}. Pour les groupes classiques, il est possible de construire explicitement des polyn\^omes invariants $a_i$ \`a l'aide de l'alg\`ebre lin\'eaire \cf \cite{H} et \cite{compagnon}.

On notera $\bbc=\Spec(k[\bbt]^\bbW)=\Spec(k[\bbg]^\bbG)$. D'apr\`es le
th\'eor\`eme ci-dessus, c'est un espace affine sur le corps $k$ de dimension
$r$ et de coordonn\'ees $a_1,\ldots, a_r$. Il est de plus muni d'une action
de $\GG_m$ donn\'ee par
$$
t(a_1,\ldots,a_r)=(t^{e_1}a_1,\ldots,t^{e_r}a_r)
$$
qu'on appellera l'action par les exposants.

L'inclusion $k[\bbt]^\bbW\subset k[\bbt]$ d\'efinit un morphisme
$\pi:\bbt\rightarrow\bbc$. C'est un morphisme fini et plat qui r\'ealise
$\bbc$ comme le quotient au sens des invariants de $\bbt$ par l'action de
$\bbW$. De plus, il existe un ouvert non vide $\bbc^\rs$ de $\bbc$
au-dessus duquel $\pi$ est un morphisme fini \'etale galoisien de groupe de
Galois $\bbW$. Ce morphisme est compatible avec l'action de $\GG_m$ par
homoth\'etie sur $\bbt$ et l'action de $\GG_m$ sur $\bbc$ par les exposants.

On notera $ \chi:\bbg\rta\bbc$ le morphisme caract\'eristique de Chevalley
qui se d\'eduit de l'inclusion d'alg\`ebres $k[\bbg]^\bbG\subset k[\bbg]$. Ce
morphisme est compatible avec l'action de $\GG_m$ par homoth\'etie sur
$\bbg$ et l'action de $\GG_m$ sur $\bbc$ par les exposants. Par
analogie avec le cas des matrices, on appellera $\chi(x)$ le polyn\^ome
caract\'eristique de $x$ et $\bbc$ l'espace des polyn\^omes
caract\'eristiques.

En plus de la propri\'et\'e de $\GG_m$-\'equivariance, $\chi$ est $\bbG$-invariant
par construction. Ceci donne naissance \`a des morphismes entre
champs alg\'e\-briques
$$[\chi]:[\bbg/\bbG]\rta \bbc\ \mbox{ et }\ [\chi/\GG_m]:
[\bbg/\bbG\times\GG_m]\rta [\bbc/\GG_m].$$

Pour tout $x\in\bbg$, notons $I_x$ son centralisateur. Quand $x$ varie, les
groupes $I_x$ s'organisent en un sch\'ema en groupes $I$ au-dessus de
$\bbg$. Pour le point g\'en\'erique $x$ de $\bbg$, $\dim(I_x)=r$ de sorte
que les points $x\in\bbg$ tels que $\dim(I_x)=r$ forment un ouvert de
$\bbg$ en vertu du th\'eor\`eme de semi-continuit\'e de la dimension des
fibres. On notera $\bbg^\reg$ cet ouvert. Rappelons le r\'esultat suivant du
\`a Kostant \cf \cite{Kos}.

\begin{lemme}
La restriction de $\chi$ \`a $\bbg^\reg$ est un morphisme lisse. Ses fibres
g\'eom\'etriques sont des espaces homog\`enes sous l'action de $G$.
\end{lemme}

Puisque $\bbg^\reg$ est lisse sur $k$, la lissit\'e de la restriction de $\chi$
\`a $\bbg^\reg$ est \'equivalente \`a l'existence des sections locales de
$\chi$ passant par n'importe quel point de $\bbg^\reg$. En fait, Kostant
a construit une section globale qu'on va maintenant rappeler.

\subsection{Section de Kostant}\label{Section de Kostant}
Fixons un \'epinglage de $\GG$ qui consiste en le choix d'un tore maximal
d\'eploy\'e $\bbT$, d'un sous-groupe de Borel $\bbB$ contenant $\bbT$ de
radical unipotent $\bbU$ et d'un vecteur $\bbx_{+}\in \Lie(\bbU)$ de la forme
$\bbx_+=\sum_{\alpha\in \Delta}\bbx_\alpha$ o\`u $\Delta$ est l'ensemble
des racines simples et o\`u $\bbx_\alpha$ est un vecteur non nul du
sous-espace propre $\Lie(\bbU)_\alpha$ de $\Lie(\bbU)$ correspondant \`a
la valeur propre $\alpha$ pour l'action de $\bbT$.

Il existe alors un unique $\mathfrak{sl}_2$-triplet $(h,\bbx_+,\bbx_-)$ dans
$\bbg$ avec l'\'el\'ement semi-simple $h\in \bbt\cap \Lie(\bbG^{\rm der})$.
Rappelons l'\'enonc\'e suivant du \`a Kostant \cite[th\'eor\`eme 0.10]{Kos}.

\begin{lemme}\label{Kostant}
Soit $\bbg^{{\bbx}_+}$ le centralisateur de ${\bbx}_+$ dans ${\bbg}$. La
restriction du morphisme caract\'eristique $\chi:\bbg\rta\bbc$ au
sous-espace affine ${\bbx}_-+\bbg^{{\bbx}_+}$ du vectoriel $\bbg$ est un
isomorphisme.
\end{lemme}

L'inverse de cet isomorphisme
\begin{equation}
\epsilon:\bbc\rta{\bbx}_-+\bbg^{{\bbx}_+}\hookrightarrow \bbg
\end{equation}
d\'efinit une section du morphisme de Chevalley $\chi:\bbg\rightarrow\bbc$
appel\'ee la section de Kostant. En fait, Kostant a construit toute une
famille de sections du morphisme de Chevalley. Pour les groupes classiques,
il est \'egalement possible de construire des sections explicites $\bbc\rightarrow \bbg$
\`a l'aide de l'alg\`ebre lin\'eaire sans utiliser l'\'epinglage \cf
\cite{compagnon}. Ces sections, qui g\'en\'eralisent la matrice compagnon
dans le cas lin\'eaire, sont probablement plus adapt\'ees aux calculs
explicites des fibres de Springer affines et des fibres de Hitchin.

\subsection{Torsion ext\'erieure}\label{Torsion exterieure}

Pour les applications arithm\'etiques, il est n\'ecessaire de consid\'erer les
formes quasi-d\'eploy\'ees du groupe $\GG$. Fixons un \'epinglage
$(\bbT,\bbB,\bbx_+)$ de $\GG$ comme dans \ref{Section de Kostant}.
Notons $\Out(\GG)$ le groupe des automorphismes de $\GG$ qui fixent cet
\'epinglage. C'est un groupe discret qui peut \^etre \'eventuellement infini.
Il agit sur l'ensemble des racines $\Phi$ en laissant stable le sous-ensemble
des racines simples. Il agit aussi sur le groupe de Weyl $\bbW$ de fa{\c c}on
compatible c'est-\`a-dire qu'on a une action du produit semi-direct
$\bbW\rtimes \Out(\GG)$ sur $\bbT$ et sur l'ensemble $\Phi$ des racines.

\begin{definition}
Une forme quasi-d\'eploy\'ee de $\GG$ sur un $k$-sch\'ema $X$ est la
donn\'ee d'un $\Out(\GG)$-torseur $\rho_G$ sur $X$ muni de la topologie
\'etale.
\end{definition}

\begin{numero} La donn\'ee de $\rho_G$ permet de tordre $\GG$ pour obtenir un
$X$-sch\'ema en groupes r\'eductif lisse $G=\rho_G\wedge^{\Out(\GG)} \GG$
qui est muni d'un \'epinglage d\'efini sur $X$, c'est-\`a-dire un triplet
$(T,B,\bfx_+)$, o\`u $B$ est un sous-sch\'ema en groupes ferm\'e de $G$
lisse au-dessus de $X$, $T$ est un sous-tore de $B$ et ${\bf x_+}$ est une
section globale de $\Lie(B)$, tel que fibre par fibre $(T,B,\bfx_+)$ est
isomorphe \`a l'\'epinglage $(\bbT,\bbB,\bbx_+)$. Inversement tout
$X$-sch\'ema en groupes lisse r\'eductif muni d'un \'epinglage $(T,B,\bfx_+)$
et localement isomorphe \`a $\GG$ pour la topologie \'etale d\'efinit un
torseur sous le groupe $\Out(\GG)$. Nous nous permettrons l'abus de
langage qui consiste \`a dire que $G$ est une forme quasi-d\'eploy\'ee en
oubliant l'\'epinglage attach\'e.
\end{numero}

\begin{numero}
Soient $\rho_G$ un $\Out(\GG)$-torseur sur $X$ et $G$ la forme
quasi-d\'eploy\'ee attach\'ee avec l'\'epinglage $(T,B,\bfx_+)$. Les
structures discut\'ees dans les deux paragraphes pr\'ec\'edents se
transportent sur la forme quasi-d\'eploy\'ee. Soient $\frakg=\Lie(G)$ et
$\frakt=\Lie(T)$. L'action de $\bbW\rtimes\Out(\GG)$ sur $\bbt$ induit une
action de $\Out(\GG)$ sur $\bbc=\bbt/\bbW$. On d\'efinit donc l'espace des
polyn\^omes caract\'eristiques de $\frakg=\Lie(G)$ comme le $X$-sch\'ema
$$\frakc=\rho_G\wedge^{\Out(\GG)} \bbc.$$
Il est muni d'un morphisme
$$\chi:\frakg\rightarrow \frakc$$
qui se d\'eduit du morphisme de Chevalley $\chi:\bbg\rightarrow\bbc$.
Comme $\Out(\GG)$ fixe le $\mathfrak {sl}_2$-triplet $(h,\bbx_+,\bbx_-)$,
la section de Kostant $\epsilon:\bbc\rightarrow \bbg$ est
$\Out(\GG)$-\'equivariant. Par torsion, on obtient un $X$-morphisme
\begin{equation}\label{section de Kostant tordue}
\epsilon:\frakc\rightarrow \frakg
\end{equation}
section du morphisme de Chevalley $\chi:\frakg\rightarrow \frakc$ qu'on
appellera aussi section de Kostant.  On a par ailleurs un morphisme fini plat
$$\pi:\frakt\rightarrow \frakc$$
qui se d\'eduit de $\pi:\bbt\rightarrow \bbc$. Le $X$-sch\'ema en groupes
fini \'etale
$$W=\rho_G \wedge^{\Out(\GG)}\bbW$$
agit sur $\frakt$. Comme $\bbc$ est le quotient de $\bbt$ par $\bbW$ au
sens des invariants, $\frakc$ est aussi le quotient de $\frakt$ par $W$ au
sens  des invariants. Au-dessus de l'ouvert
$\frakc^\rs=\rho_G\wedge^{\Out(\GG)} \bbc$, le morphisme
$\pi:\frakt^\rs\rightarrow \frakc^\rs$ est un torseur sous le sch\'ema en
groupes fini \'etale $W$.
\end{numero}

\begin{numero}
Pour \'etudier les groupes endoscopiques, il est n\'ecessaire de
consid\'erer les r\'eductions du $\Out(\GG)$-torseur $\rho_G$. Une
r\'eduction de $\rho_G$ est un torseur $\rho:X_\rho\rightarrow X$ sous un
groupe discret $\Theta_\rho$ muni d'un homomorphisme ${\bf
o}_\GG:\Theta_\rho\rightarrow \Out(\GG)$ tel que
\begin{equation}\label{reduction}
\rho\wedge^{\Theta_\rho}\Out(\GG)=\rho_G.
\end{equation}
On a alors $\frakt=X_\rho\wedge^{\Theta_\rho} \bbt$ et
$\frakc=X_\rho\wedge^{\Theta_\rho} \bbc$ et un diagramme cart\'esien
\begin{equation}\label{diagramme reduction}
\xymatrix{
 X_\rho \times\bbt \ar[d]_{} \ar[r]^{\pi}    & X_\rho \times \bbc\ar[d]^{}  \\
 \frakt \ar[r]_{\pi} & \frakc           }
\end{equation}
dans lequel les deux fl\`eches verticales sont des $\Theta_\rho$-torseurs.
Le produit semi-direct $\bbW\rtimes \Theta_\rho$ agit sur $X_\rho\times
\bbt$ et agit librement sur l'ouvert $X_\rho\times \bbt^\rs$. La fl\`eche
diagonale dans le diagramme ci-dessus
\begin{equation}\label{pi_rho}
\pi_\rho:X_\rho\times \bbt \rightarrow \frakc
\end{equation}
est alors un morphisme fini plat $\bbW\rtimes \Theta_\rho$-invariant qui
r\'ealise $\frakc$ comme le quotient de $X_\rho\times \bbt$ par
$\bbW\rtimes \Theta_\rho$ au sens des invariants.  Au-dessus de l'ouvert
$\frakc^\rs$, $\pi_\rho$ est un morphisme fini \'etale galoisien de groupe de
Galois $\bbW\rtimes \Theta_\rho$. Il est souvent commode dans les calculs
de remplacer le torseur $\frakt^\rs\rightarrow \frakc^\rs$ sous le sch\'ema
en groupes fini \'etale $W$ par le torseur $X_\rho\times \bbt^\rs\rightarrow
\frakc^\rs$ sous le groupe constant $\bbW\rtimes \Theta_\rho$.
\end{numero}

\begin{numero}
Il est parfois  commode de passer du langage des torseurs au langage plus
concret des repr\'esentations du groupe fondamental. Soient $x$ un point
g\'eom\'etrique de $X$ et $\pi_1(X,x)$ le groupe fondamental de $X$
point\'e par $x$. Soit $x_{\rm Out}$ un point g\'eom\'etrique de $\rho_G$
au-dessus de $x$. La donn\'ee de ce point d\'efinit un homomorphisme
$$\rho_G^\bullet:\pi_1(X,x)\rightarrow \Out(\GG).$$
Une r\'eduction $\rho\wedge^{\Theta_\rho}\Out(\GG)=\rho_G$ \cf
\ref{reduction} point\'ee est la donn\'ee d'un point g\'eom\'etrique $x_\rho$
de $X_\rho$ au-dessus de $x_{\rm Out}$. Cela revient \`a donner un
homomorphisme $\rho^\bullet:\pi_1(X,x)\rightarrow \Theta_\rho$ dans le
diagramme commutatif
\begin{equation}\label{deploiement}
\xymatrix{
 \pi_1(X,x)\ar[dr]_{\rho_G^\bullet} \ar[r]^{\rho^\bullet}   & \Theta_\rho \ar[d]^{{\bf o}_G}  \\
 &           \Out(\GG)}
\end{equation}
o\`u ${\bf o}_\GG:\Theta_\rho\rightarrow \Out(\GG)$ est l'homomorphisme de r\'eduction.
\end{numero}

\subsection{Centralisateur r\'egulier semi-simple}
\label{subsection : Centralisateur rs}

Notons $\frakg^\rs$ l'image r\'eciproque de l'ouvert $\frakc^\rs$ par le
morphisme $\chi:\frakg\rightarrow\frakc$. Pour tout $a\in\frakc^\rs(\bar
k)$, il est bien connu que la fibre $\chi^{-1}(a)$ est form\'ee d'une seule
orbite sous l'action adjointe de $G$. Pour tout $\gamma\in \frakg^\rs(\bar
k)$, le centralisateur $I_\gamma$ de $\gamma$ est un tore maximal de
$G\otimes_k \bar k$.

\begin{numero}
Soient $a\in \frakc^\rs(\bar k)$ et $\gamma,\gamma'\in \chi^{-1}(a)$. Il
existe alors $g\in G(\bar k)$ qui transporte $\gamma$ sur $\gamma'$
c'est-\`a-dire tel que $\ad(g) \gamma=\gamma'$. L'automorphisme int\'erieur
$\ad(g)$ d\'efinit donc un isomorphisme $\ad(g):I_\gamma \isom
I_{\gamma'}$. De plus, si $g$ et $g'$ sont deux \'el\'ements de $G(\bar k)$
qui transportent $\gamma$ sur $\gamma'$, alors $g$ et $g'$ diff\`erent par un
\'el\'ement de $I_\gamma$. Comme $I_\gamma$ est un tore, en particulier
commutatif, les deux isomorphismes
$$\ad(g)\ {\rm et}\ \ad(g'):I_\gamma \isom I_{\gamma'}$$
sont les m\^emes. Ceci d\'emontre que les tores $I_\gamma$ et
$I_{\gamma'}$ sont canoniquement isomorphes. Ceci d\'efinit donc un tore
qui ne d\'epend que de $a$ et qui est canoniquement isomorphe \`a
$I_\gamma$ pour tout $\gamma\in\chi^{-1}(a)$. Ce tore peut \^etre d\'ecrit
directement \`a partir de $a$ \`a coefficients arbitraires de la fa{\c c}on
suivante.

Soient $S$ un $X$-sch\'ema et $a\in\frakc^\rs(S)$ un $S$-point de
$\frakc^\rs$. On appellera {\em rev\^etement cam\'eral} associ\'e \`a $a$
le $W$-torseur $\pi_a:\tilde S_a \rightarrow S$ obtenu en formant le
diagramme cart\'esien
$$
\xymatrix{
 \tilde S_a \ar[d]_{\pi_a} \ar[r]^{}   & \frakt^\rs \ar[d]^{\pi}   \\
  S \ar[r]_{a} & \frakc^\rs           }
$$
Posons
\begin{equation}\label{pi a J a}
J_a=\pi_a\wedge^{W} T.
\end{equation}
Ce lemme est un cas particulier probablement bien connu d'un r\'esultat de
Donagi et Gaitsgory. On rappellera l'\'enonc\'e plus g\'en\'eral dans le
paragraphe \ref{desciption galoisienne de J}.
\end{numero}

\begin{lemme}\label{J a}
Soient $S$ un $X$-sch\'ema et $a\in\frakc^\rs(S)$ un $S$-point de
$\frakc^\rs$. Soient $x$ un $S$-point de $\frakg^\rs(S)$ tel que $\chi(x)=a$ et
$I_x$ l'image r\'eciproque du centralisateur sur $\frakg$. Alors on a un
isomorphisme canonique $J_a=I_x$.
\end{lemme}

\begin{numero}
Voici une variante de la construction ci-dessus. Consid\'erons une
r\'eduction $\rho\wedge^{\Theta_\rho}\Out(\GG)=\rho_G$ du torseur
$\rho_G$ \cf \ref{reduction}. Reprenons les notations de \cf
\ref{reduction}. Soit $\pi_{\rho,a}:\tilde S_{\rho,a}\rightarrow S$ l'image
r\'eciproque par $a:S\rightarrow\frakc^\rs$ du rev\^etement
$\pi_\rho:X_\rho\times \bbt\rightarrow \frakc$ \cf \ref{pi_rho}. Le
morphisme $\pi_{\rho,a}:\tilde S_{\rho,a}\rightarrow S$ est alors un torseur
sous le groupe $\bbW\rtimes \Theta_\rho$. On a alors une d\'efinition
\'equivalente de $J_a$
\begin{equation}\label{pi a J a 2}
J_a=\pi_{\rho,a}\wedge^{\bbW\rtimes \Theta_\rho}\bbT.
\end{equation}
\end{numero}

\subsection{Classes de conjugaison dans une classe stable}
Dans ce paragraphe, $G$ sera une forme quasi-d\'eploy\'ee de $\GG$ sur un
corps $F$ contenant $k$. Nous entendrons par classe de conjugaison stable
semi-simple r\'eguli\`ere de $\frakg$ sur $F$ un \'el\'ement
$a\in\frakc^\rs(F)$. La d\'efinition originale de Langlands des classes de
conjugaison stable est plus compliqu\'ee mais une fois restreinte aux
\'el\'ements semi-simples r\'eguliers de l'alg\`ebre de Lie, elle co\"incide avec
la n\^otre. Comme nous nous limitons aux classes de conjugaison stable
semi-simples r\'eguli\`eres, dans la suite de l'article, par classe de conjugaison stable, nous entendrons semi-simple r\'eguli\`ere sauf mention expresse du contraire.

\begin{numero}
Soit $a\in\frakc^\rs(F)$ une classe de conjugaison stable. Soit
$\gamma_0=\epsilon(a)\in\frakg(F)$ l'image de $a$ par la section de
Kostant. Le centralisateur $I_{\gamma_0}$ de $\gamma_0$ est un tore
d\'efini sur $F$ canoniquement isomorphe au tore $J_a$ d\'efini dans le
paragraphe pr\'ec\'edent \cf \ref{J a}. Soit $\gamma$ un autre $F$-point de
$\chi^{-1}(a)$. Comme \'el\'ements de $\frakg(\ovl F)$, $\gamma_0$ et
$\gamma$ sont conjugu\'es c'est-\`a-dire qu'il existe $g\in G(\ovl F)$ tel que
$\gamma=\ad(g) \gamma_0$. Cette identit\'e implique que pour tout
$\sigma\in\Gal(\ovl F/F)$, $g^{-1}\sigma(g)\in I_{\gamma_0}(\ovl F)$.
L'application $\sigma\mapsto g^{-1}\sigma(g)$ d\'efinit un \'el\'ement
$$\inv(\gamma_0,\gamma)\in \rmH^1(F,I_{\gamma_0})$$
qui ne d\'epend que de la classe de $G(F)$-conjugaison de $\gamma$ et non
du choix du transporteur $g$. L'image de cette classe dans $\rmH^1(F,G)$
est triviale par construction. D'apr\`es Langlands, l'application
$\gamma\mapsto \inv(\gamma_0,\gamma)$ d\'efinit une bijection de
l'ensemble des classes de $G(F)$-conjugaison dans l'ensemble des $F$-points
de $\chi^{-1}(a)$ sur la fibre de l'application $\rmH^1(F,I_{\gamma_0})\rta
\rmH^1(F,G)$ au-dessus de l'\'el\'ement neutre \cf \cite{Langlands}. Cette
fibre sera not\'ee
$${\rm ker}(\rmH^1(F,I_{\gamma_0})\rta \rmH^1(F,G)).$$
\end{numero}

\begin{numero}
Au lieu de $\frakg(F)$, il est souvent n\'ecessaire de consid\'erer le
groupo\"ide $[\frakg/G](F)$ des couples $(E,\phi)$ compos\'e d'un $G$-torseur
$E$ sur $F$ et d'un $F$-point $\phi$ de $\ad(E)=E\wedge^G \frakg$. Le
morphisme de Chevalley d\'efinit un foncteur $[\chi]$ de $[\frakg/G](F)$
dans l'ensemble $\frakc(F)$. Soit $a\in\frakc^\rs(F)$. Consid\'erons le
groupo\"ide des $F$-points de $[\chi]^{-1}(a)$. Les objets de
$[\chi]^{-1}(a)(F)$ sont localement isomorphes pour la topologie \'etale. Par
ailleurs, on a un point base $(E_0,\gamma_0)$ du groupo\"ide o\`u $E_0$ est
le $G$-torseur trivial et o\`u $\gamma_0\in \frakg(F)$ est l'\'el\'ement
$\gamma_0=\epsilon(a)$ d\'efini par la section de Kostant. Pour tout
$F$-point $(E,\phi)$ de $[\chi]^{-1}(a)$, on obtient un invariant
$$\inv((E_0,\gamma_0),(E,\phi))\in \rmH^1(F,I_{\gamma_0}).$$
L'application $(E,\phi)\mapsto \inv((E_0,\gamma_0),(E,\phi))$ d\'efinit une
bijection de l'en\-semble des classes d'isomorphisme de $[\chi]^{-1}(a)(F)$
sur $\rmH^1(F,I_{\gamma_0})$.

Soient $(E,\phi)$ un $F$-point de $[\chi]^{-1}(a)$ et
$\inv((E_0,\gamma_0),(E,\phi))$ son
invariant. La classe d'isomorphisme de $E$ correspond alors \`a
l'image de cet invariant dans $\rmH^1(F,G)$ par l'application
$$\rmH^1(F,I_{\gamma_0})\rightarrow \rmH^1(F,G).$$
On retrouve ainsi la bijection mentionn\'ee plus haut entre les classes de
$G(F)$-conjugaison dans $\chi^{-1}(a)(F)$ et le sous-ensemble de
$\rmH^1(F,I_{\gamma_0})$ des \'el\'ements d'image triviale dans
$\rmH^1(F,G)$.
\end{numero}

\subsection{Dualit\'e de Tate-Nakayama}\label{subsection : Tate Nakayama}

La discussion du paragraphe pr\'ec\'edent prend une forme tr\`es explicite
dans le cas d'un corps local non-archim\'edien gr\^ace \`a la dualit\'e de
Tate-Nakayama. Soient $F_v$ un corps local non-archim\'edien, $\calO_v$
son anneau des entiers et $v$ la valuation. Soit $F_v^{\rm sep}$ une
cl\^oture s\'eparable de $F_v$. Notons $\Gamma_v$ le groupe de Galois
$\Gal(F_v^{\rm sep}/F_v)$. Notons $X=\Spec(F_v)$ et $x$ le point
g\'eom\'etrique choisi.

\begin{numero}
Soit $G$ la forme quasi-d\'eploy\'ee sur $F_v$ de $\GG$ associ\'ee \`a un
homomorphisme $\rho_G^\bullet:\Gamma_v\rightarrow \Out(\GG)$. Soit
$\hat \bbG$ le dual complexe de $\bbG$. Par d\'efinition il est muni d'un
\'epinglage $(\hat\bbT,\hat\bbB,\hat\bbx)$ dont la donn\'ee radicielle
associ\'ee s'obtient \`a partir de celle de $\GG$ en \'echangeant les racines
et les coracines. En particulier $\Out(\GG)=\Out(\hat \GG)$. On dispose donc
d'une action $\rho_G^\bullet$ de $\Gamma_v$ sur $\hat \GG$ fixant l'\'epinglage.

D'apr\`es Kottwitz \cf \cite{K-CTT} et \cite{K-EST}, on a alors
$$\rmH^1(F,G)^*=\pi_0((Z_{\hat\GG})^{\rho_G^\bullet(\Gamma_v)})$$
o\`u $(Z_{\hat\GG})^{\rho_G^\bullet(\Gamma_v)}$ est le sous-groupe des
points fixes dans le centre ${Z_{\hat \GG}}$ de $\hat \GG$ sous l'action de
$\rho_G^\bullet(\Gamma_v)$. En particulier, $\rmH^1(F_v,G)^*$ est un
groupe ab\'elien de type fini.
\end{numero}

\begin{numero}\label{x_rho a}
Mettons-nous dans la situation de \ref{deploiement}. En particulier, on a un
rev\^etement fini \'etale galosien $\rho:X_\rho\rightarrow X$ de groupe de
Galois $\Theta_\rho$ avec un point $x_\rho\in X_\rho$ au-dessus de $x$.
Pour tout $a\in\frakc^\rs(F_v)$, on a un $\bbW\rtimes \Theta_\rho$-torseur
$\pi_{\rho,a}:\widetilde X_{\rho,a} \rightarrow  X$ qui provient de
\ref{pi_rho}. Choisissons un point $x_{\rho,a}$ de $\widetilde X_{\rho,a}$
au-dessus de $x_\rho$.

On a alors un homomorphisme $\pi_{\rho,a}^\bullet: \Gamma_v\rightarrow
\bbW\rtimes \Theta_\rho$ rendant commutatif le diagramme
\begin{equation}\label{triangle galoisien}
\xymatrix{
 \Gamma_v \ar[dr]_{\rho_G^\bullet} \ar[r]^{\pi_{\rho,a}^\bullet\,\,\,\,}   & \bbW\rtimes \Theta_\rho\ar[d]^{}  \\
 & \Theta_\rho          }
\end{equation}

Soit $\gamma_0=\epsilon(a)$ la section de Kostant appliqu\'ee \`a $a$.
D'apr\`es le lemme \ref{J a} on a
$$I_{\gamma_0}=J_a=\pi_{\rho,a}\wedge^{\bbW\rtimes
\Theta_\rho}\bbT.$$ D'apr\`es la dualit\'e de Tate et Nakayama locale
\cite[1.1]{K-EST}, on a alors
\begin{equation}\label{T-N}
\rmH^1(F_v,J_a)^*=\pi_0(\hat \bbT^{\pi_{\rho,a}^\bullet(\Gamma_v)}).
\end{equation}
Autrement dit le groupe des caract\`eres complexes de $\rmH^1(F_v,J_a)$
co\"incide avec le groupe des composantes connexes du groupe des points
fixes de $\pi_{\rho,a}^\bullet(\Gamma_v)$ dans $\hat\bbT$. Ici $\hat \bbT$ d\'esigne le tore
complexe dual de $\bbT$ d\'efini en \'echangeant le groupe des
caract\`eres et le groupe des cocaract\`eres.
\end{numero}

\begin{numero}
L'inclusion $\iota:\hat \bbT \hookrightarrow \hat \bbG$ est
$\Gamma$-\'equivariante modulo conjugaison c'est-\`a-dire que pour tout $t\in
\hat\bbT$ pour tout $\sigma\in \Gamma_v$, $\rho^\bullet(\sigma)(\iota(t))$
et $\iota(\pi_a^\bullet(\sigma)(t))$ sont conjugu\'es dans $\hat\GG$. On en
d\'eduit l'inclusion
$$
(Z_{\hat \GG})^{\rho_G^\bullet(\Gamma_v)}\subset \hat
\bbT^{\pi_a^\bullet(\Gamma_v)}
$$
qui induit un homomorphisme entre les groupes de composantes connexes
$$
\pi_0((Z_{\hat \GG})^{\rho_G^\bullet(\Gamma_v)})\rightarrow
\pi_0(\hat \bbT^{\pi_a^\bullet(\Gamma_v)}).
$$
Par dualit\'e on retrouve la fl\`eche $\rmH^1(F_v,I_{\gamma_0})\rta
\rmH^1(F_v,G)$ d\'efinie dans le paragraphe pr\'ec\'edent.
\end{numero}

\subsection{$\kappa$-int\'egrales orbitales} \label{subsection : integrales orbitales}

Gardons les notations du paragraphe pr\'ec\'edent. Supposons en plus que
$G$ soit donn\'e comme forme quasi-d\'eploy\'ee de $\GG$ sur $\calO_v$.
Le groupe localement compact $G(F_v)$ est alors muni d'un sous-groupe
ouvert compact maximal $G(\calO_v)$. Soit ${\rm d}g_v$ la mesure de Haar
de $G(F_v)$ normalis\'ee de sorte que $G(\calO_v)$ soit de volume un.

Pour $a\in \frakc^\rs(F_v)$, donnons-nous une mesure de Haar ${\rm d}t_v$
sur le tore $J_a(F_v)$. Pour tout $\gamma\in \frakg(F_v)$ avec
$\chi(\gamma)=a$, l'isomorphisme canonique  $J_a=I_\gamma$ permet de
transporter la mesure de Haar ${\rm d}t_v$ de $J_a(F_v)$ en une mesure
de Haar sur $I_\gamma(F_v)$.

Pour tout $\gamma$ comme ci-dessus, pour toute fonction  localement
constante $f$ \`a support compact sur $\frakg(F_v)$, on peut alors d\'efinir
l'int\'egrale orbitale
$${\bf O}_\gamma(f,{\rm d}t_v)=\int_{I_\gamma(F_v) \backslash G(F_v)}
f(\ad(g_{v})^{-1}\gamma) {\frac{{\rm d}g_v} { {\rm d}t_v}}.
$$

\begin{definition}\label{kappa integrale}
Soit $\kappa$ un \'el\'ement de $\hat \bbT^{\pi_a^\bullet(\Gamma_v)}$.
Pour toute fonction $f$ localement constante \`a support compact dans
$\frakg(F_v)$, on d\'efinit la $\kappa$-int\'egrale orbitale de $f$ sur la
classe de conjugaison stable $a\in\frakc(F_v)$ par la formule
$$
{\bf O}_a^\kappa(f,{\rm d}t_v)=\sum_{\gamma} \langle
\inv(\gamma_0,\gamma), \kappa \rangle {\bf O}_\gamma(f,{\rm d}t_v)
$$
o\`u $\gamma$ parcourt l'ensemble des classes de $G(F_v)$-conjugaison
dans la classe stable de caract\'eristique $a$, o\`u le point base
$\gamma_0=\epsilon(a)$ est d\'efini par la section de Kostant et o\`u ${\rm
d}t_v$ est une mesure de Haar du tore $J_a(F_v)$.
\end{definition}

Notons pour m\'emoire que la d\'efinition de la $\kappa$-int\'egrale orbitale
d\'epend du choix du point g\'eom\'etrique $x_{\rho,a}$ dans $\tilde
X_{\rho,a}$ sans lequel on ne peut pas relier le groupe de cohomologie
$\rmH^1(F_v,J_a)$ avec le tore dual $\hat\bbT$ et donc exprimer la
dualit\'e de Tate-Nakayama sous la forme \ref{T-N}.

\subsection{Groupes endoscopiques}\label{subsection : Groupes endoscopiques}

Par construction, le groupe dual $\hat\bbG$ est muni d'un \'epinglage
$(\hat\bbT,\hat\bbB,\hat\bbx_+)$. Soit $\kappa$ un \'el\'ement du tore
maximal $\hat\bbT$ dans cet \'epinglage. La composante neutre du
centralisateur de $\kappa$ dans $\hat\GG$ est un sous-groupe r\'eductif
qu'on notera $\hat\bbH$. L'\'epinglage de $\hat\bbG$ induit sur $\hat\bbH$
un \'epinglage avec le m\^eme tore maximal $\hat\bbT$. Soit $\bbH$ le
groupe r\'eductif d\'eploy\'e sur $k$ muni d'un \'epinglage dont le dual
complexe est $\hat\bbH$ muni de son \'epinglage. On a alors
$\Out(\bbH)=\Out(\hat\bbH)$.

Consid\'erons le centralisateur $(\hat\GG\rtimes \Out(\GG))_\kappa$ de
$\kappa$ dans le produit semi-direct $\hat\GG\rtimes \Out(\GG)$. On a la
suite exacte
$$1\rightarrow \hat\bbH \rightarrow
(\hat\GG\rtimes \Out(\GG))_\kappa \rightarrow
\pi_0(\kappa)\rightarrow 1
$$
o\`u $\pi_0(\kappa)$ est le groupe des composantes connexes de
$(\hat\GG\rtimes \Out(\GG))_\kappa$. On a alors des homomorphismes
canoniques
$$
\xymatrix{
 & \pi_0(\kappa)\ar[dl]_{{\bf o}_\bbH(\kappa)} \ar[dr]^{{\bf
o}_\bbG(\kappa)}   &   \\
 \Out(\bbH) & & \Out(\bbG)           }
$$

\begin{definition}\label{donnee endoscopique}
Soit $G$ une forme quasi-d\'eploy\'ee de $G$ sur $X$ donn\'ee par un
$\Out(\GG)$-torseur $\rho_G$. Une donn\'ee endoscopique de $G$ sur $X$
est un couple $(\kappa,\rho_\kappa)$ avec $\kappa$ comme ci-dessus et
o\`u $\rho_\kappa$ est un $\pi_0(\kappa)$-torseur qui induit $\rho_G$ par
le changement de groupes de structure ${\bf o}_\bbG(\kappa)$.

Le groupe endoscopique associ\'e \`a la donn\'ee endoscopique
$(\kappa,\rho_\kappa)$ est la forme quasi-d\'eploy\'ee $H$ sur $X$ de
$\bbH$ donn\'ee par le $\Out(\HH)$-torseur $\rho_H$ obtenu \`a partir de
$\rho_\kappa$ par le changement de groupes de structure ${\bf
o}_\bbH(\kappa)$.
\end{definition}

Il y a une variante point\'ee de la notion de donn\'ee endoscopique qui est
utile. Soit $X$ un sch\'ema avec un point g\'eom\'etrique $x$. Soit $G$ un
groupe r\'eductif connexe $X$ forme quasi-d\'eploy\'ee du groupe constant
$\GG$ d\'efinie par un homomorphisme
$\rho_G^\bullet:\pi_1(X,x)\rightarrow\Out(\GG)$.

\begin{definition}\label{donnee endoscopique pointee}
On appelle donn\'ee endoscopique point\'ee de $G$ sur $X$ un couple
$(\kappa,\rho_\kappa^\bullet)$ o\`u $\kappa\in \hat\bbT$ et
$\rho_\kappa^\bullet$ est un homomorphisme
$$\rho_\kappa^\bullet:\pi_1(X,x)\rightarrow\pi_0(\kappa)$$
au-dessus de $\rho_G^\bullet$.
\end{definition}

Avec une donn\'ee endoscopique point\'ee, on peut former un diagramme
commutatif
$$
\xymatrix{
& \pi_1(X,x) \ar[ddl]_{\rho_H^\bullet} \ar[ddr]^{\rho_G^\bullet} \ar[d]_{\rho_\kappa^\bullet} &\\
 & \pi_0(\kappa)\ar[dl]^{{\bf o}_\bbH(\kappa)} \ar[dr]_{{\bf
o}_\bbG(\kappa)}   &   \\
 \Out(\bbH) & & \Out(\bbG)       }
$$
Le groupe endoscopique $H$ associ\'e \`a la donn\'ee endoscopique
point\'ee $(\kappa,\rho_\kappa^{\bullet})$ peut \^etre alors form\'e \`a
l'aide de l'homomorphisme $\rho_H^\bullet$.

\subsection{Transfert des classes de conjugaison stable}
\label{subsection : Transfert des classes}

Soit $(\kappa,\rho_\kappa)$ une donn\'ee endoscopique comme
dans \ref{donnee endoscopique}. Soit $H$ le groupe endoscopique associ\'e.
On va maintenant d\'efinir le transfert des classes de conjugaison stable de
$H$ \`a $G$ en construisant un morphisme $\nu:\frakc_H\rightarrow
\frakc$.

Par construction, on dispose d'une r\'eduction simultan\'ee des torseurs
$\rho_G$ et $\rho_H$ au torseur
$$\rho_\kappa:X_{\rho_\kappa}\rightarrow X$$
sous le groupe $\pi_0(\kappa)$. On peut alors r\'ealiser $\frakc$  comme le
quotient au sens des invariants de $X_{\rho_\kappa}\times \bbt$ par l'action
de $\bbW\rtimes \pi_0(\kappa)$ \cf \ref{pi_rho}. De m\^eme, on peut
r\'ealiser $\frakc_H$ comme le quotient au sens des invariants de
$X_{\rho_\kappa}\times \bbt$ par l'action de $\bbW_\HH\rtimes
\pi_0(\kappa)$. Pour d\'efinir le morphisme $\frakc_H\rightarrow\frakc$ il
suffit de d\'efinir un homomorphisme
$$\bbW_\bbH\rtimes \pi_0(\kappa)\rightarrow \bbW\rtimes \pi_0(\kappa)$$
compatible avec l'action sur $X_{\rho_\kappa}\times \bbt$ ce qui nous
conduit au lemme suivant.

\begin{lemme}\label{homomorphisme produit semi-direct}
Soit $\bbW_\HH\rtimes \pi_0(\kappa)$ le produit semi-direct d\'efini par
${\bf o}_\HH(\kappa):\pi_0(\kappa)\rightarrow \Out(\HH)$. Soit
$\bbW\rtimes \pi_0(\kappa)$ le produit semi-direct d\'efini par ${\bf
o}_\GG(\kappa):\pi_0(\kappa)\rightarrow \Out(\GG)$. Il existe un
homomorphisme canonique
$$\bbW_\HH\rtimes \pi_0(\kappa) \rightarrow \bbW\rtimes
\pi_0(\kappa)
$$
qui induit l'homomorphisme \'evident sur les sous-groupes normaux
$\bbW_\HH \subset \bbW$, induit l'identit\'e sur le quotient $\pi_0(\kappa)$
et qui est compatible avec les actions de $\bbW_\HH\rtimes \pi_0(\kappa)$
et $\bbW\rtimes \pi_0(\kappa)$ sur $\bbt$.
\end{lemme}

\begin{proof}
Rappelons que le centralisateur $(\bbW\rtimes \Out(\GG))_\kappa$ de
$\kappa$ dans $\bbW\rtimes\Out(\GG)$ est canoniquement isomorphe au
produit semi-direct $\bbW_\HH \rtimes \pi_0(\kappa)$ \cf \cite[lemme
10.1]{N}. On en d\'eduit un homomorphisme $\theta:\pi_0(\kappa)
\rightarrow \bbW\rtimes\Out(\GG)$ de sorte qu'on a un homomorphisme de
groupes
$$\bbW_\HH \rtimes \pi_0(\kappa) \rightarrow \bbW\rtimes^{\theta} \pi_0(\kappa)$$
o\`u le second produit semi-direct est form\'e \`a l'aide de l'action de
$\pi_0(\kappa)$ sur $\bbW$ d\'efinie par l'homomorphisme
$\theta:\pi_0(\kappa)\rightarrow \bbW\rtimes \Out(\GG)$ et de l'action de
$\bbW\rtimes \Out(\GG)$ sur $\bbW$. Cet homomorphisme est visiblement
compatible avec les actions sur $\bbt$.

Il reste \`a construire un isomorphisme entre produits semi-directs
$$\bbW\rtimes^{\theta} \pi_0(\kappa) \rightarrow \bbW\rtimes
\pi_0(\kappa)
$$
dont le second est form\'e \`a l'aide de l'homomorphisme ${\bf
o}_\GG(\kappa):\pi_0(\kappa)\rightarrow \Out(\GG)$. Pour tout $\alpha\in
\pi_0(\kappa)$, l'\'el\'ement $\theta(\alpha)\in \bbW\rtimes\Out(\GG)$
s'\'ecrit de mani\`ere unique sous la forme $\theta(\alpha)=w(\alpha){\bf
o}_\GG(\alpha)$ o\`u $w(\alpha)\in\bbW$. Ceci  nous permet de d\'efinir un
homomorphisme
$$\pi_0(\kappa)\rightarrow \bbW\rtimes \pi_0(\kappa)$$
par la formule $\alpha\mapsto w(\alpha) \alpha$. Cet homomorphisme induit
un isomorphisme $\bbW\rtimes^{\theta} \pi_0(\kappa) \rightarrow
\bbW\rtimes \pi_0(\kappa)$ qui rend le diagramme
$$ \xymatrix{
\bbW\rtimes^{\theta} \pi_0(\kappa) \ar[dr]_{} \ar[rr]^{} &  &  \bbW\rtimes \pi_0(\kappa) \ar[dl]^{}  \\
 &\bbW\rtimes \Out(\GG) &           }
$$
commutatif. En particulier, cet isomorphisme est compatible avec les
actions sur $\bbt$.
\end{proof}

Au-dessus de l'ouvert semi-simple r\'egulier $\frakc^\rs$ de $\frakc$, on
a un morphisme fini et \'etale
$$\nu^\rs:\frakc_H^{G-\rs}\rightarrow \frakc^\rs$$
o\`u on a not\'e $\frakc_H^{G-\rs}$  l'image r\'eciproque de $\frakc^\rs$
dans $\frakc_H$. Ce morphisme r\'ealise le transfert des classes de
conjugaison stable de $H$ qui sont semi-simples et $G$-r\'eguli\`eres vers
les classes de conjugaison stable de $G$ qui sont semi-simples
r\'eguli\`eres.

\begin{lemme}\label{J a J a_H}
Soient $a_H\in \frakc_H^{G-\rs}(S)$ un point \`a valeur dans un sch\'ema $S$
et $a\in\frakc^\rs(S)$ son image. Alors, on a un isomorphisme canonique
entre le tore $J_a$ d\'efini par la formule \ref{pi a J a} et le tore
$J_{H,a_H}$ d\'efini par la m\^eme formule appliqu\'ee \`a $H$.
\end{lemme}

\begin{proof}
En prenant l'image inverse de \ref{pi_rho} par $a$, on a un $\bbW\rtimes
\pi_0(\kappa)$-torseur $\pi_{\rho_\kappa,a}$ sur $S$. De m\^eme, on a un
$\bbW_H\rtimes \pi_0(\kappa)$-torseur $\pi_{\rho_\kappa,a_H}$ sur $S$.
Par construction m\^eme du morphisme $\nu:\frakc_H\rightarrow \frakc$,
on a un morphisme
$$\pi_{\rho_\kappa,a_H} \rightarrow \pi_{\rho_\kappa,a}$$
compatible aux actions de $\bbW_H\rtimes \pi_0(\kappa)$ et de
$\bbW\rtimes \pi_0(\kappa)$. Par cons\'equent, le deuxi\`eme torseur se
d\'eduit du premier par le changement de groupes de structure
$\bbW_H\rtimes \pi_0(\kappa) \rightarrow \bbW\rtimes \pi_0(\kappa)$ \cf
\ref{homomorphisme produit semi-direct}. Il suffit maintenant d'appliquer la
formule \ref{pi a J a 2}.
\end{proof}

\begin{remarque}\label{x_a dans pi_{a_H}}
Soit maintenant $S=\Spec(F_v)$ o\`u $F_v$ est un corps local comme dans
\ref{subsection : Tate Nakayama}. En choisissant un $F_v^{\rm sep}$-point
$x_{\rho_\kappa,a}$ dans le torseur $\pi_{\rho_\kappa,a_H}$, on obtient un
homomorphisme $\pi_{\rho_\kappa,a_H}^\bullet:\Gamma_v \rightarrow
\bbW_H\rtimes \pi_0(\kappa)$. On a alors la dualit\'e de Tate-Nakayama
\ref{T-N}
$$\rmH^1(F_v,J_a)^*=\rmH^1(F_v,J_{H,a_H})^*=
\hat \bbT^{\pi_{\rho_\kappa,a_H}^\bullet(\Gamma_v)}.
$$
Par construction, $\kappa\in \hat\bbT^{\bbW_H\rtimes \pi_0(\kappa)}$ de
sorte qu'on peut d\'efinir la $\kappa$-int\'egrale orbitale ${\bf O}_a(f,{\rm
d}t_v)$ suivant \ref{kappa integrale}.
\end{remarque}

\subsection{Discriminant et r\'esultant}
Soit $\Phi$ le syst\`eme de racines associ\'e au groupe d\'eploy\'e $\GG$.
Pour toute racine $\alpha\in \Phi$, on note ${\rm d}\alpha\in k[\bbt]$ la
d\'eriv\'ee du caract\`ere $\alpha:\bbT\rightarrow\GG_m$. Formons le
discriminant
$$\underline\discrim_\GG=\prod_{\alpha\in\Phi} {\rm d}\alpha \in k[\bbt]$$
qui est clairement un \'el\'ement $\bbW$-invariant de cette alg\`ebre de
polyn\^omes. Il d\'efinit donc une fonction sur l'espace des polyn\^omes
caract\'eristiques $\bbc=\Spec(k[\bbt]^\bbW)$. Soit $\discrim_\GG$ le
diviseur de $\bbc$ d\'efini par cette fonction. Rappelons l'\'enonc\'e bien
connu.

\begin{lemme}\label{diviseur discriminant}
Le diviseur $\discrim_\GG$ est un diviseur r\'eduit de $\bbc$ stable sous
l'action de $\Out(\GG)$. L'ouvert compl\'ementaire de ce diviseur est
l'ouvert r\'egulier semi-simple $\bbc^\rs$.
\end{lemme}

\begin{proof}
L'image inverse de $\discrim_\GG$ est le diviseur des hyperplans de $\bbt$
d\'efinis par les racines, chaque hyperplan \'etant de multiplicit\'e $2$ car
$\alpha$ et $-\alpha$ d\'efinissent le m\^eme hyperplan. Le groupe $\bbW$
agit librement sur le compl\'ement de ce diviseur dans $\bbt$ si bien que le
morphisme $\pi_\bbt:\bbt\rightarrow\bbc$ est \'etale sur cet ouvert et
clairement ramifi\'e le long de ce diviseur. Par ailleurs, $\Out(\GG)$ agit sur
l'ensemble des racines si bien qu'il laisse stable ce diviseur.

Il reste \`a d\'emontrer que $\discrim_\GG$ est un diviseur r\'eduit.
Puisqu'il s'agit d'une intersection compl\`ete, il suffit de montrer qu'un
ouvert dense de $\discrim_\bbG$ est r\'eduit. On peut donc \^oter de $\bbt$ les
points appartenant \`a plus de deux hyperplans de racines.  On se ram\`ene
alors \`a un groupe de rang semi-simple un o\`u l'assertion peut \^etre
v\'erifi\'ee \`a la main.
\end{proof}

Soient $X$ un $k$-sch\'ema et $G$ une forme quasi-d\'eploy\'ee de $\GG$ sur
$X$ donn\'ee par un $\Out(\GG)$-torseur $\rho_G$. En tordant $\discrim_\GG$
par $\rho_G$, on obtient un diviseur r\'eduit $\discrim_G$ de $\frakc$ dont
l'ouvert compl\'ementaire est $\frakc^\rs$. Soit $(\kappa,\rho_\kappa)$ une
donn\'ee endoscopique \cf \ref{donnee endoscopique} de $G$. On a un
diviseur $\discrim_\bbH$ de $\bbc_\HH$ et par torsion un diviseur
$\discrim_H$ de $\frakc_H$.

\begin{lemme}
Il existe un unique diviseur effectif $\resultant_\bbH^\bbG$ de $\bbc_\HH$
tel que
$$\nu^* \discrim_\bbG=\discrim_\bbH+2 \resultant_\bbH^\bbG.$$
\end{lemme}

\begin{proof} Choisissons un sous-ensemble $\Psi\subset \Phi-\Phi_\bbH$ tel que
pour toute paire de racines oppos\'ees $\pm\alpha\in \Phi-\Phi_\bbH$, le
cardinal de $\{\pm \alpha\}\cap \Psi$ vaut un. Consid\'erons l'id\'eal de
$k[\bbt]$ engendr\'e par la fonction $\prod_{\alpha\in \Psi}{\rm d}\alpha$.
Cet id\'eal est invariant sous l'action de $\bbW_\bbH$ car la fonction
$\prod_{\alpha\in \Psi}{\rm d}\alpha$ se transforme sous l'action de
$\bbW_\HH$ par un signe. Le diviseur effectif $\bbW_\HH$-invariant d\'efini
par cet id\'eal n'\'etant pas contenu dans $\discrim_\bbH$, il existe d'un
unique diviseur effectif $\resultant_\bbH^\bbG$ de $\bbc_\HH$ qui v\'erifie
alors la relation dans l'\'enonc\'e du lemme.
\end{proof}

\begin{numero}\label{discrimnant resultant}
En tordant $\resultant_\bbH^\bbG$ par $\rho_{\kappa}$, on obtient un
diviseur effectif $\resultant_H^G$ sur $\frakc_H$ qui v\'erifie la relation
$$
\nu^* \discrim_G=\discrim_H+2 \resultant_H^G.
$$
\end{numero}

\subsection{Le lemme fondamental pour les alg\`ebres de Lie}
On est maintenant en position d'\'enoncer le lemme fondamental pour les
alg\`ebres de Lie, conjectur\'e par Langlands et Shelstad et Waldspurger.

Soient $F_v$ un corps local non-archim\'edien, $\calO_v$ son anneau des
entiers et $v:F_v^\times\rightarrow \ZZ$ la valuation discr\`ete. Soit
$\FF_q$ le corps r\'esiduel de $\calO_v$. Soit $X_v=\Spec(\calO_v)$. Soit
$F_v^{\rm sep}$ une cl\^oture s\'eparable de $F_v$. Celle-ci d\'efinit un
point g\'eom\'etrique $x$ de $X_v$.

Soit $G$ une forme quasi-d\'eploy\'ee de $\GG$ sur $\calO_v$ d\'efinie par
un homomorphisme $\rho_G^\bullet:\pi_1(X_v,x)\rightarrow \Out(\GG)$.
Consid\'erons une donn\'ee endoscopique point\'ee
$(\kappa,\rho_\kappa^\bullet)$ form\'ee d'un \'el\'ement $\kappa\in\hat\bbT$
et d'un homomorphisme $\rho_\kappa^\bullet:\pi_1(X_v,x)\rightarrow
\pi_0(\kappa)$ au-dessus de $\rho_G^\bullet$ \cf \ref{donnee endoscopique
pointee}. On a alors un sch\'ema en groupes r\'eductifs $H$ au-dessus de
$X_v$ et un morphisme de $X_v$-sch\'emas $\frakc_H\rightarrow \frakc$.

Soit $a_H\in\frakc_H(\calO_v)$ d'image $a\in\frakc(\calO_v)\cap
\frakc^\rs(F_v)$. Choisissons une mesure de Haar ${\rm d}t_v$ sur le tore
$J_a(F_v)$. D'apr\`es le lemme \ref{J a J a_H}, on a un isomorphisme entre
les tores $J_a$ et $J_{H,a_H}$ de sorte qu'on peut transporter la mesure
de Haar ${\rm d}t_v$ sur $J_{H,a_H}(F_v)$.

En choisissant un $F_v^{\rm sep}$-point $x_a$ comme dans \ref{x_a dans
pi_{a_H}}, on peut d\'efinir la $\kappa$-int\'egrale orbitale ${\bf
O}_a^\kappa$ de n'importe quelle fonction localement constante \`a
support compact dans $\frakg(F_v)$. Notons $1_{\frakg_v}$ la fonction
caract\'eristique de $\frakg(\calO_v)$ dans $\frakg(F_v)$ et $1_{\frakh_v}$
la fonction caract\'eristique de $\frakh(\calO_v)$ dans $\frakh(F_v)$.

\begin{theoreme}\label{LS}
Avec les notations ci-dessus, on a l'\'egalit\'e
$${\bf O}_a^\kappa(1_{\frakg_v},{\rm d}t_v)=q^{r_{H,v}^G(a_H)}{\bf
SO}_{a_H}(1_{\frakh_v},{\rm d}t_v)
$$
o\`u $r_{H,v}^G(a_H)=\deg_v(a_H^* \resultant_H^G)$.
\end{theoreme}

Ce th\'eor\`eme est la variante pour les alg\`ebres de Lie de la
conjecture originale de Langlands et Shelstad qui porte sur les groupes de Lie
sur un corps local non-archim\'edien. Cette variante a \'et\'e formul\'ee par
Waldspurger qui a d\'emontr\'e qu'elle implique la conjecture originale pour
les groupes de Lie. Il a aussi d\'emontr\'e que le cas o\`u $F$ est un corps
local de caract\'eristique positifve ne divisant pas l'ordre de $\bbW$ implique
le cas o\`u $F$ est un corps local de caract\'eristique nulle et dont la
caract\'eristique r\'esiduelle ne divise pas l'ordre de $\bbW$ \cf \cite{W}.
Nous nous limitons au premier cas.

L'\'enonc\'e original de Langlands-Shelstad est sensiblement plus compliqu\'e
notamment \`a cause de la pr\'esence d'un signe dans le facteur de
transfert. Dans \cite{K-FdT}, Kottwitz  a \'etabli le lien entre le facteur de
transfert de Langlands-Shelstad et la section de Kostant qui nous permet
d'\'enoncer la conjecture de Langlands et Shelstad sous cette forme plus
simple. Dans le cas des groupes classiques, Waldspurger a donn\'e une forme
plus explicite de cette conjecture dans \cite{W-pav}. Hales a \'egalement
\'ecrit un article d'exposition fort agr\'eable \cite{Hales} sur l'\'enonc\'e de
la conjecture de Langlands-Shelstad.

\begin{numero} Notons que l'\'enonc\'e ci-dessus s'\'etend trivialement au cas o\`u on part
d'un \'el\'ement $a_H\in \frakc^{G-\rs}(F_v)$ qui n'est pas dans
$\frakc_H(\calO_v)$. Dans ce cas, on a \'egalement
$a\notin\frakc(\calO_v)$. Ceci se d\'eduit en effet du crit\`ere valuatif de
propret\'e appliqu\'e au morphisme fini $\nu_H:\frakc_H\rightarrow\frakc$.
Il est alors \'evident que les int\'egrales orbitales ${\bf
O}_a^\kappa(1_{\frakg_v},{\rm d}t_v)$ et ${\bf SO}_{a_H}(1_{\frakh_v},{\rm
d}t_v)$ sont nulles.
\end{numero}

\begin{numero} Notons enfin que l'\'egalit\'e de diviseurs
$$a_*\discrim_G=a_H^* \discrim_H+ 2 a_H^* \resultant_G^H$$
qui se d\'eduit de \ref{discrimnant resultant}, implique
$$\Delta_H(a_H) \Delta_G(a)^{-1}= q^{\deg_v(a_H^* \resultant_G^H)}$$
avec $\Delta_H(a_H)=q^{-{\rm deg}(a_H^* \discrim_H)/ 2}$ et
$\Delta_G(a)=q^{-{\rm deg}(a^* \discrim_G)/ 2}$. Ceci permet de r\'ecrire la
formule \ref{LS} sous la forme plus habituelle
$$\Delta_G(a){\bf O}_a^\kappa(1_{\frakg_v},{\rm d}t_v)=\Delta_H(a_H){\bf
SO}_{a_H}(1_{\frakh_v},{\rm d}t_v).
$$
\end{numero}

\subsection{Le lemme fondamental non standard}

Dans \cite{W-tordu}, Waldspurger a formul\'e une variante de la conjecture
de Langlands-Shelstad qu'il appelle le lemme fondamental non standard.
Dans ce paragraphe, nous allons rappeler cette conjecture.

Soient $\bbG_1$ et $\bbG_2$ deux groupes r\'eductifs d\'eploy\'es sur $k$
avec \'epinglages. Pour tout $i\in\{1,2\}$, on a en particulier un tore
maximal $\bbT_i$ de $\bbG_i$,  l'ensemble des racines $\Phi_i\subset
\bbX^*(\bbT_i)$, l'ensemble des racines simples $\Delta_i\subset \Phi_i$
ainsi que l'ensemble des coracines $\Phi_i^\vee\subset \bbX_*(\bbT_i)$. Le
quintuplet
$$(\bbX^*(\bbT_i),\bbX_*(\bbT_i),\Phi_i,\Phi_i^\vee,\Delta_i)$$
est la donn\'ee radicielle associ\'ee au groupe \'epingl\'e $\bbG_i$ et qui le
d\'etermine \`a  isomorphisme unique pr\`es.

\begin{definition}\label{isogenie donnees radicielles}
Une isog\'egie des donn\'ees radicielles entre $\bbG_1$ et $\bbG_2$
consiste en un couple d'isomorphismes de $\QQ$-espaces vectoriels
$$\psi^*:\bbX^*(\bbT_2)\otimes\QQ \longrightarrow \bbX^*(\bbT_1)\otimes\QQ
$$
et
$$
\psi_*:\bbX_*(\bbT_1)\otimes\QQ \longrightarrow
\bbX_*(\bbT_2)\otimes\QQ
$$
duaux l'un de l'autre tels que $\psi^*$ envoie bijectivement l'ensemble des
$\QQ$-droites de la forme $\QQ \alpha_2$ avec $\alpha_2\in \Phi_2$ sur
l'ensemble des $\QQ$-droites de la forme $\QQ \alpha_1$ avec $\alpha_1\in
\Phi_1$ en envoyant les droites des racines simples sur les droites des
racines simples et de m\^eme pour $\psi_*$ avec les $\QQ$-droites
engendr\'ees par les coracines.
\end{definition}

\begin{exemple}
Deux groupes semi-simples ayant le m\^eme groupe
adjoint ont des donn\'ees radicielles isog\`enes. En effet dans ce cas, on a
un isomorphisme canonique $\bbX^*(\bbT_1)\otimes \QQ\rightarrow
\bbX^*(\bbT_2)\otimes\QQ$ qui respecte l'ensemble des racines et celui des
racines simples et dont l'isomorphisme dual respecte les coracines.
\end{exemple}

\begin{exemple}
L'exemple le plus int\'eressant est celui de deux groupes r\'eductifs duaux au
sens de Langlands. Par d\'efinition, la donn\'ee radicielle du groupe dual
$\hat\bbG$ s'obtient \`a partir de celle de $\bbG$ en \'echangeant le
groupe des caract\`eres avec le groupe des cocaract\`eres, l'ensemble des
racines avec l'ensemble des coracines. Consid\'erons la d\'ecomposition en
somme directe de $\QQ$-espaces vectoriels
\begin{eqnarray*}
\bbX^*(\bbT)\otimes\QQ&=&\QQ\Phi \oplus \bbX^*(Z_\GG)\otimes\QQ\\
\bbX_*(\bbT)\otimes\QQ&=&\QQ\Phi^\vee \oplus \bbX_*(Z_\GG)\otimes\QQ.
\end{eqnarray*}
o\`u $\QQ\Phi$ est le sous-espace vectoriel de $\bbX^*(\bbT)\otimes\QQ$
engendr\'e par les racines et $\bbX^*(Z_\GG)$ est le groupe des
caract\`eres du centre de $\GG$. L'application $\alpha\mapsto \alpha^\vee$
induit un isomorphisme d'espaces vectoriels $\QQ\Phi\rightarrow
\QQ\Phi^\vee$ qui envoie les droites de la forme $\QQ \beta$ pour une
certaine racine $\beta$ sur les droites de la forme $\QQ {\beta'}^\vee$ pour
une certaine coracine $\beta'$. L'application lin\'eaire duale a la m\^eme
propri\'et\'e. Les cas int\'eressants sont les cas $B_n\leftrightarrow C_n$,
$F_4$ et $G_2$ o\`u il existe des racines de longueur diff\'erente. On se
r\'ef\`ere \`a \cite[page 14]{W-tordu} pour une discussion plus d\'etaill\'ee.
\end{exemple}

\begin{numero}
Puisque la r\'eflexion associ\'ee \`a une racine $\alpha$ ne d\'epend que de la
$\QQ$-droite passant par $\alpha$, les isomorphismes $\psi^*$ et $\psi_*$
induisent un isomorphisme entre les groupes de Weyl $\bbW_1 \isom
\bbW_2$ de deux groupes r\'eductifs $\bbG_1$ et $\bbG_2$ appari\'es.
\end{numero}

\begin{numero}\label{apparies}
Soient $\bbG_1$ et $\bbG_2$ deux groupes d\'eploy\'es dont les donn\'ees
radicielles sont isog\`enes. Soit $\Out_{12}$ le groupe des automorphismes
de $\bbX_*(\bbT_1)\otimes\QQ$ qui laissent stables $\Phi_1$, $\Delta_1$
mais \'egalement $\bbX^*(\bbT_2), \Phi_2, \Delta_2$ vu comme
sous-ensembles de $\bbX_*(\bbT_1)\otimes\QQ$ et de m\^eme pour les
automorphismes duaux de $\bbX_*(T_1)\otimes\QQ$. Pour tout $k$-sch\'ema
$X$, pour tout $\Out_{12}$ torseur $\rho_{12}$, on peut tordre $\bbG_1$
et $\bbG_2$ munis de leurs \'epinglages par $\rho_{12}$ pour obtenir des
formes quasi-d\'eploy\'ees $G_1$ et $G_2$. Les couples des formes
quasi-d\'eploy\'ees obtenus par ce proc\'ed\'e sont appel\'es {\em appari\'es}.
\end{numero}

\'Etant donn\'e l'isomorphisme $\psi^*$ entre les $\QQ$-espaces vectoriels
$$\bbX^{}*(\bbT_1)\otimes\QQ\simeq\bbX^{}*(\bbT_2)\otimes\QQ,$$
on peut comparer la position de deux r\'eseaux $\bbX^*(\bbT_1)$ et
$\bbX^*(\bbT_2)$ vus comme r\'eseaux dans un m\^eme $\QQ$-espace
vectoriel. Un nombre premier $p$ est dit {\em bon} par rapport \`a $\psi^*$
si $p$ ne divise pas les entiers
$$|\bbX^*(\bbT_1)/(\bbX^*(\bbT_1) \cap \bbX^*(\bbT_2))|
\ \ {\rm et}\ \ |\bbX^*(\bbT_2)/(\bbX^*(\bbT_2) \cap \bbX^*(\bbT_1))|.
$$
Si $k$ est un corps de caract\'eristique bonne par rapport \`a $\psi^*$,
celui-ci induira un isomorphisme $\bbX^*(\bbT_1)\otimes k \isom
\bbX^*(\bbT_2)\otimes k$.

\begin{lemme}\label{c_G_1=c_G_2}
Soient $G_1$ et $G_2$ deux groupes appari\'es au-dessus d'une base $X$ de bonnes
carat\'eristiques r\'esiduelles. Soient $T_1$ et $T_2$ les
tores maximaux des \'epinglages de $G_1$ et $G_2$ et $\frakt_1, \frakt_2$
leurs alg\`ebres de Lie. Alors il existe un isomorphisme canonique
$\frakt_1\rightarrow \frakt_2$ et un isomorphisme compatible
$\nu:\frakc_{G_1}\isom \frakc_{G_2}$.
\end{lemme}

\begin{proof}
On a
$$\bbt_i=\Spec({\rm Sym}_{\calO_X}(\bbX^*(\bbT_i)\otimes \calO_X))$$
o\`u ${\rm Sym}_{\calO_X}(\bbX^*(\bbT_i)\otimes \calO_X)$ est la
$\calO_X$-alg\`ebre sym\'etrique associ\'e au $\calO_X$-module libre
$\bbX^*(\bbT_i)\otimes \calO_X$. Si les caract\'eristiques r\'esiduelles de
$X$ sont bonnes, $\psi^*$ induit un isomorphisme de $\calO_X$-modules
libres
$$\bbX^*(\bbT_2)\otimes \calO_X \longrightarrow \bbX^*(\bbT_1)\otimes \calO_X$$
et donc un isomorphisme $\bbt_1 \isom \bbt_2$. On a d\'ej\`a vu que
$(\psi^*,\psi_*)$ induit un isomorphisme $\bbW_1\isom \bbW_2$ qui est
visiblement compatible avec leurs actions sur $\bbt_1\isom \bbt_2$. Il en
r\'esulte un isomorphisme entre
$$\bbc_1=\Spec({\rm Sym}_{\calO_X}(\bbX^*(\bbT_i)\otimes \calO_X))^{\bbW_{1}}$$
et
 $$\bbc_2=\Spec({\rm Sym}_{\calO_X}(\bbX^*(\bbT_i)\otimes
\calO_X))^{\bbW_{2}}.
$$
En appliquant la torsion ext\'erieure par $\rho_{12}$, on obtient
l'isomorphisme $\frakc_1 =\frakc_2$ qu'on voulait.
\end{proof}

Mettons-nous sous l'hypoth\`ese du lemme pr\'ec\'edent. Soit
$X$ le disque $\Spec(\calO_v)$ o\`u $\calO_v=k[[\epsilon_v]]$ avec $k$ un corps fini
de caract\'eristique bonne par rapport \`a $\psi^{}*$. Soient $a_1\in
\frakc_{G_1}(\calO_v)$ et $a_2\in\frakc_{G_2}(\calO_v)$ tels que
$\nu(a_1)=a_2$. L'isog\'enie $\bbT_1\rightarrow \bbT_2$ donn\'ee par
$\psi^*$ induit par torsion une isog\'enie des tores $J_{a_1}\rta J_{a_2}$
d'apr\`es \ref{J a}. De plus, l'isog\'enie induit un isomorphisme entre les
alg\`ebres de Lie de ces tores sous l'hypoth\`ese que la caract\'eristique est
bonne par rapport \`a $\psi$. On peut donc transporter des mesures de Haar
de $J_{a_1}(F_v)$ \`a $J_{a_2}(F_v)$ et inversement. On va utiliser la m\^eme notation ${\rm d}t_v$ pour ces mesures de Haar compatibles.

\begin{theoreme}\label{Waldspurger}
On a l'\'egalit\'e suivante entre les int\'egrales orbitales stables
$${\bf SO}_{a_1}(1_{G_1},{\rm d}t_v)= {\bf SO}_{a_2}(1_{G_2},{\rm d}t_v)$$
o\`u $1_{G_i}$ est la fonction caract\'eristique du compact
$\frakg_i(\calO_v)$ dans $\frakg_i(F_v)$.
\end{theoreme}

Cet \'egalit\'e a \'et\'e conjectur\'ee par Waldspurger qui l'appelle le lemme
fondamental non standard. Dans \cite{W-tordu}, il a d\'emontr\'e que la
conjonction du lemme fondamental ordinaire \ref{LS} et du lemme
non standard ci-dessus implique le lemme fondamental tordu.

\subsection{Formule globale de stabilisation}
\label{Formule globale de stabilisation}

Revenons \`a la conjecture de Langlands-Shelstad. Le lemme fondamental
consiste en une identit\'e d'int\'egrales orbitales locales. Il sera n\'eanmoins
n\'ecessaire de le replacer dans son contexte global d'origine qui est la
stabilisation de la formule des traces. Nous allons donc passer en revue la structure
de la partie anisotrope sur un corps global de caract\'eristique
positive. Cette revue nous servira de guide plus tard pour \'etudier la structure de
la cohomologie de la fibration de Hitchin.

Soient $k=\FF_q$ et $F$ le corps des fonctions rationnelles sur une courbe
projective lisse g\'eom\'etriquement connexe $X$ sur $k$. Pour tout point
ferm\'e $v\in |X|$, notons $F_v$ la compl\'etion de $F$ selon la valuation
d\'efinie par $v$ et $\calO_v$ son anneau des entiers. Pour simplifier, nous
allons supposer dans cette discussion que $G$ est un groupe adjoint
d\'eploy\'e.

Pour toute classe $\xi\in\rmH^1(F,G)$, on note $G^\xi$ la forme int\'erieure
de $G$ d\'efinie par l'image de $\xi$ dans $\rmH^1(F,G^{\rm ad})$. On
appellera la forme $G^\xi$ ainsi d\'efinie une forme int\'erieure forte. Au
lieu de consid\'erer la formule des traces de $G$, on va consid\'erer la
somme des formules des traces sur les formes int\'erieures fortes qui sont
localement triviales. La stabilisation de la somme devient plus simple et en
plus, admet une interpr\'etation g\'eom\'etrique directe. Le processus de
stabilisation qu'on va pr\'esenter est du \`a Langlands et Kottwitz, voir
\cite{Langlands} et \cite{K-EST}. Nous reprenons ici l'exposition de
\cite{N-ICM}.

Consid\'erons donc la somme
\begin{equation}\label{1}
\sum_{\xi\in {\rm ker}^1(F,G)} \sum_{\gamma\in {\mathfrak
g}^{\xi,\ani}(F)/{\sim}} \bfO_\gamma(1_D)
\end{equation}
o\`u
\begin{itemize}
    \item ${\rm ker}^1(F,G)$
    est l'ensemble des classes d'isomorphisme des $G$-tor\-seurs sur
    $F$ d'image triviale dans les $\rmH^1(F_v,G)$ pour tous
    $v\in|X|$.
    \item ${\mathfrak g}^\xi$ est la forme de ${\mathfrak g}$ sur $F$
        d\'efinie par $\xi$.
    \item $\gamma$ parcourt l'ensemble des classes de conjugaison
        r\'eguli\`eres semi-simples de ${\mathfrak g}^\xi(F)$dont le centralisateur dans
        $\frakg^\xi(F\otimes_k \bar k)$ est un tore anisotrope.
    \item $\bfO_\gamma(1_D)$ est l'int\'egrale orbitale globale
    $$
    \bfO_\gamma(1_D)=\int_{G_\gamma^\xi(F)\backslash G({\mathbb A})}
    1_{D}({\rm ad}(g)^{-1}\gamma) dg
    $$
    de la fonction
    $$1_D=\bigotimes_{v\in |X|} 1_{D_v}:\frakg(\bbA)\longrightarrow \CC$$
    $1_{D_v}$ \'etant la fonction caract\'eristique du compact ouvert
    $\varepsilon^{-d_v}{\mathfrak g}({\mathcal O}_v)$ de ${\mathfrak
    g}(F_v)$, les entiers $d_v$ \'etant des entiers pairs, nuls sauf pour un
    nombre fini de places $v$. L'int\'egrale est convergente pour les
    classes anisotropes $\gamma$.

    \item $dg$ est la mesure de Haar normalis\'ee de $G({\mathbb A})$ de
        telle fa{\c c}on que $G({\mathcal O}_{{\mathbb A}})$ soit de volume
        un.
\end{itemize}

Consid\'erons le morphisme caract\'eristique de Chevalley
$\chi:\frakg\rightarrow \frakc$ ainsi que ses formes tordues
$$\chi^\xi :\frakg^\xi \longrightarrow \frakc$$
par les classes $\xi\in\rmH^1(F,G)$. Notons que le groupe des automorphismes
de $G$ agit sur $\frakc$ \`a travers le groupe des automorphismes
ext\'erieurs de sorte que la torsion par les $\xi$ n'affecte pas $\frakc$.
Toute classe de conjugaison $\gamma$ de $\frakg^\xi(F)$ d\'efinit donc un
\'el\'ement $a\in\frakc(F)$. Comme le centralisateur de $\gamma$
semi-simple r\'egulier ne d\'epend que de $a$, il existe un sous-ensemble
$\frakc^\ani(F)$ de $\frakc(F)$ des \'el\'ements $a$ provenant des classes
$\gamma$ semi-simples r\'eguli\`eres et anisotropes dans $\frakg^{\xi}(F\otimes_{k} \bar k)$. La somme (\ref{1})
peut se r\'e\'ecrire comme une somme sur les $a\in\frakc^\ani(F)$ :
\begin{equation}\label{2}
\sum_{a\in\frakc^\ani(F)}
\sum_{\xi\in {\rm ker}^1(F,G)}
\sum_{ {\gamma\in {\mathfrak g}^\xi(F)/{\scriptscriptstyle \sim} },\, {\chi(\gamma)=a}   }
\bfO_\gamma(1_D).
\end{equation}

Pour chaque \'el\'ement $a\in\frakc^\ani(F)$, la section de
Kostant \ref{Kostant} produit un \'el\'ement $\gamma_0=\epsilon(a)\in
{\mathfrak g}(F)$ d'image $\chi(\gamma_0)=a$. On a not\'e $J_a$ le
centralisateur $I_{\gamma_0}$ de $\gamma_0$. Puisqu'on s'est restreint \`a
la partie semi-simple r\'eguli\`ere anisotrope, $J_a$ est un tore anisotrope.
Le tore dual $\hat J_a$ d\'efini sur $\Ql$ est muni d'une action finie de
$\Gamma={\rm Gal}(\overline F/F)$ telle que le groupe $\hat J_a^\Gamma$
des points fixes est un groupe fini.

Pour tout $\xi\in {\rm ker}^1(F,G)$, les classes de conjugaison
$\gamma\in {\mathfrak g}^\xi(F)$ telles que $\chi(\gamma)=a$ sont en
bijection avec les classes de cohomologie
$$\alpha={\rm inv}(\gamma_0,\gamma)\in {\rm H}^1(F,J_a)$$
dont l'image dans ${\rm H}^1(F,G)$ est l'\'el\'ement $\xi$. Ainsi l'ensemble
des paires $(\xi,\gamma)$ de la somme (\ref{2}) o\`u $\xi\in {\rm
ker}^1(F,G)$ et $\gamma$ est une classe de conjugaison de ${\mathfrak
g}^\xi(F)$ d'image $a\in \frakc^{\rm ani}(F)$ est en bijection avec
$${\rm ker}[{\rm H}^1(F,J_a)\rightarrow \bigoplus_{v\in |X|}
{\rm H}^1(F_v,G)].$$

Pour qu'une collection de classes de conjugaison
$(\gamma_v)_{v\in|X|}$ de ${\mathfrak g}(F_v)$ avec
$\chi(\gamma_v)=a$ provienne d'une paire $(\xi,\gamma)$ de la somme
(\ref{2}), il faut et il suffit que $\gamma_v=\gamma_0$ pour presque
tout $v$ et que
\begin{equation}\label{condition}
\sum_{v\in |X|}\alpha_v|_{\hat I_a^\Gamma}=0
\end{equation}
o\`u $\alpha_v={\rm inv}_v(\gamma_0,\gamma_v)$ d'apr\`es \cite{K-CTT}.
Si c'est le cas, le nombre de paires $(\xi,\gamma)$ qui s'envoient sur cette
collection $(\gamma_v)_{v\in|X|}$ est \'egal au cardinal du groupe
$${\rm ker}^1(F,J_a)={\rm ker}[{\rm H}^1(F,J_a)\rightarrow
\bigoplus_{v\in |X|} {\rm H}^1(F_v,J_a)].
$$

On va maintenant faire entrer en jeu les int\'egrales orbitales locales.
Pour d\'efinir celles-ci, on a besoin d'une mesure de Haar du centralisateur.
Pour tout $a$, on va choisir une forme volume invariante de la fibre
g\'en\'erique de $J_a$. Une telle forme existe et unique \`a un scalaire pr\`es.
Elle induit en chaque place $v$ une mesure ${\rm d}t_v$ de $J_a(F_v)$
dont le produit tensoriel $\bigotimes_{v\in|X|} {\rm d} t_v$ ne d\'epend
plus du choix de la forme volume. C'est la mesure de Tamagawa.

La somme (\ref{2}) se r\'e\'ecrit comme suit
\begin{equation}
\sum_{a\in\frakc^\ani(F)} \left|{\rm ker}^1(F,J_a)\right|\,
\tau(J_a) \sum_{(\gamma_v)_{v\in|X|}} \prod_v
\bfO_{\gamma_v}(1_{D_v},{\rm d}t_v)
\end{equation}
o\`u les $\gamma_v$ sont des classes de conjugaison de ${\mathfrak
g}(F_v)$ v\'erifiant l'\'equation (\ref{condition}) et o\`u
$$\tau(J_a)={\rm vol}(J_a(F)\backslash J_a(\bbA), \bigotimes_{v\in|X|} {\rm
d}t_v)
$$
est un nombre de Tamagawa. En mettant en facteur le nombre de
Tamagawa $\tau(J_a)$, on trouve une somme de produits d'int\'egrales
orbitales locales $\prod_v \bfO_{\gamma_v}(1_{D_v},{\rm d}t_v)$ au lieu
des int\'egrales orbitales globales $\bfO_\gamma(1_D)$.

En appliquant la formule d'Ono \cite{O}
$$\left|{\rm ker}^1(F,J_a)\right|\, \tau(J_a)
=\left|\pi_0(\hat J_a^\Gamma)\right|,
$$
la somme (\ref{2}) devient
\begin{equation}
\sum_{a\in\frakc^\ani(F)}\left|\pi_0(\hat J_a^\Gamma)\right|
\sum_{(\gamma_v)_{v\in|X|}} \prod_v \bfO_{\gamma_v}(1_{D_v},{\rm
d}t_v)
\end{equation}
o\`u les $(\gamma_v)$ v\'erifient la condition (\ref{condition}). Notons
qu'avec l'hypoth\`ese $J_a$ anisotrope, le groupe $\hat J_a^\Gamma$ est
un groupe fini de sorte que $\pi_0(\hat J_a^\Gamma)=\hat J_a^\Gamma$.

En utilisant la transformation de Fourier sur le groupe fini $\hat
J_a^\Gamma$, la somme (\ref{2}) devient
\begin{equation}\label{kappa}
\sum_{a\in\frakc^\ani(F)} \sum_{\kappa \in \hat J_a^\Gamma}
\bfO^\kappa_a(1_D,\bigotimes_{v\in|X|}{\rm d}t_v)
\end{equation}
avec
\begin{equation}
{\bf O}^\kappa_a(1_D,\bigotimes_{v\in|X|}{\rm d}t_v)=\prod_{v\in|X|}
\sum_{
\genfrac{}{}{0pt}{}{\gamma_v\in {\mathfrak g}(F_v)/{\sim}}{ \chi(\gamma_v)=a}
}
\langle {\rm inv}_v(\gamma_0,\gamma_v),\kappa\rangle {\bf
O}_{\gamma_v}(1_{D_v},{\rm d}t_v).
\end{equation}

L'op\'eration suivante consiste \`a permuter la sommation sur les $a$ et la
sommation sur les $\kappa$. En choisissant un plongement de $\hat J_a$
dans $\hat G$, $\kappa$ d\'efinit une classe de conjugaison semi-simple
$[\kappa]$. L'intersection $\hat J_a^{\Gamma}\cap [\kappa]$ de $\hat J_a^{\Gamma}$ avec la
classe de conjugaison $[\kappa]$ ne d\'epend pas du choix de plongement
de $\hat J_a$ dans $\hat G$. La somme (\ref{2}) devient maintenant
\begin{equation}\label{[kappa]}
\sum_{[\kappa]\in \hat G/\sim} \sum_{a\in \frakc^{\rm ani}(F)}
\sum_{\kappa\in \hat J_a^\Gamma \cap [\kappa]} {\bf
O}^\kappa_a(1_D,\bigotimes_{v\in|X|}{\rm d}t_v).
\end{equation}

Rappelons qu'on a suppos\'e que $G$ est un groupe semi-simple
adjoint. Pour chaque classe de conjugaison $[\kappa]$ on choisit un
repr\'esentant $\kappa\in\hat G$. Comme $\hat G$ est un groupe
semi-simple simplement connexe, $\hat G_\kappa$ est un groupe r\'eductif
connexe. Soit $\hat H=\hat G_\kappa$ et $H$ le groupe r\'eductif dual de
$\hat H$. Comme on ne s'int\'eresse qu'\`a la partie anisotrope on \'ecarte
tous les $H$ qui ne sont pas semi-simples. Supposons donc $H$ semi-simple
et regardons le morphisme
$$\nu_H:\frakc_H^\ani(F)\longrightarrow \frakc^\ani(F).$$
Un \'el\'ement $a$ est dans l'image de $\nu_{H}$ si et seulement
si $\hat J_a^\Gamma \cap [\kappa]$ est non-vide. En g\'en\'eral, on a une
bijection canonique entre l'ensemble $\hat J_a^\Gamma \cap [\kappa]$ et
l'ensemble des $a_H\in\frakc_H^\ani(F)$ dans la pr\'eimage de $a$.

En supposant le lemme fondamental, la somme (\ref{2}) devient
\begin{equation}\label{5}
\sum_{H} \sum_{a_H\in \frakc_H^{\rm ani}(F)} {\bf
SO}_a(1_D,\bigotimes_{v\in|X|}{\rm d}t_v).
\end{equation}
o\`u la premi\`ere somme porte sur l'ensemble des classes d'\'equivalence
des groupes endoscopiques elliptiques de $G$.

Comme nous avons remarqu\'e dans \cite[1]{N}, le comptage des points
\`a valeurs dans un corps fini de l'espace de module des
fibr\'es de Higgs donne essentiellement l'expression (\ref{2}). On y a
d'ailleurs propos\'e une interpr\'etation g\'eom\'etrique du processus de stabilisation (\ref{2})=(\ref{kappa}) comme une d\'ecomposition de la cohomologie de la fibration de Hitchin par rapport \`a l'action de ses sym\'etries naturelles.  Il s'agit donc d'une d\'ecomposition en somme directe d'un complexe pur sur la base de la fibration de Hitchin.

L'\'egalit\'e (\ref{2})=(\ref{5}) peut alors \^etre interpr\'et\'e comme une \'egalit\'e dans un groupe de Grothendieck entre deux complexes purs. Le th\'eor\`eme \ref{stabilisation sur tilde A} est une variante pr\'ecise de cette interpr\'etation dont on en d\'eduira le lemme fondamental de Langlands-Shelstad \ref{LS}. 


\section{Centralisateur r\'egulier et section de Kostant}
\label{section : centralisateur regulier}

Nous rappelons dans ce chapitre la construction du centralisateur r\'egulier et du morphisme du centralisateur r\'gulier vers le centralisateur de \cite{N}. Nous rappelons aussi la description galoisienne du centralisateur r\'gulier de Donagi et Gaitsgory \cite{DG}.

On garde les notations de \ref{Torsion exterieure}. En particulier, $\GG$
est un groupe r\'eductif d\'eploy\'e sur un corps $k$ et $G$ est une forme
quasi-d\'eploy\'ee de $\GG$ sur un $k$-sch\'ema $X$. On suppose que la
caract\'eristique de $k$ ne divise pas l'ordre de $\bbW$.

\subsection{Centralisateur r\'egulier}\label{centralisateur
regulier}

Soit $I$ le sch\'ema en groupes des centralisateurs au-dessus de $\frakg$. La
fibre de $I$ au-dessus d'un point $x$ de $\frakg$ est le sous-groupe de $G$
qui centralise $x$
$$ I_x=\{g\in G|\ad(g)x=x\}.$$
La dimension de $I_x$ d\'epend en g\'en\'eral de $x$ de sorte que $I$ n'est
pas plat sur $\frakg$ mais la restriction $I^\reg$ de $I$ \`a l'ouvert
$\frakg^\reg$ est un sch\'ema en groupes lisse de dimension relative $r$.
Puisque sa fibre g\'en\'erique est un tore, c'est un sch\'ema en groupes
commutatif lisse.

Le lemme \cite[3.2]{N} suivant est le point de d\'epart de notre \'etude de
la fibration de Hitchin. Pour la commodit\'e du lecteur, nous allons rappeler
bri\`evement sa d\'emonstration.

\begin{lemme} \label{J}
Il existe un unique sch\'ema en groupes lisse commutatif $J$ sur $\frakc$
muni d'un isomorphisme $G$-\'equivariant
$$(\chi^* J)|_{\frakg^\reg} \isom I|_{\frakg^\reg}.$$
De plus, cet isomorphisme se prolonge en un homomorphisme $\chi^* J\rta
I$.
\end{lemme}

\begin{proof}
Soient $x_1,x_2$ deux points de $\frakg^\reg(\bar k)$ tels que
$\chi(x_1)=\chi(x_2)=a$. Soient $I_{x_1}$ et $I_{x_2}$ les fibres de $I$ en
$x_1$ et $x_2$. Il existe $g\in G(\bar k)$ tel que $\ad(g) x_1=x_2$. La
conjugaison par $g$ induit un isomorphisme $I_{x_1} \rightarrow I_{x_2}$
qui de surcro{\^\i}t ne d\'epend pas du choix de $g$ puisque $I_{x_1}$ est
commutatif. Ceci permet de d\'efinir la fibre $J_a$ de $J$ en $a$.

La d\'efinition de $J$ comme un sch\'ema en groupes affine au-dessus de
$\frakc$ se proc\`ede par descente fid\`element plate. Notons $I_1^\reg$
et $I_2^\reg$ les sch\'emas en groupes sur $\frakg^\reg\times_\frakc
\frakg^\reg$ images r\'eciproques de $I^\reg=I|_{\frakg^\reg}$ par la
premi\`ere et la deuxi\`eme projection. La donn\'ee de descente de
$I^\reg$ le long du morphisme $\chi^\reg:\frakg^\reg\rightarrow \frakc$
consiste un en isomorphisme $\sigma_{12}:I_2^\reg\rightarrow I_1^\reg$
qui v\'erifie une condition de cocycle. Nous allons construire $\sigma_{12}$
en laissant le soin de v\'erifier la condition de cocycle au lecteur.

La construction de l'isomorphisme $\sigma_{12}$ se fait aussi par descente.
Consi\-d\'erons le morphisme
$$G\times \frakg^\reg \rightarrow \frakg^\reg\times_\frakc \frakg^\reg$$
d\'efini par $(g,x)\rightarrow (x,\ad(g)x)$. C'est un morphisme lisse surjectif
donc a fortiori fid\`element plat. Au-dessus de $G\times \frakg^\reg$, on a
un isomorphisme canonique entre $I_1^\reg$ et $I_2^\reg$ qui consiste en
la structure $G$-\'equivariante de $I$. Pour que cet isomorphisme descende
\`a $\frakg^\reg\times_\frakc \frakg^\reg$, il faut v\'erifier une identit\'e
au-dessus du carr\'e de $G\times \frakg^\reg$ au-dessus de
$\frakg^\reg\times_\frakc \frakg^\reg$. Apr\`es avoir identifi\'e ce carr\'e
\`a $G\times I_1^\reg$, l'identit\'e \`a v\'erifier se d\'eduit de la
commutativit\'e du $\frakg^\reg$-sch\'ema en groupes $I_1^\reg$.

On a donc construit un sch\'ema en groupes lisse commutatif $J$ au-dessus
de $\frakc$ muni d'un isomorphism $G$-\'equivariant $\chi^*
J|_{\frakg^\reg}\rightarrow I|_{\frakg^\reg}$. Cet isomorphisme s'\'etend
en un homomorphisme de sch\'ema en groupes $\chi^* J\rightarrow I$
puisque $\chi^* J$ est un $k$-sch\'ema lisse, $I$ est un $k$-sch\'ema affine
et de plus $\chi^* J-\chi^* J|_{\frakg^\reg}$ est un ferm\'e de codimension
trois de $\chi^* J$.
\end{proof}

Nous appelons $J$ le {\em centralisateur r\'egulier}. On peut en fait
prendre comme d\'efinition $J:=\epsilon^*I$ o\`u $\epsilon$ est la section
de Kostant de \ref{Section de Kostant}. Notons que $J$ est muni d'une
structure $\GG_m$-\'equivariante pour l'action de $\GG_m$ sur $\frakc$
d\'efinie par les exposants. Par descente, on a un sch\'ema en
groupes sur $[\frakc/\GG_m]$ qu'on notera \'egalement $J$, voir
\cite[3.3]{N}.

\subsection{Sur le quotient $[\frakg/ G]$}

Le morphisme de Chevalley $\chi:\frakg\rightarrow \frakc$ \'etant
$G$-invariant, il se factorise par le champ quotient $[\frakg/ G]$ et le
morphisme
$$[\chi]:[\frakg/ G]\rightarrow \frakc.$$
Rappelons que $[\frakg/ G]$ associe \`a tout $k$-sch\'ema
$S$ le groupo\"ide des couples $(E,\phi)$ o\`u $E$ est un $G$-torseur sur $S$
et o\`u $\phi$ est une section du fibr\'e adjoint $\ad(E)$ associ\'ee \`a la
repr\'esentation adjointe de $G$.

Au-dessus de $\frakc$, on a d\'efini un sch\'ema en groupes commutatif lisse
$J$. Soit $\bfB J$ le classifiant de $J$ qui associe \`a
tout $\frakc$-sch\'ema $S$ le groupo\"ide de Picard des $J$-torseurs sur $S$. Le
lemme \ref{J} montre qu'on a une action de $\bfB J$ sur $[\frakg/ G]$
au-dessus de $\frakc$. En effet, on peut tordre un couple $(E,\phi)\in
[\frakg/ G](S)$ par un $J$-torseur \`a l'aide de l'homomorphisme $\chi^*
J\rightarrow I$ de \ref{J}.

\begin{proposition}\label{gerbe}
Le morphisme $[\chi^\reg]:[\frakg^\reg/G]\rta\frakc$ est une gerbe li\'ee
par le centralisateur r\'egulier $J$. De plus, cette gerbe est neutre.
\end{proposition}

\begin{proof}
L'assertion que $[\chi^\reg]$ est une gerbe se d\'eduit du fait que la
restriction de l'homomorphisme $\chi^* J\rightarrow I$ \`a $\frakg^\reg$
est un isomorphisme $G$-\'equivariant par construction de $J$. La section
de Kostant $\epsilon:\frakc\rta\frakg^\reg$ compos\'ee avec le morphisme
quotient $\frakg^\reg\rta[\frakg^\reg/G]$ neutralise la gerbe. Nous noterons
ce point $[\epsilon]:\frakc\rta [\frakg^\reg/G]$.
\end{proof}

\begin{numero}
Il n'est donc pas d\'eraisonnable de penser $[\frakg/ G]$ comme une sorte
de compactification du champ de Picard ${\bf B} J$. C'est une moule
avec laquelle on fabrique des situations g\'eom\'etriques plus concr\`etes comme la fibre
de Springer affine et la fibration de Hitchin en \'evaluant sur de diff\'erents
sch\'emas $S$. Pour les fibres de Springer affines, on l'\'evalue sur l'anneau
des s\'eries formelles \`a une variable \cf \ref{section : Fibres de Springer
affines}. Pour la fibration de Hitchin, on l'\'evalue sur une courbe projective
lisse \cf \ref{section : Fibration de Hitchin}. Pour cette derni\`ere, on a
besoin de tenir compte aussi de l'action de $\GG_m$ qui agit par
homoth\'etie sur $\frakg$.
\end{numero}

\begin{numero}
Le centralisateur r\'egulier $J$ est muni d'une action de $\GG_m$ qui
rel\`eve l'action de $\GG_m$ sur $\frakc$. Le
classifiant ${\bf B} J$ au-dessus de $[\frakc/ \GG_m]$ agit sur $[\frakg/
G\times\GG_m]$. Ce dernier contient comme ouvert
$[\frakg^\reg/G\times\GG_m]$. Le morphisme
\begin{equation}\label{mu}
[\chi^\reg/\GG_m]:[\frakg^\reg/G\times\GG_m]\rta[\frakc/\GG_m]
\end{equation}
est encore une gerbe li\'ee par $J$. Cette gerbe n'est pas neutre en
g\'en\'eral. Elle le devient n\'eanmoins apr\`es l'extraction d'une racine
carr\'e du fibr\'e inversible universel sur ${\bf B}\GG_m$. Consid\'erons
l'homomorphisme $[2]:\GG_m\rta\GG_m$ d\'efini par $t\mapsto t^2$. Il
induit un morphisme $\bfB[2]:{\bf B}\GG_m\rta {\bf B}\GG_m$ qui consiste
en l'extraction d'une racine carr\'e du fibr\'e inversible universel sur ${\bf
B}\GG_m$. Nous indiquons par un exposant $[2]$ le changement de base
par ce morphisme. En particulier, on a un morphisme
$$
[\chi/\GG_m]^{[2]}:[\frakg/G\times \GG_m]^{[2]}\rta
[\frakc/\GG_m]^{[2]}.
$$
En effet, $[\frakc/\GG_m]^{[2]}$ est le quotient de $\frakc$ par l'action de
$\GG_m$ d\'efinie comme le carr\'e de l'action par les exposants et $[\frakg/G\times
\GG_m]^{[2]}$ est le quotient de $\frakg$ par l'action adjointe de $G$ et le
carr\'e de l'homoth\'etie. Consid\'erons le compos\'e de deux
homomorphismes
$$\GG_m\rightarrow T\times \GG_m\rightarrow G\times\GG_m$$
dont le premier est d\'efini par $t\mapsto (2\rho(t),t)$ o\`u $2\rho$ est la
somme des coracines positives. La section de Kostant \ref{Section de
Kostant} $\epsilon:\frakc\rightarrow \frakg$ est \'equivariante par rapport \`a cet
homomorphisme de sorte qu'il induit une section de $[\chi/\GG_m]^{[2]}$.
Nous pouvons reformuler ce qui pr\'ec\`ede d'une fa{\c c}on plus commode
\`a l'usage.
\end{numero}

\begin{lemme}\label{epsilon D'}
Soient $S$ un $k$-sch\'ema muni d'un fibr\'e inversible $D$ et
$h_D:S\rightarrow {\bf B}\GG_m$ le morphisme associ\'e vers le classifiant
de $\GG_m$. Soit $a:S\rightarrow [\frakc/\GG_m]$ un morphisme au-dessus
de $h_D$. La section de Kostant et le choix d'une racine carr\'e $D'$ de $D$
d\'efinit une section
$$[\epsilon]^{D'}(a):S\rightarrow [\frakg^\reg/G\times \GG_m].$$
\end{lemme}

\subsection{Le centre de $G$ et les composantes connexes de $J$}
Le contenu de ce paragraphe est bien connu.

\begin{proposition} \label{centre connexe}
Si $G$ est un groupe de centre connexe, le centralisateur r\'egulier $J$ a
des fibres connexes.
\end{proposition}

\begin{proof} Dans le cas d'un \'el\'ement
nilpotent r\'egulier, il s'agit d'un th\'eor\`eme de Springer \cf\cite[III, 3.7 et
1.14]{S-S} et \cite[th\'eor\`eme 4.11]{S}. Le cas g\'en\'eral s'y ram\`ene par
la d\'ecomposition de Jordan. Soit $x\in \frakg(\bar k)$ un point
g\'eom\'etrique de $\frakg$ qui est r\'egulier. Soit $x=s+n$ sa
d\'ecomposition de Jordan o\`u $s\in\frakg(\bar k)$ est un \'el\'ement
semi-simple et $n\in\frakg(\bar k)$ est un \'el\'ement nilpotent tel que
$[x,n]=0$. D'apr\`es \cite[3, lemmes 5 et 8]{Kos}, le centralisateur $G_x$
de $x$ est un sous-groupe r\'eductif connexe de $G$ et de plus son centre
est connexe. On applique donc le th\'eor\`eme de Springer \`a l'\'el\'ement
nilpotent r\'egulier $n$ de $\Lie(G_s)$.
\end{proof}

\begin{corollaire}\label{pi 0 ZG}
Pour tout $x\in \frakg^\reg(\bar k)$, l'homomorphisme canonique
$Z_G\rightarrow I_x$ induit un homomorphisme surjectif
$\pi_0(Z_G)\rightarrow \pi_0(I_x)$.
\end{corollaire}

\begin{proof}
Soient $G^\ad$ le groupe adjoint de $G$ et $I^\ad_x$ le centralisateur de
$x$dans $G^\ad$. On a la suite exacte
$$1\rightarrow Z_G \rightarrow I_x \rightarrow I^\ad_x \rightarrow 1.$$
Puisque $I^\ad_x$ est un groupe connexe par la proposition qui pr\'ec\`ede,
la fl\`eche canonique $\pi_0(Z_G)\rightarrow \pi_0(I_x)$ est surjective.
\end{proof}

\subsection{Description galoisienne de $J$}
\label{desciption galoisienne de J}

A la suite de Donagi et Gaitsgory \cite{DG}, on peut d\'ecrire le sch\'ema en
groupes $J$ au-dessus de $\frakc$ \`a l'aide du rev\^etement fini et plat
$\pi:\frakt \rta\frakc$. Notre pr\'esentation sera un peu diff\'erente de la
leur.

Consid\'erons la restriction \`a la Weil du tore $T\times \frakt$ sur $\frakt$
$$\Pi:=\prod_{\frakt/\frakc} (T\times \frakt)=\pi_* (T\times\frakt).$$
Comme groupe ab\'elien fibr\'e au-dessus de $\frakc$, on a
$$
\Pi(S)=
\Hom_\frakt(S\times_\frakc \frakt,T\times\frakt)
$$
pour tout $\frakc$-sch\'ema $S$. Ce foncteur est repr\'esentable par un
sch\'ema en groupes commutatif au-dessus de $\frakc$ puisque le morphisme
$\frakt\rightarrow\frakc$ est fini et plat. Comme la restriction \`a la Weil
pr\'eserve la lissit\'e, $\Pi$ est un sch\'ema en groupes lisse et commutatif
au-dessus de $\frakc$ de dimension relative $r\sharp\,\bbW$. Au-dessus de
l'ouvert $\frakc^\rs$, le rev\^etement $\frakt^\rs \rta\frakc^\rs$ est fini
\'etale de sorte que la restriction de $\Pi$ \`a cet ouvert est un tore.

Le $S$-sch\'ema en groupes fini \'etale $W$ agit simultan\'ement sur $T$ et
$\frakt$. L'action diagonale de $W$ sur $T\times\frakt$ induit une action de
$W$ sur $\Pi$. Les points fixes par $W$ d\'efinit un
sous-foncteur ferm\'e $J^1$ de $\Pi$.

\begin{lemme}\label{J1}
Le sous-sch\'ema ferm\'e $J^1$ de $\Pi$ est un sch\'ema en groupes
commutatif et lisse au-dessus de $\frakc$.
\end{lemme}

\begin{proof}
Puisque l'ordre de $W$ est premier \`a la caract\'eristique, les points fixes
de $W$ dans le sch\'ema lisse $\Pi$ forme un sous-sch\'ema ferm\'e $J^1$ lisse sur $k$. Soient
$x$ un point g\'eom\'etrique de $\Pi$ fixe sous l'action de $W$ et $a$ son
image dans $\frakc$. Soient ${\bf T}_x$ l'espace tangent de $\Pi$ en $x$ et
${\bf T}_a$ l'espace tangent de $\frakc$ en $a$. Puisque $\Pi$ est lisse
au-dessus de $\frakc$, l'application ${\bf T}_x\rta {\bf T}_a$ est surjective.
L'espace tangent ${\bf T}_x J^1$ de $J^1$ en $x$ est le sous-espace
vectoriel de ${\bf T}_x$ des vecteurs fixes sous $W$. Puisque la
caract\'eristique $p$ du corps de base $k$ ne divise pas $\sharp\,\bbW$,  la
restriction de ${\bf T}_x\rta {\bf T}_a$ \`a la partie $W$-fixe est encore surjective.
Il s'ensuit que $J^1$ est un sch\'ema lisse au-dessus de $\frakc$.
\end{proof}

La proposition qui suit est un renforcement de \ref{pi a J a}.

\begin{proposition} \label{J J1}
Il existe un homomorphisme canonique $J\rightarrow
J^1$ qui de plus est un isomorphisme au-dessus de l'ouvert
$\frakc^\rs$ de $\frakc$.
\end{proposition}

\begin{proof} Commen{\c c}ons par construire un homomorphisme de $J$
dans la restriction \`a la Weil $\pi_* (T\times\frakt)$. Par adjonction, il
revient au m\^eme de construire un homomorphisme
$$\pi^* J\rightarrow T\times\frakt$$
de sch\'emas en groupes au-dessus de $\frakt$.

Pour ce faire, nous avons besoin  de la r\'esolution simultan\'ee de
Gro\-thendieck-Springer. Soit $\tilde\frakg$ le sch\'ema des couples
$(x,gB)$ o\`u $x\in \frakg$ et $gB\in G/B$ qui v\'erifie
$\ad(g)^{-1}(x)\in\Lie(B)$. Ici, $B$ d\'esigne le sous-groupe de Borel de
l'\'epinglage de $G$. Notons $\pi_\frakg:\tilde\frakg\rightarrow\frakg$ la
projection sur la variable $x$. La projection $\Lie(B)\rightarrow\frakt$
d\'efinit un morphisme $\tilde\chi:\tilde\frakg\rightarrow \frakt$ qui
compl\`ete le carr\'e commutatif
$$
\xymatrix{
  \tilde\frakg \ar[d]_{\pi_\frakg} \ar[r]^{\tilde\chi}
                & \frakt  \ar[d]^{\pi}\\
     \frakg \ar[r]_{\chi}          & \frakc             }
$$
De plus, si on se restreint \`a l'ouvert $\frakg^\reg$ de $\frakg$, on obtient
un diagramme cart\'esien
$$
\xymatrix{
  \tilde\frakg^\reg \ar[d]_{\pi_\frakg^\reg} \ar[r]^{\tilde\chi^\reg}
                & \frakt  \ar[d]^{\pi}\\
     \frakg^\reg \ar[r]_{\chi^\reg}          & \frakc             }
$$
D'apr\`es \ref{J}, $(\chi^\reg)^*J= I|_{\frakg^\reg}$ si bien que pour
construire un homomorphisme de $\frakt$-sch\'emas en groupes $\pi^*
J\rightarrow (T\times \frakt)$, il suffit de construire un homomorphisme de
$\tilde\frakg^\reg$-sch\'emas en groupes
$$
(\pi_\frakg^\reg)^*(I|_{\frakg^\reg}) \rightarrow T\times
\tilde\frakg^\reg
$$
qui soit $G$-\'equivariant. Nous avons besoin d'un lemme.

\begin{lemme}
Pour tout $(x,gB)\in \tilde\frakg^\reg(\bar k)$, on a $I_x\subset \ad(g)B$.
\end{lemme}

\begin{proof}
L'assertion est claire pour les \'el\'ements $x$ qui sont r\'eguliers
semi-simples. En effet, pour ceux-ci, le centralisateur $I_x$ est un tore qui
agit la fibre $\pi_\frakg^{-1}(x)$. Puisque celle-ci est un ensemble discret,
l'action du tore est n\'ecessairement triviale.

Consid\'erons le sch\'ema en groupes $H$ au-dessus de $\tilde\frakg^\reg$
dont la fibre au-dessus de $(x,gB)\in \tilde\frakg^\reg$ est le sous-groupe de
$I_x$ des \'el\'ements $h$ tels que $h\in gBg^{-1}$. Par construction, c'est
un sous-sch\'ema en groupes ferm\'es de $(\pi_\frakg^\reg)^*I$ qui co\"incide
avec $(\pi_\frakg^\reg)^*I|_{\frakg^\reg}$ au-dessus de l'ouvert dense des
couples $(x,gB)$ avec $x$ r\'eguliers semi-simples. Or
$(\pi_\frakg^\reg)^*I|_{\frakg^\reg}$ est plat sur $\tilde\frakg^\reg$, ce
sous-sch\'ema en groupes est n\'ecessairement \'egal \`a
$(\pi_\frakg^\reg)^*I|_{\frakg^\reg}$.
\end{proof}

Consid\'erons le sch\'ema en groupes $\underline B$ au-dessus de $G/B$
dont la fibre au-dessus de $gB$ est le sous-groupe $\ad(g)B$ de $G$. Notons
$\underline B|_{\tilde\frakg^\reg}$ le changement de base de $\underline
B$ \`a $\tilde\frakg^\reg$. Le lemme ci-dessus montre qu'au-dessus de
$\tilde\frakg^\reg$, on a un homomorphisme
$$I|_{\tilde\frakg^\reg}\rightarrow  \underline B|_{\tilde\frakg^\reg}.$$

Par ailleurs, on dispose d'un homomorphisme de $\underline B$ dans le $G/B$-tore
$T\times G/B$.  En composant, on obtient un homomorphisme
$G$-\'equivariant
$$I|_{\tilde \frakg^\reg} \rightarrow T\times \tilde\frakg^\reg$$
qui est un isomorphisme au-dessus du lieu r\'egulier semi-simple. Au-dessus
du lieu r\'egulier semi-simple, l'action de $W$ sur $I|_{\tilde \frakg^\reg}$
se transporte sur l'action diagonale de $T\times \frakg^\rs$. Par adjonction,
on obtient un homomorphisme
$$I|_{\frakg^\reg}\rightarrow (\pi_\frakg)_*(T\times \tilde\frakg^\reg).$$
qui se factorise par le sous-foncteur des points fixes sous l'action diagonale
de $W$ sur $(\pi_\frakg)_*(T\times \tilde\frakg^\reg)$.

Par descente, on obtient un homomorphisme $J\rightarrow J^1$ qui est un
isomorphisme au-dessus de $\frakc^\rs$.
\end{proof}

La proposition pr\'ec\'edente admet la variante suivante. Consid\'erons une
r\'eduction $\rho\wedge^{\Theta_\rho} \Out(\GG)=\rho_G$ du torseur
$\rho_G$ \cf \ref{reduction}. On dispose alors un rev\^etement fini plat
$\pi_{\rho}:X_\rho\times \bbt \rightarrow \frakc$ qui est g\'en\'eriquement
\'etale galoisien de groupe de Galois $\bbW\rtimes \Theta_\rho$ \cf
\ref{pi_rho}. Ceci fournit une description l\'eg\`erement diff\'erente de
$J^1$.

\begin{lemme}\label{J pi_rho_G}
$J^1$ est canoniquement isomor\-phe au
sous-sch\'ema des points fixes dans la restriction des scalaires \`a la Weil
$$\prod_{(X_\rho\times \bbt)/ \frakc}(\bbT\times X_\rho\times\bbt)$$
pour l'action diagonale de $\bbW\rtimes \Theta_\rho$ sur $\bbT\times
X_\rho\times\bbt$.
\end{lemme}

\begin{proof}
L'assertion \`a d\'emontrer \'etant locale pour la topologie \'etale de $X$,
on peut supposer que $\rho$ est trivial. Dans quel cas, elle est imm\'ediate.
\end{proof}

En suivant \cite{DG}, nous allons d\'efinir un sous-faisceau $J'$
de $J^1$ que nous d\'emontrerons qu'il co\"incide avec l'image de
$J\rightarrow J^1$. Cette d\'efi\-nition n\'ecessitera un peu de
pr\'epa\-ra\-tions. La construction qui suit est locale pour la topologie
\'etale de la base $X$ de sorte qu'on peut supposer $G$ d\'eploy\'e. Pout
toute racine $\alpha\in\Phi$, soit $h_\alpha$ l'hyperplan de $\frakt$ noyau
de l'application lin\'eaire ${\mathrm d}\alpha:\frakt\rta\GG_a$. Soit
$s_\alpha\in W$ la r\'eflexion par rapport \`a l'hyperplan $h_\alpha$. Soit
$T^{s_\alpha}$ le sous-groupe de $T$ des \'el\'ements fixes par $s_\alpha$.
On a alors l'inclusion
$$\alpha(T^{s_\alpha})\subset \{\pm 1\}.$$
Soient $x$ un point g\'eom\'etrique de $\frakt$ tel que $s_\alpha(x)=x$ et
$a$ son image dans $\frakc$. Puisque $J^1$ est la partie $W$-invariante du
faisceau $\prod_{\frakt/\frakc}(T\times\frakt)$, on a un homomorphisme
canonique de $J^1_a$ dans la fibre $T\times\{x\}$ dont l'image est
contenue dans $T^{s_\alpha}\times\{x\}$. En composant avec la racine
$\alpha:T\rta \GG_m$, on obtient un homomorphisme
$\alpha_x:J^1_a\rta\GG_m$ d'image contenue dans $\{\pm 1\}$.

Soit $J^0$ le sous-sch\'ema en groupes ouvert des composantes neutres de
$J^1$. Par construction $J^0_a$ est la composante neutre de $J^1_a$ de
sorte que $J^0_a$ est contenu dans le noyau de $\alpha_x$.

\begin{definition}\label{J'}
Soit $J'$ le sous-foncteur de $J^1$ d\'efini comme suit. Pour tout
$\frakc$-sch\'ema $S$, le groupe $J'(S)$ est le sous-groupe de $J^1(S)$ des
morphismes
$$f:S\times_\frakc \frakt \rightarrow T$$
tel que pour tout point g\'eom\'etrique $x$ de $S\times_\frakc\frakt$ tel
que $s_\alpha(x)=x$ pour une certaine racine $\alpha$, on a
$\alpha(f(x))=1$. On a des inclusions de foncteurs $J^0\subset J'\subset
J^1$.
\end{definition}

\begin{lemme}
Le foncteur $J'$ de la cat\'egorie des $\frakc$-sch\'emas dans la cat\'egorie
des groupes ab\'eliens est repr\'esentable par un sous-sch\'ema en groupes
ouvert de $J^1$.
\end{lemme}

\begin{proof}
Il suffit de d\'emontrer qu'il est repr\'esentable par un sous-sch\'ema ouvert
de $J^1$ car \'etant un foncteur en sous-groupes, le sch\'ema qui le
repr\'esente sera automatiquement sch\'ema en groupes. Consid\'erons
l'ouvert $U$ de $T\times\frakt$ compl\'ement de la r\'eunion sur toutes les
racines $\alpha\in\Phi$ des ferm\'es $\alpha^{-1}(-1)\times h_\alpha$. La
restriction \`a la Weil ${\rm Res}_{\frakt/\frakc}(U)$ est alors un
sous-sch\'ema ouvert de ${\rm Res}_{\frakt/\frakc}(T\times \frakt)$. On a
alors
$$J'={\rm Res}_{\frakt/\frakc}(U)\cap J^1$$
de sorte que $J'$ est un sous-sch\'ema ouvert de $J^1$.
\end{proof}

L'\'enonc\'e suivant est une variante d'un th\'eor\`eme de Donagi et
Gaitsgory \cite[th\'eor\`eme 11.6]{DG}. On propose ici une d\'emonstration
un peu diff\'erente.

\begin{proposition}\label{J=J'}
L'homomorphisme $J\rightarrow J^1$ de \ref{J J1} se factorise par le
sous-sch\'ema en groupes ouvert $J'$ de \ref{J'} et induit un isomorphisme
$J\rightarrow J'$.
\end{proposition}

\begin{proof} Pour d\'emontrer que que $J\rightarrow J^1$ se
factorise par $J'$, il suffit de d\'emontrer que pour tout $a\in\frakc(\bar
k)$, l'homomorphisme $\pi_0(J_a)\rightarrow \pi_0(J_a^1)$ se factorise par
$\pi_0(J'_a)$. Il suffit de v\'erifier que fibre par fibre l'homomorphisme
$\pi_0(J_a)\rightarrow \pi_0(J^1_a)$ se factorise par $\pi_0(J')$. Rappelons
que l'homomorphisme $Z_G\rightarrow J_a$ induit un homomorphisme
surjectif $\pi_0(Z_G)\rightarrow \pi_0(J_a)$ d'apr\`es \ref{pi 0 ZG}. Il suffit
donc de v\'erifier que l'homomorphisme $Z_G\rightarrow J^1_a$ se
factorise par $J'_a$. Mais ceci est \'evident car la restriction de n'importe
quelle racine \`a $Z_G$ est triviale.

Puisque $J$ et $J'$ sont des sch\'ema en groupes lisses et affines au-dessus
de $\frakc$, pour d\'emontrer que l'homomorphisme $J\rightarrow J'$ est
un isomorphisme, il suffit de le faire au-dessus d'un ouvert de $\frakc$ dont
le compl\'ement est un ferm\'e de codimension deux. L'image inverse de
$\discrim_G$ dans $\frakt$ est la r\'eunion des hyperplans $h_\alpha$ des
racines. Soit $\discrim_G^{\rm sing}$ le ferm\'e de $\discrim_G$ dont
l'image inverse est lieu des points appartenant \`a au moins deux
hyperplans $h_\alpha$. Il est clair que $\discrim_G^{\rm sing}$ est un
ferm\'e de codimension deux de $\frakc$. Il suffit de d\'emontrer que
l'isomorphisme entre $J$ et $J'$ sur $\frakc-{\discrim_G}$ se prolonge en
un isomorphisme sur $\frakc-\discrim_G^{\rm sing}$.

Soit $a\in (\frakc-\discrim_G^{\rm sing})(\bar k)$. Il suffit de d\'emontrer
que l'homomorphisme $J\rightarrow J'$ est un isomorphisme au-dessus d'un
voisinage \'etale de $a$. Si $a\notin {\discrim_G}$, il n'y a rien \`a
d\'emontrer car on a vu que $J=J'=J^1$ sur l'ouvert $\frakc-{\discrim_G}$.
Si maintenant $a\in {\discrim_G}-\discrim_G^{\rm sing}$, on peut trouver
$s\in\frakt(\bar k)$ d'image $a$ tel que $s$ annul\'e par une unique racine
$\alpha$. Soient $T_\alpha$ le noyau de $\alpha:T\rightarrow\GG_m$ et
$H_\alpha$ le centralisateur de $T_\alpha$. Soit $n$ un \'el\'ement
nilpotent r\'egulier de l'alg\`ebre de Lie $\frakh_\alpha$ de $H_\alpha$ qui
s'identifie \`a une sous-alg\`ebre de Lie de $\frakg$. L'\'el\'ement $x=s+n$
est un \'el\'ement r\'egulier de $\frakg$ d'image $a$ dans $\frakc$. Il est
aussi un \'el\'ement r\'egulier de $\frakh_\alpha$ d'image $a_{H_\alpha}$
dans l'espace des polyn\^omes caract\'eristiques $\frakc_{H_\alpha}$ de
$H_\alpha$. On v\'erifie que le morphisme $\frakc_{H_\alpha}\rightarrow
\frakc$ envoie $a_{H_\alpha}$ en $a$ et est \'etale en ce point. On peut
aussi v\'erifier que dans un voisinage de $a_{H_\alpha}$, les images
r\'eciproques de $J, J'$ et $J^1$ \`a $\frakc_{H_\alpha}$ co\"incide avec les
m\^emes groupes mais associ\'es \`a $H_\alpha$. On peut ainsi ramener le
probl\`eme au cas particulier des groupes de rang semi-simple un.

Un groupe semi-simple de rang un est isomorphe \`a un produit de $\SL_2$,
$\PGL_2$ ou $\GL_2$ avec un tore. Il suffit donc de se restreindre \`a ces
trois groupes. Par un calcul direct, on v\'erifie que le centralisateur $J$ a
une fibre non-connexe dans le cas $\SL_2$ alors que dans les deux autres
cas, ses fibres sont toutes connexes. Il en est de m\^eme pour $J'$.
\end{proof}

\begin{corollaire}\label{isom sur composantes neutres}
L'homomorphisme $J\rightarrow J^1$ de \cf \ref{J=J'}, induit un isomorphisme sur leurs
sous-sch\'emas ouverts des composantes neutres des fibres.
\end{corollaire}

\subsection{Le cas des groupes endoscopiques}
Consid\'erons  une donn\'ee endoscopique $(\kappa,\rho_\kappa)$ de $G$ et
le groupe endoscopique associ\'e $H$ \cf \ref{donnee endoscopique}. On a
d\'efini un morphisme fini plat $\nu:\frakc_H\rightarrow \frakc$ entre les
$X$-sch\'emas des polyn\^omes caract\'eristiques de $H$ et de $G$ \cf
\ref{subsection : Transfert des classes}. Le centralisateur r\'egulier $J$ sur
$\frakc$ et  le centralisateur $J_H$ r\'egulier de $H$ sur $\frakc_H$ sont
reli\'es de la fa{\c c}on suivante.

\begin{proposition}\label{J JH}
Il existe un homomorphisme canonique
$$\mu:\nu^* J \longrightarrow J_H$$
qui est un isomorphisme au-dessus de l'ouvert
$\frakc_H^{G-\rs}=\nu^{-1}(\frakc^\rs)$.
\end{proposition}

\begin{proof}
Reprenons les notations de \ref{subsection : Transfert des classes}. En
particulier, on a un rev\^etement fini plat $X_{\rho_\kappa}\times
\bbt\rightarrow \frakc$ qui permet de r\'ealiser $\frakc$ comme le quotient
invariant de $X_{\rho_\kappa}\times\bbt$ par $\bbW\rtimes \pi_0(\kappa)$.
De m\^eme, $\frakc_H$ est le quotient invariant de
$X_{\rho_\kappa}\times\bbt$ par $\bbW_\HH\rtimes \pi_0(\kappa)$.
Rappelons aussi que le morphisme $\nu:\frakc_H\rightarrow \frakc$ a \'et\'e
construit dans \ref{subsection : Transfert des classes} en produisant un
homomorphisme
$$\bbW_\HH\rtimes \pi_0(\kappa) \rightarrow \bbW\rtimes \pi_0(\kappa)$$
compatible avec les actions de ces deux groupes sur $X_{\rho_\kappa}\times
\bbt$.

Consid\'erons le sch\'ema en groupes
$$J^1=\prod_{X_\rho\times\bbt/\frakc}
(X_\rho\times\bbt\times \bbT)^{\bbW\rtimes \pi_0(\kappa)}
$$
et son analogue pour $H$
$$J_H^1=\prod_{X_\rho\times\bbt/\frakc_H}
(X_\rho\times\bbt\times \bbT)^{\bbW_\HH\rtimes \pi_0(\kappa)}
$$
Le morphisme \'evident
$$X_{\rho_\kappa}\times\bbt \rightarrow (X_{\rho_\kappa}\times\bbt)\times_{\frakc} \frakc_H$$
qui est $\bbW_\HH\rtimes \pi_0(\kappa)$-\'equivariant, induit un
homomorphisme $\nu^* J^1\rightarrow J_H^1$. On a v\'erifi\'e que cet
homomorphisme est un isomorphisme au-dessus de l'ouvert
$\frakc_H^{G-\rs}$ dans \ref{J a J a_H}.

D'apr\`es \ref{isom sur composantes neutres}, on a un homomorphisme
$J\rightarrow J^1$ qui induit un isomorphisme sur leur sous-sch\'ema
ouvert des composantes neutres des fibres. Pour v\'erifier que
l'homomorphisme $\nu^* J^1\rightarrow J_H^1$ se restreint en un
homomorphisme $\nu^* J \rightarrow J_H$, il suffit de d\'emontrer que
pour tout $a_H\in \frakc_H(\bar k)$ d'image $a\in \frakc(\bar k)$,
l'homomorphisme
$$\pi_0(J^1_a)\rightarrow \pi_0(J^1_{H,a_H})$$
qui se d\'eduit de $J^1_a \rightarrow J^1_{H,a_H}$, envoie le sous-groupe
$\pi_0(J_a)\subset \pi_0(J^1_a)$ dans le sous-groupe
$\pi_0(J_{H,a_H})\subset \pi_0(J^1_{H,a_H})$. Puisque l'ensemble des
racines de $H$ est un sous-ensemble de l'ensemble des racines pour $G$, les conditions
qui d\'elimitent le sous-groupe $\pi_{0}(J_{H,a_H})$ dans le groupe
$\pi_0(J^1_{H,a_H})$ sont satisfaites par les \'el\'ements de
$\pi_0(J_a)$ \cf \ref{J'}. La proposition suit.
\end{proof}


\section{Fibres de Springer affines}
\label{section : Fibres de Springer affines}

Par analogie avec les fibres de Springer dans la r\'esolution simultan\'ee de
Grothendieck-Springer, Kazhdan et Lusztig ont introduit les fibres de
Springer affines et ont \'etudi\'e leur propri\'et\'e g\'eom\'etrique. Goresky,
Kottwitz et MacPherson ont r\'ealis\'e le lien entre les fibres de Springer
affine et les int\'egrales orbitales stables via le comptage de points de certain quotient des fibres de Springer affines. Dans ce chapitre, nous allons passer en revue les propri\'et\'e g\'eom\'etriques des
fibres de Springer affine suivant Kazhdan et Lusztig en donnant quelques
compl\'ements. Le comptage de points sera revu dans le chapitre
\ref{section : comptage}.

Voici les notations qui seront utilis\'ees dans ce chapitre. Soient $k$ un
corps fini \`a $q$ \'el\'ements et $\bar k$ une cl\^oture s\'eparable de $k$.
Soient $F_v$ un corps local d'\'egales caract\'eristiques, $\calO_v$ son
anneau des entiers dont le corps r\'esiduel $k_v$ est une extension finie de
$k$.  On notera $X_v=\Spec(\calO_v)$ le disque formel associ\'e et
$X_v^\bullet$ le disque formel point\'e. Soient $v$ le point ferm\'e de $X_v$
et $\eta_v$ son point g\'en\'erique.

Soient $\bar\calO_v=\calO_v\hat\otimes_k \bar k$ et $\bar
X_v=\Spec(\bar\calO_v)$. L'ensemble des composantes connexes de $\bar
X_v$ est en bijection l'ensemble des plongements de l'extension $k_v$ de
$k$ dans $\bar k$
$$\bar X_v=\bigsqcup_{\bar v:k_v\rightarrow \bar k} \bar X_{\bar v}.$$

En choisissant un point g\'eom\'etrique $\bar\eta_v$ dans la fibre
g\'eom\'etrique de $\bar X_{\bar v}$, on obtient la suite exacte habituelle
$$1\rightarrow I_v \rightarrow \Gamma_v \rightarrow \Gal(\bar k/\bar k_v)
\rightarrow 1
$$
o\`u $\Gamma_v=\pi_1(\eta_v,\bar\eta_v)$ est le groupe de Galois de
$F_v$ et $I_v=\pi_1(\bar X_{\bar v},\bar \eta_v)$ son sous-groupe d'inertie.

Soit $G$ une forme quasi-d\'eploy\'ee de $\GG$ sur $X_v$ associ\'ee \`a un
$\Out(\GG)$-torseur $\rho_G$. En choisissant un point g\'eom\'etrique de
$\bar\eta_{\Out,v}$ de $\rho_G$ au-dessus de $\bar\eta_v$, on obtient un
homomorphisme $\rho_G^\bullet:\Gamma_v\rightarrow \Out(\GG)$ qui se
factorise \`a travers $\Gal(\bar k/ k_v)$. Au-dessus de $\bar X_v$, $\rho_G$
est le torseur trivial. Le point de $\bar\eta_{\Out,v}$ fournit une
trivialisation de $\rho_G$ au-dessus de $\bar X_{\bar v}$.

\subsection{Rappels sur la grassmannienne affine}
\label{subsection : grassmannienne affine}

Pour tout $k$-sch\'ema $S$, notons $X_v\hat\times S$ la compl\'etion
$v$-adique de $X_v\times S$ et $X^\bullet_v\hat\times S$ l'ouvert compl\'ementaire
de $\{v\}\times S$ dans $X_v\hat\times S$. La grassmannienne affine
est le foncteur $\calQ_v$ qui associe \`a tout $k$-sch\'ema le groupo\"ide des
$G$-torseurs $E_v$ sur $X_v\hat\times S$ munis d'une trivialisations sur
$X_v^\bullet\hat\times  S$. Notons qu'un automorphisme de $E_v$, trivial
sur $X_v^\bullet\hat\times S$ est n\'ecessairement trivial de sorte que ce
groupo\"ide est une cat\'egorie discr\`ete. On pourra aussi bien le remplacer
par l'ensemble des classes d'isomorphisme.

D'apr\`es \cite{BL} et \cite{Fa-loop}, on sait que $\calQ_v$ est strictement
repr\'esentable par un ind-sch\'ema sur $k$. Plus pr\'ecis\'ement, il existe
syst\`eme injectif de $k$-sch\'emas projectifs dont les fl\`eches sont
immersions ferm\'ees et dont la limite inductive repr\'esente le foncteur
$\calQ_v$. Comme $G$ est un sch\'ema en groupes de fibres connexes,
l'ensemble des $k$-points de $\calQ_v$ s'exprime comme un
quotient
$$\calQ_v(k)=G(F_v)/G(\calO_v).$$
Quand $G=\GL_r$, cet ensemble s'identifie naturellement \`a l'ensemble
des $\calO_v$-r\'eseaux dans le $F_v$-espace vectoriel $F_v^{\oplus r}$.

La m\^eme d\'efinition vaut quand on remplace $k$ par $\bar k$, $X_v$ par
$\bar X_{\bar v}$ pour tout plongement $\bar v:k_v\rightarrow \bar k$. On
a alors la grassmannienne affine $\calG_{\bar v}$ d\'efinie sur $\bar k$. On a
la formule
$$\calG\otimes_k \bar k=\prod_{\bar v:k_v\rightarrow \bar k} \calG_{\bar v}.$$

\subsection{Fibres de Springer affines}\label{Springer}
Nous gardons les notations de \ref{Torsion exterieure}. En particulier, on a
un sch\'ema $\frakc$ au-dessus de $X_{v}$ obtenu par torsion ext\'erieure de
l'espace des polyn\^omes caract\'eristiques $\bbc$ de $\bbg$
comme dans le th\'eor\`eme de Chevalley \ref{Chevalley}.

Notons
$$\frakc^\heartsuit(\calO_v)=\frakc(\calO_v)\cap
\frakc^{\rs}(F_v)
$$
l'ensemble des $\calO_v$-points de $\frakc$ dont la fibre g\'en\'erique est
r\'eguli\`ere semi-simple. Rappelons qu'on a la section de Kostant
$\epsilon:\frakc\rightarrow \frakg^\reg$  \cf \ref{section de Kostant
tordue}. Pour tout $a\in \frakc^\heartsuit(\calO_v)$, on a un point
$$[\epsilon](a)\in [\frakg^\reg/G]$$
qui consiste en le $G$-torseur triviale $E_0$ sur $X_v$ et une section
$\gamma_0\in \Gamma(X_v,\ad(E_0))$ ayant $a$ comme polyn\^ome
caract\'eristique.

Pour chaque $a\in\frakc^\heartsuit(\calO_v)$, on d\'efinit la fibre de
Springer affine $\calM_v(a)$ comme suit. Le foncteur $\calM_v(a)$ associe
\`a tout $k$-sch\'ema $S$ le groupo\"ide des couples $(E,\phi)$ form\'e d'un
$G$-torseur $E$ sur $X_v\hat\times S$ et d'une section $\phi$ de $\ad(E)$,
munis d'un isomorphisme avec $(E_0,\gamma_0)$ sur
$X_v^\bullet\hat\times S$. En particulier
$$[\chi](E,\phi)=[\chi](E_0,\gamma_0)=a.$$

Notons que par construction, le $G$-torseur $E_0$ est le torseur trivial si
bien que $E$ d\'etermine un point de la Grassmannienne affine.

\begin{proposition}\label{localement de type fini}
Le foncteur d'oubli $(E,\phi)\mapsto E$ d\'efinit un morphisme de la fibre
de Springer affine $\calM_v(a)$ dans la grassmannienne affine $\calQ_v$ qui
est une immersion ferm\'ee. En particulier, $\calM_v(a)$ est strictement
repr\'esentable par un ind-sch\'ema. De plus, le r\'eduit $\calM_v^{\rm
red}(a)$de $\calM_v(a)$ est repr\'esentable par un sch\'ema localement de
type fini.
\end{proposition}

\begin{proof}
La premi\`ere assertion est imm\'ediate. La seconde assertion a \'et\'e
d\'emontr\'ee par Kazhdan et Lusztig \cite{KL}.
\end{proof}

Consid\'erons l'ensemble des $k$-points de $\calM_v(a)$. Soit $(E,\phi)$ un
objet de $\calM_v(a,k)$. Les fibres g\'en\'eriques de $E$ et de $E_0$ \'etant
identifi\'ees, la donn\'ee de $E$ consiste en une classe $g\in
G(F_v)/G(\calO_v)$. Puisque $\phi$ est identifi\'ee \`a $\gamma_0$ sur la
fibre g\'en\'erique, pour que $\phi$ d\'efinisse une section sur $X_v$ de
$\ad(E)\otimes D$, il faut et il suffit que $\ad(g)^{-1}\gamma_0 \in
\frakg(\calO_v)$. On a donc
$$
\calM_v(a,k)=\{g\in G(F_v)/G(\calO_v) \mid \ad(g)^{-1}\gamma_0 \in
\frakg(\calO_v)\}.
$$

Nous allons maintenant consid\'erer une variation triviale de la construction
pr\'ec\'edente en pr\'esence d'un faisceau inversible $D'$ sur $X_v$. Cette
variation sera n\'ecessaire pour r\'ealiser le passage entre les fibres de
Springer affines et les fibres de Hitchin.

Soient $D={D'}^{\otimes 2}$ et $h_D:X_v\rta \bfB\GG_m$ le morphisme dans
le classifiant de $\GG_m$ correspondant au fibr\'e en droites $D$. Fixons un
morphisme $h_a:X_v\rta [\frakc/\GG_m]$ qui s'ins\`ere dans le diagramme
commutatif :
$$\xymatrix{
  X_v \ar[r]^{\!\!\!\!\!h_a} \ar[dr]_{h_D}
                & [\frakc/\GG_m] \ar[d]^{}  \\
                & \bfB \GG_m             }$$
Notons $[\epsilon]^{D'}(a)$ le point de Kostant de $a$ construit comme
dans le lemme \ref{epsilon D'}. On notera $a^\bullet$ et $h_a^\bullet$ les
restrictions de $a$ et de $h_a$ \`a $X_v^\bullet$ et on notera le point de
Kostant associ\'e $[\epsilon]^{D'}(a^\bullet)$.

\begin{definition}\label{fibre de Springer}
On d\'efinit la fibre de Springer $\calM_v(a)$ comme le foncteur qui associe
\`a tout $k$-sch\'ema $S$ l'ensemble $\calM_v(a,S)$ des classes
d'isomorphisme des morphismes $h_{E,\phi}:X_v\hat\times S\rta
[\frakg/G\times \GG_m]$ s'ins\'erant dans le diagramme commutatif
$$
\xymatrix{
  X_v\hat \times S \ar[r]^{h_{E,\phi}\,\,\,\,\,\,\,\,\,\,} \ar[dr]_{h_a}
                & [\frakg/G\times\GG_m] \ar[d]^{[\chi]}  \\
                & [\frakc/\GG_m]             }
$$
munis d'un isomorphisme entre la restriction de $h_{E,\phi}$ \`a
$X_v^\bullet\hat\times S$, et le point de Kostant $[\epsilon]^{D'}
(a^\bullet)$.
\end{definition}

La m\^eme d\'efinition vaut quand on remplace $k$ par $\bar k$, $X_v$ par
$\bar X_{\bar v}$ pour tout plongement $\bar v:k_v\rightarrow \bar k$.
Pour tout $a\in \frakc(\bar\calO_{\bar v})\cap \frakc^\rs(\bar F_{\bar v})$,
on a alors la fibre de Springer affine $\calM_{\bar v}(a)$ d\'efinie sur $\bar
k$. Pour tout $a\in \frakc(\calO_{v})\cap \frakc^\rs(F_{v})$, on a la formule
$$\calM_v(a)\otimes_k \bar k=\prod_{\bar v:k_v\rightarrow \bar k} \calM_{\bar v}(a).$$

\subsection{Sym\'etries d'une fibre de Springer affine}
\label{symetries d'une fibre de Springer}

Le centralisateur r\'egulier permet de  d\'efinir le groupe des sym\'etries
d'une fibre de Springer affine. Soient $a\in\frakc^\heartsuit(\calO_v)$ et
$h_a:X_v\rightarrow \frakc$ le morphisme correspondant. Soit $J_a=h_a^*
J$ l'image r\'eciproque du centralisateur r\'egulier.

Consid\'erons le groupo\"ide de Picard $\calP_v(J_a)$ fibr\'e au-dessus de
$\Spec(k)$ qui associe \`a tout $k$-sch\'ema $S$ le groupo\"ide de Picard
$\calP_v(J_a,S)$ des $J_a$-torseurs sur $X_v\hat\times S$ munis d'une
trivialisation sur $X_v^\bullet\hat\times S$. On peut v\'erifier que pour tout
$k$-sch\'ema $S$, $\calP_v(J_a,S)$ est une cat\'e\-go\-rie de Picard
discr\`ete. De plus, le foncteur qui associe \`a $S$ le groupe des classes
d'isomorphisme de $\calP_v(J_a,S)$ est repr\'esentable par un ind-sch\'ema
en groupes sur $k$ qu'on notera $\calP_v(J_a)$. Le groupe des
$\bar k$-points $\calP_v(J_a,k)$ s'identifie canoniquement au quotient
$J_a(\bar F_v)/J_a(\bar \calO_v)$. Si les fibres de $J_{a}$ sont connexes,
l'ensemble des $k$-points de $\calP_{v}(J_{a})$ s'identifie au quotient
$J_{a}(F_{v})/J_{a}(\calO_{v})$.

Le lemme \ref{J} permet de d\'efinir une action de $\calP_v$ sur $\calM_v$.
En effet, pour tout $(E,\phi)\in \calM_a(S)$, on a un homomorphisme de
faisceaux en groupes au-dessus de $X_v\hat\times S$
$$J_a \longrightarrow \underline{\rm Aut}(E,\phi)$$
qui se d\'eduit de \ref{J}. Ceci permet de tordre $(E,\phi)$ par un
$J_a$-torseur sur $X_v\hat\times S$ qui est trivialis\'e sur
$X_v^\bullet\hat\times S$.

Sur les $k$-points, on peut d\'ecrire concr\`etement cette action. Pour simplifier
l'exposition, supposons que $J_{a}$ a des fibres connexes. L'action du
groupe $\calP_v(J_a,k)=J_a(F_v)/J_a(\calO_v)$ sur l'ensemble
$$\calM_v(a,k)=\{g\in G(F_v)/G(\calO_v) \mid \ad(g)^{-1}\gamma_0 \in
\frakg(\calO_v)\}
$$
se d\'ecrit concr\`etement comme suit. D'apr\`es \ref{J}, il existe un isomorphisme
canonique de $J_a(F_v)$ sur le centralisateur $G_{\gamma_0}(F_v)$ de
$\gamma_0$ qu'on va noter $j\mapsto \theta(j)$. On fait agir $J_a(F_v)$
sur l'ensemble des $g\in G(F_v)$ tels que $\ad(g)^{-1}(\gamma_0)\in
\frakg(\calO_v)$ par $j.g=\theta(j)g$. Pour que ceci induise une action de
$J_a(F_v)/J_a(\calO_v)$ sur $\calM_v(a,k)$, il faut et il suffit que pour tout
$g\in\calM_v(a,k)$, on a l'inclusion
$$\theta(J_a(\calO_v)) \subset \ad(g) G(\calO_v).$$
Soit $\gamma=\ad(g)^{-1} \gamma_0 \in \frakg(\calO_v)$. D'apr\`es \ref{J},
l'isomorphisme $\ad(g)^{-1}\circ\theta : I_a \rta G_{\gamma_0}$ se prolonge
en un homomorphisme de $X_v$-sch\'emas en groupes sur
$$\ad(g)^{-1}\circ\theta: J_a \longrightarrow I_\gamma$$
ce qui implique en particulier que
$$\theta(J_a(\calO_v)) \subset \ad(g)(I_{\gamma_0}(\calO_v))
\subset \ad(g) G(\calO_v).$$

Nous consid\'erons le sous-foncteur $\calM_v^\reg(a)$ de la fibre de
Springer affine $\calM_v(a)$ dont les points sont les morphismes $h_{E,\phi}:
X_v \rightarrow [\frakg/G\times \GG_m]$ qui se factorisent par
l'ouvert $[\frakg^\reg/G\times\GG_m]$. Il est
clair que $\calM_v^\reg(a)$ est un ouvert de $\calM_v(a)$.

\begin{lemme}
L'ouvert $\calM_v^\reg(a)$ est un espace principal homog\`ene sous l'action
de $\calP_v(J_a)$.
\end{lemme}

\begin{proof}
C'est une cons\'equence du lemme \ref{gerbe}.
\end{proof}

Soit $\bar v:k_v \rightarrow \bar k$. Pour tout $a\in \frakc(\bar\calO_{\bar
v})\cap \frakc^\rs(\bar F_{\bar v})$, on a le groupe des symm\'etries
$\calP_{\bar v}(J_a)$ d\'efinie sur $\bar k$ de la fibre affine $\calM_{\bar
v}(a)$. Si $a\in \frakc(\calO_{v})\cap \frakc^\rs(F_{v})$, on a la formule
$$\calP_v(J_a)\otimes_k \bar k=\prod_{\bar v:k_v
\rightarrow \bar k} \calP_{\bar v}(J_a).
$$
Dans la suite du chapitre, on va passer \`a $\bar k$ et discuter des
propri\'et\'es g\'eom\'etriques de $\calP_{\bar v}(J_a)$. Rappelons
qu'au-dessus de $\bar X_{\bar v}$, $G$ est un groupe d\'eploy\'e.

\subsection{Quotient projectif d'une fibre de Springer affine}

Soit $a\in\frakc(\bar\calO_{\bar v})$ dont la fibre g\'en\'erique est dans
$\frakc^\rs(\bar F_{\bar v})$. D'apr\`es Kazhdan et Lusztig, le foncteur
$\calM_{\bar v}(a)$ a un sch\'ema r\'eduit sous-jacent $\calM_{\bar v}^{\rm
red}(a)$ qui est localement de type fini \cf \ref{localement de type fini} et
\cite{KL}. Ils ont aussi d\'emontr\'e que $\calM_{\bar v}^{\rm red}(a)$,
quotient\'e par un groupe discret convenable, est un sch\'ema projectif.
Rappelons leur \'enonc\'e de fa{\c c}on plus pr\'ecise.

Soit $\Lambda$ le quotient libre maximal de $\pi_0(\calP_{\bar v}(J_a))$.
Choisissons un rel\`eve\-ment arbitraire
$$\Lambda\rightarrow \calP_{\bar v}(J_a)$$
qui induit une action de $\Lambda$ sur $\calM_{\bar v}(a)$ et $\calM_{\bar
v}^{\rm red}(a)$.

\begin{proposition}\label{quotient projectif}
Le groupe discret $\Lambda$ agit librement sur $\calM_{\bar v}^{\rm
red}(a)$ et le quotient $\calM_{\bar v}^{\rm red}(a)/ \Lambda$ est un $\bar
k$-sch\'ema projectif.
\end{proposition}

\begin{proof}
La proposition 1 de \cite[page 138]{KL} montre que $\Lambda$ agit
librement sur $\calM_{\bar v}^{\rm red}(a)$ et qu'en tronquant la fibre de
Springer affine r\'eduite $\calM_{\bar v}^{\rm red}(a)$, on obtient des
sch\'emas projectifs qui se surjectent sur le quotient $\calM_{\bar v}^{\rm
red}(a)/ \Lambda$. Il s'ensuit que le quotient est \'egalement un sch\'ema
projectif.
\end{proof}

\subsection{Approximation}
\label{subsection : approximation}

D'apr\`es un th\'eor\`eme bien connu de Harish-Chandra, les int\'egrales
orbitales semi-simples r\'eguli\`eres sont localement constantes. Dans ce
paragraphe, nous allons montrer une variante g\'eom\'e\-trique, plus forte,
de ce th\'eor\`eme.

\begin{proposition}\label{HC}
Soient $a\in\frakc(\calO_v)\cap \frakc^\rs(F_v)$, $\calM_a$  la fibre de Springer
affine et $\calP_v(J_a)$ le groupe des sym\'etries de $\calM_v(a)$. Il
existe alors un entier $N$ tel que pour toute extension finie $k'$ de $k$,
pour tout $a'\in\frakc(\calO_v\otimes_k k')$ tel que
$$a\equiv a' \mod \epsilon_v^N,$$
la fibre de Springer affine $\calM_{v}(a')$ munie de l'action de
$\calP_v(J_{a'})$ est isomorphe \`a $\calM_v(a)\otimes_k k'$ munie de
l'action de $\calP_v(J_a)\otimes_k k'$.
\end{proposition}

\begin{proof}
Consid\'erons le rev\^etement cam\'eral $\widetilde X_{a,v}$ de $X_v$
associ\'e \`a $a$ construit en formant le diagramme cart\'esien
$$
\xymatrix{
 \widetilde X_{a,v} \ar[d]_{\pi_a} \ar[r]^{}   & \ar[d]^{} \frakt  \\
 X_v \ar[r]_{a} & \frakc          }
$$
Le morphisme $\pi_a$ est fini plat et g\'en\'eriquement \'etale. Le
rev\^etement $\widetilde X_{a,v}$ est muni d'une action de $W$ qui se
d\'eduit de l'action de $W$ sur $\frakt$. On associe aussi \`a $a'\in
\frakc(\calO_v\otimes_k k')$ son rev\^etement cam\'eral $\widetilde
X_{a',v}\rightarrow X_v\otimes_k k'$.

\begin{lemme}
Soit $a\in\frakc(\calO_v)\cap \frakc^\rs(F_v)$. Il existe un entier $N$ tel
que pour toute extension finie $k'$ de $k$, pour tout
$a'\in\frakc(\calO_v\otimes_k k')$ tel que
$$a\equiv a' \mod \epsilon_v^N,$$
le rev\^etement $\widetilde X_{a',v}$ de $X_v\otimes_k k'$ muni de
l'action de $W$ est isomorphe au rev\^etement $\widetilde
X_{a,v}\otimes_k k'$ muni de l'action de $W$. On peut demander en plus que
l'isomorphisme rel\`eve l'isomorphisme \'evident module $\epsilon_{v}^{N}$.
\end{lemme}

\begin{proof}
C'est un cas particulier d'un lemme d'Artin-Hironaka \cite[lemme 3.12]{Artin}.
La seconde assertion est implicite dans la d\'emonstration d'Artin.
\end{proof}

Supposons que les rev\^etement cam\'eraux $\widetilde X_{a,v}$ et
$\widetilde X_{a',v}$ de $X_v\otimes_k k'$ sont isomorphes, d\'emontrons
que  $\calM_v(a)\otimes_k k'$ munie de l'action
de $\calP_v(J_a)\otimes_k k'$ et $\calM_v(a')$ munie de l'action de $\calP_v(J_{a'})$
sont isomorphes. En rempla{\c c}ant $k$ par $k'$, on peut d\'esormais supposer
$k=k'$ .

Pour cela, il nous faut une description de la fibre de Springer affine en
termes du centralisateur r\'egulier. Soit
$\gamma_0=\epsilon(a):X_v\rightarrow \frakg$ la section de Kostant de $a$.
On a alors un isomorphisme canonique entre $J_a$ et le $X_v$-sch\'ema en
groupes $I_{\gamma_0}=\gamma_0^* I$ o\`u $I$ est le sch\'ema en groupes
des centralisateurs au-dessus de $\frakg$. L'homomorphisme \'evident
$I_{\gamma_0}\rightarrow G$ induit un homomorphisme injectif de fibr\'es
vectoriels sur $X_v$
$${\rm Lie}(I_{\gamma_0})\rightarrow \frakg.$$
L'\'enonc\'e suivant montre que la fibre de Springer affine $\calM_v(a)$ ne d\'epend
pas de $\gamma_0$ mais seulement de la sous-alg\`ebre de Lie
commutative $\Lie(I_{\gamma_0})$ de $\frakg$.

\begin{lemme}
Soit $g\in G(F_v)$. Alors on a $\ad(g)^{-1}(\gamma_0)\in \frakg(\calO_v)$ si
et seulement si $\ad(g)^{-1} {\rm Lie}(I_{\gamma_0}) \subset
\frakg(\calO_v)$.
\end{lemme}

\begin{proof} Comme $\gamma_0\in \Lie(I_{\gamma_0})$,
$\ad(g)^{-1} {\rm Lie}(I_{\gamma_0}) \subset \frakg(\calO_v)$ implique
imm\'ediatement $\ad(g)^{-1}(\gamma_0)\in \frakg(\calO_v)$. Inversement,
soit $g\in G(F_v)$ tel que $\gamma=\ad(g)^{-1}(\gamma_0) \in
\frakg(\calO_v)$. Soit $I_\gamma=\gamma^* I$. Le fait \cf \ref{J} que l'isomorphisme
en fibre g\'en\'erique
$$\ad(g)^{-1}:J_{a,F_v}=I_{\gamma_0,F_v} \rightarrow I_{\gamma,F_v}$$
se prolonge en un homomorphisme $J_a \rightarrow I_\gamma$ implique en
particulier que $\ad(g)^{-1} {\rm Lie}(I_{\gamma_0}) \subset
\frakg(\calO_v)$.
\end{proof}

Pour terminer la d\'emonstration de la proposition \ref{HC}, il suffit de
d\'emontrer le lemme suivant.
\end{proof}

\begin{lemme}
Soient $a,a'\in \frakc(\calO_v)\cap \frakc_{\rs}(F_v)$ tels que $a\equiv a'
\mod \epsilon_v$. Supposons qu'il existe un isomorphisme entre les
rev\^etements cam\'eraux $\widetilde X_{a,v}$ et $\widetilde X_{a',v}$
munis de l'action de $W$ qui rel\`eve l'isomorphisme \'evident dans la fibre
sp\'eciale. Soient $\gamma_0=\epsilon(a)$ la section de Kostant de $a$ et
$\gamma'_0=\epsilon(a')$ la section de Kostant de $a'$. Soient
$I_{\gamma_0}$ et $I_{\gamma_0'}$ les sous-sch\'emas en groupes de $G$
centralisateurs des sections $\gamma_0$ et $\gamma_0'$. Alors il existe
$g\in G(\calO_v)$ tel que
$$\ad(g)^{-1}I_{\gamma_0}= I_{\gamma_0'}.$$
\end{lemme}

\begin{proof}
Ce lemme est une cons\'equence d'un r\'esultat de Donagi et Gaitsgory
\cite[th\'eor\`eme 11.8]{DG}. Pour la commodit\'e du lecteur, nous allons en
extraire la portion utile \`a notre propos. Puisque les
sch\'emas en groupes $I_{\gamma_0}$ et $I_{\gamma'_0}$ sont
compl\`etement d\'etermin\'es par les rev\^etements cam\'eraux
$\widetilde X_{a,v}$ respectivement $\widetilde X_{a',v}$ \cf \ref{J=J'},
l'isomorphisme entre $\widetilde X_{a,v}$ et $\widetilde X_{a',v}$ induit
un isomorphisme $\iota: I_{\gamma_0} \rightarrow I_{\gamma'_0}$. Cet
isomorphisme transporte $\gamma_0\in \Lie(I_{\gamma_0})$ en un
\'el\'ement $\iota(\gamma_0)\in \Lie(I_{\gamma'_0})$. Puisque les fibres
sp\'eciales de $\iota(\gamma_0)$ et $\gamma_0'$ co\"incident,
$\iota(\gamma_0):X_v \rightarrow \frakg$ se factorise par l'ouvert
$\frakg^\reg$. On en d\'eduit l'\'egalit\'e
$I_{\iota(\gamma_0)}=I_{\gamma'_0}$ de sous-sch\'emas en groupes de
$G$.

On dispose de deux sections
$$\gamma_0, \iota(\gamma_0): X_v \rightarrow \frakg^\reg$$
qui ont le m\^eme polyn\^ome caract\'eristique $a$ et qui sont \'egales
module $\epsilon_v$. Puisque le morphisme
$$G\times_X \frakg^\reg \rightarrow \frakg^\reg\times_{\frakc} \frakg^\reg$$
est un morphisme lisse, il existe $g\in G(\calO_v)$ tel que $g\equiv 1 \mod
\epsilon_v$ et tel que $\ad(g^{-1}(\gamma_0))=\iota(\gamma_0)$. On en
d\'eduit que
$$\ad(g)^{-1}(I_{\gamma_0})= I_{\iota(\gamma_0)}=I_{\gamma_0'}.$$
C'est ce qu'on voulait.
\end{proof}

\subsection{Cas lin\'eaire}\label{Cas lineaire local}
Examinons maintenant les fibres de Springer affines du groupe lin\'eaire en
suivant la pr\'esentation de Laumon \cite{L}. Nous r\'ef\'erons \`a
\cite{compagnon} pour une discussion similaire dans le cas des groupes
classiques. Soit $G=\GL(r)$.

On garde les notations fix\'ees au d\'ebut de ce chapitre. Un point
$a\in\frakc(\bar\calO_{\bar v})$ est repr\'esent\'e par un polyn\^ome
unitaire de degr\'e $r$ de variable $t$
$$P(a,t)=t^r-a_1 t^{r-1}+\cdots+(-1)^r a_r\in \bar\calO_{\bar v}[t]$$
Formons la $\bar\calO_{\bar v}$-alg\`ebre finie et plate de rang $r$
$$B=\bar\calO_{\bar v}[t]/P(a,t)$$
et notons $E=B\otimes_{\bar\calO_{\bar v}} \bar F_{\bar v}$. L'hypoth\`ese
$a\in\frakc^\heartsuit(\bar\calO_{\bar v})$ implique que $E$ est une $\bar
F_{\bar v}$-alg\`ebre finie \'etale de dimension $r$. On peut \'ecrire $E$ comme
un produit $E_1\times\cdots\times E_s$ de $s$ extensions s\'eparables de
$\bar F_{\bar v}$ avec $s\leq r$.

\begin{numero}
Dans cette situation, on a une description de la fibre de Springer en termes
des r\'eseaux. Les $\bar k$-points de la fibre de Springer $\calM_{\bar v}(a)$
est l'ensemble des $B$-r\'eseaux dans $E$ c'est-\`a-dire des $B$-modules
$\calL$ qui sont contenus dans le $\bar F_{\bar v}$-espace vectoriel $E$
comme un $\bar \calO_{\bar v}$-r\'eseau. Les $\bar k$-points de la partie
r\'eguli\`ere $\calM_{\bar v}^\reg(a)$ consistent des $B$-r\'eseaux de $E$
qui sont des $B$-modules libres. Les $\bar k$-points du groupe des
sym\'etries $\calP_{\bar v}(J_a)$ est le groupe $E^\times /B^\times$.
\end{numero}

\begin{numero}
La normalisation $B^\flat$ de $B$ est l'anneau des entiers de $E_{\bar v}$.
On a alors un d\'evissage de $E^\times/B^\times$
$$1\rightarrow (B^\flat)^\times/B^\times \rightarrow
E^\times/B^\times \rightarrow E^\times/(B^\flat)^\times\rta 1
$$
o\`u $(B^\flat)^\times/B^\times$ est le groupe des $\bar k$-points de la
composante neutre du groupe $\calP_{\bar v}(J_a)$. Par cons\'equent, le
groupe des composantes connexes $\pi_0(\calP_{\bar v}(J_a))$ est
canoniquement isomorphe \`a $E^\times/(B^\flat)^\times$ qui est un groupe
ab\'elien libre muni d'une base index\'ee par l'ensemble des composantes
connexes de $\Spec(B^\flat)$.
\end{numero}

\begin{numero}
Dans ce cas, la dimension de $\calP_{\bar v}(J_a)$  est \'egale \`a l'invariant
$\delta$ de Serre
$$\delta_{\bar v}(a)=\dim(\calP_{\bar v}(J_a))=\dim_k(B^\flat/B).$$
Par ailleurs, cet entier peut \^etre exprim\'e en fonction du discriminant.
Soit $d_{\bar v}(a)={\rm val}_{\bar v}(\discrim(a))$ la valuation $\bar
v$-adique du discriminant de $a$. D'apr\`es \cite[III.3 proposition 5 et III.6
corollaire 1]{Se}, on a la formule
$$\delta_{\bar v}(a)=(d_{\bar v}(a)-c_{\bar v}(a))/ 2$$
o\`u $c_{\bar v}(a)=r-s$.
\end{numero}

\subsection{Dimension}
\label{dimension}

Dans \cite{KL}, Kazhdan et Lusztig ont montr\'e que
$$\dim(\calM_{\bar v}(a))=\dim(\calM_{\bar v}^\reg(a)).$$
On peut en fait d\'eduire un \'enonc\'e plus pr\'ecis \`a partir de leurs
r\'esultats.

\begin{proposition}\label{dimension du complement}
Le compl\'ementaire de l'ouvert r\'egulier $\calM_{\bar v}^\reg(a)$ de la
fibre de Springer $\calM_{\bar v}(a)$ est de dimension strictement plus
petite que celle de $\calM_{\bar v}(a)$.
\end{proposition}

\begin{proof}
Pour discuter de la dimension, on peut n\'egliger les nilpotents dans les
anneaux structuraux de $\calM_{\bar v}(a)$. Comme rappel\'e dans le
paragraphe \ref{Springer}, $\calM_{\bar v}(a)$ est la sous-vari\'et\'e de la
Grassmannienne affine des $g\in G(F_{\bar v})/G(\calO_{\bar v})$ tel que
$\ad(g)^{-1}(\gamma_0)\in\frakg(\calO_{\bar v})$.

A la suite de Kazhdan et Lusztig, consid\'erons aussi la sous-vari\'et\'e des
points fixes $\calB_{\bar v}(a)$ des drapeaux affines fixes sous $\gamma_0$
c'est-\`a-dire $g\in G(F_{\bar v})/{\rm Iw}_{\bar v}$ tel que
$\ad(g)^{-1}\gamma_0\in \Lie({\rm Iw}_{\bar v})$. Ici ${\rm Iw}_{\bar v}$
est un sous-groupe d'Iwahori de $G(\calO_{\bar v})$. Dans \cite{KL},
Kazhdan et Lusztig ont d\'emontr\'e que les fibres de Springer affines pour
les sous-groupes d'Iwahori sont \'equidimen\-sionnelles.
L'\'equidimen\-sionalit\'e des fibres de Springer classiques a \'et\'e
d\'emontr\'ee auparavant par Spaltenstein \cite{Sp}.

Le morphisme $\calB_{\bar v}(a)\rightarrow \calM_{\bar v}(a)$ est un
morphisme fini au-dessus de $\calM_{\bar v}^\reg(a)$. Les fibres au-dessus
des point $x\in\calM_{\bar v}-\calM_{\bar v}^\reg$ sont non vides et de dimension sup\'erieure ou \'egale \`a un. L'\'equidimensionalit\'e de $\calB_{\bar v}(a)$
implique donc que les composantes irr\'eductibles de $\calM_a-\calM_a^\reg$
sont de dimension strictement plus petite que $\dim(\calM_a^\reg)$.
\end{proof}

\begin{corollaire}\label{dim zero}
Si $\dim(\calM_v(a))=0$, alors l'ouvert dense $\calM_v^\reg(a)$ est
$\calM_v(a)$ tout entier.
\end{corollaire}

Kazhdan et Lusztig ont aussi conjectur\'e une formule qui exprime la
dimension ci-dessus en fonction du discriminant et d'un terme d\'efectif
reli\'e \`a la monodromie qui a \'et\'e d\'emontr\'ee plus tard par
Bezrukavnikov dans \cite{Bz}. Rappelons cette formule.

Gardons les notations fix\'ees au d\'ebut de ce chapitre. Soit $a:\bar X_{\bar
v} \rightarrow \frakc$ un morphisme dont l'image n'est pas contenue dans le
diviseur du discriminant ${\discrim_G}$. En prenant l'image r\'eciproque de
${\discrim_G}$, on obtient un diviseur de Cartier effectif de $\bar X_{\bar
v}$ support\'e par son point ferm\'e. Son degr\'e est un entier naturel que
nous allons noter
$$d_{\bar v}(a):=\deg_{\bar v}(a^*{\discrim_G}).$$

En prenant l'image r\'eciproque du rev\^etement $\pi:\frakt^\rs \rightarrow
\frakc^\rs$ par le morphisme $a:\bar X_{\bar v}^\bullet \rightarrow
\frakc^\rs$, on obtient un $W$-torseur $\pi_a$ sur $\bar X_{\bar
v}^\bullet$. Rappelons qu'au-dessus de $\bar X_{\bar v}$ la forme
quasi-d\'eploy\'ee se d\'eploie canoniquement de sorte qu'apr\`es le
changement de base \`a $\bar X_{\bar v}$, on a $W=\bbW$. En choisissant
un point g\'eom\'etrique de ce torseur au-dessus du point g\'eom\'etrique
$\bar\eta_{\bar v}$ de $\bar X_{\bar v}$, on obtient un homomorphisme
\begin{equation}\label{pi_a^bullet}
\pi_a^\bullet: I_v\rightarrow \bbW
\end{equation}
Puisque $p$ ne divise pas l'ordre de $\bbW$, $\pi_a^\bullet$ se factorise
par le quotient mod\'er\'e $I_v^{\rm tame}$ de $I_v$ de sorte que son
image est un sous-groupe cyclique. Notons
\begin{equation}\label{c_v(a)}
c_{\bar v}(a):=\dim(\bbt) - \dim(\bbt^{\pi_a^\bullet(I_v)}).
\end{equation}
L'\'enonc\'e suivant a \'et\'e d\'emontr\'e dans \cite{Bz} par Bezrukavnikov.

\begin{proposition}\label{dimension Bezru} On a les \'egalit\'es
$$\dim(\calM^\reg_{\bar v}(a))=\dim(\calP_{\bar v}(J_a))=\frac{d_{\bar v}(a)-c_{\bar v}(a)}{2}.$$
\end{proposition}

Nous allons noter $\delta_{\bar v}(a)=\dim(\calP_{\bar v}(J_a))$ et l'appeler
l'{\em invariant $\delta$ local}.

\subsection{Mod\`ele de N\'eron}
\label{subsection : Modele de Neron}

La dimension de $\calP_{\bar v}(J_a)$ peut \^etre calcul\'ee autrement \`a
l'aide du mod\`ele de N\'eron de $J_a$. En fait, le mod\`ele de N\'eron
nous permet d'analyser compl\`etement la structure de $\calP_{\bar
v}(J_a)$. D'apr\`es Bosch, Lutkebohmer et Raynaud \cite{BLR}, il existe un unique
sch\'ema en groupes lisse de type fini $J^\flat_a$ sur $\bar X_{\bar v}$ de
m\^eme fibre g\'en\'erique que $J_a$ et maximal pour cette propri\'et\'e
c'est-\`a-dire pour tout autre sch\'ema en groupes lisse de type fini $J'$ sur
$\bar X_{\bar v}$ de m\^eme fibre g\'en\'erique, il existe un
homomorphisme canonique $J'\rightarrow J^\flat_a$ qui induit l'identit\'e
sur les fibres g\'en\'eriques. En particulier, on a un homomorphisme
canonique $J_a\rightarrow J^\flat_a$. Au niveau des points entiers, cet
homomorphisme d\'efinit les inclusions
$$J_a(\bar\calO_{\bar v})\subset J^\flat_a(\bar\calO_{\bar v})\subset J_a(\bar F_{\bar v}).$$
o\`u $J^\flat_a(\bar \calO_{\bar v})$ est le sous-groupe born\'e maximal de
$J_a(F_{\bar v})$. On appellera $J^\flat_a$ le mod\`ele de N\'eron de
$J_a$. Remarquons que dans la terminologie de \cite{BLR}, $J^\flat_a$
sera appel\'e le mod\`ele de N\'eron de type fini \`a distinguer avec leur mod\`ele
de N\'eron qui est seulement localement de type fini.

En rempla{\c c}ant dans la d\'efinition \ref{symetries d'une fibre de
Springer} de $\calP_{\bar v}(J_a)$ le sch\'ema en groupes $J_a$ par son
mod\`ele de N\'eron $J^\flat_a$, on obtient un ind-sch\'ema en groupes
$\calP_{\bar v}(J^\flat_a)$ sur $\bar k$. L'homomorphisme de sch\'emas en
groupes $J_a\rta J^\flat_a$ induit un homomorphisme de ind-groupes
$$\calP_{\bar v}(J_a)\rta \calP_{\bar v}(J^\flat_a)$$
qui induit un d\'evissage de $\calP_{\bar v}(J_a)$.

\begin{lemme}\label{R_v(a)}
Le groupe $\calP_{\bar v}(J_a^\flat)$ est hom\'eomorphe \`a un groupe
ab\'e\-lien libre de type fini. L'homomorphisme $p_{\bar v}:\calP_{\bar
v}(J_a)\rta \calP_{\bar v}(J^\flat_a)$ est surjectif. Le noyau $\calR_{\bar
v}(a)$ de $p_{\bar v}$  est un sch\'ema en groupes affine de type fini sur
$\bar k$.
\end{lemme}

\begin{proof}
Puisque $J_a^\flat(\bar\calO_{\bar v})$ est le sous-groupe born\'e maximal
dans le tore $J_a(\bar F_{\bar v})$, le quotient $J_a(\bar F_{\bar
v})/J_a^\flat(\bar\calO_{\bar v})$ est un groupe ab\'elien libre de type fini.
L'homomorphisme
$$\calP_{\bar v}(J_a)(\bar k)=J_a(\bar F_{\bar v})/J_a(\bar \calO_{\bar v}) \rightarrow J_a(\bar F_{\bar v})/
J_a^\flat(\bar \calO_{\bar v})=\calP_{\bar v}(J_a^\flat)(\bar k)
$$
est manifestement surjectif.

Pour un entier $N$ assez grand, $J_a(\calO_{\bar v})$ contient le noyau de
l'homomorphisme $J_a^\flat(\bar\calO_{\bar v}) \longrightarrow
J_a^\flat(\bar\calO_{\bar v}/\varepsilon_{\bar v}^N \bar\calO_{\bar v})$. Il
s'ensuit que $\calR_{\bar v}(a)$ est un quotient de la restriction \`a la Weil
$$\prod_{\Spec(\bar\calO_{\bar v}/\varepsilon_{\bar v}^N \bar\calO_{\bar v})/\Spec(\bar k)}
J_a^\flat\otimes_{\bar\calO_{\bar v}}(\bar\calO_{\bar v}/\varepsilon_{\bar v}^N \bar\calO_{\bar v})
$$
qui est un $\bar k$-groupe alg\'ebrique affine lisse de type fini. Le lemme
s'en d\'eduit.
\end{proof}

Le mod\`ele de N\'eron peut \^etre explicitement construit \`a l'aide de la
normalisation du rev\^etement cam\'eral. Notons $\widetilde X_{a,\bar v}$
l'image r\'eciproque de $\frakt \rightarrow \frakc$ par le morphisme $a:\bar
X_{\bar v}\rightarrow \frakc$. Consid\'erons la normalisation $\widetilde
X_{a,\bar v}^\flat$ de $\widetilde X_{a,v}$ qui est un sch\'ema fini et plat
au-dessus de $\bar X_{\bar v}$ muni d'une action de $W$. Puisque le corps
r\'esiduel de $\bar X_{\bar v}$ est alg\'ebriquement clos, $\widetilde
X_{a,\bar v}^\flat$ est un sch\'ema semi-local complet r\'egulier de
dimension un.

\begin{proposition}\label{J a flat}
Soit $\widetilde X_{a,\bar v}^\flat=\Spec(\widetilde \calO_{\bar v}^\flat)$ la
normalisation de $\widetilde X_{a,\bar v}$. Alors le mod\`ele de N\'eron
$J_a^\flat$ de $J_a$ est le groupe des points fixes sous l'action diagonale
de $W$ dans la restriction des scalaires de $\widetilde X_{a,\bar v}^\flat$
\`a $\bar X_{\bar v}$ du tore $T\times_{\bar X_{\bar v}} \widetilde
X_{a,\bar v}^\flat$
$$J_a^\flat=\prod_{\widetilde
X_{a,\bar v}^\flat/\bar X_{\bar v}}(T\times_{\bar X_{\bar v}} \widetilde X_{a,\bar v}^\flat)^W.$$
\end{proposition}

\begin{proof}
Notons $\pi_a^\flat$ le morphisme $\widetilde X_{a,\bar v}^\flat\rightarrow
\bar X_{\bar v}$. Comme dans le lemme \ref{J1}, le sch\'ema des points
fixes sous l'action diagonale de $W$ sur la restriction des scalaires \`a la
Weil $\prod_{\widetilde X_{\bar v}^\flat/\bar X_{\bar v}}(T\times_{\bar
X_{\bar v}} \widetilde X_{\bar v}^\flat)$ est un sch\'ema en groupes lisse de
type fini sur $\bar X_{\bar v}$. Il sera plus commode de raisonner avec le
faisceau $(\pi_{a *}^{\flat}T)^W$ que repr\'esente la restriction \`a la Weil
ci-dessus.

D'apr\`es la description galoisienne du centralisateur r\'egulier
\ref{desciption galoisienne de J}, la restriction de $\pi_a^{\flat*} J_a$ \`a
$\widetilde X_{a,\bar v}^{\flat\bullet}=\widetilde X_{a,\bar
v}^\flat\times_{\bar X_{\bar v}} \bar X_{\bar v}^\bullet$ est canoniquement
isomorphe au tore $T\times_{\bar X_{\bar v}^\bullet} \widetilde X_{a,\bar
v}^{\flat\bullet}$. Le mod\`ele de N\'eron de $T\times_{\bar X_{\bar
v}^\bullet} \widetilde X_{a,\bar v}^{\flat\bullet}$ \'etant le tore
$T\times_{\bar X_{\bar v}} \widetilde X_{a,\bar v}^{\flat}$, on a un
homomorphisme canonique
$$
\pi_a^{\flat *} J_a \longrightarrow T\times_{\bar X_{\bar v}}
\widetilde X_{a,\bar v}^{\flat}.
$$
Par adjonction, on a un homomorphisme $J_a \longrightarrow \pi_{a
*}^{\flat}T$. En fibre g\'en\'erique, cet homomorphisme se factorise par le
sous-tore des points fixes sous l'action diagonale $W$ dans $T\times_{\bar
X_{\bar v}^\bullet} \widetilde X_{a,\bar v}^{\flat\bullet}$. On en d\'eduit un
homomorphisme $J_a \longrightarrow (\pi_{a *}^{\flat} T)^W$. Le m\^eme
raisonnement s'applique en fait \`a n'importe quel sch\'ema en groupes lisse
de type fini ayant la m\^eme fibre g\'en\'erique que $J_a$. Le lemme en
r\'esulte donc.
\end{proof}

\begin{corollaire}\label{delta local}
On a la formule
$$\dim(\calP_{\bar v}(J_a))=\dim_{\bar k}(\frakt\otimes_{\bar\calO_{\bar v}} \widetilde
\calO_{\bar v}^\flat/\widetilde\calO_{\bar v})^{W}.
$$
\end{corollaire}

On obtient ainsi une autre formule pour l'invariant $\delta_{\bar
v}(a)$. Notons au passage que l'entier $c_{\bar
v}(a)$ de \ref{c_v(a)} est \'egale \`a la chute du rang torique du mod\`ele
de N\'eron $J^\flat_a$ c'est-\`a-dire la diff\'erence entre $r$ et le rang
torique de la fibre sp\'eciale de $J_a^\flat$.

\subsection{Composantes connexes}
Dans ce paragraphe, nous allons d\'ecrire le groupe des composantes
connexes $\pi_0(\calP_{\bar v}(a))$.

Soit $a:\bar X_{\bar v}\rightarrow \frakc$ un morphisme dont l'image n'est
pas contenue dans le diviseur discriminant. On a un alors un sch\'ema en
groupes lisse $J_a$ sur $\bar X_{\bar v}$ dont la fibre g\'en\'erique est un
tore. Soit $J_a^0$ le sous-sch\'ema en groupes ouvert des composantes
neutres de $J_a$. Au-dessus de $\bar F_{\bar v}$, $J_a$ est un tore de
sorte que l'homomorphisme $J_a^0\rightarrow J_a$ induit un isomorphisme
au-dessus de $\bar F_{\bar v}$. En tant qu'homomorphisme de faisceaux en
groupes ab\'eliens, c'est un homomorphisme injectif dont le conoyau est
support\'e par la fibre sp\'eciale de $\bar X_{\bar v}$ avec comme fibre
$$\pi_0(J_{a,\bar v})=J_a(\bar \calO_{\bar v})/J_a^0(\bar \calO_{\bar v}).
$$

Les inclusions $J_a^0(\bar\calO_{\bar v})\subset J_a(\bar\calO_{\bar
v})\subset J_a(\bar F_{\bar v})$ induisent une suite exacte
$$
1\rightarrow \pi_0(J_{a,v})\rightarrow J_a(\bar F_{\bar v})/J_a^0(\bar \calO_{\bar v})\rightarrow
J_a(\bar F_{\bar v})/J_a(\bar \calO_{\bar v}) \rightarrow 1.
$$
On a donc un homomorphisme surjectif
\begin{equation}\label{P_v(J_a^0)->P_v(J_a)}
\calP_{\bar v}(J_a^0)\rightarrow \calP_{\bar v}(J_a)
\end{equation}
dont le noyau est le groupe fini $\pi_0(J_{a,v})$. On en d\'eduit une suite
exacte
$$\pi_0(J_{a,v})\rightarrow \pi_0(\calP_{\bar v}(J_a^0))\rightarrow \pi_0(\calP_{\bar v}(J_a)) \rightarrow 1.$$
Pour d\'eterminer $\pi_0(\calP_{\bar v}(J_a))$, il suffit donc de d\'ecrire le
groupe $\pi_0(\calP_{\bar v}(J_a^0))$ et l'image de la fl\`eche
$\pi_0(J_{a,v})\rightarrow \pi_0(\calP_{\bar v}(J_a^0))$.

Comme dans le cas de la dualit\'e de Tate-Nakayama, le groupe des
composantes connexes $\pi_0(P_a)$ s'exprime plus ais\'ement \`a l'aide de
la dualit\'e. Pour tout groupe ab\'elien de type fini $\Lambda$, nous notons
$$\Lambda^*=\Spec(\Ql[\Lambda])$$
le $\Ql$-groupe diagonalisable de groupe des caract\`eres $\Lambda$ et
inversement  pour tout $\Ql$-groupe diagonalisable $A$, nous notons $A^*$
son groupe des caract\`eres qui est un groupe ab\'elien de type fini.

Pour \'ecrire des formules explicites, fixons une trivialisation de
$\rho_{\Out}$ au-dessus de $\bar X_v$. Ceci permet en particulier d'identifier
$W$ et $\bbW$. Soit $\widetilde X_{a,\bar v}^\bullet$ l'image r\'eciproque du
rev\^etement fini \'etale $\pi:\frakt^\rs \rightarrow \frakc^\rs$ par le
morphisme $a:\bar X_{\bar v}\rightarrow \frakc$. En choisissant un point
g\'eom\'etrique de $\widetilde X_{a,\bar v}$ au-dessus du point
g\'eom\'etrique $\bar\eta_{v}$ de $\bar X_{\bar v}$, on obtient un
homomorphisme $\pi_a^\bullet:I_v \rightarrow \bbW$.

\begin{proposition}\label{pi 0 local}
Avec le choix d'un point g\'eom\'etrique de $\widetilde X_{a,\bar v}$
au-dessus du point g\'eom\'etrique $\bar\eta_{v}$ de $\bar X_{\bar v}$, on a
un isomorphisme canonique entre groupes diagonalisables
$$\pi_0(\calP_{\bar v}(J_a^0))^*=\hat\bbT^{\pi_a^\bullet(I_v)}.$$
De m\^eme, on a un isomorphisme
$$\pi_0(\calP_{\bar v}(J_a))^*=\hat\bbT(\pi_a^\bullet(I_v))$$
o\`u $\hat\bbT(\pi_a^\bullet(I_v))$ est le sous-groupe de
$\hat\bbT^{\pi_a^\bullet(I_v)}$ form\'e des \'el\'ements $\kappa\in \hat\bbT$ tel
que $\pi_0(\calP_{\bar v}(J_a^0))$ est contenu dans le groupe de Weyl de la
composante neutre $\hat\bbH$ du centralisateur de $\kappa$ dans
$\hat\bbG$.
\end{proposition}

\begin{proof}
Consid\'erons  le sch\'ema en groupes des composantes neutres $J_a^{\flat,0}$ du
mod\`ele de N\'eron $J_a^\flat$. Dans la terminologie de \cite{BLR}, c'est
le mod\`ele de N\'eron connexe. Puisque $J_a^0$ a des fibres connexes,
l'homomorphisme $J_a^0 \rightarrow J_a^\flat$ se factorise par
$J_a^{\flat,0}$.

\begin{lemme}
L'homomorphisme $\calP_{\bar v}(J_a^0)\rightarrow \calP_{\bar
v}(J_a^{\flat,0})$ induit un isomorphisme de $\pi_0(\calP_{\bar v}(J_a^0))$
sur $J_a(\bar F_{\bar v})/J_a^{\flat,0}(\bar \calO_{\bar v})$.
\end{lemme}

\begin{proof}
Comme $J_a^0$ et $J_a^{\flat,0}$ a des fibres connexes, l'homomorphisme
$\calP_{\bar v}(J_a^0)\rightarrow \calP_{\bar v}(J_a^{\flat,0})$ induit un
isomorphisme sur les groupes des composantes connexes. Topologiquement,
$\calP_{\bar v}(J_a^{\flat,0})$ est le groupe discret $J_a(\bar F_{\bar
v})/J_a^{\flat,0}(\bar \calO_{\bar v})$.
\end{proof}

Pour tout tore $A$ sur $\bar F_{\bar v}$, on va consid\'erer le mod\`ele de
N\'eron connexe $A^{\flat,0}$ de $A$ sur $\bar\calO_{\bar v}$ et le groupe
ab\'elien de type fini $A(\bar F_{\bar v})/ A^{\flat,0}(\bar\calO_{\bar v})$.
On peut v\'erifier \cf \cite{Ra} que le foncteur $A\mapsto A(\bar F_{\bar v})/
A^{\flat,0}(\bar\calO_{\bar v})$ v\'erifie les axiomes du lemme
\cite[2.2]{K-Iso1}.  Suivant Kottwitz, on obtient une formule g\'en\'erale pour pour $A(\bar F_{\bar v})/
A^{\flat,0}(\bar\calO_{\bar v})$ et en particulier, on
obtient le lemme suivant.

\begin{lemme}\label{Kottwitz local}
Avec les notations ci-dessus, on a un isomorphisme $J_a(\bar F_{\bar
v})/J_a^{\flat,0}(\bar\calO_{\bar v})=(\bbX_*)_{\pi_a^\bullet(I_v)}$.
\end{lemme}

La conjonction des deux lemmes pr\'ec\'edents donne l'isomorphisme
$$
\pi_0(\calP_{\bar v}(J_a^0))=(\bbX_*)_{\pi_a^\bullet(I_v)}
$$
qui induit par dualit\'e la premi\`ere assertion du th\'eor\`eme
$$\pi_0(\calP_{\bar v}(J_a^0))^*=\hat\bbT^{\pi_a^\bullet(I_v)}.$$

La d\'emonstration de la deuxi\`eme assertion utilise l'astuce des $z$-exten\-sions.
Suivant \cite[7.5]{K-EST}, il existe une suite exacte
$$1\rightarrow G\rightarrow G_1\rightarrow C\rightarrow 1$$
de sch\'ema en groupes r\'eductifs au-dessus de $X$ qui se d\'eduit par
torsion ext\'erieure d'une suite exacte de groupes r\'eductifs d\'eploy\'es
$$1\rightarrow \bbG\rightarrow \bbG_1 \rightarrow \mathbb C\rightarrow 1$$
o\`u $\mathbb C$ est un tore et o\`u $\bbG_1$ est un groupe r\'eductif de
centre connexe. Son groupe dual $\hat\bbG_1$ a un groupe d\'eriv\'e
simplement connexe. Il existe un tore maximal $\bbT_1$ de $\bbG_1$ tel
qu'on a une suite exacte
$$1\rightarrow \bbT\rightarrow \bbT_1 \rightarrow \mathbb C\rightarrow 1.$$

En rempla{\c c}ant $G$ par $G_1$, on va noter $\frakc_1$ \`a la place de
$\frakc$. L'homo\-morphisme $G\rightarrow G_1$ induit un morphisme
$\alpha:\frakc\rightarrow \frakc_1$. Au-dessus de $\frakc_1$, on a le
sch\'ema en groupes des centralisateurs r\'eguliers $J_1$ qui est un
sch\'ema en groupes lisse \`a fibres connexes puisque le centre de $G_1$
est connexe \cf \ref{centre connexe}. On a une suite exacte de sch\'emas
en groupes lisse commutatif
$$1\rightarrow J\rightarrow \alpha^* J_1\rightarrow C\rightarrow 1.$$
Soit $a:\bar X_{\bar v} \rightarrow \frakc$ tel que $a|_{\bar X_{\bar
v}^\bullet}$ est \`a l'image dans $\frakc^\rs$. Notons encore $\alpha(a):\bar
X_{\bar v}\rightarrow \frakc_1$ le morphisme obtenu en composant $a$
avec $\alpha$. Sur $\bar X_{\bar v}$, on a une suite exacte
$$1\rightarrow J_a\rightarrow J_{1,\alpha(a)}\rightarrow C\rightarrow 1$$
avec $J_a=a^* J$ et $(J_1)_{\alpha(a)}=\alpha(a)^* J_1$.

\begin{proposition}
L'homomorphisme $$\pi_0(\calP_{\bar v}(J_a))\rightarrow \pi_0(\calP_{\bar
v}(J_{1,\alpha(a)}))$$ est injectif.
\end{proposition}

\begin{proof}
Puisque $J\rightarrow \alpha^* J_1$ est une immersion ferm\'ee, en
particulier propre, on a $J_a(\bar\calO_{\bar v})=J_a(\bar F_{\bar v})\cap
(J_1)_{\alpha(a)}(\bar\calO_{\bar v})$. Il s'ensuit que l'homomorphisme
$$
j_{\alpha(a)}:J_a(\bar F_{\bar v})/J_a(\bar \calO_{\bar v})\rightarrow J_{1,\alpha(a)}
(\bar F_{\bar v})
/J_{1,\alpha(a)}(\bar\calO_{\bar v})
$$
est injectif. Puisque $C$ est un tore sur $\bar X_{\bar v}$,
$C(\bar F_{\bar v})/C(\bar \calO_{\bar v})$ est un groupe ab\'elien libre de type fini et discret. La
composante neutre de
$$J_{1,\alpha(a)}(\bar F_{\bar v})/J_{1,\alpha(a)}(\bar\calO_{\bar v})$$ appartient donc dans
l'image de $j_{\alpha(a)}$. Par cons\'equent, $j_{\alpha(a)}$ induit un
isomorphisme entre la composante neutre de $J_a(\bar F_{\bar v})/J_a(\bar \calO_{\bar v})$ et
celle de $J_{1,\alpha(a)}(\bar F_{\bar v})/J_{1,\alpha(a)}(\bar\calO_{\bar v})$. Il  en r\'esulte que
l'homomorphisme $\pi_0(\calP_{\bar v}(J_a))\rightarrow \pi_0(\calP_{\bar
v}((J_1)_{\alpha(a)}))$ est injectif.
\end{proof}

\begin{corollaire}
Il existe un isomorphe canonique entre $\pi_0(\calP_{\bar v}(J_a))$ et
l'image de l'homomorphisme $$\pi_0(\calP_{\bar v}(J_a^0))\rightarrow
\pi_0(\calP_{\bar v}(J_{1,\alpha(a)})).$$

\end{corollaire}

\begin{proof}
En plus de la proposition pr\'ec\'edente, il suffit d'invoquer la surjectivit\'e
de $\pi_0(\calP_{\bar v}(J_a^0))\rightarrow \pi_0(\calP_{\bar v}(J_a))$.
\end{proof}

En choisissant un point g\'eom\'etrique de $\widetilde X_{a,\bar v}^\bullet$,
on peut identifier le groupe $\pi_0(\calP_{\bar v}(J_a))$ \`a l'image de
l'homomorphisme
$$(\bbX_*)_{I_v} \rightarrow (\bbX_{1,*})_{I_v}$$
o\`u $\bbX_{1,*}=\Hom(\GG_m,\bbT_1)$. Notons que l'\'egalit\'e
$\pi_0(\calP_{\bar v}(J_{1,\alpha(a)}))=(\bbX_{1,*})_{I_v}$ r\'esulte de la
connexit\'e des fibres de $J_1$ et de \ref{Kottwitz local}. Dualement, il
existe un isomorphisme entre le sous-groupe $\Spec(\Ql[\pi_0(\calP_{\bar
v}(J_a))])$ de $\hat\bbT^{I_v}$ et l'image de l'homomorphisme
$$\hat\bbT_1^{I_v}=\Spec(\Ql[(\bbX_{1,*})_{I_v}])\rightarrow
\Spec (\Ql[(\bbX_*)_{I_v}])=\hat\bbT^{I_v}.$$

Il reste maintenant \`a d\'emontrer que l'image de $\hat\bbT_1^{I_v}$ dans
$\hat\bbT^{I_v}$ est bien le sous-groupe des \'el\'ements $\kappa\in
\hat\bbT$ tels que $\pi_a^\bullet(I_v)$ soit contenu dans le groupe de Weyl
de la composante neutre $\hat\bbH$ du centralisateur $\hat\bbG_\kappa$.
Pour tout $\kappa\in\hat\bbT$, pour tout $\kappa_1\in\hat\bbT_1$ d'image
$\kappa$, le centralisateur $\hat\bbH_1$ de $\kappa_1$ dans $\hat\bbG_1$
est connexe et son image dans $\hat\bbG$ est $\hat\bbH$. En effet, $\hat\bbG_{1}$ a un groupe d\'eriv\'e simplement connexe de sorte que le centralisateur d'un \'el\'ement semi-simple de $\hat\bbG_{1}$ est connexe. On en d\'eduit ce
qu'on voulait.
\end{proof}

On a une autre description de $\pi_0(\calP_{\bar v}(J_a))$ plus explicite
mais finalement moins commode \`a l'usage.

\begin{proposition}
Avec le choix d'un point g\'eom\'etrique de $\widetilde X_{a,\bar v}$
au-dessus du point g\'eom\'etrique $\bar\eta_{v}$ de $\bar X_{\bar v}$, on a
un isomorphisme entre $\pi_0(\calP_{\bar v}(J_a))$ est le conoyau du
compos\'e de deux fl\`eches
$$
\pi_0(Z_{\GG}) \rightarrow \pi_0(\bbT^{\pi_a^\bullet(I_v)})
\rightarrow (\bbX_*)_{\pi_a^\bullet(I_v)}
$$
dont la premi\`ere se d\'eduit de l'homomorphisme \'evident $Z_{\GG}
\rightarrow \bbT^{\pi_a^\bullet(I_v)}$ et dont la seconde est un
isomorphisme canonique de $\pi_0(\bbT^{\pi_a^\bullet(I_v)})$ sur la partie
de torsion de $(\bbX_*)_{\pi_a^\bullet(I_v)}$.
\end{proposition}

\begin{proof}
On sait que $\pi_0(\calP_{\bar v}(J_a))$ est le conoyau de l'homomorphisme
$\pi_0(J_{a,v})\rightarrow \pi_0(\calP_{\bar v}(J_a^0))$. D'apr\`es \ref{pi 0
ZG}, on sait que $\pi_0(J_{a,v})$ et $\pi_0(Z_\GG)$ ont la m\^eme image
dans $(\bbX_*)_{\pi_a^\bullet(I_v)}$.
\end{proof}

\subsection{Densit\'e de l'orbite r\'eguli\`ere}
On reporte  \`a la section \ref{dense} pour une esquisse de la d\'emonstration
de la proposition suivante. Elle se d\'eduira de son analogue global qui sera
d\'emontr\'ee de mani\`ere d\'etaill\'ee. Seul analogue global sera utilis\'e
dans la suite de l'article.

\begin{proposition}\label{densite locale}
L'ouvert $\calM_v^\reg(a)$ est dense dans $\calM_v(a)$.
\end{proposition}

Notons qu'on sait d\'ej\`a que le ferm\'e compl\'ementaire
$\calM_v(a)-\calM_v^\reg(a)$ est de dimension strictement plus petite que
que $\dim(\calM_v^\reg(a))$ \cf \ref{dimension du complement}.

\begin{corollaire}
L'ensemble des composantes irr\'eductibles de la fibre de Springer affine
$\calM_v(a)$ est en bijection canonique avec le groupe des composantes connexes du groupe
$\calP_v(J_a)$.
\end{corollaire}

\begin{proof}
On a en effet un isomorphisme $\calP_v(J_a)\rightarrow \calM_v^\reg(a)$
en faisant agir $\calP_v(J_a)$ sur le point de Kostant.
\end{proof}

\subsection{Le cas d'un groupe endoscopique}

Soit $H$ un groupe endoscopique de $G$ au sens de \ref{subsection :
Groupes endoscopiques}. Soit $a_H\in \frakc_H(\bar\calO_{\bar v})$ d'image
$a\in\frakc(\bar\calO_{\bar v})\cap \frakc^\rs(\bar F_{\bar v})$. Il n'y a pas
de relation directe entre les fibres se Springer affines $\calM_{H,v}(a_H)$
et $\calM_v(a)$. Il y a pourtant une relation \'evidente entre leurs groupes
de sym\'etries.

On alors deux $\bar X_{\bar v}$-sch\'emas en groupes $J_a=a^* J$ et
$J_{H,a_H}=a_H^* J_H$ qui sont reli\'es par un homomorphisme
$$\mu_{a_H}:J_a\rightarrow J_{H,a_H}$$
qui est un isomorphisme au-dessus du disque point\'e  $\bar X_{\bar v}^\bullet$. Cet homomorphisme
est construit en prenant l'image r\'eciproque par $a_H$ de
l'homomorphisme $\mu$ construit dans \ref{J JH}.

Soit $\calR_{H,\bar v}^G(a_H)$ le groupe alg\'ebrique affine sur $\bar k$
dont le groupe des $\bar k$-points est
$$\calR_{\bar v}(a_H)(\bar k)=J_{H,a_H}(\bar \calO_{\bar v})/ J_a(\bar \calO_{\bar v}).$$
On a alors la suite exacte
\begin{equation}\label{RPP local}
1\rightarrow \calR_{H,\bar v}^G(a_H)\rightarrow \calP_{\bar v}(J_a) \rightarrow
\calP_{\bar v}(J_{H,a_H})\rightarrow 1.
\end{equation}

\begin{lemme}
On a
$$\dim(\calR_{H,\bar v}^G(a_H))=r_{H,\bar v}^G(a_H)$$
o\`u $r_{H,\bar v}^G(a_H)=\deg_{\bar v}(a_H^*\resultant_G^H)$, le diviseur
$\resultant_G^H$ de $\frakc_H$ \'etant d\'efini dans \ref{discrimnant
resultant}.
\end{lemme}

\begin{proof}
En vertu de la suite exacte \ref{RPP local}, on a
$$\dim(\calR_{H,\bar v}^G(a_H))=\dim(\calP_{\bar v}(J_a))-\dim(\calP_{\bar v}(J_{H,a_H})).$$
Comme l'homomorphisme $J_{H,a_H}\rightarrow J_a$ est un isomorphisme
dans la fibre g\'en\'erique, on l'\'egalit\'e
$$c_{\bar v}(a_H)=c_{\bar v}(a)$$
o\`u $c_{\bar v}(a_H)$ et $c_{\bar v}(a)$ sont les invariants galoisiens qui
apparaissent dans la formule de dimension de Bezrukavnikov \cf
\ref{dimension Bezru}. En appliquant cette formule  \`a $a$ et $a_H$, on
trouve
$$\dim(\calR_{H,\bar v}^G(a_H))={\deg_{\bar v}(a^*\discrim_G)-\deg_{\bar v}(a_H^*\discrim_H)}.$$
Il suffit maintenant d'\'evoquer \ref{discrimnant resultant} pour conclure.
\end{proof}

Soit maintenant $a_H\in \frakc_H(\calO_v)$ d'image
$a\in\frakc(\calO_v)\cap \frakc^\rs(F_v)$. Avec la m\^eme d\'efinition que
ci-dessus, on obtient un groupe $\calR_v(a_H)$ d\'efini sur $k$. On a la
formule
$$\calR_{H,v}^G(a_H)\otimes_k \bar k=\prod_{\bar v:k_v \rightarrow \bar k}\calR_{\bar v}(a)$$
qui implique la formule de dimension
$$\dim \calR_{H,v}^G(a_H)=\deg(k_v/ k)\deg_v (a_H^* \resultant^H_G).$$

En identifiant $\calP_v(J_a)$ avec l'ouvert $\calM_v^\reg(a)$ de
$\calM_v(a)$ et en identifiant $\calP_v(J_{H,a_H})$ avec l'ouvert
$\calM_{H,v}(a_H)$, puis en prenant l'adh\'erence du graphe de
l'homomorphisme $\mu_{a_H}$ dans $\calM_v(a)\times \calM_v(a_H)$, on
obtient une correspondance int\'eressante entre ces deux fibres de Springer
affines. Nous n'allons pas utiliser cette correspondance dans ce travail.


\section{Fibration de Hitchin}
\label{section : Fibration de Hitchin}

Dans son article m\'emorable \cite{H}, Hitchin a observ\'e que le fibr\'e cotangent de l'espace de module
des fibr\'es stables sur une courbe forme un syst\`eme hamiltonien
compl\`etement int\'egrable. Pour cela, il construit explicitement une
famille de fonctions commutantes pour le crochet de Poisson en nombre
\'egal \`a la moiti\'e de la dimension de ce cotangent. Ces fonctions
d\'efinissent un morphisme de ce cotangent vers un espace affine dont la
fibre g\'en\'erique est essentiellement une vari\'et\'e ab\'elienne. Ce
morphisme particuli\`erement joli est souvent appel\'e la fibration
de Hitchin.

Nous adopterons un point de vue diff\'erent en consid\'erant les fibres de la
fibration de Hitchin comme l'analogue global des fibres de Springer affines.
En particulier, nous rempla{\c c}ons le fibr\'e canonique de la courbe par un
fibr\'e inversible de degr\'e tr\`es grand. En faisant ainsi, nous perdons la
forme symplectique tout en gardant une fibration ayant la m\^eme allure
que la fibration originale de Hitchin.

Comme nous avons observ\'e dans \cite{N}, le comptage des points dans un
corps fini de la fibration de Hitchin g\'en\'eralis\'ee dans ce sens donne
formellement le c\^ot\'e g\'eom\'etrique de la formule des traces pour
l'alg\`ebre de Lie. Ceci continue \`a nous servir de guide dans ce chapitre
pour \'etudier les propri\'et\'es g\'eom\'etriques de la fibration de Hitchin.
Le comptage de points sera pass\'e en revue dans le chapitre \ref{section : comptage}.

Notre outil favori pour explorer la g\'eom\'etrie de la fibration de Hitchin $f:\calM\rightarrow \calA$ et un
un champ de Picard $\calP\rightarrow \calA$ agissant sur $\calM$. Cette action est fond\'ee sur la construction du centralisateur r\'egulier \ref{section : centralisateur regulier}. On d\'emontre en particulier que qu'il existe un ouvert $\calM^\reg$ de $\calM$ \cf \ref{M reg} sur lequel $\calP$ agit simplement transitivement et qui est dense dans chaque fibre $\calM_a$ \cf \ref{densite globale}. On peut analyser la structure de $\calP_a$ en d\'etails gr\^ace \`a la courbe cam\'erale et \`a la th\'eorie du mod\`ele de N\'eron. 
Ceci permet en particulier de d\'efinir un ouvert $\calA^\diamondsuit$ au-dessus duquel $\calM$ est essentiellement un sch\'ema ab\'elien.

On fera aussi certains calculs utiles pour la suite comme le calcul des dimension \ref{subsection : dimension}, celui du groupe des composante connexes de $\calP_a$ \cf \ref{subsection : pi_0(P_a)} et celui du groupe des automorphismes des fibr\'es de Higgs \cf \ref{subsection : Automorphismes}.

Le point cl\'e de ce chapitre est la formule de produit \ref{produit} qui \'etablit la relation entre la fibre de Hitchin et les fibres de Springer affine. Cette formule a \'et\'e d\'emontr\'ee dans \cite{N}. Cette formule de produt joue un r\^ole crucial dans le chapitre sur le comptage \ref{section : comptage} et en fait elle est \'egalement en filigrane dans la d\'emonstration du th\'eor\`eme du support dans le chapitre \ref{support}.

On \'etablit enfin un lien entre la fibration de Hitchin de $G$ et celle d'un groupe endoscopique. Il existe un morphisme canonique $\nu:\calA_H\rightarrow \calA$ et si $a_H\mapsto a$, on a un homomorphisme canonique $\mu:\calP_a \rightarrow \calP_{H,a_H}$. En revanche, il n'y a pas de relation directe entre les fibres de Hitchin $\calM_a$ et $\calM_{H,a_H}$ mais une correspondance qui se d\'eduit de $\mu$. On utlisera pas cette correspondance dans la suite de l'article mais il vaut probablement le coup de l'exploiter davantage.

\medskip

Voici les notations qui seront utilis\'ees dans ce chapitre. On fixe une courbe
$X$ propre lisse et g\'eom\'etriquement connexe de genre $g$ sur un corps
fini $k$. Soit $\bar k$ une cl\^oture alg\'ebrique de $k$. On note $\bar
X=X\otimes_k\bar k$.

On note $F$ le corps des fonctions rationnelles sur $F$. Soit $|X|$
l'ensemble des points ferm\'es de $X$. Chaque \'el\'ement $v\in|X|$
d\'efinit une valuation $v:F^\times\rightarrow \ZZ$. Notons $F_v$ la
compl\'etion de $F$ par rapport \`a cette valuation, $\calO_v$ son anneau
des entiers et $k_v$ son corps r\'esiduel. On note $X_v=\Spec(\calO_v)$ le
disque formel en $v$ et $X_v^\bullet=\Spec(F_v)$ le disque formel point\'e.

Soit $\bbG$ un groupe de Chevalley avec un groupe de Weyl $\bbW$ dont l'ordre
n'est pas divisible par la caract\'eristique de $k$. Soit $G$ un $X$-sch\'ema en
groupes r\'eductif forme quasi-d\'eploy\'ee de $\GG$ d\'efinie par un
$\Out(\GG)$-torseur $\rho_G$ sur $X$. $G$ est alors muni d'un \'epinglage
$(T,B,\bfx_+)$. Reprenons la suite des notations de \ref{Torsion
exterieure}. En particulier, on a un $X$-sch\'ema des
polyn\^omes caract\'eristique $\frakc$, le morphisme de Chevalley
$\chi:\frakg\rightarrow \frakc$ o\`u $\frakg=\Lie(G)$ et un morphisme fini
plat $\pi:\frakt\rightarrow \frakc$ qui au-dessus de l'ouvert $\frakc^\rs$ est
un torseur sous le sch\'ema en groupes fini \'etale $W$ obtenu en tordant
$\bbW$ par $\rho_G$.

Nous fixons un fibr\'e inversible $D$ sur $X$ qui est le carr\'e
$D={D'}^{\otimes 2}$ d'un autre fibr\'e inversible $D'$. De r\`egle
g\'en\'erale, nous supposons que le degr\'e de $D$ sera plus grand que $2g$
o\`u $g$ est le genre de $X$ ce qui est contraire \`a \cite{H} o\`u $D$ est le
fibr\'e canonique. Nous indiquerons n\'eanmoins les endroits o\`u cette
hypoth\`ese est vraiment n\'ecessaire.

Notons que $\GG_m$ agit sur $\frakg$ et $\frakt$ par homoth\'etie, et agit
sur $\frakc$ de fa{\c c}on compatibles. On notera $\frakg_D$, $\frakt_D$ et
$\frakc_D$ les objets tordus par le $\GG_{m}$-torseur attach\'e au fibr\'e
inversible $D$.

\subsection{Rappels sur ${\rm Bun}_G$}

Nous consid\'erons le champ ${\rm Bun}_G$ qui associe \`a tout
$k$-sch\'ema $S$ le groupo\"ide des $G$-torseurs sur $X\times S$. Le champ
${\rm Bun}_G$ est un champ alg\'ebrique au sens d'Artin \cite{L-MR}. Le
groupo\"ide des $k$-points de ${\rm Bun}_G$ peut s'exprimer comme une
r\'eunion disjointe des doubles quotients
$${\rm Bun}_G(k)=\bigsqcup_{\xi\in {\rm ker}^1(F,G)}
[G^\xi(F)\backslash \check\prod_{v\in |X|} G(F_v)/G(\calO_v)]
$$
sur l'ensemble ${\rm ker}^1(F,G)$ des $G$-torseurs sur $F$ qui sont
localement triviaux. Ici, $G^\xi$ d\'esigne la forme int\'erieure forte de $G$
sur $F$ donn\'ee par la classe $\xi\in{\rm ker}^1(F,G)$ et le produit
$\check\prod $ d\'esigne un produit restreint.

Pour tout $v\in |X|$, on dispose comme dans \ref{subsection :
grassmannienne affine} de la grassmannienne affine $\calQ_v$ et d'un
morphisme
$$\zeta:\calQ_v\longrightarrow {\rm Bun}_G$$
qui consiste \`a recoller le $G$-torseur sur $X_v\hat\times S$ muni d'une
trivialisation sur $X_v^\bullet\hat\times S$, avec le $G$-torseur trivial sur
$(X-v)\times S$. L'existence de ce recollement formel est un r\'esultat de
Beauville et Laszlo \cite{BL}. Au niveau des $k$-points, ce morphisme est le
foncteur \'evident
$$G(F_v)/G(\calO_v) \longrightarrow
[G(F)\backslash \check\prod_{v\in |X|} G(F_v)/G(\calO_v)]
$$
qui envoie $g_v\in G(F_v)/ G(\calO_v)$ sur l'uplet constitu\'e de $g_v$ et
des \'el\'ements neutres de $G(F_{v'})/ G(\calO_{v'})$ pour toutes les places
$v'\not= v$.

\subsection{Construction de la fibration}\label{Fibration de Hitchin}
Rappelons la d\'efinition de l'espace de module de Hitchin.

\begin{definition}\label{definition fibration de Hitchin}
L'espace total de Hitchin est le groupo\"ide fibr\'e $\calM$ qui associe \`a
tout $k$-sch\'ema $S$ le groupo\"ide $\calM(S)$ des couples $(E,\phi)$
constitu\'e d'un $G$-torseur $E$ sur $X\times S$ et d'une section
$$\phi\in\rmH^0(X\times S,\ad(E)\otimes_{\calO_X} D)$$
o\`u $\ad(E)$ est le fibr\'e en alg\`ebres de Lie obtenu en tordant $\frakg$
muni de l'action adjointe par le $G$-torseur $E$.
\end{definition}

\begin{numero}
Il revient au m\^eme de dire que $\calM(S)$ est le groupo\"ide des
morphismes $h_{E,\phi}$ qui s'ins\`erent dans le diagramme commutatif :
$$X\times S \rightarrow [\frakg_D/ G].$$
Ce champ est un fibr\'e vectoriel au-dessus du champ alg\'ebrique ${\rm
Bun}_G$ classifiant les $G$-torseurs au-dessus de $X$ si bien qu'il est
lui-m\^eme un champ alg\'ebrique.
\end{numero}

\begin{numero}
Le morphisme caract\'eristique de Chevalley $\chi:\frakg\rta\frakc$ induit
un morphisme
$$[\chi]:[\frakg_D/G]\rta \frakc_D.$$
On obtient un morphisme
$$f:\calM\rta\calA$$
o\`u pour tout $k$-sch\'ema $S$, $\calA(S)$ est le groupo\"ide des
morphismes $a$ qui s'ins\`erent dans le diagramme commutatif :
$$a: X\times S\rightarrow \frakc_D.$$
Il revient au m\^eme de dire que $\calA$ est l'ensemble des sections du
fibr\'e vectoriel $\frakc_D$ au-dessus de $X$. Il s'ensuit que $\calA$ est un
$k$-espace vectoriel de dimension finie.
\end{numero}

\begin{numero}\label{Kostant Hitchin}
\'Etant donn\'ee une racine carr\'e $D'$ de $D$, on obtient une section $\epsilon_{D'}:\calA\rightarrow \calM$ du morphisme $f:\calM\rightarrow \calA$ en utilisant \ref{epsilon D'}. Cette section est essentiellement la m\^eme que les sections qu'a construites Hitchin de fa\c con plus explicite dans le cas des groupes classiques. On l'appellera la section de Kostant-Hitchin.
\end{numero}

\subsection{Sym\'etries d'une fibre de Hitchin}
\label{subsection : symetries Hitchin}

Comme pour les fibres de Springer affines, la construction des sym\'etries naturelles
d'une fibre de Hitchin est fond\'ee sur le lemme \ref{J}, voir \cite{N}.

\begin{numero}
Pour tout $S$-point $a$ de $\calA$, on a un morphisme $h_a:X\times S\rta
[\frakc/\GG_m]$. On note $J_a=h_a^* J$ l'image r\'eciproque de $J$ sur
$[\frakc/\GG_m]$ et on consid\`ere le groupo\"ide de Picard $\calP_a(S)$ des
$J_a$-torseurs au-dessus de $X\times S$. Quand $a$ varie,
cette construction d\'efinit un groupo\"ide de Picard $\calP$ fibr\'e au-dessus de $\calA$.
\end{numero}

\begin{numero}
L'homomorphisme $\chi^*J\rta I$ du lemme 2.2.1 induit pour tout $S$-point
$(E,\phi)$ au-dessus de $a$ un homomorphisme
$$J_a \rta \Aut_{X\times S}(E,\phi)=h_{E,\phi}^* I.$$
Par cons\'equent, on peut tordre $(E,\phi)$ par n'importe quel
$J_a$-torseur. Ceci d\'efinit une action du groupo\"ide de Picard $\calP_a(S)$
fibr\'e sur le groupo\"ide $\calM_a(S)$. En laissant le point $a$ varier, on
obtient une action de $\calP$ sur $\calM$ relativement
\`a la base $\calA$.
\end{numero}

Comme pour les fibres de Springer affines, nous consid\'erons l'ouvert
$\calM^\reg$ de $\calM$ dont les points sont les morphisme $h_{E,\phi}:X\times
S\rightarrow [\frakg_{D}/G]$ qui se factorisent par l'ouvert
$[\frakg^\reg_{D}/G]$.

\begin{proposition}\label{M reg}
$\calM^\reg$ est un ouvert de $\calM$ ayant des fibres non vides au-dessus
de $\calA$. De plus, c'est torseur sous l'action de $\calP$.
\end{proposition}

\begin{proof}
Pour tout $a\in\calA(\bar k)$, le point $[\epsilon]^{D'}(a)$ construit dans \cf \ref{epsilon D'} est dans l'ouvert
r\'egulier. Ceci montre que le morphisme $\calM^\reg\rightarrow \calA$ a
les fibres non vides. On d\'eduit du lemme \ref{gerbe} que $\calM_a^\reg$
est un torseur sous $\calP_a$.
\end{proof}

\begin{numero}
La section de Kostant-Hitchin \ref{Kostant Hitchin} $\epsilon_{D'}:\calA\rightarrow \calM$ se factorise par l'ouvert $\calM^\reg$ car la section de Kostant $\epsilon:\frakc\rightarrow\frakg$ se factorise \`a travers $\frakg^\reg$.
\end{numero}

\subsection{Dimensions}\label{subsection : dimension}
Nous allons utiliser la m\^eme notation $\frakc_D$ pour d\'esigner le
$X$-sch\'ema en vectoriels et le $\calO_X$-module localement libre. Si $G$
est d\'eploy\'e, le $\calO_X$-module $\frakc_D$ a une expression explicite
$$\frakc_D=\bigoplus_{i=1}^r D^{\otimes e_i}$$
o\`u $e_1,\ldots,e_r$ sont des entiers naturels qui apparaissent dans
\ref{Chevalley}. En g\'en\'eral, $\frakc_D$ prend cette forme apr\`es un
changement de base fini \'etale $\rho:X_\rho \rightarrow X$ qui d\'eploie
$G$. Par d\'efinition, le $X$-sch\'ema en vectoriel $\frakc_D$ est
$$\frakc_D=\Spec({\rm Sym}_{\calO_X}(\frakc_D^*))$$
o\`u ${\rm Sym}_{\calO_X}[\frakc_D^*]$ est le faisceau en
$\calO_X$-alg\`ebre puissance sym\'etrique du $\calO_X$-module dual
$\frakc_D^*$.

\begin{lemme}\label{dim A}
Si $D>2g-2$, $\calA$ est un $k$-espace affine de dimension
$$\dim(\calA)=\sharp\, \Phi\deg(D)/ 2+r(1-g+\deg(D))$$
o\`u $r$ est le rang de $\GG$ et $\sharp\,\Phi$ est le nombre de ses
racines.
\end{lemme}

\begin{proof}
Soit $\rho:X_\rho\rightarrow X$ le rev\^etement fini \'etale galosien qui
d\'eploie $G$. Alors $\rho^*\frakc_D$ est isomorphe \`a une somme directe
$\rho^*D^{\otimes e_i}$.  Il s'ensuit que
$$\deg(\frakc_D)=(e_1+\cdots+e_r)\deg(D).$$
On a l'\'egalit\'e
$$\dim \rmH^0 (X,\frakc_D)+\dim\rmH^1(X,\frakc_D)=(e_1+\cdots+e_r)\deg(D)+r(1-g)$$
par le th\'eor\`eme de Riemann-Roch. On sait d'apr\`es Kostant que les
entiers $e_i-1$ sont les exposants du syst\`eme de racines $\Phi$ de sorte
que
$$e_1+\cdots+e_r=r+\sharp\, \Phi/ 2.$$
Il suffit donc de d\'emontrer que $\rmH^1(X,\frakc_D)=0$. Puisque
$\deg(D)> 2g-2$ et puisque $\rho$ est fini \'etale, on a
$$\deg(\rho^* D)> \deg (\rho^* \Omega_{X/ k})=\deg(\Omega_{X_\rho/ k})$$
d'o\`u r\'esulte l'annulation de $\rmH^1(X_\rho, \rho^* D^{\otimes e_i})$.
Pour d\'emontrer l'annulation de $\rmH^1(X,\frakc_D)$, il suffit de
remarquer que $\rmH^1(X,\frakc_D)$ est un facteur direct de
$\rmH^1(X_\rho, \rho^* \frakc_D)$.
\end{proof}

Si $\deg(D)$ est un entier fix\'e plus grand que $2g-2$, la dimension de la
base de Hitchin $\calA$ ne d\'epend donc ni de $D$ ni de la forme
quasi-d\'eploy\'ee.

\begin{proposition}\label{P lisse}
Le champ de Picard $\calP$ est lisse au-dessus de $\calA^\heartsuit$.
\end{proposition}

\begin{proof} Puisque $J_a$ est un sch\'ema en groupe lisse commutatif,
l'obstruction \`a la d\'eformation d'un $J_a$-torseur g{\^\i}t dans le groupe
$\rmH^2(\bar X,\Lie(J_a))$. Celui-ci est nul car $\bar X$ est un sch\'ema de
dimension un.
\end{proof}

\begin{proposition}
Pour tout $a\in\calA(\bar k)$, on a un isomorphisme canonique
$\Lie(J_a)=\frakc_D^*\otimes D$.
\end{proposition}

Dans la d\'emonstration qui suit lorsque $f:X\rightarrow Y$ et si $\calL$ est un
$\calO_Y$-module, on \'ecrira simplement $\calL$ pour d\'esigner aussi son
image inverse $f^*\calL$. Cet abus de notation ne devrait pas causer de
confusion car le sch\'ema qui porte le module sera toujours clair par le contexte.

\begin{proof}
D'apr\`es \ref{J=J'}, l'homomorphisme $J_a\rightarrow J^1_a$ induit un
isomorphisme $\Lie(J_a) \rightarrow \Lie(J^1_a)$ sur les alg\`ebres de Lie.
Par construction de $J^1$, $\Lie(J_a^1)$ peut se calculer \`a l'aide du
rev\^etement cam\'eral $\pi_a:\widetilde X_a \rightarrow \bar X$ par la formule
$$\Lie(J^1_a)=((\pi_a)_*\frakt)^{W}.$$
Rappelons que le rev\^etement cam\'eral est d\'efini en formant le diagramme
cart\'esien
$$ \xymatrix{
 \widetilde X_a \ar[d]_{\pi_a} \ar[r]^{}   & \frakt_D \ar[d]^{\pi}  \\
 \bar X \ar[r]_{a} &   \frakc_D        }
$$
Puisque $\pi$ est fini et plat, la formation de $\pi_* \frakt$ commute \`a
tout changement de base en particulier $(\pi_a)_*\frakt=a^* \pi_* \frakt$. Il
suffit donc de calculer $(\pi_*\frakt)^W$.

Puisque $\frakc_D$ est le quotient invariant de $\frakt_D$ par l'action de
$W$, il en est de m\^eme pour les espaces totaux des fibr\'es tangents
$T_{\frakt_D/X}$ et $T_{\frakc_D/X}$. On en d\'eduit
$$(\pi_* \Omega_{\frakt_D/X})^W=\Omega_{\frakc_D/X}.$$
Par ailleurs, comme $\frakt_D$ est un fibr\'e vectoriel sur $X$, on a
$\Omega_{\frakt_D/X}=\frakt_{D^{-1}}$. De m\^eme, on a
$\Omega_{\frakc_D/X}=\frakc_{D}^*$. On obtient finalement une \'egalit\'e
de $\calO_{\frakc_D}$-modules $(\pi_* \frakt_{D^{-1}})^{W}=\frakc_{D}^*$
d'o\`u
$$(\pi_*\frakt)^W=\frakc_D^* \otimes D.$$
En prenant l'image inverse de cette \'egalit\'e par la section $a:\bar
X\rightarrow \frakc_D$, on obtient l'\'egalit\'e $\Lie(J_a)=\frakc_D^*
\otimes_{\calO_X} D$.
\end{proof}

Dans le cas o\`u $G$ est d\'eploy\'e, le lemme permet d'exprimer
$\Lie(J_a)$ en termes de $D$ et des exposants du syst\`eme de racines $\Phi$. Soient
$e_1,\ldots,e_r$ les degr\'es des polyn\^omes invaraints homog\`enes comme
dans l'\'enonc\'e de \ref{Chevalley}, on a
$$\Lie(J_a)=D^{-e_1+1} \oplus\cdots\oplus D^{-e_r+1}.$$
Si $G$ n'est pas d\'eploy\'e, $\Lie(J_a)$ devient isomorphe \`a la somme
directe de droite sur un rev\^etement fini \'etale galoisien de $X$ qui
d\'eploie $G$. En particulier
$$\deg(\Lie(J_a))=\sum_{i=1}^r (-e_r+1)\deg(D)= -\sharp\,\Phi\deg(D)/ 2.$$

\begin{corollaire}\label{dim fibre}
Pour tout $a\in \calA^\heartsuit(\bar k)$,
$$\dim(\calP_a)= \sharp\,\Phi\deg(D) / 2 + r(g-1).$$
\end{corollaire}

\begin{proof}
On a
$$\dim(\calP_a)=\dim(\rmH^1(\bar X,\Lie(J_a)))-\dim(\rmH^0(\bar X,\Lie(J_a)))$$
de sorte que l'\'egalit\'e \`a d\'emontrer r\'esulte de la formule de
Riemann-Roch.
\end{proof}

\begin{numero}
En comparant avec la formule \ref{dim A}
$$\dim(\calA)=\sharp\, \Phi\deg(D)/ 2+r(1-g+\deg(D))$$
on obtient
$$\dim(\calP)=(r+\sharp\,\Phi)\deg(D).$$
\end{numero}

\begin{numero}
Dans \ref{densite globale}, on d\'emontrera que $\calM_a^\reg$ est dense
dans $\calM_a$ de sorte qu'on aura alors les m\^emes formules de dimension pour
$\calM$
$$\dim(\calM_a)= \sharp\,\Phi\deg(D) / 2 + r(g-1)$$
et
$$\dim(\calM)=(r+\sharp\,\Phi)\deg(D).$$
\end{numero}

\subsection{L'ouvert $\calA^\heartsuit$ et le rev\^etement cam\'eral}

Consid\'erons le diagramme cart\'esien
$$
\xymatrix{
 \widetilde X \ar[d]_{} \ar[r]^{}   & \frakt_D \ar[d]^{\pi}  \\
 X\times \calA \ar[r] &       \frakc_D    }
$$
o\`u le morphisme de bas associe un couple $(x,a)$ avec $x\in X$ et
$a\in\calA$ l'image $a(x)\in\frakc_D$. Le morphisme de gauche qui se
d\'eduit de $\pi$ par changement de base est un morphisme fini et plat avec
une action de $W$.

En prenant la fibre en chaque point $a\in\calA(\bar k)$, on obtient le
rev\^etement cam\'eral $\pi_a:\widetilde X_a\rightarrow \bar X$ de Donagi.
Dans la suite, on va se restreindre aux param\`etres $a$ tels que la courbe
cam\'erale est r\'eduite. Ces param\`etres forment un ouvert de $\calA$ qui
peut \^etre d\'ecrit comme suit.

Consid\'erons l'image inverse $U$ de $\frakc^\rs_D\subset \frakc_D$ dans
$X\times \calA$. Le morphisme de $U\rightarrow \calA$ \'etant lisse, son
image est un ouvert de $\calA$ que nous allons noter $\calA^\heartsuit$.
Ses $\bar k$-points sont d\'ecrits comme suit
$$\calA^\heartsuit(\bar k)=\{a\in\calA(\bar k)\mid a(\bar X)\not\subset \discrim_{G,D}\}.$$
Si $\deg(D)>2g$, l'ouvert $\calA^\heartsuit$ est non vide. On reporte la
d\'emonstration \`a \ref{A diamond non vide} o\`u on d\'emontre un
\'enonc\'e plus fort.

\begin{lemme}
Pour tout point g\'eom\'etrique $a\in \calA^\heartsuit(\bar k)$, le
rev\^ete\-ment $\pi_a:\widetilde X_a\rta X\otimes_k \bar k$ est
g\'en\'eriquement un torseur sous $W$. De plus,  $\widetilde X_a$ est une
courbe r\'eduite.
\end{lemme}

\begin{proof}
Par d\'efinition de $\calA^\heartsuit$, l'intersection $U_a$ de $U$ avec la
fibre $X\times \{a\}$ est un ouvert non vide de $\bar X$. Par construction,
$\pi_a$ est un $W$-torseur au-dessus de cet ouvert dense. Puisque $\pi_a$
est un morphisme fini plat, $\pi_{a*}\calO_{\widetilde X_a}$ est un
$\calO_{\bar X}$-module sans torsion. S'il est g\'en\'eriquement r\'eduit, il
est partout r\'eduit.
\end{proof}

La courbe cam\'erale $\widetilde X_a$ est munie d'une famille de
rev\^etements finis \'etales naturels. Consid\'erons une r\'eduction du
torseur $\rho_G$ comme dans \cf \ref{reduction} d'o\`u nous reprenons les
notations. En particulier, on a un rev\^etement fini \'etale $\rho:X_\rho
\rightarrow X$ galoisien de groupe de Galois $\Theta_\rho$ qui induit
$\rho_{\Out}$ par le changement de groupes structuraux
$\Theta_\rho\rightarrow \Out(\GG)$. On a alors un diagramme cart\'esien
\begin{equation}\label{diagramme deploiement 2}
\xymatrix{
 X_\rho \times\bbt \ar[d]_{} \ar[r]^{\pi}    & X_\rho \times \bbc\ar[d]^{}  \\
 \frakt \ar[r]_{\pi} & \frakc           }
\end{equation}
Pour tout $a\in\calA(\bar k)$, construisons le rev\^etement
$\pi_{\rho,a}:\widetilde X_{\rho,a}\rightarrow \bar X$ en formant le produit
cart\'esien
\begin{equation}\label{tilde X rho a}
\xymatrix{
 \widetilde X_{\rho,a} \ar[d]_{\pi_{\rho,a}} \ar[r]^{}   & X_\rho\times \bbt_D \ar[d]^{\pi_\rho}  \\
 \bar X \ar[r]_a &       \frakc_D    }
\end{equation}
On a alors un morphisme fini \'etale $\widetilde X_{\rho,a}\rightarrow
\widetilde X_a$ qui est galoisien de groupe de Galois $\Theta_\rho$.

La discussion ci-dessus s'applique encore si on dispose d'une r\'eduction
$\bar \rho$ du $\Out(\GG)$-torseur $\bar \rho_G$ sur $\bar X$.

\begin{proposition}\label{connexe}
Soit $\bar\rho:\bar X_{\bar \rho}\rightarrow \bar X$ un rev\^etement fini
\'etale galoisien et connexe. Supposons que $\deg(D)>2g$. Alors pour tout
$a\in\calA^\heartsuit(\bar k)$, la courbe $\widetilde X_{\bar \rho,a}$ est
connexe.
\end{proposition}

\begin{proof}
Rappelons un r\'esultat de connexit\'e g\'en\'eral de Fulton et Hansen \cf
\cite[th\'eor\`eme 4.1]{FL}.

\begin{theoreme} Soient $M,N$ des $\bar k$-sch\'emas irr\'eductibles
avec des morphismes propres $f:M\rightarrow \bbP^h$ et $g:N\rightarrow
\bbP^h$ tels que $\dim(f(M))+\dim(g(N))>h$. Alors $M\times_{\bbP^h} N$ est
connexe.
\end{theoreme}

Voici comment appliquer ce r\'esultat \`a notre cas particulier. En passant
au rev\^etement $\bar\rho:\bar X_{\bar\rho}\rightarrow \bar X$, $\frakc_D$
se d\'ecompose en une somme directe $\frakc_D=\bigoplus_{i=1}^r
D^{\otimes e_i}$. La section $a\in \Gamma(X,\frakc_D)$ d\'efinit alors des
sections $a_i\in \Gamma(\bar X_{\bar \rho}, D^{\otimes e_i})$ pour tous
$i=1,\ldots,r$. En tant que sch\'ema au-dessus de $\bar X_{\bar \rho}$,
$$\frakc_D \times_X \bar X_{\bar \rho}= \bar X_{\bar \rho}\times\bbc_D$$
est un produit des fibr\'es en droites  $D^{\otimes e_i}$. Compactifions
chacun de ces fibr\'es en droites $D^{\otimes e_i}$ en un fibr\'e en droites
projective $\bbP^1(D^{\otimes e_i}\oplus \calO_{X_\rho})$. Leur produit
$$\bbP^{e_1,\ldots,e_r}_{X_\rho}(D)=
\bbP^1(D^{\otimes e_1}\oplus \calO_{\bar X_{\bar \rho}}) \times_{\bar
X_{\bar \rho}}\times\cdots \times_{\bar X_{\bar \rho}} \bbP^1(D^{\otimes
e_r}\oplus \calO_{\bar X_{\bar \rho}})
$$
est une compactification de $\frakc_D \times_X \bar X_{\bar \rho}$. Il est
g\'eom\'etriquement connexe car $\bar X_{\bar \rho}$ l'est par hypoth\`ese.
Soit $M$ la normalisation de $\bbP^{e_1,\ldots,e_r}_{X_\rho}(D)$ dans le
rev\^etement fini et plat $\bar X_{\bar \rho}\times \bbt_D \rightarrow \bar
X_{\bar \rho}\times\bbc_D$.

Consid\'erons aussi le produit au-dessus de $k$
$$\bbP^{e_1,\ldots,e_r}(D)=
\bbP^1(D^{\otimes e_1}\oplus \calO_{\bar X_{\bar \rho}}) \times\times\cdots \times
\bbP^1(D^{\otimes e_r}\oplus \calO_{\bar X_{\bar \rho}})
$$
qui contient $\bbP^1(D^{\otimes e_i}\oplus \calO_{\bar X_{\bar \rho}})$
comme un sous-sch\'ema ferm\'e. Chaque facteur $\bbP^1(D^{\otimes e_i})$
est muni d'un fibr\'e en droites $\calO(1)$ tr\`es ample relativement \`a
$k$ car sous l'hypoth\`ese $\deg(D)>2g$. Il induit un plongement de
$\bbP^1(D^{\otimes e_i})$ dans un espace projectif $\bbP^{h_i}$. La
composante $a_i\in \Gamma(\bar X_{\bar \rho},D^{\otimes e_i})$ de $a$
d\'efinit un hyperplan $H_{a_i}$ dans ce projectif. L'image de la section
$$a:\bar X_{\bar \rho}\rightarrow \bar X_{\bar \rho}\times\bbc_D$$
est alors l'image r\'eciproque dans $\bbP^{e_1,\ldots,e_r}_{\bar X_{\bar
\rho}}(D)$ de
$$H_1 \times\cdots \times H_r \subset \bbP^{h_1}\times\cdots\times \bbP^{h_r}.$$
Il s'ensuit que $\widetilde X_{\bar \rho,a}$ est l'image r\'eciproque dans $M$
du m\^eme produit d'hyperplans. Il ne reste qu'\`a prendre le plongement
de Veronese de $\bbP^{h_1}\times\cdots\times \bbP^{h_r}$ pour \^etre en
position d'appliquer le th\'eor\`eme de Fulton-Hansen qui implique que la
courbe $\widetilde X_{\bar \rho,a}$ est connexe.
\end{proof}

\subsection{L'ouvert $\calA^\diamondsuit$}
\label{subsection : A diamond}

Commen{\c c}ons par \'etudier l'ouvert $\calA^\diamondsuit$ o\`u les fibres
de $\calM_a$ sont aussi simples que possible. Cet ouvert est d\'efini comme
suit. Un point $a\in \calA(\bar k)$ appartient \`a cet ouvert si la section
$h_a:\bar X \rightarrow \frakc_D$ coupe transversalement le diviseur
$\discrim_{D}$. Ici $\discrim_{D}$ d\'esigne le diviseur de $\frakc_D$
obtenu en tordant le diviseur $\discrim\subset \frakc$ par le
$\GG_m$-torseur $L_D$ associ\'e au fibr\'e inversible $D$.

\begin{proposition} \label{A diamond non vide}
Si $\deg(D)>2g$, l'ouvert $\calA^\diamondsuit$ est non vide.
\end{proposition}

La d\'emonstration est compl\`etement similaire \`a celle du th\'eor\`eme de Bertini due \`a Zariski.
Commen{\c c}ons par d\'emontrer un lemme.

\begin{lemme}\label{deg(D)>2g}
Supposons $\deg(D)>2g$ o\`u $g$ est le genre de $X$. Pour tout $x\in \bar
X(\bar k)$ d\'efini par un id\'eal $\mathfrak m_x$. La fl\`eche
$$\rmH^0(\bar X,\frakc)\rightarrow \frakc \otimes_{\calO_{\bar X}}
\calO_{\bar X}/\mathfrak m_x^2.
$$
est surjective. Ici le fibr\'e vectoriel $\frakc$ est vu comme un
$\calO_X$-module localement libre de rang $r$.
\end{lemme}

\begin{proof}
En passant au rev\^etement fini \'etale $\rho: X_\rho\rightarrow X$,
$\frakc$ devient isomorphe \`a une somme directe
$$\rho^* \frakc =\bigoplus_{i=1}^r \rho^* D^{\otimes e_i}$$
o\`u $e_i$ sont d\'efinis dans \ref{Chevalley} et en particulier sont des
entiers plus grand ou \'egal \`a $1$. Comme $\frakc$ est un facteur direct
de $\rho_* \rho^* \frakc$, il suffit de d\'emontrer la surjectivit\'e de la
fl\`eche
$$\rmH^0(\bar X',\rho^*\frakc)\rightarrow \rho^*\frakc \otimes_{\calO_{\bar X}}
\calO_{\bar X}/\mathfrak m_x^2.
$$
Il suffit de d\'emontrer cette surjectivit\'e pour chacun des facteurs directs
$\rho^*D^{\otimes e_i}$.

Il est loisible de supposer ici que $X_\rho$ est g\'eom\'etriquement
connexe. Notons $g'$ son genre. On a alors $2g'-2=n(2g-2)$ o\`u $n$ est le
degr\'e de $\rho$. La surjectivit\'e ci-dessus se d\'eduit de l'in\'egalit\'e
$$\deg(D^{\otimes e_i}) > 2ng = (2g'-2)+ 2n.$$
Le lemme en r\'esulte.
\end{proof}

\begin{proof} Revenons \`a la proposition \ref{A diamond non vide}.
Dans $\discrim_{G,D}$, on a un ouvert lisse
${\discrim_{D}}-\discrim_{D}^{\rm sing}$ compl\'ement d'un ferm\'e
$\discrim_{D}^{\rm sing}$ de codimension $2$ dans $\frakc_D$.

Consid\'erons le sous-sch\'ema $Z_1$ de
$({\discrim_{G,D}}-\discrim_{G,D}^{\rm sing})\times \calA$ consitiu\'e des
couples $(c,a)$ tels que la section $a(\bar X)$ passe par le point $c$ et
intersecte avec le diviseur ${\discrim_G}$ en ce point avec une
multiplicit\'e au moins $2$. D'apr\`es le lemme ci-dessus
$$\dim(Z_1)\leq \dim(\calA)-1$$
de sorte que la projection $Z_1\rightarrow \calA$ n'est pas surjective.

Consid\'erons le sous-sch\'ema $Z_2$ de $\discrim_G^{\rm sing}\times \calA$
des couples $(c,a)$ tels que la section $a(\bar X)$ passe par $c$. De
nouveau d'apr\`es le lemme, on a une estimation de dimension
$$\dim(Z_2)\leq \dim(\calA)-1$$
de sorte que la r\'eunion des images de $Z_1$ et de $Z_2$ est contenu dans
un sous-sch\'ema ferm\'e strict de $\calA$. Il existe donc un point
$a\in\calA^{\heartsuit}$ telle qu la section $a(\bar X)$ ne coupe pas le lieu
singulier $\discrim_{G,D}^{\rm sing}$ du discriminant et coupe le lieu lisse
de ce diviseur transversalement.
\end{proof}

\begin{lemme}\label{cameral lisse}
Un point $a\in\calA(\bar k)$ est dans l'ouvert $\calA^\diamondsuit(\bar k)$
si et seulement si la courbe cam\'erale $\widetilde X_a$ est lisse.
\end{lemme}

\begin{proof}
Supposons que $a\in\calA^\diamondsuit(\bar k)$ c'est-\`a-dire la section
$a(\bar X)$ dans $\frakc_D$ coupe transversalement le diviseur
$\discrim_{G,D}$. Montrons que l'image inverse de cette section sur le
rev\^etement $X_\rho\times \bbt_D$ est lisse. En dehors du diviseur
$\discrim_{G,D}$, ce rev\^etement est \'etale si bien qu'il n'y rien \`a
v\'erifier. La consition $a(\bar X)$ dans $\frakc_D$ coupe transversalement
le diviseur $\discrim_{G,D}$ implique en particulier qu'elle ne coupe pas le
lieu singulier $\discrim_{G,D}^{\rm sing}$ de $\discrim_{G,D}$. Un couple
$(v,x)\in \frakt_D(\bar k)$ est compos\'e d'un point $v\in X(\bar k)$ et d'un point $x$ dans la
fibre de $\frakt_D$ au-dessus de $x$. Au-dessus de $v$, le groupe $G$ est
d\'eploy\'e de sorte qu'on peut parler des hyperplans de racines dans la fibre
de $\frakt_D$ au-dessus de $v$. Si $(v,x)\in \frakt_D(\bar k)$ est au-dessus
d'un point intersection de $a(\bar X)$ avec
$\discrim_{G,D}-\discrim_{G,D}^{\rm sing}$, $x$ appartient \`a un
unique hyperplan de racines. On peut donc se ramener au cas d'un groupe
de rang semi-simple un. Dans ce cas, un calcul direct montre que le
compl\'et\'e formel $\widetilde X_a$ en $(v,x)$ est de la forme $\bar
k[[\epsilon_v]] [t]/ (t^2-\epsilon_v^m)$ o\`u $\epsilon_v$ est un
uniformisant de $\bar X$ en le point $v$ et $m$ est la multiplicit\'e
d'intersection de $a(\bar X)$ avec $\discrim_{G,D}$ en ce point. Dans le cas
transversal $m=1$ ceci implique que $\widetilde X_a$ est lisse en $(\tilde
v,x)$.

Supposons maintenant que $a\notin \calA^\diamondsuit(\bar k)$. Si $a(\bar
X)$ coupe le lieu lisse de $\discrim_{G,D}$ avec une multiplicit\'e plus
grand que un, le calcul ci-dessus montrer que $\widetilde X_a$ n'est pas
lisse. Supposons maintenant que $a(\bar X)$ coupe le lieu $\discrim_{G,D}^{\rm sing}$
en un point $v\in \bar X$. Supposons que $\widetilde X_a$ est lisse en le
point $(v,x)\in \frakt_D(\bar k)$ au-dessus de $v$. Le point $x$ appartient
alors \`a au moins deux hyperplans de racine diff\'erents de sorte que le
groupe de monodromie locale $\pi_a^\bullet(I_v)$ \cf \ref{pi_a^bullet}
contient deux involutions diff\'erentes. Ceci n'est pas possible car sous
l'hypoth\`ese que la caract\'eristique de $k$ ne divise l'ordre de $\bbW$, le
groupe de monodromie locale $\pi_a^\bullet(I_v)$ est un groupe cyclique.
\end{proof}

\begin{corollaire}\label{irreductible}
Soit  $\bar X_{\bar \rho}\rightarrow \bar X$ un rev\^etement fini \'etale
galoisien connexe. Supposons $\deg(D)>2g$. Alors pour tout $a\in
\calA^\diamondsuit(\bar k)$, la courbe $\widetilde X_{\bar\rho,a}$ est
irr\'eductible.
\end{corollaire}

\begin{proof}
D'apr\`es \ref{connexe}, $\widetilde X_{\rho,a}$ est connexe. Comme
$\widetilde X_a$ est lisse, le rev\^etement \'etale $\widetilde X_{\rho,a}$
l'est aussi. Elle est donc irr\'eductible.
\end{proof}

\begin{definition} Un champ ab\'elien est le quotient d'une vari\'et\'e
ab\'elienne par l'action triviale d'un groupe diagonalisable.
\end{definition}

L'exemple typique d'un champ ab\'elien est le champ classifiant les fibr\'es
inversibles de degr\'e z\'ero sur une courbe projective lisse
g\'eom\'etriquement connexe. Si cette courbe est munie d'une action d'un
groupe fini d'ordre premier \`a la caract\'eristique, les champs de Prym
associ\'es sont des champs ab\'eliens.

\begin{proposition}\label{A lozenge}
Pour tout $a\in \calA^\diamondsuit(\bar k)$, l'ouvert $\calM_a^\reg$ est
$\calM_a$ tout entier de sorte que $\calM_a$ est un torseur sous $\calP_a$.
La composante neutre de $\calP_a$ est un champ ab\'elien.
\end{proposition}

On renvoie \`a \cite[proposition 4.2]{N} pour la d\'emonstration de la
premi\`ere assertion. La d\'emonstration de la seconde assertion est
report\'ee \`a fin du paragraphe \ref{subsection : Modele de Neron global}.

\subsection{Cas lin\'eaire}\label{Cas lineaire}

On va analyser les fibres de $\calM$ et de $\calP$ au-dessus d'un point
$a\in\calA^\heartsuit(\bar k)$. Commen{\c c}ons par le cas du groupe
lin\'eaire : soit $G=\GL(r)$. Dans ce cas, on peut d\'ecrire les fibres de
Hitchin \`a l'aide des courbes spectrales  en suivant Hitchin \cite{H} et
Beauville-Narasimhan-Ramanan \cite{BNR}. On pourrait faire de m\^eme
pour les groupes classiques \cf \cite{H} et \cite{compagnon}.

\begin{numero}
Dans le cas $G=\GL(r)$, l'espace affine $\calA$ est l'espace vectoriel
$$\calA=\bigoplus_{i=1}^r \rmH^0(X,D^{\otimes i}).$$
La donn\'ee d'un point $a=(a_1,\ldots,a_r)\in\calA(\bar k)$ d\'etermine une
courbe spectrale $Y_a$ trac\'ee sur l'espace total $\Sigma_D$ du fibr\'e en
droites $D$. Cette courbe est donn\'ee par l'\'equation
$$t^r-a_1 t^{r-1}+\cdots+(-1)^r a_r=0.$$
On d\'efinit un ouvert $\calA^\heartsuit$ de $\calA$ dont les points
g\'eom\'etriques $a\in\calA^\heartsuit(\bar k)$ d\'efinissent une courbe
spectrale r\'eduite. Si $a\in\calA^\heartsuit(\bar k)$, la fibre de Hitchin
$\calM_a$ est le groupo\"ide $\ovl\Pic(Y_a)$ des $\calO_{Y_a}$-modules sans
torsion de rang un d'apr\`es \cite{BNR}. La fibre $\calP_a$ est le groupo\"ide
$\Pic(Y_a)$ des $\calO_{Y_a}$-modules inversibles. $\Pic(Y_a)$ agit sur
$\ovl\Pic(Y_a)$ par produit tensoriel et $\ovl\Pic(Y_a)$ contient
$\Pic(Y_a)$ comme un ouvert.
\end{numero}

\begin{numero}
Si la courbe spectrale $Y_a$ est lisse alors il n'y pas de diff\'erence entre
$\Pic(Y_a)$ et $\ovl\Pic(Y_a)$ c'est-\`a-dire $\calP_a$ agit simplement
transitivement sur $\calM_a$. De plus, dans ce cas, la structure de
$\calP_a$ est aussi simple que possible. Le groupe des composantes
connexes de $\calP_a$ est isomorphe \`a $\ZZ$ par l'application degr\'e car
$Y_a$ est connexes. La composante neutre de $\calP_a$ est le quotient de
la jacobienne de $Y_a$ qui est une vari\'et\'e ab\'elienne, par le groupe
$\GG_m$ agissant trivialement.
\end{numero}

\begin{numero}
Soit $\xi:Y_a^\flat\rightarrow Y_a$ la normalisation de $Y_a$. Le foncteur
d'image r\'eciproque induit un homomorphisme
$$\xi^*:\Pic(Y_a)\rightarrow \Pic(Y_a^\flat)$$
qui induit un isomorphisme
$$\pi_0(\Pic(Y_a)) \isom \pi_0(\Pic(Y_a^\flat))=\ZZ^{\pi_0(Y_a^\flat)}.$$
Le noyau de $\xi^*$ est un groupe affine commutatif de dimension
$$\delta_a=\dim \rmH^0(Y_a, \xi_* \calO_{Y_a^\flat}/\calO_{Y_a}).$$
\end{numero}

\begin{numero}
Par construction, $Y_a$ est un courbe trac\'ee sur une surface lisse en
particulier n'a que des singularit\'es planes. D'apr\`es Altman, Iarrobino et
Kleiman \cite{AIK}, $\Pic(Y_a)$ est alors un ouvert dense de
$\ovl\Pic(Y_a)$.
\end{numero}

Nous allons maintenant g\'en\'eraliser la discussion ci-dessus \`a un groupe
r\'eductif g\'en\'eral.

\subsection{Mod\`ele de N\'eron global}
\label{subsection : Modele de Neron global}

Dans ce paragraphe, on fixe un point $a\in\calA^\heartsuit(\bar k)$ et on
note $U$  l'image r\'eciproque de l'ouvert $\frakc^\rs_D$ par le morphisme
$a:\bar X\rightarrow \frakc^\rs_D$.

Comme dans le cas local \ref{subsection : Modele de Neron}, la structure du
champ de Picard $\calP_a$ des $J_a$-torseurs sur $\bar X$ peut \^etre
analys\'ee \`a l'aide du mod\`ele de N\'eron $J_a^\flat$ de $J_a$. C'est un
sch\'ema en groupes lisse de type fini au-dessus de $\bar X$ muni d'un
homomorphisme $J_a \rightarrow J_a^\flat$ qui est un isomorphisme
au-dessus de $U$. Il est de plus caract\'eris\'e par la propri\'et\'e suivante :
pour tout sch\'ema en groupes lisse de type fini $J'$ sur $\bar X$ avec un
homomorphisme $J_a \rightarrow J'$ qui est un isomorphisme sur $U$, il
existe un unique homomorphisme $J'\rightarrow J_a^\flat$ tel que le
triangle \'evident commute. L'existence de ce mod\`ele de N\'eron est un
r\'esultat de Bosch, Lutkebohmer et Raynaud \cf \cite{BLR} qu'ils l'appellent
le mod\`ele de N\'eron de type fini \`a distinguer avec le mod\`ele de
N\'eron localement de type fini.

Ce mod\`ele de N\'eron global s'obtient \`a partir des mod\`eles de N\'eron
locaux \cf \ref{subsection : Modele de Neron} comme suit. Pour
tout $v\in \overline X-U$, on note $\bar X_v$ la compl\'etion de $\bar X$
en $v$ et $\bar X_v^\bullet= \bar X_v-\{v\}$. Consid\'erons le mod\`ele de N\'eron du
tore $J_a |_{{\bar X}_v^\bullet}$. En recollant les mod\`eles de N\'eron en les
diff\'erents points $v\in \overline X-U$ avec le tore $J_a|_U$, on obtient
un sch\'ema en groupes commutatifs lisse $J_a^\flat$ au-dessus $\bar X$
muni d'un homomorphisme de sch\'emas en groupes $J_a\rightarrow
J_a^\flat$.

Comme dans \ref{J a flat}, $J_a^\flat$ peut \^etre exprim\'e \`a l'aide de la
normalisation $\widetilde X_a^\flat$ de la courbe cam\'erale $\widetilde
X_a$. L'action de $W$ sur $\widetilde X_a$ induit une action de ce groupe
sur $\widetilde X_a^\flat$. Notons $\pi_a^\flat:\widetilde X_a^\flat
\rightarrow \bar X$ le morphisme vers $\bar X$. Voici la cons\'equence
globale de l'\'enonc\'e local \ref{J a flat}.

\begin{corollaire}\label{J a flat global}
$J_a^\flat$ s'identifie au sous-groupe des points fixes sous l'action
diagonale de $W$ dans $\prod_{\widetilde X_a^\flat /\bar X} (T \times_{\bar
X} \widetilde X_a^\flat )$.
\end{corollaire}

La variante suivante sera utile dans la suite. Consid\'erons une r\'eduction
de $\bar \rho_G$ donn\'ee sous la forme $\bar\rho:\bar X_{\bar
\rho}\rightarrow \bar X$ torseur sous $\Theta_{\bar \rho}$. Pour tout $a\in
\calA^\heartsuit(\bar k)$, on a un rev\^etement fini plat
$\pi_{\bar\rho,a}:\widetilde X_{\bar\rho,a}\rightarrow \bar X$ comme dans
\ref{tilde X rho a} qui est g\'en\'eriquement \'etale galoisien de groupe de
Galois $\bbW\rtimes \Theta_{\bar\rho}$. Soit $\widetilde
X_{\bar\rho,a}^\flat$ la normalisation de $\widetilde X_{\bar\rho,a}$ qui est
aussi munie d'une action de $\bbW\rtimes \Theta_{\bar\rho}$. Soit
$\pi_{\bar\rho,a}^\flat$ sa projection sur $\bar X$. On a alors
$$J_a^\flat=\prod_{\widetilde X_{\bar\rho,a}^\flat /\bar X}
(\bbT \times \widetilde X_{\rho,a}^\flat )^{\bbW\rtimes \Theta_{\bar\rho}}.$$

Consid\'erons le groupo\"ide de Picard  $\calP_a^\flat$ des
$J_a^\flat$-torseurs. L'homomor\-phisme de sch\'emas en groupes
$J_a\rightarrow J_a^\flat$ induit un homomorphisme de groupo\"ides de
Picard $\calP_a\rightarrow \calP_a^\flat$. Cet homomorphisme r\'ealise
essentiellement la structure g\'en\'erale d'un groupe alg\'ebrique sur un
corps alg\'ebri\-quement clos comme l'extension d'une vari\'et\'e ab\'elienne
par un groupe alg\'ebrique affine \cf \cite{Ros}. La d\'emonstration qui suit
s'inspire d'un argument de Raynaud \cite{Ray}.

\begin{proposition}\label{P a flat}
\begin{enumerate}
\item L'homomorphisme $\calP_a(\overline k)\rightarrow
    \calP_a^\flat(\overline k)$ est essentiellement surjectif.

\item La composante neutre $(\calP_a^\flat)^0$ de $\calP_a^\flat$ est un
    champ ab\'elien.

\item Le noyau $\calR_a$ de $\calP_a\rightarrow \calP_a^\flat$ est un
    produit de groupes alg\'ebriques affines de type fini $\calR_v(a)$ qui
    sont d\'efinis dans le lemme \ref{R_v(a)}. Ceux-ci sont triviaux sauf en
    un nombre fini de points $v\in |\bar X|$.

\end{enumerate}
\end{proposition}

\begin{proof}
\begin{enumerate}
\item Par construction m\^eme des mod\`eles de N\'eron,
    l'homomorphisme $J_a\rightarrow J_a^\flat$ est injectif en tant
    qu'homomorphisme entre faisceaux en groupes ab\'eliens pour la
    topologie \'etale de $\overline X$. Consid\'erons la suite exacte
\begin{equation}\label{suite courte}
1\rightarrow J_a \rightarrow J_a^\flat \rightarrow J_a^\flat/J_a
\rightarrow 1
\end{equation}
o\`u le quotient $J_a^\flat/J_a$ est support\'e par le ferm\'e fini
 $\overline X-U$ et la suite exacte longue de cohomologie qui s'en d\'eduit.
 Comme $\rmH^1(\bar X,J_a^\flat/J_a)=0$, l'homomorphisme
$$\rmH^1(\overline X,J_a)\rightarrow \rmH^1(\overline
X,J_a^\flat)$$ est surjectif.

\item Soit $\widetilde X_a$ l'image r\'eciproque du rev\^etement
    $\frakt_{D}\rightarrow\frakc_{D}$ par le morphisme
    $h_a:\overline X\rightarrow \frakc_{D}$. Soit $\widetilde
    X_a^\flat$ la normalisation de $\widetilde X_a$. D'apr\`es la
    proposition \ref{J a flat}, le mod\`ele de N\'eron $J_a^\flat$ ne
    d\'epend que du rev\^etement $\widetilde X_a^\flat$. Plus
    pr\'ecis\'ement $J_a^\flat$ consiste en les points fixes sous l'action
    diagonale de $W$ dans la restriction \`a la Weil $\prod_{\widetilde
    X_a^\flat/\overline X}(T\times_{\bar X} \widetilde X_a^\flat)$. Il en
    r\'esulte qu'\`a isog\'enie pr\`es, $\calP_a^\flat$ est un facteur du
    groupo\"ide des $T$-torseurs sur $\widetilde X_a^\flat$ lequel est
    isomorphe au produit de $r$ copies du $\Pic(\widetilde X_a^\flat)$.
    Puisque $\widetilde X_a^\flat$ est une courbe projective lisse
    \'eventuellement non connexe, la composante neutre de
    $\Pic(\widetilde X_a^\flat)$ est le quotient d'un produit de vari\'et\'es
    jacobiennes par un produit de $\GG_m$ agissant trivialement.

\item La derni\`ere assertion r\'esulte aussi de la suite exacte longue de
    cohomologie qui se d\'eduit de la suite exacte courte (\ref{suite
    courte}). Ayant d\'efini le noyau comme la cat\'egorie des
    $J_a$-torseurs munis d'une trivialisation du $J_a^\flat$-torseur qui
    s'en d\'eduit, on n'a pas en fait \`a se pr\'eoccuper des $\rmH^0$ dans la
    suite longue.
\end{enumerate}
Ceci termine la d\'emonstration de la proposition.
\end{proof}

\begin{proof}
Revenons maintenant \`a la se\-conde assertion de
\ref{A lozenge}. Soit $a\in\calA^\diamondsuit(\bar k)$. La courbe cam\'erale
$\widetilde X_a$ est alors une courbe lisse de sorte que le conoyau de
l'homomorphisme $J_a\rightarrow J_a^\flat$ est un faisceau de groupes
finis support\'e par un sous-sch\'ema fini de $\overline X$. Il s'ensuit que
$\calP_a\rightarrow \calP_a^\flat$ est une isog\'enie de groupo\"ides de Picard
c'est-\`a-dire un homomorphisme essentiellement surjectif \`a noyau fini.
Puisque $\calP_a^\flat$ est un champ ab\'elien, il en est de m\^eme de
$\calP_a$.
\end{proof}

\subsection{Invariant $\delta_a$}
Pour tout $a\in\calA^\heartsuit(\bar k)$, consid\'erons le noyau
\begin{equation}\label{R_a}
\calR_a:={\rm ker}[\calP_a \longrightarrow \calP_a^\flat]
\end{equation}
Il classifie les $J_a$-torseurs sur $X$ dont le $J_a^\flat$-torseur qui s'en
d\'eduit est muni d'une trivialisation. C'est un groupe alg\'ebrique affine de
type fini qui se d\'ecompose en produit
$$\calR_a=\prod_{v\in\bar X-U} \calR_v(a)$$
o\`u $\calR_v(a)$ est le groupe d\'efini en \ref{R_v(a)}.

On d\'efinit l'{\em invariant $\delta_a$} comme la dimension de $\calR_a$
qui s'\'ecrit donc comme une somme d'invariants $\delta$ locaux
$$\delta_a:=\dim(\calR_a)=\sum_{v\in \bar X-U} \delta_v(a).$$
La conjonction de \ref{P a flat} et de  la formule de dimension du groupe
des sym\'etries locales \ref{delta local} donne une formule pour l'invariant
$\delta$ global.

\begin{corollaire}\label{delta global}
Pour tout $a\in\calA(\overline k)$, l'invariant $\delta_a$ d\'efini comme
ci-dessus est \'egal \`a
$$
\delta_a=\dim \rmH^0(\bar X, \frakt\otimes_{\calO_{\bar X}} (\pi_{a*}^\flat
\calO_{\widetilde X^\flat_a}/\pi_{a*} \calO_{\widetilde X_a}))^W.$$
\end{corollaire}

De nouveau, on a une autre formule qui exprime l'invariant $\delta$ global en
fonction du discriminant corrig\'e par des invariants monodromiques locaux comme
dans la formule de Bezrukavnikov.

Rappelons que le discriminant $\discrim_G$ est un polyn\^ome homog\`ene
de degr\'e $\sharp\,\Phi$ sur $\frakc$, $\sharp\,\Phi$ \'etant le nombre de
racines dans le syst\`eme de racines $\Phi$. Il s'ensuit que pour tout
$a\in\calA^\heartsuit(\bar k)$, $a^* \discrim_{G,D}$ est un diviseur
lin\'eairement \'equivalent \`a $D^{\otimes (\sharp\Phi)}$ de sorte que
$$\deg(a^* \discrim_{G,D})=\sharp\, \Phi \deg(D).$$
Ecrivons
$$a^* \discrim_{G,D}=d_1 v_1 +\cdots + d_n v_n$$
o\`u $v_1,\ldots,v_n$ sont des points deux \`a deux distincts de $\bar X$
et o\`u $d_i$ est la multiplicit\'e de $v_i$. Pour tout $i=1,\ldots,n$, notons
$c_i$ la chute du rang torique de $J_a^\flat$ en le point $v_i$.
La formule suivante est un corollaire imm\'ediat de \ref{dimension Bezru}.

\begin{proposition}\label{dimension Bezru globale}
On a l'\'egalit\'e
$$2\delta_a=\sum_{i=1}^n (d_i-c_i)=\sharp\,\Phi\deg(D)-\sum_{i=1}^n c_i.$$
\end{proposition}

\subsection{Le groupe $\pi_0(\calP_a)$}
\label{subsection : pi_0(P_a)}

Dans ce paragraphe, nous allons d\'ecrire le groupe de composantes
connexes de $\calP_a$ dans l'esprit de la dualit\'e de Tate-Nakayama. Pour
cela, il est n\'ecessaire de faire un certain nombre de choix et de fixer quelques notations.

Fixons un point $\infty\in X(\bar k)$. Consid\'erons l'ouvert $\calA^\infty$
de $\bar\calA$ qui consiste en les points $a\in\calA(\bar k)$ tels que
$a(\infty)\in \bar\frakc_D^\rs$. C'est un sous-sch\'ema ouvert de
$\bar\calA^\heartsuit$. Si $\infty$ est d\'efini sur $k$, alors l'ouvert
$\calA^\infty$ est aussi d\'efini sur $k$.

On choisit un d\'eploiement de $G$ au-dessus de $\bar X$ comme dans
\ref{deploiement}. Ceci consiste en un homomorphisme
$\bar\rho^\bullet:\pi_1(\bar X,\infty)\rightarrow \Theta_{\bar\rho}$ qui
rel\`eve $\bar \rho_G^\bullet:\pi_1(\bar X,\infty)\rightarrow \Out(\GG)$.
Cet homomorphisme d\'etermine un rev\^etement fini \'etale galoisien
$\bar\rho:\bar X_{\bar\rho}\rightarrow \bar X$ de groupe de Galois
$\Theta_{\bar\rho}$ muni d'un point $\infty_{\bar\rho}\in \bar X_{\bar\rho}$
au-dessus de $\infty$. On a alors un rev\^etement fini et plat
$$\pi_{\bar \rho}:\bar X_{\bar \rho}\times \bbt \rightarrow \bar\frakc_D$$
qui au-dessus de l'ouvert $\bar\frakc_D^\rs$ est fini \'etale galoisien de
groupe de Galois $\bbW\rtimes \Theta_{\bar\rho}$.

Soit $a\in \calA^\infty(\bar k)$. Consid\'erons le rev\^etement fini et plat
$\pi_{\bar\rho,a}:\widetilde X_{\bar\rho,a}\rightarrow\bar X$ d\'efini en
formant le produit cart\'esien \ref{tilde X rho a}. G\'en\'eriquement, c'est un
rev\^etement fini \'etale galoisien de groupe de Galois $\bbW\rtimes
\Theta_{\bar\rho}$. Par construction, il est fini et \'etale au-dessus de
$\infty$. Choisissons un point $\tilde\infty_{\bar\rho}$ de $\widetilde
X_{\bar\rho,a}$ au-dessus de $\infty_{\bar\rho}$. Ce choix
permet d'identifier la fibre de $J_a=a^* J$ au-dessus de $\infty$ avec le
tore fixe $\bbT$ \`a l'aide de \ref{J pi_rho_G}.

Soit $\widetilde X_{\bar\rho,a}^\flat$ la normalisation de $\widetilde
X_{\bar\rho,a}$. Puisque $\tilde\infty_{\bar\rho}$ est un point lisse de
$\widetilde X_{\bar\rho,a}$, il d\'etermine un point de sa normalisation.
Notons $\tilde a=(a,\tilde \infty_{\bar\rho})$. Soit $C_{\tilde a}$ la
composante de $\widetilde X_{\bar\rho,a}^\flat$ qui contient le point
$\tilde\infty$. Soit $W_{\tilde a}$ le sous-groupe de $\bbW\rtimes
\Theta_{\bar\rho}$ qui laisse stable cette composante. Consid\'erons aussi
le sous-groupe normal $I_{\tilde a}$ de $W_{\tilde a}$ engendr\'e par les
\'el\'ements de $W_{\tilde a}$ ayant au moins un point fixe dans $C_{\tilde
a}$.

Soit $J_a^0$ le sous-sch\'ema en groupes des composantes neutres de
$J_a$. Consid\'erons le champ de Picard $\calP'_a$ des $J_a^0$-torseurs sur
$\bar X$. L'homomorphisme de faisceaux $J_a^0\rightarrow J_a$ induit un
homomorphisme de champs de Picard
$\calP'_a\rightarrow \calP_a$.

\begin{lemme}\label{P'_a P_a}
L'homomorphisme $\calP'_a\rightarrow \calP_a$  est surjectif et a un noyau
fini. Il en est de m\^eme de l'homomorphisme
$\pi_0(\calP'_a)\rightarrow \pi_0(\calP_a)$ qui s'en d\'eduit.
\end{lemme}

\begin{proof}
On a une suite exacte courte de faisceaux
$$0\rightarrow J_a^0 \rightarrow J_a \rightarrow \pi_0(J_a) \rightarrow 0$$
o\`u $\pi_0(J_a)$ est un faisceau de support fini dont la fibre en un point
$v\in\bar X$ est  le groupe $\pi_0(J_a)_v$ des composantes connexes de la
fibre $J_{a,v}$ de $J_a$. On en d\'eduit une suite exacte longue
$$\rmH^0(\bar X,\pi_0(J_a))\rightarrow \rmH^1(\bar X,J_a^0)
\rightarrow \rmH^1(\bar X,J_a)\rightarrow \rmH^1(\bar X,\pi_0(J_a))=0.$$
L'annulation du dernier terme  r\'esulte du fait que
$\pi_0(J_a)$ est support\'e par un sch\'ema  de dimension z\'ero. On en
d\'eduit la surjectivit\'e de $\calP'_a\rightarrow \calP_a$ de noyau
$\rmH^0(\bar X,\pi_0(J_a))$. La finitude de ce noyau vient de la finitude
des fibres de $\pi_0(J_a)$. L'assertion sur $\pi_0(\calP'_a)\rightarrow
\pi_0(\calP_a)$ s'en suit imm\'ediatement.
\end{proof}

Au lieu des groupes ab\'eliens $\pi_0(\calP'_a)$ et $\pi_0(\calP_a)$, il sera
plus commode de d\'ecrire les groupes diagonalisables duaux. Dualement,
on a une inclusion des groupes diagonalisables
$$\pi_0(\calP_a)^*\subset \pi_0(\calP'_a)^*.$$
o\`u on a not\'e
\begin{eqnarray*}
\pi_0(\calP_a)^*&=&\Spec(\Ql[\pi_0(\calP_a)]) \\
\pi_0(\calP'_a)^*&=&\Spec(\Ql[\pi_0(\calP'_a)]).
\end{eqnarray*}
On utilisera l'exposant $(\_)^*$ pour d\'esigner la dualit\'e entre les groupes
ab\'eliens de type fini et les groupes diagonalisables de type fini sur $\Ql$.

\begin{proposition}\label{pi 0 global}
Pour tout $\tilde a=(a,\tilde\infty_{\bar\rho})$ comme ci-dessus, on a des
isomorphismes de groupes diagonalisables
$$\pi_0(\calP'_a)^*=\hat \bbT^{W_{\tilde a}}$$
et
$$\pi_0(\calP_a)^*=\hat\bbT(I_{\tilde a},W_{\tilde a})$$
o\`u $\hat\bbT(I_{\tilde a},W_{\tilde a})$ est le sous-groupe de
$\hat\bbT^{W_{\tilde a}}$ form\'e des \'el\'ements $\kappa$ tels que
$W_{\tilde a}\subset (\bbW\rtimes \Theta_{\bar\rho})_\kappa$ et $I_{\tilde
a}\subset \bbW_\HH$ o\`u $\bbW_\HH$ est le groupe de Weyl de la
composante neutre du centralisateur de $\kappa$ dans $\hat\bbG$.
\end{proposition}

\begin{proof}
D'apr\`es \cite[corollaire 6.7]{N}, on a un isomorphisme canonique
$$
(\bbX_*)_{W_{\tilde a}} \longrightarrow \pi_0(\calP'_a)
$$
du groupe des $W_{\tilde a}$-coinvariants de $\bbX_*=\Hom(\GG_m,\bbT)$
dans le groupe des composantes connexes de $\calP'_a$. Ceci est est
essentiellement un cas particulier d'un lemme de Kottwitz \cite[lemme
2.2]{K-Iso1}. Dualement, on a un isomorphisme de groupes diagonalisables
$$\pi_0(\calP'_a)^*=\hat \bbT^{W_{\tilde a}}$$
o\`u $\hat \bbT$ est le $\Ql$-tore dual de $\bbT$.

Notons $U=a^{-1}(\bar\frakc^\rs_D)$. Comme dans la d\'emonstration de
\ref{P'_a P_a}, on a une suite exacte
$$\rmH^0(\bar X,\pi_0(J_a)) \rightarrow \pi_0(\calP'_{a})\rightarrow \pi_0(\calP_a)\rightarrow 0$$
o\`u $\rmH^0(\bar X,J_a/J_a^0)=\bigoplus_{v\in \bar X-U} \pi_0(J_{a,v})$
o\`u $\pi_0(J_{a,v})$ d\'esigne le groupe des composantes connexes de la
fibre de $J_a$ en $v$. Pour tout $v\in \bar X-U$, on a une suite exacte
locale analogue
$$\pi_0(J_{a,v})\rightarrow \pi_0(\calP_v(J_a^0))\rightarrow \pi_0(\calP_v(J_a))\rightarrow 0$$
compatible avec la suite globale. Consid\'erons les suites duales des groupes
diagonalisables
$$0\rightarrow \pi_0(\calP_v(J_a))^{*} \rightarrow  \pi_0(\calP_v(J_a^0))^{*} \rightarrow
\pi_0(J_{a,v})^{*}$$
Le sous-groupe
$$\pi_0(\calP_a)^*\subset \pi_0(\calP'_a)^*$$
est alors l'intersection des images inverses des sous-groupes
$$\pi_0(\calP_v(J_a))^*\subset\pi_0(\calP_v(J_a^0))^*$$ pour tout
$v\in\bar X-U$. La proposition se d\'eduit maintenant de \ref{pi 0 local}.
\end{proof}

\subsection{Automorphismes}
\label{subsection : Automorphismes}

Soit $(E,\phi)\in\calM(\bar k)$ d'image $a\in\calA^\heartsuit(\bar k)$. Nous
allons d\'eterminer des bornes pour le groupe des automorphismes
$\Aut(E,\phi)$ en fonction de $a$. Soit $U$ l'ouvert $\bar X$ o\`u $a$ est
semi-simple r\'egulier.

Consid\'erons le faisceau des automorphismes $\underline\Aut(E,\phi)$ qui
associe \`a tout $\bar X$-sch\'ema $S$ le groupe $\Aut((E,\phi)|_S)$. Ce
faisceau est repr\'esentable par le sch\'ema en groupes
$I_{(E,\phi)}=h_{(E,\phi)}^*I$ qui est l'image r\'eciproque du centralisateur $I$ sur
$\frakg$ par la fl\`eche $h_{(E,\phi)}:\bar X\rightarrow [\frakg_{D}/G]$. La restriction de $I_{(E,\phi)}$ \`a l'ouvert $U$ est un
tore mais au-dessus de $\bar X$, le sch\'ema en groupes $I_{(E,\phi)}$ n'est
ni lisse ni m\^eme plat. On peut n\'eanmoins consid\'erer sa lissification au
sens de \cite{BLR}. D'apr\`es {\it loc. cit}, il existe un unique sch\'ema en
groupes lisse $I_{(E,\phi)}^{\rm lis}$ au-dessus de $\bar X$ tel que pour tout
$\bar X$-sch\'ema $S$ lisse on a
$$\Aut((E,\phi)|_S)=\Hom_X(S,I_{(E,\phi)}^{\rm lis}).$$
La fl\`eche tautologique $I_{(E,\phi)}^{\rm lis} \rightarrow I_{(E,\phi)}$ est
un isomorphisme au-dessus de l'ouvert $U$. Notons que la caract\'erisation
de $I_{(E,\phi)}^{\rm lis}$ implique l'\'egalit\'e
\begin{equation}
\Aut(E,\phi)=\rmH^0(\bar X,I_{(E,\phi)}^{\rm lis}).
\end{equation}

Puisque $J_a$ est lisse, l'homomorphisme canonique $J_a\rightarrow
I_{(E,\phi)}$ induit un homomorphisme
$$J_a\rightarrow I_{(E,\phi)}^{\rm lis}$$
qui est un isomorphismisme au-dessus de $U$. Par la propri\'et\'e
universelle du mod\`ele de N\'eron, on a un homomorphisme
$$I_{(E,\phi)}^{\rm lis}\rightarrow J_a^\flat$$
qui est un isomorphisme sur $U$.

\begin{proposition}\label{borne automorphisme}
Pour tout $(E,\phi)\in\calM(\bar k)$ au-dessus d'un point
$a\in\calA^\heartsuit(\bar k)$, on a des inclusions canoniques
$$\rmH^0(\bar X,J_a)\subset \Aut(E,\phi) \subset \rmH^0(\bar X,J_a^\flat).$$
\end{proposition}

\begin{proof}
 Il suffit de v\'erifier que les fl\`eches $J_a\rightarrow I_{(E,\phi)}^{\rm lis}$ et
$I_{(E,\phi)}^{\rm lis}\rightarrow J_a^\flat$ sont injectifs en tant
qu'homomomorphismes de faisceaux pour la topologie \'etale. Pour cela, il
suffit de v\'erifier l'injectivit\'e sur les voisinages formels en chaque point
$v\in\bar X-U$. Notons $\bar\calO_v$ le compl\'et\'e formel de $\calO_{\bar
X}$ en $v$ et $\bar F_v$ le corps des fractions de $\bar\calO_v$. On a alors
les inclusions
$$
J_a(\bar\calO_v) \subset I_{(E,\phi)}^{\rm lis}
(\bar\calO_v) \subset J_a^\flat(\bar\calO_v)
$$
de sous-groupes de $J_a(\bar F_v)$.
\end{proof}

Pour d\'ecrire explicitement la borne sup\'erieure $\rmH^0(\bar
X,J_a^\flat)$, faisons des choix comme dans le paragraphe pr\'ec\'edent. En
particulier, on a un point $\infty\in X(\bar k)$ tel que $a(\infty)\in
\bar\frakc_D^\rs$. On choisit aussi un d\'eploiement $\bar\rho:\bar
X_{\bar\rho}\rightarrow \bar X$ de groupe de Galois $\Theta_{\bar \rho}$
avec un point $\infty_\rho$ \cf \ref{deploiement}. On a alors un
rev\^etement fini plat $\pi_{\bar\rho,a}:\widetilde
X_{\bar\rho,a}\rightarrow \bar X$ qui est g\'en\'eriquement \'etale galoisien
de groupe de Galois $\bbW\rtimes \Theta_{\bar\rho}$. En choisissant un
point $\tilde\infty_{\bar\rho}$ de $\widetilde X_{\bar\rho,a}$ au-dessus de
$\infty$, on d\'efinit un sous-groupe $W_{\tilde a}$ de $\bbW\rtimes
\Theta_{\bar\rho}$ et un sous-groupe normal $I_{\tilde a}$ de $W_{\tilde
a}$ comme dans le paragraphe \ref{subsection : pi_0(P_a)}. Avec \ref{J a
flat global} et \ref{J pi_rho_G}, on a la formule
$$\rmH^0(\bar X,J_a^\flat)=\bbT^{W_{\tilde a}}.$$
On en d\'eduit le corollaire suivant.

\begin{corollaire} Soit $\tilde a=(a,\tilde\infty)$ comme ci-dessus.
Pour tout $(E,\phi)\in\calM(\bar k)$ au-dessus d'un point
$a\in\calA^\heartsuit(\bar k)$, ${\rm Aut}(E,\phi)$ s'identifie \`a un
sous-groupe de $\bbT^{W_{\tilde a}}$.
\end{corollaire}

Les travaux r\'ecents de Frenkel et Witten sugg\`erent que cette borne
n'est pas optimale. En fait, on devrait avoir l'inclusion
$${\rm Aut}(E,\phi) \subset \bbT(I_{\tilde a},W_{\tilde a})$$
o\`u $\bbT(I_{\tilde a},W_{\tilde a})$ est le sous-groupe de
$\bbT^{W_{\tilde a}}$ d\'efini comme dans \ref{pi 0 global} en rempla{\c
c}ant $\hat\bbT$ par $\bbT$. De plus, l'\'egalit\'e devrait \^etre atteinte
aux points les plus singuliers de la fibre $\calM_a$.

Soit $\calA^\ani(\bar k)$ le sous-ensemble des points  $a\in\calA(\bar k)$
tels que $\bbT^{W_{\tilde a}}$ soit fini. Cette propri\'et\'e ne d\'epend pas
du choix du point $\tilde\infty$. On d\'emontrera plus loin qu'il existe un
ouvert $\calA^{\ani}$ de $\calA$ dont $\calA^{\ani}(\bar k)$ est l'ensemble
des $\bar k$-points. Au-dessus de cet ouvert, $\calM$ est alors un champ de
Deligne-Mumford.

Calculons enfin le groupe des automorphismes des points de $\calP_a$.

\begin{proposition}\label{automorphisme P_a}
Fixons un d\'eploiement $\bar\rho:\bar X_{\bar\rho}\rightarrow \bar X$ de
groupe de Galois $\Theta_{\bar\rho}$ avec $\bar X_{\bar\rho}$ connexe.
Pour tout $a\in\calA^\heartsuit(\bar k)$, on a
$$\rmH^0(\bar X,J_a)=(Z \GG)^{\bbW\rtimes \Theta_{\bar\rho}}$$
et
$$\rmH^0(\bar X,\Lie(J_a))=\Lie(Z \GG)^{\bbW\rtimes \Theta_{\bar\rho}}.$$
En particulier, si au-dessus de $\bar X$ le centre de $G$ ne contient pas de
tore d\'eploy\'e, alors $\calP_a$ est un champ de Picard de
Deligne-Mumford.
\end{proposition}

\begin{proof}
On dispose d'un homomomorphisme canonique de faisceaux $ZG\rightarrow
J_a$ sur $\bar X$. Il s'agit donc de d\'emontrer que la fl\`eche induite sur
les sections globales est un isomorphisme. D'apr\`es \ref{J=J'}, on a un
homomomorphisme injectif de faisceaux $J_a \rightarrow J_a^1$ de
conoyau fini. De plus, $J^1_a$ est d\'etermin\'e par le rev\^etement
cam\'eral $\pi_a:\widetilde X_{\bar\rho,a} \rightarrow \bar X$ avec la
formule
$$J_a^1=((\pi_a)_*(\bbT\times \widetilde X_{\bar\rho,a}))^{\bbW\rtimes \Theta_{\bar\rho}}.$$
Puisque $\widetilde X_{\bar\rho,a}$ est connexe \cf \ref{connexe}, on a
$$\rmH^0(\bar X,J_a^1)=\bbT^{\bbW\rtimes \Theta_{\bar\rho}}.$$
Le fait que le sous-groupe $\rmH^0(\bar X,J_a)$ de $\rmH^0(\bar X,J_a^1)$
s'identifie \`a $(Z \GG)^{\bbW\rtimes \Theta_{\bar\rho}}$ se d\'eduit de la
description explicite du sous-faisceau $J'$ de $J^1$ dans la proposition
\ref{J=J'}.
\end{proof}

\subsection{Calcul de d\'eformation}
\label{subsection : deformation}

Les d\'eformations des fibr\'es de Higgs ont \'et\'e \'etudi\'ees
par Biswas et Ramanan dans \cite{BR}. Pour la commodit\'e du lecteur, nous
allons reprendre leur calcul.

Rappelons d'abord le calcul usuel des d\'eformations d'un torseur sous un
groupe lisse. Soient $S$ un $k$-sch\'ema et $G$ un $S$-sch\'ema en groupes
lisse. Soit $\bfB G$ le classifiant de $G$. Le $G$-torseur
universel ${\bf E}G$ est alors $S$ au-dessus de $[S/G]$.  Notons
$$\pi_{{\bf E}G}:{\bf E}G\longrightarrow \bfB G$$
le $G$-torseur tautologique. Consid\'erons le triangle distingu\'e des
complexes cotangents
$$\pi_{{\bf E}G}^* L_{\bfB G/ S}\longrightarrow L_{{\bf E}G/ S} \longrightarrow
L_{{\bf E}G/\bfB G} \longrightarrow \pi_{{\bf E}G}^* L_{\bfB
G/ S}[1].
$$
Puisque $S={\bf E} G$, le complexe cotangent $L_{{\bf E}G/ S}$ est nul
alors que $L_{{\bf E}G/{\bfB G}}$ est le fibr\'e vectoriel $\frakg^*$ plac\'e
en degr\'e $0$. Il en r\'esulte un isomorphisme
$$L_{{\bf E}G/\bfB G} \isom \pi_{{\bf E}G}^* L_{\bfB G/ S}[1].$$
On en d\'eduit un isomorphisme
$$\frakg^*[-1] \isom \pi_{{\bf E}G}^* L_{\bfB G/ S}$$
qui par descente le long de $\pi_{{\bf E}G}$ induit un isomorphisme
$$(  {\bf E}G\wedge^G\frakg^*) [-1]\isom L_{\bfB G/ S}.$$
Le complexe cotangent $L_{\bfB G/k}$ du classifiant de $G$ est donc le
fibr\'e vectoriel obtenu en tordant par le torseur ${\bf E} G$ l'espace
vectoriel $\frakg^*$ muni de la repr\'esentation coadjointe, plac\'e en
degr\'e $1$.

Ainsi, pour tout $S$-sch\'ema $X$, pour tout $G$-torseur $E$ sur $X$
correspondant \`a une fl\`eche $h_E:X\rightarrow \bfB G$, l'obstruction \`a
la d\'eformation de $E$ g{\^\i}t dans le groupe
$$\rmH^1(X,\underline{\rm RHom}(h_E^* L_{\bfB G/ S},
\calO_X))=\rmH^2(X,  E\wedge^G \frakg)
$$
et si cette obstruction s'annule, les d\'eformations forment un espace
principal homog\`ene sous le groupe
$$\rmH^0(X,\underline{\rm RHom}(h_E^* L_{\bfB G/ S},
\calO_X))=\rmH^1(X,  E\wedge^G \frakg)
$$
alors que le groupe des automorphismes infinit\'esimaux est
$\rmH^0(X,E\wedge^G \frakg)$.

Soit maintenant $V$ un fibr\'e vectoriel sur $S$ muni d'une action de $G$
et consid\'erons le champ quotient $[V/G]$. Le $G$-torseur
$\pi_V:V\rightarrow [V/G]$ d\'efinit un morphisme $[\nu]:[V/G] \rightarrow
\bfB G$ qui s'ins\`ere dans un diagramme cart\'esien
$$
\xymatrix{
  V \ar[d]_{\pi_V} \ar[r]^{\nu}
                & {\bf E} G \ar[d]^{\pi_{{\bf E}G}}  \\
  [V/G] \ar[r]_{[\nu]}
                & \bfB G             }
$$
Consid\'erons le triangle distingu\'e des complexes cotangents
$$\pi_V^* L_{[V/G]/ S} \longrightarrow L_{V/ S}
\longrightarrow L_{V/[V/G]} \longrightarrow \pi_V^* L_{[V/G]/k}[1].
$$
Le terme $L_{V/ S}$ est fibr\'e vectoriel constant de valeur $\nu^*V^*$
plac\'e en degr\'e $0$ o\`u $V^*$ est le $S$-fibr\'e vectoriel dual de $V$ et
$\nu:V\rightarrow S$ est la projection sur $S$. Le terme $L_{V/[V/G]}$ se
calcule par changement de base $L_{V/[V/G]}=\nu^*L_{{\bf E}G/\bfB G}$ et
est donc le fibr\'e vectoriel $\nu^*\frakg^*$ plac\'e aussi en degr\'e $0$.
Au-dessus de chaque point $v\in V$, l'action de $G$ au voisinage de $v$
d\'efinit une application lin\'eaire
$$\alpha_v:\frakg \longrightarrow T_v V=V$$
dont le dual est la fibre en $v$
$$\alpha_v^*:(L_{V/ S})_v=V^* \longrightarrow (L_{V/[V/G]})_v=\frakg^*.$$
de la fl\`eche $L_{V/k} \rightarrow L_{V/[V/G]}$ du triangle distingu\'e.
Cette fl\`eche descend \`a $[V/G]$ en une fl\`eche
$$\alpha_v^*\wedge^G \pi_V
: \pi_V \wedge^G V^*  \longrightarrow \pi_V \wedge^G \frakg^*
$$
dont le c\^one est isomorphe \`a $L_{[V/G]/ S}$.

Appliquons le calcul ci-dessus pour caluler les d\'eformations des paires de
Hitchin. Reprenons les notations fix\'ees au d\'ebut du chapitre \ref{section
: Fibres de Springer affines}. En particulier, $G$ est le sch\'ema en groupes
r\'eductif sur la courbe $X$ et $\frakg$ son alg\`ebre de Lie qui est munie
de l'action adjointe de $G$ et de l'action de $\GG_m$ par homoth\'etie. Le
fibr\'e inversible $D$ d\'efinit un $\GG_m$ torseur $L_D$. Consid\'erons le
champ $[ \frakg_{D}/ G]$ obtenu en tordant $\frakg$ par
$L_D$ puis divis\'e par $G$. Le complexe cotangent $L_{[ \frakg_{D}/ G]/X}$ s'identifie donc \`a
$$L_{[\frakg/G]/X} \wedge^{\GG_m} L_D$$
qui est le c\^one de
$$ (\pi_{D,\frakg}\wedge^{G}\frakg^* )\otimes D^{-1}
\longrightarrow \pi_{D,\frakg}\wedge^{G}\frakg^*
$$
o\`u $\pi_{D,\frakg}$ est le $G$-torseur \'evident sur le quotient $[ \frakg_{D}/ G]$.

Soit $(E,\phi)$ un champ de Higgs sur $X$ \`a valeur dans $\bar k$. Elle
correspond \`a une fl\`eche
$$
h_{E,\phi}:\bar X\rightarrow [ \frakg_{D}/ G].
$$
La d\'eformation de $(E,\phi)$ est contr\^ol\'ee par le complexe
$$\underline{\rm RHom}(h_{E,\phi}^*(L_D\wedge^{\GG_m} L_{[\frakg/G]/X}  ),
\calO_X)
$$
qui s'exprime maintenant simplement
$$\ad(E,\phi):=[\ad(E) \rightarrow \ad(E)\otimes D]$$
o\`u
\begin{itemize}
\item $\ad(E)$ est le fibr\'e vectoriel $\frakg\wedge^G E$ ;
\item $\ad(E)$ est plac\'e en degr\'e $-1$ en $\ad(E)\otimes D$ est plac\'e
    en degr\'e $0$ ;
\item la fl\`eche est donn\'ee par $x\mapsto [x,\phi]$.
\end{itemize}

Rappelons la proposition 5.3 de \cite{N}. Le lecteur notera une diff\'erence
dans le d\'ecalage de $\ad(E,\phi)$ par rapport \`a {\it loc. cit}. Nous
donnerons ici une d\'emonstration un peu diff\'erente.

\begin{theoreme}\label{lisse}
Soit $(E,\phi)\in \calM(\bar k)$ un point au-dessus d'un point
$a\in\calA^\heartsuit(\bar k)$. Alors le groupe $\rmH^1(X,\ad(E,\phi))$ o\`u
g{\^\i}t l'obstruction \`a la d\'eformation de la paire $(E,\phi)$ est nul dans
l'un des cas suivants
\begin{itemize}
\item $\deg(D)> 2g-2$ ,
\item $\deg(D)= 2g-2$ et $a\in \calA^\ani(\bar k)$.
\end{itemize}
Si l'une de ces deux hypoth\`eses sont v\'erifi\'ees, alors $\calM$ est est
lisse au point $(E,\phi)$.
\end{theoreme}

\begin{proof}
Fixons une forme sym\'etrique non d\'eg\'en\'er\'ee et invariante sur
$\frakg$. On identifie alors le dual du complexe $\ad(E,\phi)$ \`a
$$\ad(E,\phi)^*=[\ad(E)\otimes D^{-1} \rightarrow \ad(E)]$$
dont les deux termes non nuls sont plac\'es en degr\'es $-1$ et $0$ et dont la
diff\'erentielle est donn\'ee par $x\mapsto [x,\phi]$. Le faisceau de
cohomologie $\rmH^{-1}$ de ce complexe s'identifie \`a
$$\Lie(I_{E,\phi}^{\rm lis})\otimes D^{-1}$$
o\`u $I_{E,\phi}^{\rm lis}$ est le sch\'ema en groupes lisse sur $\bar X$
introduit dans le paragraphe \ref{subsection : Automorphismes}. Par
dualit\'e de Serre, le groupe $\rmH^1(X,\ad(E,\phi))$ est dual du groupe
$$\rmH^0(\bar X,\Lie(I_{E,\phi}^{\rm lis})\otimes D^{-1}\otimes \Omega_{X/ k}).$$

Comme dans \ref{subsection : Automorphismes}, on a un homomorphisme
injectif de $\calO_{\bar X}$-modules
$$\Lie(I_{E,\phi}^{\rm lis}) \rightarrow \Lie(J_a^\flat)$$
de sorte que pour d\'emontrer la nullit\'e du groupe des obstructions
envisag\'ee, il suffit de d\'emontrer que
$$\rmH^0(\bar X,\Lie(J_a^\flat) \otimes D^{-1} \otimes \Omega_{X/ k})=0.$$
On est donc amen\'e \`a d\'emontrer le lemme suivant.
\end{proof}

\begin{lemme}\label{annulation J a flat}
Pour tout $a\in\calA^\heartsuit(\overline k)$, $\rmH^0(\bar X, {\rm
Lie}(J_a^\flat)\otimes L)=0$ pour tout fibr\'e en droite $L$ de degr\'e
strictement n\'egatif. La m\^eme conclusion vaut sous l'hypoth\`ese
$a\in\calA^\ani(\bar k)$ et $\deg(L)\leq 0$.
\end{lemme}

\begin{proof}
Choisissons un d\'eploiement $\bar\rho:\bar X_{\bar\rho}\rightarrow \bar X$
connexe de $G$ \cf \ref{deploiement}. On a alors un rev\^etement fini plat
$\widetilde X_{\bar\rho,a} \rightarrow\bar X$ qui est g\'en\'eriquement
\'etale galoisien de groupe de Galois $\bbW\rtimes \Theta_{\bar\rho}$.
Notons $\widetilde X_{\bar\rho,a}^\flat$ la normalisation de $\widetilde
X_{\bar\rho,a} \rightarrow\bar X$ et $\pi_{\bar\rho,a}^\flat$ la projection
sur $\bar X$. D'apr\`es \ref{J a flat global}, $\Lie(J_a^\flat)$ peut \^etre
calcul\'e \`a partir de  $\widetilde X_{\bar\rho,a}^\flat$ avec la formule
$$\Lie(J_a^\flat)=(\pi_{\bar\rho,a}^\flat)_*
(\calO_{\widetilde X_{\bar\rho,a}^\flat}\otimes \bbt)^{\bbW\rtimes \Theta_{\bar\rho}}$$
de sorte que
$$\Lie(J_a^\flat) \otimes L=
(\pi_a^\flat)_*((\pi_a^\flat)^*L\otimes \bbt)^{\bbW\rtimes \Theta_{\bar\rho}}.
$$

Si $\deg(L)<0$, $(\pi_a^\flat)^*L$ est un fibr\'e en droites de degr\'e
strictement n\'egatif sur chaque composante connexe de $\widetilde
X_{\bar\rho,a}^\flat$ et ne peut pas avoir de sections globales non nulles.

Si $\deg(L)=0$, si $(\pi_a^\flat)^*L$ a de sections globales non nulles si et
seulement s'il est isomorphe \`a $\calO_{\widetilde X_{\bar\rho,a}^\flat}$.
On a dans ce cas
$$\rmH^0(\widetilde X_{\bar\rho,a}^\flat,
((\pi_{\bar\rho,a}^\flat)^*L\otimes
\bbt)^{\bbW\rtimes \Theta_{\bar\rho}}=\bbt^{W_{\tilde a}}
$$
o\`u $W_{\tilde a}$ est le sous-groupe de $\bbW\rtimes \Theta_{\bar\rho}$
d\'efini comme dans le paragraphe \ref{subsection : pi_0(P_a)} et qui
d\'epend du choix d'un point $\widetilde\infty$ de $\widetilde X_a^\flat$.
Sous l'hypoth\`ese $a\in\calA^\ani(\bar k)$, le groupe des $W_{\tilde
a}$-invariants $\bbt^{W_{\tilde a}}$ est nul.
\end{proof}

\subsection{Formule de produit}

Nous allons \`a pr\'esent rappeler le lien entre les fibres de Hitchin et les
fibres de Springer affines qui en sont des analogues locaux.

Soit $a\in\calA^\heartsuit(\bar k)$. Soit $U$ l'image r\'eciproque de l'ouvert
r\'egulier semi-simple $\frakc^\rs$ de $\frakc$ par le morphisme $a:\bar
X\rightarrow [\frakc/\GG_m]$. Le recollement avec la section de Kostant
d\'efinit un morphisme
$$
\prod_{v\in\bar X-U}\calM_v(a)\rightarrow \calM_a.
$$
De m\^eme, on a un homomorphisme de groupes $\prod_{v\in \bar X-U}
\calP_v(J_a) \longrightarrow \calP_a$. Ceux-ci induisent un morphisme
$$\zeta:\prod_{v\in\bar X-U} \calM_v(a)
\wedge^{\prod_{v\in\bar X-U}\calP_v(J_a)} \calP_a \longrightarrow \calM_a.
$$
D'apr\`es le th\'eor\`eme 4.6 de \cite{N}, ce morphisme induit une
\'equivalence sur la cat\'egorie des $\bar k$-points. Le lecteur remarquera
des changements de notations par rapport \`a {\it loc. cit.} : $\calM_v(a)$ y
\'etait d\'esign\'e par $\calM_{v,a}^\bullet$ et ce qui y \'etait d\'esign\'e par
$\calM_{v,a}$ n'appara{\^\i}tra plus ici.

\begin{proposition}\label{produit}
Pour tout $a\in\calA^\ani(\bar k)$, le quotient de
$$\prod_{v\in\bar X-U} \calM_v^{\rm red}(a)\times  \calP_a$$
par l'action diagonale de  $\prod_{v\in\bar X-U}\calP_v^{\rm red}(J_a)$ est
un champ de Deligne-Mumford propre. De plus, le morphisme
$$\prod_{v\in\bar X-U}
\calM_v^{\rm red}(a) \wedge^{\prod_{v\in\bar X-U}\calP_v^{\rm red}(J_a)} \calP_a \rightarrow
\calM_a$$ est un hom\'eomorphisme.
\end{proposition}

\begin{proof}
L'homomorphisme $\calP_v(J_a) \rightarrow \calP_a$ induit un
homomorphisme $\pi_0(\calP_v(J_a)) \rightarrow \pi_0(\calP_a)$ sur les
groupes des composantes connexes. Puisque $a\in\calA^\ani(\bar k)$,
$\pi_0(\calP_a)$ est un groupe fini, le noyau de cet homomorphisme est un
sous-groupe d'indice finie de $\pi_0(\calP_v(J_a))$. Il existe donc un
sous-groupe ab\'elien libre d'indice finie $\Lambda_v$ de
$\pi_0(\calP_v(J_a))$ contenu dans ce noyau.

Choisissons un rel\`evement $\Lambda_v \rightarrow \calP_v(J_a)$.
Puisque $a$ est d\'efini sur un corps fini, quitte \`a remplacer $\Lambda_v$
par un sous-groupe d'indice finie, on peut supposer que $\Lambda_v$ est
contenu dans le noyau de $\calP_v(J_a) \rightarrow \calP_a$.

Le groupe $\prod_{v\in\bar X-U}\Lambda_v$ agit sur $\prod_{v\in\bar X-U}
\calM_v^{\rm red}(a)\times  \calP_a$ en agissant librement sur le premier
facteur et trivialement sur le second facteur. Le quotient est
$$\prod_{v\in\bar X-U} (\calM_v^{\rm red}(a)/ \Lambda_v)\times  \calP_a
$$
o\`u chaque $\calM_v^{\rm red}(a)/ \Lambda_v$ est un $\bar k$-sch\'ema
projectif d'apr\`es Kazhdan et Lusztig \cf \ref{quotient projectif}.

Il reste \`a quotienter par $\prod_{v\in\bar X-U} (\calP_v(J_a)/
\Lambda_v)$. Pour tout $v$, l'homomorphisme $\calR_v(a)\rightarrow
\calP_v(J_a)/ \Lambda_v$ est injectif et de noyau fini. Rappelons qu'on a
une suite exacte
$$1\rightarrow \calR_a \rightarrow \calP_a \rightarrow \calP_a^\flat \rightarrow 1$$
o\`u $\calP_a^\flat$ est un champ de Deligne-Mumford propre. Le quotient
de
$$\prod_{v\in\bar X-U} (\calM_v^{\rm red}(a)/ \Lambda_v)\times  \calP_a$$
par l'action diagonale de $\calR_a=\prod_v \calR_v(a)$ est donc une
fibration localement triviale au-dessus de $\calP_a^\flat$ de fibres
isomorphes \`a $$\prod_{v\in\bar X-U} (\calM_v^{\rm red}(a)/ \Lambda_v).$$
Ce quotient est donc un champ de Deligne-Mumford propre.

Il reste donc \`a quotienter par le groupe fini $\prod_v \calP_v(J_a)/
(\calR_v(a)\times \Lambda_v)$. Le quotient final est aussi un champ de
Deligne-Mumford propre.

Puisque $\calM_a$ est un champ de Deligne-Mumford, en particulier
s\'epar\'e, le morphisme
$$\zeta:\prod_{v\in\bar X-U}
\calM_v^{\rm red}(a) \wedge^{\prod_{v\in\bar X-U}\calP_v^{\rm red}(J_a)} \calP_a
\rightarrow \calM_a
$$
est un morphisme propre. Puisqu'il induit une \'equivalence sur les $\bar
k$-points, c'est donc un hom\'eomorphisme. En particulier, $\calM_a$ est un
champ de Deligne-Mumford propre.
\end{proof}

On s'attend \`a ce que l'\'enonc\'e ci-dessus s'\'etend \`a $a\in
\calA^\heartsuit(\bar k)$.

\subsection{Densit\'e}\label{dense}
On va maintenant \'enoncer et d\'emontrer l'analogue global de \ref{densite
locale}.

\begin{proposition}\label{densite globale}
Pour tout point g\'eom\'etrique $a\in\calA^\heartsuit(\bar k)$, la fibre
$\calM^\reg_a$ de est dense dans la fibre $\calM_a$.
\end{proposition}

\begin{proof} La
formule de produit \ref{produit} implique que le ferm\'e compl\'ementaire
$\calM_a^\reg$ dans $\calM_a$ est de dimension strictement plus petite
que $\calM_a^\reg$. Comme l'espace total de Hitchin $\calM$ est lisse sur
$k$ d'apr\`es \ref{lisse}, les fibres de Hitchin $\calM_a$ sont localement
une intersection compl\`ete. En particulier, elles ne peuvent pas admettre
des composantes irr\'educ\-tibles de dimension strictement plus petite que la
sienne. Ceci d\'emontre que $\calM_a^\reg$ est dense dans $\calM_a$.
\end{proof}

La d\'emonstration ci-dessus est essentiellement la m\^eme sur celle de
Altman, Iarrobino et Kleiman dans \cite{AIK} pour la densit\'e de la
jacobienne dans la jacobienne compactifi\'ee d'une courbe projective
r\'eduite irr\'educ\-tible ayant des singularit\'es planes.

\begin{corollaire}
La partie r\'eguli\`ere $\calM_v^\reg(a)$ de la fibre de Springer
$\calM_v(a)$ est dense.
\end{corollaire}

On commence par construire une situation globale \`a partir de la situation
locale donn\'ee en proc\'edant comme dans \cf \ref{lemme fondamental}.
L'assertion locale se d\'eduit alors de l'assertion globale \`a l'aide de la
formule de produit \cf \ref{produit}. Nous laissons au lecteur les d\'etails de
la d\'emonstration de ce corollaire qui ne sera pas utilis\'e dans la suite de
l'article.

\begin{corollaire}
Pour tout $a\in\calA^\heartsuit(\bar k)$, $\calM_a$ est \'equidimension\-nelle
de dimension \'egale \`a
$$\sharp\,\Phi\deg(D) / 2 + r(g-1).$$
De plus, l'ensemble des composantes irr\'eductibles de $\calM_a$ s'identifie
au groupe $\pi_0(P_a)$.
\end{corollaire}

\begin{proof}
La formule de dimension r\'esulte de \ref{dim fibre}. L'identification de
l'ensemble des composantes irr\'eductibles de $\calM_a$ avec le groupe
$\pi_0(P_a)$ est rendue possible par la section de Kostant.
\end{proof}

\begin{corollaire}
Si $\deg(D)>2g-2$, le morphisme $f^\heartsuit:\calM^\heartsuit\rightarrow
\calA^\heartsuit$ est un morphisme plat de dimension relative $d$. Ses
fibres sont g\'eom\'e\-triquement r\'eduites.
\end{corollaire}

\begin{proof}
D'apr\`es \ref{lisse}, $\calM^\heartsuit$ et $\calA^\heartsuit$ sont lisses sur
$k$. Pour d\'emontrer que $f$ est plat, il suffit alors de d\'emontrer que la dimension des
fibres v\'erifient l'\'egalt\'e
$$\dim(\calM_a)=\dim(\calM)-\dim(\calA).$$
Ceci d\'ecoule des \'egalit\'es $\dim(\calM_a)=\dim(P_a)$,
$\dim(\calM)=\dim(P)$ et pour $P$ lisse sur $\calA$ \cf \ref{P lisse}, on a
l'\'egalit\'e
$$\dim(\calP^\heartsuit)=\dim(\calA^\heartsuit)+\dim(\calP_a).$$
Comme dans la d\'emonstration de \ref{densite globale}, on sait que la fibre
$\calM_a$ est localement une intersection compl\`ete. Puisque qu'elle
admet un ouvert dense lisse $\calM_a^{\rm reg}$, elle est n\'ecessairement
r\'eduite.
\end{proof}

\subsection{Le cas des groupes endoscopiques}
\label{subsection : Fibration de Hitchin de groupes endoscopiques}

Soit $(\kappa,\rho_\kappa)$ une donn\'ee endoscopique de $G$ au-dessus de
$X$ \cf \ref{donnee endoscopique}. Soit $H$ le groupe endoscopique
associ\'e. Comme dans \ref{subsection : Transfert des classes}, on a un
morphisme $\nu:\frakc_H \rightarrow \frakc$. En tordant par $D$, on
obtient un morphisme $\nu:\frakc_{H,D} \rightarrow \frakc_D$. En prenant
les points \`a valeurs dans $X$, on obtient un morphisme qu'on note encore
$$\nu:\calA_H \rightarrow \calA.$$
On sait d'apr\`es \cite[7.2]{N} que la restriction de ce morphisme \`a
l'ouvert $\calA^\heartsuit$ ce morphisme est fini et non ramifi\'e.

\begin{numero}\label{codim A_H}
Notons
$$r_H^G(D)=(|\Phi|-|\Phi_H|)\deg(D)/ 2.$$
D'apr\`es \ref{dim A}, on sait
$$\dim(\calA)-\dim(\calA_H)=r_H^G(D)$$
de sorte que l'image de $\calA_H^{G-\heartsuit}=\nu^{-1}(\calA^\heartsuit)$
dans $\calA^\heartsuit$ est un sous-sch\'ema ferm\'e de codimension
$r_H^G(D)$.
\end{numero}

\begin{numero}
Au-dessus de $\calA_H$, on dispose de la fibration de Hitchin du groupe $H$
$$f_H:\calM_H\rightarrow \calA_H.$$ On a \'egalement le champ de Picard
$\calP_H\rightarrow \calA_H$ agissant sur $\calM_H$. Il n'y a pas de relation
directe entre $\calM$ et $\calM_H$ mais $\calP$ et $\calP_H$ sont reli\'es
de fa{\c c}on simple. Soit $a_H\in\calA_H(\bar k)$ d'image
$a\in\calA^\heartsuit(\bar k)$. L'homomorphisme $\mu:\nu^* J\rightarrow
J_H$ de \ref{J JH} induit un homomorphisme $J_a \rightarrow J_{H,a_H}$
qui est g\'en\'eriquement un isomorphisme. On obtient donc un
homomorphisme surjectif
$$\calP_a\rightarrow \calP_{H,a_H}$$
de noyau
$$\calR_{H,a_H}^G=\rmH^0(\ovl X,J_{H,a_H}/J_a)$$
qui est un groupe affine de dimension
$$\dim(\calR_{H,a_H}^G)=\dim(\calP_a)-\dim(\calP_{H,a_H}).$$
D'apr\`es la formule \ref{dim fibre}, on \'egalement
$$\dim(\calR_{H,a_H}^G)=(|\Phi|-|\Phi_H|)\deg(D)/ 2=r_H^G(D).$$
\end{numero}

\begin{numero}\label{delta-dellta_H}
Soit $J_{H,a_H}^\flat$ le mod\`ele de N\'eron de $J_{H,a_H}$. On a des
homomorphismes
$$J_a \rightarrow J_{H,a_H}\rightarrow J_{H,a_H}^\flat$$
qui sont des isomorphismes sur un ouvert non-vide de $\ovl X$. Il s'ensuit
que $J_{H,a_H}^\flat$ est aussi le mod\`ele de N\'eron de $J_a$. En
combinant avec \ref{R_a}, on a la suite exacte
$$1\rightarrow \calR_{H,a_H}^G \rightarrow \calR_a \rightarrow \calR_{H,a_H} \rightarrow 1.$$
On en d\'eduit la formule de dimension
$$\delta_a-\delta_{H,a_H}=r_H^G(D)$$
o\`u $\delta_a=\dim(\calR_a)$ et $\delta_{H,a_H}=\dim(\calR_{H,a_H})$
sont les invariants $\delta$ de $a$ et de $a_H$ par rapport aux groupes $G$
et $H$ respectivement.
\end{numero}

\subsection{Le cas des groupes appari\'es}
\label{Hitchin apparies} Soient $G_1$ et $G_2$ deux $X$-sch\'emas en
groupes appari\'es au sens de \ref{apparies}. On a alors un isomorphisme
$\frakc_{G_1}= \frakc_{G_2}$ \cf \ref{c_G_1=c_G_2}. On en d\'eduit un
isomorphisme $\frakc_{G_1,D}= \frakc_{G_2,D}$ puis un isomorphisme entre
les bases des fibrations de Hitchin
$$\calA=\calA_{1}= \calA_{2}.$$
Il n'y a pas de relation directe entre les fibrations
$f_{1}:\calM_{1}\rightarrow\calA_{1}$ et
$f_{2}:\calM_{2}\rightarrow\calA_{2}$ de $G_1$ et $G_2$ mais une relation
entre les champs de Picard associ\'es $\calP_{1}$ et $\calP_{2}$.

\begin{proposition}\label{isogenie P1 P2}
Il existe un  homomorphisme entre $\calA$-champs de Picard
$$\calP_{1}\rightarrow \calP_{2}$$
qui induit une isog\'enie entre leurs composantes neutres.
\end{proposition}

\begin{proof}
Comme dans la d\'emonstration de \ref{c_G_1=c_G_2}, on a un
isomorphisme $\frakt_1\isom \frakt_2$ au-dessus de l'isomorphisme
$\frakc_{G_1}\isom \frakc_{G_2}$. En vertu de \ref{J=J'}, on en d\'eduit un
isomorphisme entre les composantes neutres $J_{G_1}^0 \isom J_{G_2}^0$
des centralisateurs r\'eguliers de $G_1$ et $G_2$. Puisque $J_{G_1}$ et
$J_{G_2}$ sont des sch\'emas en groupes de type fini, en composant
l'isomorphisme $J_{G_1}^0 \isom J_{G_2}^0$ avec la multiplication par un
entier $N$ assez divisible, on obtient un homomorphisme $J_{G_1}^0
\rightarrow J_{G_2}^0$ qui s'\'etend en un homomorphisme $J_{G_1}
\rightarrow J_{G_2}^0$ de noyau et de conoyau finis. On en d\'eduit un
homomorphisme entre $\calP_{1}\rightarrow \calP_{2}$ qui induit une
isog\'enie entre les composantes neutres.
\end{proof}

\section{Stratification par normalisation en famille}
\label{section : stratification}

Dans la suite de l'article, nous aurons besoin de deux stratifications de la base de Hitchin $\calA$ dont l'une est relative au groupe des composantes connexes de la fibre $\calP_a$ et l'autre est relative \`a l'invariant $\delta_a$ qui est en quelques sortes la dimension de la partie affine de $\calP_a$. Ces stratifications existent par des arguments g\'en\'eraux de semi-continuit\'e. Le but de ce chapitre est de comprendre comment ces deux invariants sont reli\'es \`a la courbe cam\'erale et de rendre la construction de ces stratifications plus explicites.

L'observation de base est la suivante : comme ces deux invariants sont compl\`etement d\'etermin\'es par le rev\^etement cam\'eral $\tilde X_a\rightarrow \bar X$, dans le cas o\`u les normalisations de la courbe cam\'erale vivent en famille, les invariants $\delta_a$ et $\pi_0(\calP_a)$ sont localement constants. On va donc d\'efinir la stratification par normalisation en famille ayant au-dessus de chaque strate une courbe cam\'erale qui se normalise en famille apr\`es le changement de base \`a un rev\^etement radiciel de la strate. Cette stratification est alors plus fine que les stratifications relatives \`a $\delta_a$ et \`a $\pi_0(\calP_a)$. On commencera par discuter le cas $\GL(r)$ o\`u on se ram\`ene \`a au probl\`eme bien connu de normalisation en famille des courbes sur une surface.

Un r\'esultat cl\'e de ce chapitre est la proposition \ref{codimension} qui donne une borne inf\'erieure de la codimension de la strate $\calA_\delta$ sous une hypoth\`ese sur $\deg(D)$. En caract\'eristique $0$, on dispose une d\'emonstration de cette in\'egalit\'e fond\'ee sur le caract\`ere symplectique de la fibration de Hitchin sans le recours \`a l'hypoth\`ese sur $\deg(D)$. On pr\'esentera cet argument dans un autre article o\`u l'accent sera mis plus sur la caract\'eristique $0$. En caract\'eristique positive, cet argument ne marche pas. Voici comment on va contourner l'obstacle. Goresky, Kottwitz et MacPherson  ont fait des calculs de codimension similaires dans le cadre local \cf \cite{GKM-codim}. Le calcul global se ram\`ene au calcul local pourvu que $\deg(D)$ soit grand.

On introduira un ouvert \'etale $\tilde\calA$ de $\calA$ au-dessus duquel $\pi_0(\calP)$ devient un quotient d'un faisceau constant. Avec cette rigidit\'e suppl\'ementaire, la d\'ecomposition endoscopique qu'on \'etudiera dans le chapitre suivant, aura une forme agr\'eable sur $\tilde\calA$.

\subsection{Normalisations en famille des courbes spectrales}
Nous allons introduire dans ce chapitre une stratification de la base de
Hitchin $\calA$ adapt\'ee au champ de Picard $\calP$. Le cas g\'en\'eral va
\^etre model\'e sur le cas du groupe lin\'eaire dont on va discuter
maintenant.

Dans le cas $G=\GL(r)$, on peut associer \`a tout point
$a\in\calA^\heartsuit(\bar k)$ une courbe r\'eduite $Y_a$ trac\'ee sur
l'espace total du fibr\'e en droites $D$ \cf \ref{Cas lineaire}. Le groupe de
sym\'etries $\calP_a$ est alors le champ de Picard $\Pic(Y_a)$ des
$\calO_{Y_a}$-modules inversibles. La structure de $\Pic(Y_a)$ peut \^etre
analys\'ee \`a l'aide de la normalisation de $Y_a$. Soit
$\xi:Y_a^\flat\rightarrow Y_a$ la normalisation de $Y_a$. Le foncteur
$\calL\mapsto \xi^*\calL$ qui associe \`a tout $\calO_{Y_a}$-module
inversible son image inverse par $\xi$ d\'efinit un homomorphisme
$\Pic(Y_a)\rightarrow \Pic(Y_a^\flat)$. La suite exacte longue de
cohomologie associ\'ee \`a la suite exacte courte de faisceaux sur $Y_a$
$$
1\rightarrow \calO_{Y_a}^\times \rightarrow \xi_*\calO_{Y_a^\flat}^\times
\rightarrow \xi_*\calO_{Y_a^\flat}^\times/\calO_{Y_a}^\times\rightarrow 1
$$
nous fournit des renseignements suivants. Le noyau de $\xi^*$ qui  consiste en la cat\'egorie des $\calO_{Y_a}$-modules inversibles
$\calL$ munis d'une trivialisation de $\xi^*\calL$, est un groupe
alg\'ebrique dont l'ensemble des points est
$\rmH^0(Y_a,\xi_{*}\calO_{Y_a^\flat}^\times/ \calO_{Y_a}^\times)$. De plus, le
foncteur $\xi^*$ est essentiellement surjectif.

On peut attacher \`a cette situation deux invariants :
\begin{itemize}
\item l'ensemble $\pi_0(Y_a^\flat)$ des composantes connexes de
    $Y_a^\flat$ ;
\item l'entier
$\delta_a=\dim \rmH^0(Y_a,\xi_*\calO_{Y_a^\flat}/\calO_{Y_a})$
appel\'e l'invariant $\delta$ de Serre.
\end{itemize}
En prenant le degr\'e sur les composantes de $Y_a^\flat$, on obtient un
homomorphisme
$$\pi_0(\Pic(Y_a))\rightarrow \ZZ^{\pi_0(Y_a^\flat)}$$
qui est un isomorphisme. Quant \`a l'invariant $\delta$, il mesure la
dimension du groupe affine $\ker(\xi^*)$.

Teissier a introduit dans \cite{T} l'espace de module des normalisations en
famille d'une famille de courbes planes. Rappelons d'abord la d\'efinition
d'une normalisation en famille.

\begin{definition}\label{normalisation en famille}
Soit $y:Y\rightarrow S$ un morphisme propre, plat et de fibres r\'eduites de
dimension un. Une normalisation en famille de $Y$ est un morphisme
propre birationnel $\xi:Y^\flat\rightarrow Y$ qui est un isomorphisme
au-dessus d'un ouvert $U$ de $Y$ dense dans chaque fibre de $Y$ au-dessus
de $S$ et tel que le compos\'e $y \circ\xi$ est un morphisme propre et
lisse.
\end{definition}

Dans une normalisation en famille les invariants $\delta$ et $\pi_0(\calP)$
restent localement constants.

\begin{proposition}\label{localement constant}
\begin{enumerate}
\item L'image directe $y_* (\xi_*\calO_{Y^\flat}/\calO_Y)$ est un
    $\calO_S$-module localement libre de type fini.
\item Il existe un faisceau $\pi_0(Y^\flat)$ localement constant pour la
    topologie \'etale de $S$ dont la fibre en chaque point g\'eom\'etrique
    $s\in S$ est l'ensemble des composantes connexes de $Y^\flat_s$.
\end{enumerate}
\end{proposition}

\begin{proof}
Soit $s$ un point g\'eom\'etrique de $S$. La restriction de
$\xi_*\calO_{Y^\flat}/\calO_Y$ est support\'e par $Y_s-U_s$ qui est un
sch\'ema de dimension z\'ero. Il s'ensuit que
$\rmH^1(Y_s,\xi_*\calO_{Y^\flat}/\calO_Y)=0$ et
$$\dim \rmH^0(Y_s,\xi_*\calO_{Y^\flat}/\calO_Y)$$
est la diff\'erence entre le genre arithm\'etique de $Y_s$ et celui de
$Y_s^\flat$. C'est donc une fonction localement constante en $s$. Il reste
\`a appliquer le th\'eor\`eme de changement de base pour les faisceaux
coh\'erents \cite[5, corollaire 2]{Mum}.

La seconde assertion est une cons\'equence du fait que $\xi_* \Ql$ est un
faisceau  localement constant.
\end{proof}

Nous allons nous restreindre aux cas des courbes spectrales. Soit $\calB$
l'espace de module des normalisations en famille des courbes
spectrales $Y_a$. Il associe \`a tout $k$-sch\'ema $S$ le groupo\"ide
$\calB(S)$ des triplets $(a,Y_a^\flat,\xi)$ o\`u $a\in \calA^\heartsuit(S)$ un
$S$-point de $\calA^\heartsuit$ o\`u $Y_a^\flat$ est une $S$-courbe
projective lisse et o\`u $\xi:Y_a^\flat\rightarrow Y_a$ est une
normalisation en famille de la courbe spectrale $Y_a$ associ\'ee \`a $a$. Le
foncteur $\calB$ est repr\'esentable par un $k$-sch\'ema de type fini.

Le morphisme d'oubli $\calB\rightarrow\calA^\heartsuit$ induit une
bijection au niveau des $\bar k$-points. En effet, pour tout
$a\in\bbA^\heartsuit(\bar k)$, la normalisation $Y_a^\flat$ de $Y_a$ est
uniquement d\'etermin\'ee. Toutefois, $\calB$ a plus de composantes
connexes que $\calA$. En effet, les deux invariants $\pi_0(\widetilde Y_a)$
et $\delta_a$ sont localement constants d'apr\`es la proposition
pr\'ec\'edente. Soit $\Psi$ l'ensemble des composantes connexes de $\calB\otimes_{k}\bar k$.
Pour tout $\psi\in \Psi$, il existe un entier $\delta(\psi)$ \'egal \`a $\delta_a$
pour tout $a\in\calB_\psi(\bar k)$.

\begin{proposition}
Si $\deg(D)$ est grand par rapport \`a $\delta(\psi)$, alors $\calB_\psi$ est
de dimension inf\'erieure ou \'egale \`a $\dim(\calA)-\delta(\psi)$.
\end{proposition}

Teissier \cf \cite{T}, Diaz, Harris \cf \cite{DH}, Fantechi, Gottsche, Van
Straten \cf \cite{FGV} et Laumon \cf \cite{L} ont d\'emontr\'e l'analogue
local de ce r\'esultat. Le passage du local au global n\'ecessite pour le moment
l'hypoth\`ese sur $\deg(D)$. Nous n'utiliserons ce r\'esultat que comme un guide.
A partir du paragraphe suivant, les lettres $\calB$, $\Psi$ et $\psi$ prendront des
significations un peu diff\'erentes.

\subsection{Normalisation en famille des courbes cam\'erales}
\label{foncteur B}

Dans le cas des groupes classiques, on dispose encore des courbes spectrales
\cite{H}, \cite{compagnon}. N\'eanmoins, avec ces courbes spectrales on est
amen\'e \`a travailler cas par cas. On va en fait examiner le cas g\'en\'eral
en \'etudiant l'espace de module des normalisations des courbes cam\'erales
munies de l'action de $W$. Ceci conduira m\^eme dans le cas $\GL(r)$ \`a
construire une stratification diff\'erente, plus fine que celle d\'efinie
l'espace de module des normalisations des courbes spectrales.

Soit $S$ un $k$-sch\'ema. Un $S$-point $a$ de $\calA^\heartsuit$ d\'efinit
un morphisme $a:X\times S\rightarrow\frakc_D$. En prenant l'image
r\'eciproque du rev\^etement $\pi:\frakt_D\rightarrow\frakc_D$, on obtient
un rev\^etement fini plat $\widetilde X_a$ de $X\times S$ qui dans chaque fibre
est g\'en\'eriquement un torseur sous $W$.

Consid\'erons le foncteur $\calB$ qui associe \`a tout $k$-sch\'ema $S$ le
groupo\"ide des triplets $(a,\widetilde X_a^\flat,\xi)$ o\`u :

\begin{itemize}
\item
$a\in \calA^\heartsuit(S)$ est un $S$-point de $\calA^\heartsuit$.

\item $\widetilde X_a^\flat$ est une $S$-courbe propre et lisse munie
    d'une action de $W$ et d'un morphisme $\pi_a^\flat: \widetilde
    X_a^\flat\rightarrow X\times S$ qui est fini, plat et fibre par fibre
    g\'en\'eriquement un torseur sous $W$.

\item $\xi:\widetilde X_a^\flat\rightarrow \widetilde X_a$ est morphisme
    birationnel $W$-\'equivariant qui est une normalisation en famille \cf
    \ref{normalisation en famille}.
\end{itemize}

Soit $b=(a,\widetilde X_a^\flat,\xi)\in\calB(S)$ un point de $\calB$ \`a
valeur dans un sch\'ema connexe $S$. Soit $\pr_S:X\times S\rightarrow S$ la
projection sur $S$. D'apr\`es \ref{normalisation en famille}, l'image directe
$(\pr_S)_*(\xi_*\calO_{\widetilde X_a^\flat}/\calO_{X_a^\flat})$ est un
$\calO_S$-module localement libre. Il en est de m\^eme de
$$((\pr_S)_*(\xi_*\calO_{\widetilde X_a^\flat}/\calO_{X_a^\flat})\otimes_{\calO_X} \frakt)^W$$
sous l'hypoth\`ese que l'ordre de $\bbW$ est premier \`a la caract\'eristique
de $k$. Puisque $S$ est connexe, ce $\calO_S$-module localement libre a
un rang qu'on notera $\delta(b)$.

Pour tout $a\in\calA^\heartsuit(\bar k)$, la courbe cam\'erale $\widetilde
X_a$ admet une unique normalisation $\widetilde X_a^\flat$ qui est alors
une courbe lisse sur $\bar k$. Il en r\'esulte que le morphisme d'oubli
$\calB\rightarrow\calA^\heartsuit$ induit une bijection sur les $\bar
k$-points. Elle induit par ailleurs une injection sur les points \`a valeurs
dans n'importe quel corps par l'unicit\'e de la normalisation. En
caract\'eristique $p$ et sur un corps non parfait, la normalisation peut ne
pas \^etre lisse de sorte que l'application ci-dessus peut ne pas \^etre
surjective quand elle est \'evalu\'ee sur un corps non parfait.

\begin{theoreme}\label{representable}
Le foncteur $\calB$ d\'efini ci-dessus est repr\'esentable par un
$k$-sch\'ema de type fini.
\end{theoreme}

\begin{proof}
Consid\'erons le foncteur $\calB'$ qui associe \`a tout sch\'ema $S$
l'ensemble des couples $(\widetilde X_a^\flat, \gamma)$ o\`u :

\begin{itemize}
\item $\widetilde X_a^\flat$ est une courbe propre et lisse au-dessus de
    $S$ muni d'une action de $W$ munie d'un morphisme $\pi_a^\flat:
    \widetilde X_a^\flat\rightarrow X\times S$ qui est fini, plat et dans
    chaque fibre est g\'en\'eriquement un torseur $W$.
\item $\gamma:\widetilde X_a^\flat \rightarrow \frakt_{D} \times S$ est un morphisme $W$-\'equivariant tel que l'image inverse de l'ouvert r\'egulier
    semi-simple de $\frakt_{D}$ est dense dans chaque fibre.
\end{itemize}
Soit $\calH$ le foncteur qui associe \`a tout sch\'ema $S$ l'ensemble des
classes d'isomorphisme de courbes projectives lisses $\widetilde X_a^\flat$
sur $S$ muni d'une action de $W$ et d'un morphisme $\pi_a^\flat:\widetilde
X_a^\flat \rightarrow X\times S$ comme ci-dessus. Ce foncteur est
repr\'esentable par un $k$-sch\'ema quasi-projectif. Le morphisme
$h:\calB'\rightarrow \calH$ est \'egalement repr\'esentable de sorte que
$\calB'$ est repr\'esentable par un $k$-sch\'ema quasi-projectif.

Par ailleurs, on a un morphisme $\calB\rightarrow\calB'$. Il associe au point
$b=(a,\widetilde X_a^\flat,\xi)\in\calB(S)$ le point $b'=(\widetilde
X_a^\flat, \gamma)$ o\`u $\gamma$ est le compos\'e de $\xi:\widetilde
X_a^\flat \rightarrow \widetilde X_a$ avec l'immersion ferm\'ee $\widetilde
X_a\rightarrow \frakt_{D}\times S$. Pour d\'emontrer
la repr\'esentabilit\'e de $\calB$, il suffit de v\'erifier l'assertion suivante.
\end{proof}

\begin{lemme} Le morphisme $\calB\rightarrow \calB'$ est un isomorphisme.
\end{lemme}

\begin{proof}
Pour d\'emontrer que le morphisme $\calB\rightarrow \calB'$ est un
isomorphisme, on va en construire un inverse. Soit $b'=(\widetilde
X_a^\flat, \gamma)\in\calB'(S)$. Consid\'erons la partie $W$-invariante dans
l'image directe $(\pi_a^\flat)_*\calO_{X_a^\flat}$. C'est un faisceau en
$\calO_{X\times S}$-alg\`ebres finies qui fibre par fibre au-dessus de $S$
est isomorphe g\'en\'eriquement \`a $\calO_X$. Puisque $X$ est normal,
ceci implique que
$$((\pi_a^\flat)_*\calO_{\widetilde X_a^\flat})^{W}=\calO_{X\times S}.$$
En utilisant l'\'egalit\'e $k[\bbt]^\bbW=\bbc$ \cf \ref{Chevalley}, le
morphisme $W$-\'equivariant $\gamma:\widetilde X_a^\flat\rightarrow
\frakt_{D}$ induit un morphisme $a:X\times
S\rightarrow \frakc_{D}$. Soit $\widetilde X_a$ la
courbe cam\'erale associ\'ee \`a $a$. Le morphisme $\gamma$ se factorise
alors par un morphisme
$$\xi:\widetilde X_a^\flat \rightarrow \widetilde X_a$$
qui est fibre par fibre au-dessus de $S$ une normalisation de $X_a$. On a
donc construit un point $b=(a,\widetilde X_a^\flat,\xi)\in\calB(S)$. Le
morphisme $\calB'\rightarrow\calB$ ainsi construit est l'inverse du
morphisme $\calB\rightarrow \calB'$ dans l'\'enonc\'e du lemme.
\end{proof}

Soit $\Psi$ l'ensemble de ses composantes connexes de $\calB\otimes_k \bar
k$. Pour tout $\psi\in\Psi$, soit $\calB_\psi$ le r\'eduit de la composante
connexe de $\calB\otimes_k \bar k$ index\'ee par $\psi$. Le groupe de
Galois $\Gal(\bar k/ k)$ agit sur l'ensemble $\Psi$. Pour tout $\psi\in \Psi$,
soit $\calA_\psi$ l'image de $\calB_\psi$ dans $\calA^\heartsuit\otimes_k
\bar k$. A priori, $\calA_\psi$ est un sous-ensemble constructible d'apr\`es
un th\'eor\`eme de Chevalley. On sait que c'est un sous-ensemble
localement ferm\'e.

\begin{lemme}
Pour tout $\psi\in \Psi$, $\calA_\psi$ est un sous-sch\'ema localement
ferm\'e de $\calA\otimes_k \bar k$ et  $\calB_\psi\rightarrow \calA_\psi$ est
un morphisme fini radiciel.
\end{lemme}

\begin{proof}
Le morphisme $\calB_\psi\rightarrow \calA\otimes_k \bar k$ est un
morphisme quasi-fini car il induit une injection sur les points
g\'eom\'etriques. D'apr\`es le th\'eor\`eme principal de Zariski, il peut se
factoriser en une immersion ouverte d'image dense $\calB_\psi\rightarrow
\bar \calB_\psi$ suivie d'un morphisme fini $\nu:\bar \calB_\psi\rightarrow
\calA\otimes_k \bar k$. Notons $\bar\calA_\psi$ le sous-sch\'ema ferm\'e de
$\calA^\heartsuit \otimes_k\bar k$ image de $\nu$. Il suffit de d\'emontrer
que $\bar\calB_\psi\rightarrow \bar\calA_\psi$ induit une bijection sur les
points g\'eom\'etriques. En effet, si c'est le cas $\calA_\psi$ sera le
compl\'ementaire de l'image de $\bar \calB_\psi-\calB_\psi$ et par
cons\'equent est un ouvert de $\bar\calA_\psi$. De plus, $\calB_\psi=
\nu^{-1}(\calA_\psi)$ si bien que le morphisme $\calB_\psi\rightarrow
\calA_\psi$ est un morphisme fini radiciel.

Soit $a\in\bar\calA_\psi(\bar k)$ tel qu'il existe deux points diff\'erents
$b,b'\in \bar\calB_\psi(\bar k)$ d'image $a$. Soit $\tilde a$ le point
g\'en\'erique de $\bar\calA_\psi$. L'image inverse $\nu^{-1}(\tilde a)$ est un
\'epaississement du point g\'en\'erique de $\calB_\psi$. Faisons passer un
trait formel $S$ avec le point g\'en\'erique en $\tilde a$ et le point ferm\'e
en $a$. L'image inverse $\nu^{-1}(S)$ est un sch\'ema fini au-dessus de $S$
avec dans sa fibre sp\'eciale deux points distincts $b$ et $b'$. Elle a donc
au moins deux composantes connexes contenant respectivement $b$ et
$b'$ par le lemme de rel\`evement d'idempotents. Comme le morphisme
$\calB_\psi \rightarrow \calA_\psi$ induit une bijection sur les points \`a
valeur dans un corps alg\'ebriquement clos, l'une de ses deux composantes
connexes a une fibre g\'en\'erique vide. Ceci contredit l'hypoth\`ese que
$b$ et $b'$ sont dans l'adh\'erence de $\calB_\psi$.
\end{proof}

\begin{proposition}\label{stratification}
Il existe une unique stratification
$$\calA^\heartsuit\otimes_k \bar k=\bigsqcup_{\psi\in \Psi}\calA_\psi$$
par des sous-sch\'emas localement ferm\'es r\'eduits connexes $\calA_\psi$
telle que
\begin{enumerate}
\item au-dessus de chaque strate $\calA_\psi$, la courbe cam\'erale $\widetilde X_\psi$ admet une
    normalisation en famille apr\`es un changement de
    base fini radiciel.
\item pour tout $\bar k$-sch\'ema $S$ connexe r\'eduit qui est muni d'un
    morphisme $s:S\rightarrow\calA^\heartsuit\otimes_k \bar k$ tel que la
    restriction de la courbe cam\'erale \`a $S$ admet une normalisation
    famille, alors le morphisme $s$ se factorise par l'un des $\calA_\psi$.
\end{enumerate}
\end{proposition}

\begin{proof}
Les $\calA_\psi$ construits dans le lemme pr\'ec\'edent forment une
stratification qui v\'erifie la premi\`ere condition. Si
$s:S\rightarrow\calA^\heartsuit\otimes_k \bar k$ est un morphisme comme
dans la seconde condition, il doit n\'ecessairement se factoriser par $\calB$.
Comme $S$ est connexe, et il se factorise par l'un des $\calB_\psi$. Il
s'ensuit qu'il se factorise par $\calA_\psi$ de sorte que la seconde condition
est \'egalement satisfaite. Il est clair que les deux conditions caract\'erisent
compl\`etement la stratification.
\end{proof}

Consid\'erons l'ordre partiel dans $\Psi$ dans lequel $\psi'\leq \psi$ si et seulement
si la strate $\calA_{\psi'}$ est contenue dans l'adh\'erence de $\calA_\psi$.
Puisque $\calA^\heartsuit\otimes_k\bar k$ est irr\'eductible, $\Psi$ a un
\'el\'ement  maximal qu'on notera $\psi_G$.

\begin{lemme} Supposons que $\deg(D)>2g$. Alors la grosse strate
$\calA_{\psi_G}$ est \'egale \`a l'ouvert $\calA^\diamondsuit$ de
\ref{subsection : A diamond}. En particulier, $a\in\calA_{\psi_G}$ si et
seulement si la courbe cam\'erale $\psi_G$ est lisse.
\end{lemme}

\begin{proof}
On sait par \ref{A diamond non vide} que $\calA^\diamondsuit$ est un
ouvert non vide de $\calA^\heartsuit$. Il s'ensuit que
$\calA^\diamondsuit\cap \calA_{\psi_G} \not= \emptyset$. D'apr\`es
\ref{cameral lisse}, $a\in\calA^\diamondsuit(\bar k)$ si et seulement si la
courbe cam\'erale est lisse. Cette propri\'et\'e \'etant pr\'eserv\'ee dans une
normalisation en famille de courbe cam\'erale, on en d\'eduit que
$\widetilde X_a$ est lisse pour tout $a\in \calA_{\psi_G}(\bar k)$ d'o\`u
$\calA_{\psi_G}\subset \calA^\diamondsuit(\bar k)$. Par ailleurs, au-dessus
de $\calA^\diamondsuit$, on a une normalisation en famille triviale des
courbes cam\'erales de sorte que $\calA_{\psi_G}=\calA^\diamondsuit$.
\end{proof}

Le groupe $\Gal(\bar k/ k)$ agit sur l'ensemble fini $\Psi$. La r\'eunion des
strates $\calA_\psi$ pour les indices $\psi$ dans une orbite sous $\Gal(\bar
k/ k)$ descend en un sous-sch\'ema localement ferm\'e de
$\calA^\heartsuit$.

\subsection{Stratification \`a $\delta$ constant}
D'apr\`es la formule \ref{delta global}, l'invariant $\delta_a$ est constant dans une
normalisation en famille de base connexe. On en d\'eduit une application
$$\delta:\Psi\rightarrow \NN$$
de l'ensemble $\Psi$ des composantes connexes de $\calB$ dans $\NN$ tel
que pour tout $a\in\calA_\psi(\bar k)$, on a $\delta_a=\delta(\psi)$. Notons
que la fonction $\delta:\Psi\rightarrow \NN$ est une fonction d\'ecroissante.

\begin{lemme}\label{delta decroit}
Soient $\psi,\psi'\in \Psi$ avec $\psi\geq \psi'$. Alors on a $\delta(\psi)\leq
\delta(\psi')$.
\end{lemme}

\begin{proof}
Soit $S=\Spec(R)$ un trait formel de point g\'en\'erique
$\eta=\Spec(k(\eta))$ et de point sp\'ecial $s=\Spec(k(s))$ alg\'ebriquement
clos. Soit $a:S\rightarrow \calA^\heartsuit$ un $S$-point de
$\calA^\heartsuit$ dont le point g\'en\'erique est dans $\calA_\psi$ et dont
le point sp\'ecial est dans $\calA_{\psi'}$. Soit
$$\widetilde X_a \rightarrow S\times X$$
le rev\^etement cam\'eral associ\'e.

Quitte \`a remplacer $\eta$ par une extension ins\'eparable et $S$ par son
normalis\'e dans cette extension, on peut supposer que le morphisme
$a:\eta\rightarrow \calA_\psi$ se factorise par $\calB_\psi$. Alors la
normalisation $\widetilde X_{a,\eta}^\flat$ de la fibre g\'en\'erique
$\widetilde X_{a,\eta}$ est une courbe lisse.

Consid\'erons la normalisation $\widetilde X_a^\flat$ de $\widetilde X_a$
dans $\widetilde X_{a,\eta}^\flat$ et notons $\xi$ le morphisme
$$\xi:\widetilde X_a^\flat \rightarrow \widetilde X_a.$$
Le $R$-sch\'ema $\widetilde X_{a,\eta}^\flat$ est plat. En effet, comme $R$
est un anneau de valuation discr\`ete,  pour qu'un $R$-module $M$ soit
plat, il suffit qu'il soit sans torsion. Par ailleurs, l'op\'eration de
normalisation n'introduit pas de torsion. Le $R$-module
$$\rmH^0(\widetilde X_a, \xi_* \calO_{\widetilde X_a^\flat}/
\calO_{\widetilde X_a})
$$
est alors un $R$-module plat $M$ fini et plat muni d'une action de $W$. En
effet, en tordant par un fibr\'e inversible sur $\widetilde X_a$ de degr\'e
tr\`es n\'egatif puis en prenant la suite exacte longue de cohomologie on
peut r\'ealiser $\rmH^0(\widetilde X_a, \xi_* \calO_{\widetilde X_a^\flat}/
\calO_{\widetilde X_a})$ comme le noyau d'un morphisme surjectif entre
deux $R$-modules plats.

D'apr\`es la formule \ref{delta global}, la fibre g\'en\'erique de ce module
calcule l'invariant $\delta(\psi)$
$$\delta(\psi)=\dim_{k(\eta)} (\frakt\otimes_{\calO_X} M_{\eta})^{W}.$$
En d\'eployant de $G$ dans un voisinage du lieu de ramification de
$\widetilde X_{a,k(s)}\rightarrow X\times S$, la formule ci-dessus se r\'ecrit
$$\delta(\psi)=\dim_{k(\eta)} (\bbt\otimes_{k} M_{\eta})^{\bbW}.$$
Puisque $M$ est $R$-plat et l'ordre de $\bbW$ est premier \`a la
caract\'eristique, on peut d\'ecomposer le $R[\bbW]$-module $M$ en
somme directe de repr\'esentations irr\'eductibles de $\bbW$ de sorte que
$$\dim_{k(\eta)} (\bbt\otimes_k M_{\eta})^{\bbW}=
\dim_{k(s)} (\bbt\otimes_k M_s)^{\bbW}.
$$

La fibre sp\'eciale de $\widetilde X_{a,s}^\flat$ de $\widetilde X_{a}^\flat$
n'est a priori qu'une normalisation partielle de $\widetilde X_a$. Pour
calculer l'invariant $\delta(\psi')$, on est amen\'e \`a prendre la
normalisation compl\`ete $(\widetilde X_{a,s}^\flat)^\flat$ de $\widetilde
X_{a,s}^\flat$. Le $k(s)[\bbW]$-module $M_s$ est alors un facteur direct de
$\rmH^0(\calO_{(\widetilde X_{a,s}^\flat)^\flat}/\calO_{\widetilde
X_{a,s}})$. On en d\'eduit la formule
$$\delta(\psi')\geq \dim_l (\bbt\otimes_k M_l)^{W}$$
d'o\`u le lemme.
\end{proof}

Pour tout entier $\delta\in\NN$, la r\'eunion
$$
\ovl \calA_\delta\otimes_k \bar k:= \bigsqcup_{\delta(\psi)\geq \delta} \calA_\psi
$$
est un ferm\'e de $\calA\otimes_k \bar k$. On va poser $\calA_\delta:=\ovl
\calA_\delta-\ovl\calA_{\delta+1}$ qui est alors un sous-sch\'ema localement
ferm\'e de $\calA$. On obtient ainsi une stratification
\begin{equation}\label{stratification delta}
\calA^\heartsuit\otimes_k \bar k=\bigsqcup_{\delta\in\NN}\calA_\delta
\end{equation}
appel\'ee {\em la stratification \`a $\delta$ constant}.

Consid\'erons la grosse strate $\calA_{\psi_G}$ pour l'\'el\'ement maximal
$\psi_G$ de $\Psi$. Cette strate consiste en les points $a\in\calA(\bar k)$
telle que la courbe cam\'erale $\widetilde X_a$ est lisse. L'invariant
$\delta$ vaut alors $0$. En particulier, la strate $\calA_{0}$ dans la
stratification $\delta$ constant ci-dessus est un ouvert non vide de $\bbA$.
En g\'en\'eral, il contient strictement l'ouvert $\calA_{\psi_G}$, comme le
montre l'exmple suivant.

Dans le cas $G=\GL_r$, la description de $\calP_a$ en termes de la courbe
spectrale implique que $\delta_a=0$ si et seulement si la courbe spectrale
$Y_a$ est lisse. Il peut arriver que la courbe spectrale est lisse avec un
point de branchement $y\in Y_a$ d'ordre sup\'erieur \`a $2$.
Dans ce cas la courbe cam\'erale $\widetilde X_a$ n'est pas lisse.
On voit ainsi que dans le cas $\GL_r$ la stratification
par normalisation en famille des courbes cam\'erales raffine strictement la
stratification par normalisation en famille des courbes spectrales.

\subsection{Stratification par les valuations radicielles}
\label{subsection : valuations radicielles}

Soit $\bar v$ un point g\'eom\'etrique de $\bar X$ et notons $\bar
\calO_{\bar v}$ la compl\'etion de $\bar X$ en $\bar v$ et $\bar F_{\bar v}$
son corps des fractions. En choisissant un uniformisant $\varepsilon_{\bar
v}$, on peut identifier $\bar\calO_{\bar v}$ avec l'anneau $\bar
k[[\varepsilon_{\bar v}]]$. Nous choisissons une trivialisation de la
restriction $\rho_G$ \`a $\bar\calO_{\bar v}$ qui d\'eploie $G$ et qui
fournit en particulier un isomorphisme $W=\bbW$.

Nous allons passer bri\`evement en revue l'analyse \cite{GKM-codim} de la
stratification de $\frakc^\heartsuit(\calO_{\bar v})$ par les valuations
radicielles, due \`a Goresky, Kottwitz et MacPherson. Leurs strates de
valuations radicielles sont plus fines que nos strates d\'efinies par les
normalisations en familles et a fortiori plus fines que les strates \`a $\delta$
constants.

Soient $a\in\frakc^\heartsuit(\bar\calO_{\bar v})$ et $J_a=a^* J$ le sch\'ema
en groupes lisse sur $\bar\calO_{\bar v}$ qui s'en d\'eduit. La fibre
g\'en\'erique de $J_a$ est un tore dont la monodromie peut \^etre d\'ecrite
\`a l'aide du rev\^etement cam\'eral \cf \ref{J=J'}. Soit $\bar F_{\bar
v}^{\rm sep}$ la cl\^oture s\'eparable de $\bar F_{\bar v}$. Soit $x\in
\frakt(\bar F_{\bar v}^{\rm sep})$ un $\bar F_{\bar v}^{\rm sep}$-point de
$\frakt$ d'image $a\in\frakc(\bar F_{\bar v})$. Le choix de ce point d\'efinit
un homomorphisme
$$\pi_a^\bullet:I_{\bar v} \rightarrow \bbW$$
o\`u $I_{\bar v}=\Gal(\bar F_{\bar v}^{\rm sep}/\bar F_{\bar v})$. Puisque la
caract\'eristique de $k$ ne divise par l'ordre de $\bbW$, $\pi_a^\bullet$ se
factorise par le quotient mod\'er\'e $I_{\bar v}^{\rm tame}$ de $I_{\bar v}$.
Pour adh\'erer aux notations de \cite{GKM-codim}, choisissons un
g\'en\'erateur topologique de $I_{\bar v}^{\rm tame}$ et notons $w_a$
l'image de ce g\'en\'erateur par $\pi_a^\bullet$.

Pour toute racine $\alpha\in \Phi$, on a un entier
$$r(\alpha):={\rm val}_{\bar v}(\alpha(x))$$
o\`u ${\rm val}_{\bar v}$ est l'unique prolongement de la valuation ${\rm
val}_{\bar v}(\varepsilon_{\bar v})=1$ sur $\bar F_{\bar v}$ \`a $\bar
F_{\bar v}^{\rm sep}$. On obtient ainsi une fonction $r:\Phi \rightarrow
\QQ_+$.

Le couple $(w_a,r)$ d\'epend du choix de $x$ mais l'orbite sous $\bbW$ de
ce couple n'en d\'epend pas. Nous allons noter $[w_a,r]$ l'orbite sous
$\bbW$ du couple $(w_a,r)$.

On a l'\'egalit\'e \'evidente
$$\sum_{\alpha\in \Phi} r(\alpha)=\deg_{\bar v}(a^*{\discrim_G})=d_{\bar v}(a).$$
L'invariant
$$c_{\bar v}(a)=\dim(\bbt)-\dim(\bbt^{w_a})$$
est la chute du rang torique du mod\`ele de N\'eron de $J_a$. D'apr\`es la
formule de  Bezrukavnikov, on a
$$\delta_{\bar v}(a)=\frac {d_{\bar v}(a)-c_{\bar v}(a)}{2}.$$

Soit $\frakc^\heartsuit(\calO_{\bar v})_{[w,r]}$ le sous-ensemble des $a\in
\frakc^\heartsuit(\calO_{\bar v})$ avec l'invariant $[w,r]$ donn\'e. D'apr\`es \cite{GKM-codim},
cet ensemble est admissible
c'est-\`a-dire qu'il existe un entier $N$ et un sous-sch\'ema localement ferm\'e
$Z$ de $\frakc(\calO_{\bar v}/\varepsilon_{\bar v}^N \calO_{\bar v})$ vu
comme $\bar k$-sch\'ema tel que $\frakc^\heartsuit(\calO_{\bar
v})_{[w,r]}$ soit l'image r\'ecipro\-que de $Z(\bar k)$ par l'application
$\frakc(\calO_{\bar v})\rightarrow \frakc(\calO_{\bar v}/\varepsilon_{\bar
v}^N \calO_{\bar v})$. On dira que $\frakc^\heartsuit(\calO_{\bar
v})_{[w,r]}$ est admissible d'\'echelon $N$.

Ils d\'efinissent alors la {\em codimension de la strate}
$\frakc^\heartsuit(\calO_{\bar v})_{[w,r]}$ comme la codimension de $Z$
dans $\frakc(\calO_{\bar v}/\varepsilon_{\bar v}^N \calO_{\bar v})$ vue
comme $\bar k$-sch\'emas. Cette codimension ne d\'epend visiblement pas
de l'\'echelon $N$ choisi pourvu que celui-ci soit assez grand. Notons ${\rm
codim}[w,r]$ cette codimension. D'apr\`es \cite[8.2.2]{GKM-codim} on a la
formule explicite
$${\rm codim}[w,r]=d(w,r)+\frac{d_{\bar v}(a)+c_{\bar v}(a)}{2}$$
o\`u l'entier $d(w,r)$ est la codimension de $\frakt_w(\calO)_r$ dans
$\frakt_w(\calO)$ dans les notations de {\em loc. cit}. Nous nous
contenterons d'une estimation plus grossi\`ere.

\begin{proposition} Si $\delta_a>0$, on a l'in\'egalit\'e
$${\rm codim}[w,r] \geq \delta_a+1.$$
\end{proposition}

\begin{proof}
Il est clair que
$${\rm codim}[w,r]= \delta_{\bar v}(a)+ c_{\bar v}(a)+
d(w,r)$$ o\`u $\delta_{\bar v}(a)={(d_{\bar v}(a)-c_{\bar v}(a))/2}$. Si $w$
n'est pas l'\'el\'ement trivial de $W$, on a $c_{\bar v}(a)\geq 1$. Si $w=1$, par
d\'efinition de \cite[8.2.2]{GKM-codim}, $d(w,r)$ est la codimention de
$\frakt(\bar\calO_{\bar v})_r$ dans $\frakt(\bar\calO_{\bar v})$ o\`u
$\frakt(\bar\calO_{\bar v})_r$ est la partie admissible de
$\frakt(\bar\calO_{\bar v})$ constitu\'ee des \'el\'ements ayant la valuation
radicielle $r$. Si $\delta_{\bar v}(a) >0$, alors $r\not=0$ de sorte que cette partie
est de codimension strictement positive. Donc $d(w,r)\geq 1$. Dans les
deux cas, on obtient donc l'in\'egalit\'e qu'on voulait.
\end{proof}

\begin{proposition}\label{codimension}
Si le degr\'e de $D$ est tr\`es grand par rapport \`a $\delta$, la strate \`a
$\delta$ constant $\calA_\delta$ est de codimension plus grande ou \'egale
\`a $\delta$.
\end{proposition}

\begin{proof}
Pour toute partition $\delta_\bullet$ de $\delta$ en une somme d'entiers
naturels $\delta=\delta_1+\cdots+\delta_n$, consid\'e\-rons le sous-sch\'ema
$Z_{\delta_\bullet}$ de $\calA^\heartsuit \times X^j$ qui consiste en les
uplets $(a;x_1,\ldots,x_n)$ avec $a\in\calA^\heartsuit(\bar k)$ et
$x_1,\ldots,x_n\in X(\bar k)$ tels que l'invariant $\delta$ local
$\delta_{x_i}(a)$ vaut $\delta_i$. On peut stratifier $Z_{\delta_\bullet}$ en
r\'eunion des strates $Z_{[w_\bullet,r_\bullet]}$ des $(a;x_1,\ldots,x_n)$
tel que l'image dans $\frakc^\heartsuit(\bar\calO_{x_i})$ soit dans la strate
de valuation radicielle $\frakc^\heartsuit(\bar\calO_{x_i})_{[w_i,r_i]}$.
Supposons que cette strate est admissible d'\'echelon $N_i$. Pour $\deg(D)$
tr\`es grand, l'application lin\'eaire
$$\calA \longrightarrow \prod_{i=1}^n \frakc(\bar\calO_{x_i}/\varepsilon^{N_i}
\bar\calO_{x_i})
$$
est surjective. Il s'ensuit que $Z_{[w_\bullet,r_\bullet]}$ est de
codimension au moins \'egale \`a
$$\sum_{i=1}^n (\delta_i+1)$$
dans $\calA\times X^n$ Il s'ensuit que son image dans $\calA$ est de
codimension au moins \'egale \`a $\delta=\sum_{i=1}^n \delta_i$. La
proposition s'en d\'eduit.
\end{proof}

On pense que cette in\'egalit\'e est valide sans
l'hypoth\`ese que $\deg(D)$ soit tr\`es grand par rapport \`a $\delta$. Un
calcul de l'action infinit\'esimale de $\calP_a$ sur $\calM_a$ montre que
c'est vrai en caract\'eristique z\'ero. Nous pr\'esenterons sur ce calcul dans une
autre occasion.

\subsection{L'ouvert \'etale $\tilde\calA$ de $\calA$}
\label{subsection : tilde A}

Dans l'\'enonc\'e de la dualit\'e de Tate-Nakayama \ref{T-N}, il est
n\'ecessaire de choisir un certain point base pour d\'efinir la
$\kappa$-int\'egrale orbitale pour un \'el\'ement $\kappa\in \hat\bbT$. On
va de m\^eme construire un ouvert \'etale $\tilde\calA$ de $\calA$ qui
consiste \`a faire le choix d'un point base de la courbe cam\'erale dans le
lieu o\`u celle-ci est \'etale au-dessus de $X$. On verra dans la suite
l'\'enonc\'e de stabilisation deviendra beaucoup plus plaisant sur
$\tilde\calA$.

Soit $\infty\in X(\bar k)$ un point g\'eom\'etrique de $X$. Consid\'erons
l'ouvert $\calA^{\infty}$ de $\calA^\heartsuit \otimes_k \bar k$ qui consiste
en les points $a\in\calA^\heartsuit(\bar k)$ tel que le rev\^etement
cam\'eral $\widetilde X_a\rightarrow \bar X$ est \'etale au-dessus du point
$\infty$. En variant le point $\infty$, on obtient un recouvrement de
$\calA^\heartsuit\otimes_k \bar k$ par les ouverts de Zariski
$\calA^{\infty}$. Notons que si le point $\infty$ est d\'efini sur $k$,
$\calA^\infty$ l'est aussi.

Consid\'erons le rev\^etement $\tilde \calA$ dont les $\bar k$-points
sont des couples $(a,\tilde\infty)$ o\`u $a\in\calA^\infty(\bar k)$ et
$\tilde\infty$ est un point de $\widetilde X_a$ au-dessus de $\infty$. Le
morphisme d'oubli de $\tilde\infty$
\begin{equation}\label{mu infty}
{\calW}_\infty:\tilde\calA \rightarrow\calA^\infty
\end{equation}
est un torseur sous le groupe $W_{\infty}$ qui est la fibre de $W$ au-dessus
de $\infty$. Si le point $\infty$ est d\'efini sur $k$, $\tilde\calA$ est aussi
d\'efini sur $k$.

\begin{lemme}\label{connexe tilde A}
Supposons que $\deg(D)>2g$. Alors $\tilde\calA$ est lisse et irr\'e\-ductible.
\end{lemme}

\begin{proof}
On a un diagramme cart\'esien
$$
\xymatrix{
 \tilde\calA \ar[d]_{} \ar[r]^{}   & {\frakt_{D,\infty}}  \ar[d]\\
 \calA^\infty\ar[r] & \frakc_{D,\infty}           }
$$
o\`u la fl\`eche de bas est une partie ouverte de l'application lin\'eaire de
restriction de $\calA=\rmH^0(X,\frakc_D)$ \`a la fibre de $\frakc_D$ en
$\infty$ et o\`u la fl\`eche de droite se d\'eduit de $\bbt\rightarrow \bbc$
par torsion. Sous l'hypoth\`ese $\deg(D)>2g$, cette application lin\'eaire est
surjective \cf \ref{deg(D)>2g}. On en d\'eduit que la fl\`eche du haut du
diagramme est un morphisme lisse de fibres connexes. Puisque
$\frakt_{D,\infty}$ est un vectoriel, $\tilde\calA$ est lisse et irr\'eductible.
\end{proof}

Consid\'erons l'image r\'ecipoque de la stratification
\begin{equation}
\calA\otimes_k \bar k=\bigsqcup_{\psi\in\Psi}\calA_\psi.
\end{equation}
par le morphisme $\calW_\infty:\tilde\calA\rightarrow \calA^\infty$. En
d\'ecomposant $\calW_\infty^{-1}(\calA_\psi)$ pour tout $\psi\in \Psi$ en
composantes connexes, on obtient ainsi une stratification
\begin{equation}\label{stratification tilde A}
\widetilde\calA =\bigsqcup_{\tilde \psi\in \tilde \Psi}
\widetilde\calA_{\tilde \psi}.
\end{equation}
Sur l'ensemble $\tilde \Psi$, on consid\`ere la relation d'ordre d\'efinie par
$\tilde \psi_1\leq \tilde \psi_2$ si et seulement si $\widetilde\calA_{\tilde
\psi_1}$ est inclu dans l'adh\'erence de $\widetilde\calA_{\tilde \psi_2}$. Le
groupe $W_\infty$ agit sur l'ensemble $\tilde \Psi$ et l'ensemble quotient
de $\tilde \Psi$ par l'action de $W_\infty$ s'identifie au sous-ensemble de
$\Psi$ des strates $\calA_\psi$ ayant une intersection non vide avec
$\calA^\infty$.

Puisque $\tilde\calA$ est irr\'eductible, l'ensemble $\tilde \Psi$ admet un
\'el\'ement maximal not\'e $\tilde\psi$. Il est stable sous l'action de
$W_\infty$ et son image est l'\'el\'ement maximal $\psi_G$ de $\Psi$.

\subsection{Stratification par les invariants monodromiques}
\label{subsection : stratification monodromique}

On va maintenant construire une autre stratification de
$\calA^\heartsuit\otimes_k \bar k$ fond\'ee sur la monodromie de la courbe
cam\'erale. Comme la stratification \`a $\delta$ constant, la stratification
par les invariants monodromiques sera obtenue en regroupant les strates
dans la stratification par normalisation en famille des courbes cam\'erales
\ref{stratification}.

Les invariants monodromiques qu'on va d\'efinir d\'ependent d'une
r\'educ\-tion du torseur $\rho_G$ \cf \ref{reduction}. On se donne donc un
torseur $\rho:X_\rho\rightarrow X$ sous un groupe discret $\Theta_\rho$ et
un homomorphisme $\Theta_\rho\rightarrow \Out(\GG)$ tel que
$\rho_G=\rho\wedge^{\Theta_\rho} \Out(\GG)$. On se donne aussi un $\bar
k$-point $\infty_\rho$ de $X_\rho$ au-dessus de $\infty$. Ce point d\'efinit
un point g\'eom\'etrique $x_G$ de $\rho_G$ au-dessus de $\infty$ et fournit
en particulier une identification $W_\infty=\bbW$.

Pour tout $a\in\calA^\infty(\bar k)$, on a un rev\^etement \'etale
$$\widetilde X_{\rho,a}=\widetilde X_a\times_X X_\rho$$
de la courbe cam\'erale $\widetilde X_a$. La courbe $\widetilde
X_{\rho,a}$ est alors munie d'une action de $\bbW\rtimes \Theta_\rho$. Soit
$\widetilde X_{\rho,a}^\flat$ la normalisation de $\widetilde X_{\rho,a}$
qui est aussi munie d'une action de $\bbW\rtimes \Theta_\rho$. La donn\'ee
d'un point g\'eom\'etrique $\tilde\infty$ de $\widetilde X_a$ d\'efinit alors
un point g\'eom\'etrique
$$\tilde\infty_\rho=(\tilde\infty, \infty_\rho)$$
dans le lieu lisse de $\widetilde X_{\rho,a}$ et donc un point
g\'eom\'etrique de $\widetilde X_{\rho,a}^\flat$.

Soit $\tilde a=(a,\tilde\infty)\in \widetilde\calA(\bar k)$. Soit $C_{\tilde a}$
la composante connexe de  $\widetilde X_{\rho,a}^\flat$ qui contient le
point $\tilde\infty$. Soit $W_{\tilde a}$ le sous-groupe de $\bbW\rtimes
\Theta_\rho$ constitu\'es des \'el\'ements qui laissent stable cette
composante. On consid\`ere aussi le sous-groupe $I_{\tilde a}$ engendr\'e
par les \'el\'ements de $W_{\tilde a}$ qui admettent au moins un point fixe
dans $C_{\tilde a}$. C'est un sous-groupe normal de $W_{\tilde a}$. Il est
clair que la $I_{\tilde a}$ est contenu dans le noyau de la projection
$W_{\tilde a}\rightarrow \Theta_\rho$ de sorte que $I_{\tilde a}\subset
W_{\tilde a}\cap \bbW$.

\begin{proposition} \label{WIPHI}
Il existe une application $\tilde\psi\mapsto (I_{\tilde \psi},W_{\tilde \psi})$
de l'ensemble $\tilde\Psi$ des strates dans l'ensemble des couples
$(I_*,W_*)$ form\'es d'un sous-groupe $W_*$ de $\bbW\rtimes \Theta$ et
d'un sous-groupe normal $I_*$ de $W_*$ telle que pour tout $\tilde
a=(a,\tilde\infty)\in \widetilde\calA_{\tilde\psi}(\bar k)$ on a
$$(I_{\tilde a},W_{\tilde a})=(I_{\tilde\psi},W_{\tilde\psi}).$$
\end{proposition}

\begin{proof}
Consid\'erons l'image inverse de $\calB_\psi\rightarrow \calA_\psi$ par le
morphisme fini \'etale $\calW_\infty:
\calW_\infty^{-1}(\calA_\psi)\rightarrow \calA_\psi$. Chosissons une de ses
composantes connexes  $\widetilde \calB_{\tilde\psi}$. Par construction,
au-dessus de $\calB_\psi$, on a une courbe propre et lisse $\widetilde
X_\psi^\flat$ qui normalise fibre par fibre la courbe cam\'erale $\widetilde
X_\psi$. Au-dessus de $\widetilde \calB_{\tilde\psi}$,  $\widetilde
X_{\rho,\tilde\psi}^\flat$ admet une section tautologique. Cette section
d\'efinit une composante connexe $C_{\tilde\psi}$ qui fibre par fibre est la
composante connexe $C_{\tilde a}$. C'est une courbe projective lisse de
fibres connexes au-dessus de $\calB_{\tilde\psi}$. Soit $W_{\tilde\psi}$ le
sous-groupe de $\bbW\rtimes \Theta_\rho$ qui laisse stable cette
composante. Pour tout $\tilde a\in \widetilde\calA_{\tilde\psi}(\bar k)$, on a
$W_{\tilde a}=W_{\tilde\psi}$.

Soit $w\in W_{\tilde\psi}$ un \'el\'ement ayant au moins un point fixe dans
$C^{\tilde a}$. On a alors $w\in\bbW$. Puisque le rev\^etement $C_{\tilde
a}\rightarrow \bar X$ est g\'en\'eriquement \'etale galoisien de groupe de
Galois $W_{\tilde a}$, les points fixes de $w$ dans $C_{\tilde a}$ sont
isol\'es. Puisque l'ordre de $\bbW$ est premier \`a la caract\'eristique de
$k$, les points fixes sont de multiplicit\'e un. En appliquant la formule des
points fixes de Lefschetz, on a
$$\tr(w,\rmH^*(C_{\tilde a},\Ql))\not=0.$$
Puisque $C_{\tilde\psi}\rightarrow \widetilde \calB_{\tilde\psi}$ est un
morphisme propre lisse, les groupes de cohomologie $\rmH^*(C_{\tilde
a'},\Ql)$ s'organisent en un syst\`eme local quand $\tilde a'$ varie dans
$\calB_{\tilde\psi}$. On en d\'eduit que
$$\tr(w,\rmH^*(C_{\tilde a'},\Ql))\not=0$$
pour tout $\tilde a'\in \calB_{\tilde\psi}(\bar k)$. Il en r\'esulte que $w$ a au
moins un point fixe dans $C_{\tilde a'}$. Ainsi, le sous-groupe de
$W_{\tilde\psi}$ engendr\'e par les \'el\'ements ayant au moins un point fixe
dans $C_{\tilde a}$ ne d\'epend pas du choix du point g\'eom\'etrique
$\tilde a$.
\end{proof}

On obtient donc une application ${\tilde\psi}\mapsto
(I_{\tilde\psi},W_{\tilde\psi})$ de $\tilde\Psi$ dans l'ensemble des couples
form\'es d'un sous-groupe $W_*$ de $\bbW\rtimes \Theta_\rho$ et d'un
sous-groupe normal $I_*$ de $W_*\cap\bbW$. Cette application est
\'equivariante par rapport \`a l'action par conjugaison de $\bbW$. On en
d\'eduit par passage au quotient une application de $\tilde\Psi/\bbW$ dans
l'ensemble des classes de $\bbW$-conjugaison des couples $(I_*,W_*)$. On a
ainsi d\'efini une application $\psi\mapsto [I_{\psi},W_{\psi}]$ sur le
sous-ensemble de $\Psi$ des strates $\calA_\psi$ ayant une intersection no
vide avec $\calA^\infty$. On v\'erifie sans difficult\'e que ces applications
sont compatibles pour les pour les choix diff\'erents de $\infty$ et donc se
recollent en une application $\psi\mapsto [I_{\psi},W_{\psi}]$ d\'efini sur
l'ensemble $\Psi$ tout entier.

Il existe une relation d'ordre \'evident sur l'ensemble des couples
$(I_*,W_*)$ d\'efinie par $(I_1,W_1)\leq (I_2,W_2)$ si et seulement si
$I_1\subset I_2$ et $W_1\subset W_2$. Elle induit une relation d'ordre sur
l'ensemble des classes de conjugaison $[I_*,W_*]$ d\'efinie par
$[I_1,W_1]\leq [I_2,W_2]$ si et seulement s'il existe $w\in \bbW\rtimes
\Theta$ tel que $wI_1 w^{-1} \subset I_2$ et $wW_1w^{-1}\subset W_2$.

\begin{lemme}\label{I W croissant}
L'application $\tilde \psi\mapsto (I_{\tilde \psi},W_{\tilde\psi})$ est
croissante. Il en est de m\^eme de l'application $\psi \mapsto
[I_\psi,W_\psi]$.
\end{lemme}

\begin{proof} On raisonne comme dans la d\'emonstration de \ref{delta decroit}.
Soit $S=\Spec(R)$ un trait formel   de point g\'en\'erique
$\eta=\Spec(k(\eta))$ et de point ferm\'e $s=\Spec(k(s))$ g\'eom\'etrique.
Soit $\tilde a:S\rightarrow\tilde\calA$ un morphisme avec $\tilde
a(\eta)\in\tilde\calA_{\tilde\psi}$ et $\tilde a(s)\in\tilde\calA_{\tilde\psi'}$.
On doit d\'emontrer que $(I_{\tilde \psi},W_{\tilde\psi})\geq (I_{\tilde
\psi'},W_{\tilde\psi'})$.

Consid\'erons le rev\^etement $\widetilde X_{\rho,a}$ de $X\times S$ qui est
d\'efini comme l'image r\'eciproque par $a$ du rev\^etement $X_\rho\times
\bbt_D \rightarrow \frakc_D$. Consid\'erons la normalisation $\widetilde
X_{\rho,a}^\flat$ de  $\widetilde X_{\rho,a}$. Quitte \`a faire un
changement radiciel du trait, on peut supposer que la fibre g\'en\'erique
$(\widetilde X_{\rho,a}^\flat)_\eta$ est une courbe lisse sur $k(\eta)$ de
sorte qu'on peut calculer $(I_{\tilde \psi},W_{\tilde\psi})$ \`a partir de
cette fibre g\'en\'erique. En revanche, la fibre sp\'eciale $(\widetilde
X_{\rho,a}^\flat)_s$ n'est pas normale en g\'en\'eral et il faut prendre sa
normalisation $(\widetilde X_{\rho,a}^\flat)_s^\flat$ pour calculer
$(I_{\tilde \psi'},W_{\tilde\psi'})$.

Le point $\tilde a$ d\'efinit une section de $(\widetilde
X_{\rho,a}^\flat)_\eta$. Notons $C_{\tilde a}$ la composante connexe de
$(\widetilde X_{\rho,a}^\flat)_\eta$ contenant cette section. Le groupe
$W_{\tilde\psi}$ est alors le sous-groupe de $\bbW\rtimes \Theta_\rho$
form\'e des \'el\'ements qui laissent stable cette composante. Soit
$C_{\tilde a(s)}$ la composante connexe de $(\widetilde
X_{\rho,a}^\flat)_s^\flat$ contenant le point d\'efinit par $\tilde a(s)$. Le
groupe $W_{\tilde\psi'}$ est le sous-groupe de $\bbW\rtimes \Theta_\rho$
form\'e des \'el\'ements qui laissent stable $C_{\tilde a(s)}$. Comme
$C_{\tilde a(s)}$ est une composante connexe de la normalisation de la
fibre sp\'eciale de $C_{\tilde a}$, on a une inclusion $W_{\tilde \psi'}\subset
W_{\tilde\psi}$. Un \'el\'ement de $W_{\tilde \psi'}$ ayant un point fixe dans
$C_{\tilde a(s)}$ a n\'ecessairement un point fixe dans $C_{\tilde a}$ d'o\`u
la seconde inclusion $I_{\tilde \psi'}\subset I_{\tilde \psi}$.
\end{proof}

Pour tout $(I_*,W_*)$ comme ci-dessus, notons $\tilde\calA_{(I_*,W_*)}$ la
r\'eunion des strates $\calA_{\tilde \psi}$ telles que $(I_{\tilde
\psi},W_{\tilde \psi})$. D'apr\`es \ref{I W croissant}, c'est une partie
localement ferm\'ee de $\tilde\calA$. On obtient ainsi une stratification de
$\tilde\calA$
\begin{equation}\label{stratification tilde A I W}
\tilde\calA=\bigsqcup_{(I_*,W_*)} \tilde\calA_{(I_*,W_*)}
\end{equation}
De m\^eme, on a une stratification compatible de $\calA^\heartsuit$
\begin{equation}\label{stratification monodromique}
\calA^\heartsuit \otimes_k \bar k=\bigsqcup_{[I_*,W_*]} \calA_{[I_*,W_*]}
\end{equation}

\begin{lemme}\label{I W psi_G} Consid\'erons une r\'eduction de la restriction de
$\rho_G$ \`a $\ovl X$ \`a un torseur irr\'eductible $\bar\rho:X_{\bar
\rho}\rightarrow \bar X$ sous un groupe $\Theta_{\bar \rho}$. Consid\'erons
les invariants monodromiques relatifs \`a cette r\'eduction. Soit
$\tilde\psi_G$ l'\'el\'ement maximal de $\tilde\Psi$. On a alors
$$(I_{\tilde\psi_G},W_{\tilde\psi_G})=(\bbW,\bbW\rtimes \Theta_{\bar\rho}).$$
De m\^eme, on a $[I_{\psi_G},W_{\psi_G}]=[\bbW,\bbW\rtimes
\Theta_{\bar\rho}]$ o\`u $\psi_G$ est l'\'el\'ement maximal de $\Psi$.
\end{lemme}

\begin{proof} Soit $\tilde a=(a,\tilde\infty)$ un $\bar k$-point de la strate
ouverte $\tilde\calA_{\tilde\psi_G}$. D'apr\`es \cf \ref{irreductible},
$\widetilde X_{\bar \rho,a}$ est alors une courbe lisse et irr\'eductible. On
en d\'eduit que $W_{\tilde a}=\bbW\rtimes \Theta_{\bar\rho}$.

Le sous-groupe normal $I_{\tilde a}$ de $W_a$ est alors un sous-groupe
normal de $W$. La courbe $\widetilde X_{\bar\rho,a}$ coupe
transversalement tous les murs de $\mathbbm h_\alpha$ dans
$X_{\bar\rho}\times \bbt$ associ\'es aux racines $\alpha$ de $\bbg$, le
groupe $I_{\tilde a}$ est un sous-groupe normal de $\bbW$ contenant
toutes les r\'eflexions $s_\alpha$ associ\'es au mur $h_\alpha$. Il s'ensuit
que $I_a=\bbW$.
\end{proof}

Sous l'hypoth\`ese du lemme, on a un ouvert dense
$\calA_{[\bbW,\bbW\rtimes \Theta_{\bar\rho}]}$ de $\calA$. En g\'en\'eral,
il contient strictement la grosse strate $\calA^\diamondsuit
=\calA_{\psi_G}$. En effet, dans la strate $\calA_{[\bbW,\bbW\rtimes
\Theta]}$, on demande \`a la courbe cam\'erale $\widetilde X_{\bar\rho,a}$
d'\^etre irr\'eductible alors que dans $\calA^\diamondsuit$, celle-ci doit
\^etre lisse et irr\'eductible.

\subsection{Description de $\pi_0(\calP)$} \label{subsection : description pi_0(P)}

Dans ce paragraphe, on va d\'ecrire le faisceau $\pi_0(\calP)$ le long des
strates $\calA_{I_*,W_*}$ de la stratification \ref{stratification
monodromique}. Rappelons la d\'efinition de ce faisceau. Le champ de
Picard $\calP\rightarrow \calA^\heartsuit$ \'etant lisse \ref{P lisse}, il existe
un unique faisceau $\pi_0(\calP)$ pour la topologie \'etale de
$\calA^\heartsuit$ tel que la fibre de $\pi_0(\calP)$ en un point
$a\in\calA^\heartsuit(\bar k)$ est le groupe $\pi_0(\calP_a)$ des
composantes connexes de $\calP_a$. Ceci est une cons\'equence d'un
th\'eor\`eme de Grothendieck \cf \cite[15.6.4]{EGA4}, voir aussi
\cite[6.2]{N}.

On va aussi consid\'erer le probl\`eme interm\'ediaire de d\'eterminer le
faisceau $\pi_0(\calP')$ des composante connexes du champ de Picard
$\calP'$ dont la fibre en chaque point $a\in\calA^\heartsuit(\bar k)$ classifie
des $J_a^0$ torseurs sur $\bar X$. L'homomorphisme surjectif
$\calP'\rightarrow \calP$ induit un homomorphisme surjectif
$\pi_0(\calP')\rightarrow \pi_0(\calP)$.

On commence par regarder la restriction de ces faisceaux \`a l'ouvert \'etale
$\tilde\calA$ de $\calA^\heartsuit$ \cf \ref{subsection : tilde A}. Cet ouvert
d\'epend du choix d'un point $\infty\in X(\bar k)$. On a une stratification
(\ref{stratification tilde A I W})
$$\tilde\calA=\bigsqcup_{(I_*,W_*)} \tilde\calA_{(I_*,W_*)}$$
par les invariants monodromiques avec $I_*\subset W_*\subset \bbW\rtimes
\Theta_\rho$. On a aussi fix\'e une r\'eduction du
$\Out(\GG)$-torseur $\rho_G$ en un torseur $\rho:X_\rho\rightarrow X$
sous le groupe $\Theta_\rho$ muni d'un point g\'eom\'etrique $\infty_\rho$ de
$X_\rho$ au-dessus de $\infty$. Ceci permet en particulier d'identifier la
fibre $W_\infty$ de $W$ en $\infty$ avec $\bbW$.

D'apr\`es \ref{pi 0 global}, pour tout point $\tilde a\in \tilde\calA(\bar k)$,
la fibre de $\pi_0(P'_{\tilde a})$ s'exprime
$$\pi_0(\calP'_a)=(\hat \bbT^{W_{\tilde a}})^*$$
o\`u l'exposant $(\_)^*$ d\'esigne la dualit\'e entre les groupes ab\'eliens de
type finis et les groupes diagonalisables de type fini sur $\Ql$. De m\^eme,
$\pi_0(\calP_a)$ s'identifie au quotient de $\pi_0(\calP'_a)$ dual au
sous-groupe $\hat\bbT(I_{\tilde a},W_{\tilde a})$ d\'efini dans \ref{pi 0
global}. L'\'enonc\'e suivant permet de d\'eterminer compl\`etement la
restriction des faisceaux $\pi_0(\calP')$ et $\pi_0(\calP)$ \`a $\tilde\calA$.

\begin{proposition}\label{surjectif}
Les fl\`eches surjectives de la proposition \ref{pi 0 global}
$$\bbX_*\rightarrow \pi_0(\calP'_a)=(\bbX_*)_{W_{\tilde a}}$$
d\'efinies pour tout $\tilde a=(a,\tilde\infty)\in\tilde\calA(\bar k)$
s'interpolent en un homomorphisme surjectif canonique du faisceau
constant de valeur $\bbX_*$ sur $\tilde\calA$ dans
$\pi_0(\calP')|_{\tilde\calA}$.
\end{proposition}

\begin{proof}
Soit $\tilde a=(a,\tilde\infty)\in \tilde\calA(\bar k)$. Le point g\'eom\'etrique
$\tilde\infty_\rho=(\tilde\infty,\infty_\rho)$ de $\widetilde X_{\rho,a}$
permet d'identifier la fibre de $J_a$ en $\infty$ avec le tore fixe $\bbT$ \cf
\ref{J=J'}. Cette identification permet de d\'efinir un homomorphisme du
faisceau constant sur $\tilde\calA$ de valeur le groupe $\bbX_*$ des
cocaract\`eres de $\bbT$ dans l'image r\'eciproque de $\calP'$ sur
$\tilde\calA$, voir la d\'emonstration de \cite[6.8]{N} et donc un
homomorphisme
$$\bbX_* \times \tilde\calA
\rightarrow \pi_0(\calP')|_{\tilde\calA}.
$$
Fibre par fibre c'est l'homomorphisme surjectif de \ref{pi 0 global}.
\end{proof}

\begin{remarque}
Notons que cet homomorphisme surjectif d\'epend du choix du $\bar
k$-point $\infty_\rho$. Il est d\'efini sur $k$ si le point $\infty_\rho$ est
d\'efini sur $k$. Notons que $\infty_\rho$ est d\'efini sur $k$ entra{\^\i}ne
que $\infty$ est aussi d\'efini sur $k$. Dans le langage des repr\'esentations
du groupe fondamental, il revient au m\^eme de dire que l'homomorphisme
$$\rho^\bullet:\pi_1(X,\infty)=\pi_1(\ovl X,\infty)\rtimes \Gal(\bar k/ k)
\rightarrow \Theta_\rho
$$
d\'efini \`a partir de $\infty_\rho$ est trivial sur le facteur $\Gal(\bar k/ k)$.
\end{remarque}

A l'aide de ce lemme, on obtient une description explicite des restrictions
de $\pi_0(\calP')$ et de $\pi_0(\calP)$ \`a $\tilde\calA$. Pour tout ouvert
\'etale $U$ de $\tilde\calA$, la stratification (\ref{stratification tilde A I W})
induit sur $U$ une stratification
$$U=\bigsqcup_{(I_*,W_*)} U_{(I_*,W_*)}.$$
Pour tout couple $(I_1,W_1)$ form\'e d'un sous-groupe $W_1$ de
$\bbW\rtimes \Theta$ et d'un sous-groupe normal $I_1$ de $W_1$, $U$ sera
dit un petit ouvert de type $(I_1,W_1)$ si $U_{(I_1,W_1)}$ est l'unique
strate ferm\'ee non vide dans la stratification ci-dessus. Il est clair que les
petits ouverts de diff\'erents types forment une base de la topologie \'etale
de $\tilde\calA$ dans le sens que tout ouvert peut \^etre recouvert par une
famille de petits ouverts. Pour d\'efinir un faisceau pour la topologie \'etale
de $\tilde\calA$, il suffit donc de sp\'ecifier ses sections sur les petits
ouverts et les fl\`eches de transition.

Consid\'erons les faisceaux $\Pi'$ et $\Pi$ d\'efinis comme suit. Pour un
petit ouvert $U_1$ de type $(I_1,W_1)$, on pose
\begin{eqnarray*}
\Gamma(U,\Pi')&=&(\hat \bbT^{W_{1}})^*=(\bbX_*)_{W_{1}}\\
\Gamma(U,\Pi)&=&\hat \bbT(I_1,W_1)^*
\end{eqnarray*}
Soit $U_2$ un petit ouvert \'etale de $U_1$ de type $(I_2,W_2)$. Puisque
$U_{(I_1,W_1)}$ est l'unique strate ferm\'ee non vide de $U_1$, on a
l'in\'egalit\'e
$$(I_1,W_1)\leq (I_2,W_2).$$
On a alors une inclusion \'evidente des sous-groupes des invariants de
$\hat\bbT$ sous $W_1$ et $W_2$
$$\hat\bbT^{W_2}\subset \hat\bbT^{W_1}$$
d'o\`u la fl\`eche de transition
$$\Gamma(U_1,\Pi')\rightarrow \Gamma(U_2,\Pi').$$
La d\'efinition des fl\`eches de transition du faisceau $\Pi$ r\'esulte du
lemme suivant.

\begin{lemme}\label{T(I,W) inclusion}
Si $(I_1,W_1)\leq (I_2,W_2)$, on a l'inclusion
$$\hat\bbT(I_2,W_2)\subset \hat \bbT(I_1,W_1).$$
\end{lemme}

\begin{proof}
Soient $\kappa$ un \'el\'ement de $\hat\bbT$, $\hat\GG_\kappa$ son
centralisateur dans $\hat\bbG$ et $\hat\bbH$ la composante neutre de
celui-ci. Soient $(\bbW\rtimes\Theta)_\kappa$ le centralisateur de $\kappa$
dans $\bbW\rtimes \Theta$ et $\bbW_\bbH$ le groupe de Weyl de
$\hat\bbH$. Si $\kappa\in \hat\bbT(I_2,W_2)$ alors $I_2\subset \bbW_\bbH$
et $W_2\subset (\bbW\rtimes\Theta)_\kappa$. On en d\'eduit que
$I_1\subset \bbW_\bbH$ et $W_1\subset (\bbW\rtimes\Theta)_\kappa$.
\end{proof}

L'action de $\bbW$ sur $\tilde\calA$ se rel\`eve de fa{\c c}on \'evidente sur
les faisceaux $\Pi$ et $\Pi'$. Le morphisme $\tilde\calA \rightarrow
\calA^\infty$ \'etant un morphisme fini \'etale galoisien de groupe de Galois
$\bbW$, les faisceaux $\Pi$ et $\Pi'$ descendent donc \`a l'ouvert
$\calA^\infty$ de $\calA^\heartsuit$. D\'esignons encore par $\Pi$ et $\Pi'$
les faisceaux ainsi d\'efinis sur $\calA^\infty$.

L'\'enonc\'e suivant est une cons\'equence imm\'ediate de \ref{surjectif} et
\ref{pi 0 global}.

\begin{corollaire}\label{pi}
Il existe un isomorphisme canonique entre la restriction de $\pi_0(\calP')$,
\`a $\calA^\infty$ et $\Pi'$. De m\^eme, on a
$\pi_0(\calP)|_{\calA^\infty}=\Pi$.
\end{corollaire}

Les ouverts de $\calA^\infty$ de $\calA^\heartsuit$ recouvrent
$\calA^\heartsuit$ de sorte que l'\'enonc\'e ci-dessus fournit une description
compl\`ete des faisceaux $\pi_0(\calP)$ et $\pi_0(\calP')$. En particulier,
on a l'\'enonc\'e suivant.

\begin{corollaire}\label{pi A diamond}
Consid\'erons une r\'eduction de la restriction de $\rho_G$ \`a $\ovl X$ \`a
un torseur irr\'eductible $\bar\rho:X_{\bar \rho}\rightarrow \bar X$ sous un
groupe discret $\Theta_{\bar \rho}$. Consid\'erons les invariants monodromiques
relatifs \`a cette r\'eduction. La restriction de $\pi_0(\calP)$ \`a l'ouvert
$\calA^\diamondsuit\otimes_k \bar k$ de \ref{subsection : A diamond} est le
faisceau constant de valeur $(Z_{\hat\bbG}^{\Theta_{\bar\rho}})^*$. Ici
$Z_{\hat\bbG}^{\Theta_{\bar\rho}}$ d\'esigne le sous-groupe des
\'el\'ements $\Theta_{\bar\rho}$-invariants dans le centre $Z_{\hat\bbG}$
de $\hat\bbG$.
\end{corollaire}

\begin{proof}
Rappelons que $\calA^\diamondsuit$ co\"incide avec la strate ouverte
$\calA_{\psi_G}$ pour l'\'el\'ement maximal $\psi_G$ de $\Psi$. Le
corollaire est une cons\'equence de l'\'egalit\'e
$(I_{\tilde\psi_G},W'_{\tilde\psi_G})=(\bbW,\bbW\rtimes\Theta_{\bar\rho})$
\cf \ref{I W psi_G}. En effet, on peut v\'erifier que
$$\hat\bbT(\bbW,\bbW\rtimes\Theta_{\bar\rho})=Z_{\hat\bbG}^{\Theta_{\bar\rho}}.$$
en prenant une extension centrale $\hat\bbG'\rightarrow \hat\bbG$ avec un
group d\'eriv\'e simplement connexe munie d'une action \'equivariante de
$\Theta_{\bar\rho}$.
\end{proof}


\section{Cohomologie au-dessus de l'ouvert anisotrope}
\label{section : anisotrope}

Dans ce chapitre, nous allons introduire l'ouvert $\calA^\ani$ de $\calA$ au-dessus duquel la fibration de Hitchin $f^\ani:\calM^\ani\rightarrow\calA^\ani$ est un morphisme propre, en particulier de type fini. Comme $\calM^\ani$  est un champ de Deligne-Mumford lisse, en appliquant le th\'eor\`eme de puret\'e de Deligne, on obtient la puret\'e du complexe $f^\ani_*\Ql$.

Suivant \cite{N}, on \'etudie la d\'ecomposition des faicseaux pervers de cohomologie $^p\rmH^n(f^\ani_*\Ql)$ sous l'action du faisceau $\pi_0(\calP)$ dont la restriction \`a $\calA^\ani$ est un faisceau de groupes ab\'eliens finis. La d\'ecomposition prendra une forme plus agr\'eable apr\`es le changement de base \`a l'ouvert \'etale $\tilde\calA^\ani$ au-dessus duquel $\pi_0(\calP)$ est un quotient du faisceau constant de valeur le groupe $\bbX_*$ des cocaract\`eres du tore fixe $\bbT$. En particulier, au-dessus de $\tilde\calA^\ani$, on n'a plus besoin du formalisme des co-faisceaux qui a \'et\'e introduit dans \cite{N} pour \'ecrire la d\'ecomposition endoscopique sur $\calA^\ani$.

Pour chaque $\kappa\in\hat\bbT$, on d\'efinit le ferm\'e $\tilde\calA_\kappa$ de $\tilde\calA$ o\`u $\kappa:\bbX_*\rightarrow\Ql^\times$ se factorise par $\pi_0(\calP_{\tilde a})$. L'intersection $\tilde\calA_\kappa\cap \tilde\calA^\ani$ est vide sauf pour les $\kappa$ dans un ensemble fini qu'on notera $\kappa^\ani$. Pour $\kappa\in \kappa^\ani$, on a un facteur direct $^p\rmH^n(f^\ani_*\Ql)_\kappa$ support\'e par $\tilde\calA_\kappa\cap \tilde\calA^\ani$. On donnera une description de $\tilde\calA_\kappa$ en termes des groupes endoscopiques.

Le point important de ce chapitre est l'\'enonc\'e du th\'eor\`eme de stabilisation g\'eom\'etrique \ref{stabilisation sur tilde A}. Il s'agit d'une \'egalit\'e dans le groupe Grothendieck de deux complexes purs sur $\tilde\calA_H$. Cet \'enonc\'e implique le lemme fondamental de Langlands et Shelstad \ref{LS} et le reste de l'article sera consacr\'e \`a sa d\'emonstration.

Notons que le vrai \'enonc\'e de stabilisation g\'eom\'etrique qui correspond \`a la stabilisation usuelle du c\^ot\'e g\'eom\'etrique de la formule des traces, devrait \^etre sur $\calA^\ani$ au lieu de $\tilde\calA^\ani$. Cet \'enonc\'e est nettement plus compliqu\'e et n\'ecessite en particulier le formalisme peu familier des cofaisceaux de \cite{N}. Mais il n'est rien d'autre que la descente de \ref{stabilisation sur tilde A} de $\tilde\calA_H$ \`a $\calA_H^\ani$. Nous le laisserons au lecteur comme exercice.

\subsection{L'ouvert anisotrope}
\label{subsection : ouvert anisotrope}

Consid\'erons le sous-ensemble $\Psi^\ani$ de $\Psi$ des \'el\'ements
$\psi\in \Psi$ tels que  le groupe des $W_\psi$-coinvariants de $\bbX_*$ est
un groupe fini. Il revient au m\^eme de dire que le sous-groupe des $W_\psi$-invariants dans $\bbX_*$ est trivial. Le groupe $W_\psi$ est seulement d\'efini modulo
$\bbW$-conjugaison mais la finitude des coinvariants est bien une
propri\'et\'e invariante de sorte que $\Psi^\ani$ est un sous-ensemble bien
d\'efini de $\Psi$. Consid\'erons la r\'eunion
$$\calA^\ani=\bigsqcup_{\psi\in \Psi^\ani} \calA_\psi.$$
D'apr\`es \ref{I W croissant}, $\calA^\ani$ est un ouvert de $\calA$. Sous les
hypoth\`eses que le centre de $G$ ne contient pas de tore d\'eploy\'e sur $\bar X$ et
que $\deg(D)>2g$, $\calA^\ani$ contient la grosse strate
$\calA^\diamondsuit$ qui est non vide \cf \ref{A diamond non vide}, de sorte que
$\calA^\ani$ est aussi non vide.

\begin{lemme}\label{finitude}
$\calA^\ani(\bar k)$ est le sous-ensemble de $\calA^\heartsuit(\bar k)$
form\'es des points $a\in \calA^\heartsuit(\bar k)$ tels que $\pi_0(\calP_a)$
est fini.
\end{lemme}

\begin{proof}
D'apr\`es \ref{pi 0 global}, $a\in \calA^\ani(\bar k)$ si et seulement si
$\pi_0(\calP'_a)$ est un groupe fini. Le lemme r\'esulte du fait que
$\pi_0(\calP'_a) \rightarrow  \pi_0(\calP_a)$ est un homomorphisme
surjectif de noyau fini \cf \ref{P'_a P_a}.
\end{proof}

\begin{lemme}\label{ani auto}
Soit $a\in\calA^\ani(\bar k)$ et $(E,\phi)\in \calM_a(\bar k)$. Alors le groupe des automorphismes ${\rm Aut}(E,\phi)$ est un groupe fini. En tant que $\bar k$-sch\'ema en groupes, il est r\'eduit.
\end{lemme}

\begin{proof} La finitude est une cons\'equence imm\'ediate de \ref{borne automorphisme}. En tant que $\bar k$-sch\'ema en groupes, il est contenu dans le sous-groupe des points fixes d'un sous-groupe de $\bbW$ agissant sur $\bbT$. Ce dernier est r\'eduit car $p$ ne divise l'ordre de $\bbW$.
\end{proof}

\begin{numero}
Dans \cite[II.4]{Fa}, Faltings a d\'emontr\'e le th\'eor\`eme de r\'eduction semi-stable pour les fibr\'es de Higgs qui dit qu'un fibr\'e de Higgs semi-stable sur $X$ \`a coefficients dans le corps des fractions d'un anneau des valuation discr\`ete peut s'\'etendre en un fibr\'e de Higgs sur l'anneau apr\`es une extension finie s\'eparable et de plus, si le fibr\'e de Higgs dans la fibre sp\'eciale est stable, cette extension est unique.
Nous renvoyons \`a \cite{Fa} pour la d\'efintion des fibr\'es de Higgs semi-stables et stables. Disons seulement que c'est exactement la m\^eme d\'efinition que pour les $G$-torseurs sauf qu'on ne consid\`ere que les r\'eductions paraboliques de $E$ compatibles avec le champ de Higgs $\phi$. Ceci suffit pour d\'emontrer le lemme suivant.
\end{numero}

\begin{lemme}\label{ani > stable} Soit $a\in \calA^\ani(\bar k)$ et $(E,\varphi)\in\calM^\ani(\bar k)$. Alors
$(E,\phi)$ est stable.
\end{lemme}

\begin{proof} Puisque $a\in \calA^\ani(\bar k)$, $E$ n'a pas de r\'eduction parabolique compatible \`a $\phi$.
\end{proof}

\begin{proposition} \label{propre}
La restriction $\calP^\ani$ du champ de Picard $\calP$ \`a $\calA^\ani$ est un
champ de Deligne-Mumford s\'epar\'e lisse de type fini au-dessus de $\calA^\ani$.
L'ouvert $\calM^\ani:=\calM\times_\calA \calA^\ani$ de $\calM$ est un
champ de Deligne-Mumford s\'epar\'e lisse de type fini au-dessus de $k$. Le
morphisme $f^\ani:\calM^\ani\rta\calA^\ani$ est un morphisme propre.
\end{proposition}

\begin{proof}
On sait d\'ej\`a que $\calP$ est un champ de Picard lisse au-dessus de
$\calA^\heartsuit$ \cf \ref{P lisse} et $\calM$ est lisse sur $k$ \cf \ref{lisse}.

D'apr\`es \cite[II.4]{Fa} et \ref{ani > stable}, $\calM^\ani$ est s\'epar\'e.  Autrement dit le morphisme diagonal
de $\calM^\ani$ est universellement ferm\'e. D'apr\`es \ref{ani auto}, il est quasi-fini et de fibres r\'eduites. On en d\'eduit qu'il est fini et non ramifi\'e. Par cons\'equent $\calM^\ani$ est un champ de Deligne-Mumford s\'epar\'e. Il en est de m\^eme de $\calP^\ani$ car $\calP^\ani$ s'identifie \`a un ouvert de $\calM^\ani$.

Il reste \`a d\'emontrer que $\calP$ est de type fini et $\calM$ est propre
sur $\calA^\ani$. Disposant du crit\`ere valuatif de propret\'e \cf \cite[II.4]{Fa},
il suffit de d\'emontrer qu'il est de type  fini sur $\calA^\ani$.
On sait que $\calM$ est localement de type fini. Pour tout
$a\in\calA^\ani(\bar k)$, on peut choisir une famille d'ouverts de type fini
de $\calM^\ani$ qui recourvrent la fibre $\calM^\ani$. D'apr\`es la formule
de produit \ref{produit}, la fibre $\calM_a$ est noeth\'e\-rienne si bien qu'il
existe une famille finie d'ouverts de type fini de $\calM^\ani$ qui recouvre
$\calM_a$. Puisque $f:\calM^\ani\rightarrow \calA^\ani$ est plat, l'image de
ces ouverts de $\calA^\ani$ sont des ouvert de $\calA^\ani$ contenant $a$.
Soit $V_a$ leur intersection. Il est clair que l'ouvert $f^{-1}(V_a)$ peut
\^etre recouvert par une famille finie d'ouverts de type fini. Il reste \`a
remarquer que $\calA^\ani$ est aussi noeth\'erien de sorte qu'il existe un
nombre fini de points $a$ tels que les ouverts $V_a$ comme
ci-dessus recouvrent $\calA^\ani$. La proposition s'en d\'eduit.
\end{proof}

\subsection{La $\kappa$-d\'ecomposition sur $\tilde\calA^\ani$}
\label{subsection : kappa-decomposition}

Le morphisme
$$f^\ani:\calM^\ani\rightarrow \calA^\ani$$
est un morphisme propre. Sa source $\calM^{\ani}$ est un champ de Deligne-Mumford
lisse. D'apr\`es le th\'eor\`eme de puret\'e de Deligne \cite{Weil2}, $f^\ani_* \Ql$ est un
complexe pur. D'apr\`es \cite[5.4.5]{BBD}, il est isomorphe \`a la somme
directe de ses faisceaux pervers de cohomologie
$$f^\ani_* \simeq \bigoplus_n \,^p\rmH^n(f^\ani_*\Ql)[-n].$$
o\`u $^p\rmH^n(f^\ani_*\Ql)$ est un faisceau pervers pur de poids $n$.

\begin{numero}
Puisque $\calP^\ani$ agit sur $\calM^\ani$ au-dessus de $\calA^\ani$,
$\calP^\ani$ agit sur l'image directe $f_*^\ani\Ql$. D'apr\`es le lemme
d'homotopie \cf \cite[3.2.3]{LN}, l'action de $\calP^\ani$ sur les faisceaux
pervers de cohomologie $^p\rmH^n(f_*^\ani\Ql)$ se factorise par le
faisceau en groupes ab\'eliens finis $\pi_0(\calP^\ani)$. Pour comprendre
cette action, il est commode de se restreindre d'abord \`a l'ouvert \'etale
$\tilde\calA^\ani=\tilde\calA\times_{\calA}\calA^{\ani}$ avec
$\tilde\calA$ d\'efini dans le paragraphe \ref{subsection : tilde
A}.
\end{numero}

\begin{numero}
Au-dessus de $\tilde\calA$, on a l'homomorphisme surjectif \ref{surjectif}
$$\bbX_*\times \tilde\calA\rightarrow\pi_0(\calP)|_{\tilde\calA}.$$
Notons $\tilde\calA^\ani$ l'image r\'eciproque de $\calA^\ani$ \`a
$\tilde\calA$ et $\tilde f^\ani$ celle du morphisme $f^\ani$. Au-dessus de
l'ouvert $\tilde\calA^\ani$, l'homomorphisme
$\bbX_*\rightarrow\pi_0(\calP)|_{\tilde\calA}$ se factorise par un quotient
fini de sorte que pour tout $\kappa\in\hat\bbT$, on peut d\'efinir un facteur
direct
$$^p\rmH^n(\tilde f_*^\ani\Ql)_\kappa$$
sur lequel $\bbX_*$ agit \`a travers le caract\`ere
$\kappa:\bbX_*\rightarrow \Ql^\times$.
\end{numero}

\begin{numero}\label{tilde A_kappa}
Pour tout $\kappa\in\hat\bbT$, soit $\tilde\calA_\kappa$ la r\'eunion  des
strates $\tilde\calA_{\tilde \psi}$ avec $\tilde\psi\in \Psi$ et
$\kappa\in \hat\bbT(I_{\tilde \psi},W_{\tilde\psi})$. D'apr\`es \ref{T(I,W) inclusion} et \ref{I W croissant}, $\tilde \calA_\kappa$
est un ferm\'e de $\tilde\calA$.
\end{numero}

\begin{numero}\label{kappa ani}
Nous notons  $\kappa^\ani$ le sous-ensemble de $\hat\bbT$ form\'e des \'el\'ements
$\kappa\in\hat\bbT$ tel que $\kappa\in \hat\bbT(I_{\tilde \psi},
W_{\tilde\psi})$ pour une certaine strate anisotrope $\tilde\psi\in
\tilde\Psi^\ani$. Comme $\tilde\Psi^\ani$ est un ensemble fini et pour tout $\tilde\psi\in
\tilde \Psi^\ani$, $\hat\bbT(I_{\tilde \psi}, W_{\tilde\psi})$ est un
sous-groupe fini de $\hat\bbT$, l'ensemble $\kappa^\ani$ est un
sous-ensemble fini de $\hat\bbT$. Le facteur direct $^p\rmH^n(\tilde f_*^\ani\Ql)_\kappa$ est non nul seulement si $\kappa\in \kappa^\ani$.
On a donc une d\'ecomposition en somme directe finie
\begin{equation}\label{decomposition tilde f^ani}
^p\rmH^n(\tilde f_*^\ani\Ql)= \bigoplus_{\kappa \in \kappa^\ani}
\,^p\rmH^n(\tilde f_*^\ani\Ql)_{\kappa}.
\end{equation}
\end{numero}

\begin{proposition}\label{contenu dans tilde A kappa}
Le support du faisceau pervers $^p\rmH^n(\tilde
f_*^\ani\Ql)_\kappa$ est contenu dans $\tilde\calA_\kappa\cap \calA^{\rm ani}$.
L'intersection $\tilde\calA_\kappa\cap \tilde\calA^\ani$ est non vide si et seulement si $\kappa$ est un \'el\'ement de l'ensemble fini $\kappa^\ani$.
\end{proposition}

\begin{proof}
D'apr\`es la description de
$\pi_0(\calP)$ sur $\tilde\calA$, au-dessus de l'ouvert
$U=\tilde\calA^\ani-\tilde\calA_\kappa$, l'homomorphisme
$\bbX_*\rightarrow\pi_0(\calP)$ se factorise par un quotient $\Gamma$ de
$\bbX_*$ tel que $\kappa$ soit en dehors du sous-groupe $\Gamma^*$ de
$\hat\bbT$. Il s'ensuit que la restriction de $^p\rmH^n(\tilde
f_*^\ani\Ql)_\kappa$ \`a $U$ est nulle.
\end{proof}

\subsection{L'immersion ferm\'ee de $\tilde\calA_H$ dans $\tilde\calA$}
\label{subsection : immersion fermee}

Soit $(\kappa,\rho_\kappa)$ une donn\'ee endoscopique de $G$ sur $\ovl X$
\cf \ref{donnee endoscopique}. On lui a associ\'e un groupe endoscopique
$H$ sur $\ovl X$. On a construit un morphisme
$$\nu:\calA_H\rightarrow \bar\calA$$
entre les bases des fibrations de Hitchin pour $H$ et pour $G$
\cf \ref{subsection : Fibration de Hitchin de groupes endoscopiques}.

Soit $(\kappa,\rho_\kappa^\bullet)$ une donn\'ee endoscopique point\'ee de
$G$ sur $\ovl X$ \cf \ref{donnee endoscopique pointee}. On a alors un
$\pi_0(\kappa)$-torseur $\rho_\kappa:\ovl X_{\rho_\kappa}\rightarrow \ovl
X$ avec un point $\infty_{\rho_\kappa}$ au-dessus de $\infty$. On en
d\'eduit un point $\infty_G$ du $\Out(\GG)$-torseur $\rho_G$ et un point
$\infty_H$ du $\Out(\HH)$-torseur $\rho_H$.

\begin{numero}
Avec la donn\'ee endoscopique point\'ee, on peut d\'efinir un
morphisme canonique
$$\tilde \nu:\tilde\calA_H\rightarrow \tilde\calA$$
de la fa{\c c}on suivante.
Soit $a_H\in\calA_H^\infty(\bar k)$ d'image $a\in\calA^\infty(\bar k)$.
Rappelons que les courbes cam\'erales $\widetilde X_{a_H}$ et $\widetilde
X_a$ ne sont pas directement reli\'ees mais on a un morphisme entre leurs
rev\^etements \'etales
$$\widetilde X_{\rho_\kappa,a_H}\rightarrow \widetilde X_{\rho_\kappa,a}.$$
En effet, on a le diagramme suivant
\begin{equation}\label{immense diagramme}
 \xymatrix{
& \widetilde X_{\rho_\kappa,a} \ar[dd] \ar[dr]& & \\
\widetilde X_{\rho_\kappa,a_H} \ar@{.>}[ur] \ar[dr] \ar[rr]
& & X_{\rho_\kappa}\times \bbt_D \ar[dd] \ar[dr] & \\
& \ovl X \ar[rr]^{a_H\ \ \ \ \ \ \ \ \ } \ar[dr]_a & & \frakc_{H,D} \ar[dl]^{\nu}\\
& & \frakc_D  &}
\end{equation}
avec deux parall\'elogrammes cart\'esiens. Il s'ensuit que le morphisme
$\nu$ d\'etermine le morphisme en pointill\'e qu'on voulait construire.

Un $\bar k$-point $\tilde a_H=(a_H,\tilde\infty)$ de $\tilde\calA_H$
consiste en un point $a_H\in\calA_H^\infty(\bar k)$ plus un point
$\tilde\infty$ dans $\widetilde X_{a_H}$ au-dessus de $\infty$. La donn\'ee
du point $\infty_{\rho_\kappa}$ de $\rho_\kappa$ d\'etermine alors un point
$\tilde\infty_{\rho_\kappa}=(\tilde\infty, \infty_{\rho_\kappa})$
de $\widetilde X_{\rho_\kappa,a_H}$. Il revient au m\^eme se donner un
point de $\tilde\calA_H$ que de se donner un couple
$(a_H,\tilde\infty_{\rho_\kappa})$ avec $a_H\in\calA_H^\infty(\bar k)$ et
$\tilde\infty_{\rho_\kappa}\in \widetilde X_{\rho_\kappa,a_H}(\bar k)$
au-dessus de $\infty$. Avec cette nouvelle description, en utilisant le
morphisme $\widetilde X_{\rho_\kappa,a_H}\rightarrow \widetilde
X_{\rho_\kappa,a}$ construit ci-dessus, on obtient le morphisme
$$\tilde \nu:\tilde\calA_H\rightarrow \tilde\calA$$
qu'on voulait.
\end{numero}

\begin{remarque}
Si le point $\infty_{\rho_\kappa}$ est d\'efini sur $k$, le morphisme $\tilde
\nu:\tilde\calA_H\rightarrow \tilde\calA$ ci-dessus est aussi d\'efini sur $k$.
\end{remarque}

\begin{proposition}\label{immersion fermee}
Le morphisme $\tilde \nu:\tilde\calA_{H}\rightarrow \tilde \calA\otimes_k
\bar k$ est une immersion ferm\'ee.
\end{proposition}

\begin{proof}
On sait par \cite[10.3]{N} que c'est un morphisme fini non ramifi\'e. Il suffit
maintenant de v\'erifier qu'il induit une injection sur les points
g\'eom\'etriques. Voici un \'enonc\'e plus pr\'ecis.
\end{proof}

\begin{proposition}\label{Galois}
Un point $\tilde a\in \tilde \calA(\bar k)$ peut \^etre d\'ecrit comme un
couple $(a,\tilde\infty_{\rho_\kappa})$ avec $a\in\calA^\infty(\bar k)$ et
$\infty_{\rho_\kappa}$ dans $\widetilde X_{\rho_\kappa,a}$ au-dessus de
$\infty_{\rho_\kappa}$. Soit $U$ l'ouvert de $\bar X$ d\'efini par
$U=a^{-1}(\frakc_D^\rs)$. Le point $\tilde\infty_{\rho_\kappa}$ de
$\widetilde X_{\rho_\kappa,a}$ d\'efinit un homomorphisme
$$\pi_{\rho_\kappa,a}^\bullet : \pi_1(U,\infty) \rightarrow \bbW\rtimes \pi_0(\kappa).$$
Le point $\tilde a$ provient d'un point $\tilde a_H\in \tilde\calA_{H}(\bar
k)$ si et seulement si l'image de $\pi_{\rho_\kappa,a}^\bullet$ est contenue
dans le sous-groupe $\bbW_\HH\rtimes \pi_0(\kappa)$ de $\bbW\rtimes
\pi_0(\kappa)$ d\'efini par le lemme \ref{homomorphisme produit
semi-direct}.
\end{proposition}

\begin{proof}
Les points $a_H\in\calA_H(\bar k)$ au-dessus de $a$ correspondent
bijectivement aux composantes irr\'eductibles de l'image inverse de la
section $a(\bar X)\subset \frakc_D$ par le morphisme fini plat
$\nu:\frakc_{H,D}\rightarrow \frakc_D$ dont la projection sur $\bar X$ est
un isomorphisme. Cette image inverse $\nu^{-1}(a(\bar X))$ est par ailleurs
le quotient au sens des invariants de $\widetilde X_{\rho_\kappa,a}$ par le
sous-groupe $\bbW_H\rtimes \pi_0(\kappa)$ de $\bbW\rtimes \pi_0(\kappa)$
d\'efini par le lemme \ref{homomorphisme produit semi-direct}. Le point
$(a,\tilde\infty_{\rho_\kappa})$ provient d'un point
$(a_H,\tilde\infty_{\rho_\kappa})\in \tilde\calA_{H}(\bar k)$ si et seulement
si la projection sur $\bar X$ de la composante irr\'eductible passant par
l'image de $\tilde\infty_{\rho_\kappa}$ dans le quotient au sens des
invariants
$$\nu^{-1}(a(\bar X))=\widetilde X_{\rho_\kappa,a} / (\bbW_H\rtimes\pi_0(\kappa))
$$
est un isomorphisme. Puisque ces composantes irr\'eductibles sont finies et
plates au-dessus de $\bar X$, il suffit de v\'erifier ce crit\`ere au-dessus de
l'ouvert $\ovl U$ o\`u les projections sur $\bar X$ envisag\'ees sont finies et
\'etales. On en d\'eduit le crit\`ere galoisien dans l'\'enonc\'e du lemme.
\end{proof}

Nous allons maintenant d\'ecrire le ferm\'e $\tilde\calA_\kappa$ comme une
r\'eunion disjointes des ferm\'es $\tilde\calA_H$ pour de diff\'erentes
donn\'ees endoscopiques point\'ees. Il faut pour cela choisir une r\'eduction
du torseur $\rho_G$ bonne pour tous les groupes endoscopiques envisag\'es.

Soit $\bar\rho_G$ la restriction de $\rho_G$ \`a $\ovl X$. Fixons un $\bar
k$-point $\infty_G$ de $\rho_G$ au-dessus de $\infty$. Ce point d\'efinit un
homomorphisme $\rho_G^\bullet:\pi_1(\ovl X,\infty)\rightarrow \Out(\GG)$.
Nous allons fixer un \'el\'ement $\kappa\in\hat\bbT$. Rappelons qu'une
donn\'ee endoscopique point\'ee \ref{donnee endoscopique pointee} avec
ce $\kappa$ fixe, est un homomorphisme
$$\rho_\kappa^\bullet:\pi_1(\ovl X,\infty)\rightarrow \pi_0(\kappa)$$
au-dessus de $\rho_G^\bullet$.

\begin{lemme}
Il existe un quotient fini $\rho^\bullet:\pi_1(\ovl X,\infty) \rightarrow
\Theta_{\rho}$ tel que $\rho_G^\bullet$ se factorise par $\rho^\bullet$
ainsi que toutes les donn\'ees endoscopiques point\'ees
$\rho_\kappa^\bullet$.
\end{lemme}

\begin{proof}
Puisque le noyau de $\pi_0(\kappa)\rightarrow \Out(\GG)$ est un groupe
fini, il n'y a qu'un nombre fini d'homomorphismes $\pi_1(\ovl
X,\infty)\rightarrow \pi_0(\kappa)$ au-dessus de $\rho_G^\bullet$.
\end{proof}

\begin{proposition}
Le ferm\'e $\tilde\calA_\kappa$ de \ref{tilde A_kappa} est la r\'eunion
disjointe index\'ee par l'ensemble des donn\'ees endoscopiques point\'ees
$(\kappa,\rho_\kappa^\bullet)$ des ferm\'es $\tilde\calA_H$ pour les
groupes endoscopiques $H$ associ\'es \`a $(\kappa,\rho_\kappa^\bullet)$.
\end{proposition}

\begin{proof}
Comme dans le lemme pr\'ec\'edent, fixons donc une factorisation de
$\rho_G^\bullet$ un quotient fini $\rho^\bullet:\pi_1(\ovl
X,\infty)\rightarrow \Theta_{\rho}$ suivi d'un homomorphisme ${\bf o}_\GG:
\Theta_\rho\rightarrow \Out(\GG)$ comme ci-dessus. Puisque $\rho^\bullet$
est surjectif, une donn\'ee endoscopique point\'ee
$(\kappa,\rho_\kappa^\bullet)$ sur $\ovl X$ est alors \'equivalente \`a un
homomorphisme ${\bf o}:\Theta_\rho\rightarrow \pi_0(\kappa)$ au-dessus
de ${\bf o}_\GG$. La donn\'ee de $\rho^\bullet$ d\'etermine un
rev\^etement fini \'etale $\rho:\ovl X_\rho\rightarrow \ovl X$ de groupe de
Galois $\Theta_\rho$ muni d'un point $\infty_\rho$ au-dessus de $\infty$.

Soit ${\bf o}:\Theta_\rho\rightarrow \pi_0(\kappa)$ un homomorphisme
au-dessus de ${\bf o}_\GG:\Theta_\rho\rightarrow \Out(\GG)$. Soit $H$ le
groupe endoscopique associ\'e. Soit $\tilde
a=(a,\tilde\infty_{\rho})\in\calA(\bar k)$ un point dans le ferm\'e $\calA_H$
comme dans \ref{Galois}. Montrer qu'il est dans $\tilde\calA_\kappa$
revient \`a montrer que le groupe de monodromie $W_{\tilde a}$ est
contenu dans $(\bbW\rtimes\Theta_\rho)_\kappa$ et son sous-groupe
normal $I_{\tilde a}$ est contenu dans $\bbW_\HH$.

D'apr\`es la proposition \ref{Galois}, le groupe de monodromie $W_{\tilde
a}$ est contenu dans le sous-groupe $\bbW_\HH\rtimes\Theta_\rho$ de
$\bbW\rtimes \Theta_\rho$.  Ce sous-groupe fixe $\kappa$ de sorte qu'on a
$W_{\tilde a}\subset (\bbW\rtimes\Theta_\rho)_\kappa$. Puisqu'on a un
morphisme $\widetilde X_{\rho,a} \rightarrow \bar X_\rho$ qui est
$\bbW\rtimes \Theta_{\rho}$-\'equivariant avec $\Theta_\rho$ agissant sans
points fixes sur $\bar X_\rho$, le groupe $I_{\tilde a}$ est n\'ecessairement
contenu dans le noyau de $\bbW\rtimes \Theta_\rho\rightarrow
\Theta_\rho$ donc $I_{\tilde a}\subset \bbW$. On a alors
$$I_{\tilde a}\subset \bbW \cap (\bbW_\HH\rtimes \Theta_\rho) =\bbW_\HH$$
de sorte que $\tilde a$ est bien un point de $\tilde\calA_\kappa$.

Soit maintenant $\tilde a=(a,\tilde\infty_{\rho})\in\tilde\calA_\kappa$.
Montrons d'abord qu'il d\'etermine un unique homomorphisme ${\bf
o}:\Theta_\rho\rightarrow \pi_0(\kappa)$.

Soit $C_{\tilde a}$ la composante de $\widetilde X_{\rho,a}^\flat$
contenant $\tilde\infty_\rho$. Rappelons qu'on a not\'e $W_{\tilde a}$ le
sous-groupe de $\bbW\rtimes \Theta_\rho$ des \'el\'ements qui laissent
stable la composante $C_{\tilde a}$ et $I_{\tilde a}$ est le sous-groupe
normal de $W_{\tilde a}$ engendr\'e par des \'el\'ements ayant au moins un
point fixe dans $C_{\tilde a}$. On sait que $I_{\tilde a}\subset \bbW$. Le
quotient au sens des invariants $D_{\tilde a}=C_{\tilde a}/ I_{\tilde a}$ est
alors un rev\^etement fini \'etale galoisien de $\bar X$ de groupe de Galois
$W_{\tilde a}/ I_{\tilde a}$. L'image du point $\tilde\infty_\rho$ dans le
rev\^etement fini \'etale $D_{\tilde a}$ d\'efinit un homomorphisme
$\varrho_{\tilde a}^\bullet:\pi_1(\bar X,\infty)\rightarrow W_{\tilde a}/
I_{\tilde a}$. Comme l'image de $\tilde\infty_\rho$ dans le rev\^etement
fini \'etale galoisien $\bar X_\rho$ est le point $\infty_\rho$ fix\'e, on a un
diagramme commutatif
$$
\xymatrix{
 \pi_1(\bar X,\infty) \ar[dr]_{\rho^\bullet} \ar[r]^{\varrho_{\tilde a}^\bullet}
 & W_{\tilde a}/ I_{\tilde a}\ar[d]^{}  \\
 &     \Theta_\rho      }
$$
o\`u la fl\`eche verticale se d\'eduit de l'inclusion $W_{\tilde a}\subset
\bbW\rtimes \Theta_\rho$ et $I_{\tilde a}\subset \bbW$. En particulier, la
fl\`eche verticale est n\'ecessairement surjective car $\rho^\bullet$ l'est.

Puisque que $\tilde a\in\tilde\calA_\kappa(k)$, on a $I_{\tilde a}\subset
\bbW_\HH$ et $W_{\tilde a}\subset (\bbW\rtimes\Theta_\rho)_\kappa$. On
en d\'eduit  un homomorphisme $W_{\tilde a}$ dans $(\bbW\rtimes
\Out(\GG))_\kappa$ qui envoie $I_{\tilde a}$ dans $\bbW_\HH$. Il en
r\'esulte un homomorphisme
$${\bf o}_{\tilde a}:W_{\tilde a}/ I_{\tilde a}
\rightarrow (\bbW\rtimes \Out(\GG))_\kappa/\bbW_\HH=\pi_0(\kappa).
$$
Puisqu'on a choisi le quotient $\Theta_\rho$ suffisamment grand, il existe
un unique homomorphisme ${\bf o}:\Theta_\rho\rightarrow \pi_0(\kappa)$
tel que le diagramme suivant est commutatif
\begin{equation}
\label{triangle commutatif}
\xymatrix{ \pi_1(X,\infty) \ar[dr]_{\rho^\bullet} \ar[r]^{\varrho_{\tilde a}^\bullet} &
W_{\tilde a}/ I_{\tilde a} \ar[d]_{} \ar[dr]^{{\bf o}_{\tilde a}}   &  \\
& \Theta_\rho \ar[r]_{{\bf o}} & \pi_0(\kappa)}
\end{equation}
Ainsi un point $\tilde a\in \tilde\calA_\kappa(k)$ d\'etermine une donn\'ee
endoscopique point\'ee associ\'ee \`a $(\kappa,{\bf o})$.

Montrons maintenant $\tilde a\in\tilde\calA_H$ o\`u $H$ est le groupe
endoscopique associ\'e \`a $(\kappa,{\bf o})$. Le diagramme commutatif
$$
\xymatrix{
1 \ar[r]& W_{\tilde a}\cap \bbW \ar[r]_{}& W_{\tilde a} \ar[d]_{} \ar[r]^{}
& \Theta_\rho \ar[d]^{\bf o} \ar[r]^{} & 1 \\
1 \ar[r] & \bbW_\HH \ar[r]& (\bbW\rtimes \Out(\GG))_{\kappa} \ar[r] & \pi_0(\kappa)\ar[r]  & 1        }
$$
implique imm\'ediatement $W_{\tilde a}\cap \bbW \subset \bbW_\HH$. On
conna{\^\i}t de plus un scindage canonique de la suite du bas
$$\theta:\pi_0(\kappa)\rightarrow (\bbW\rtimes \Out(\GG))_{\kappa} $$
de sorte qu'on a un homomorphisme
$$W_{\tilde a} \rightarrow \bbW_\HH\rtimes \Theta_\rho$$
o\`u le produit semi-direct est form\'e \`a l'aide de ${\bf o}$ et de
l'homomorphisme $\theta$ dans la d\'emonstration du lemme
\ref{homomorphisme produit semi-direct}. En composant avec l'inclusion
$\bbW_\HH\rtimes \Theta_\rho \rightarrow \bbW\rtimes \Theta_\rho$, on
obtient un homomorphisme
$$W_{\tilde a}\rightarrow \bbW\rtimes \Theta_\rho.$$
Comparons cet homomorphisme avec l'inclusion naturelle de $W_{\tilde
a}\subset \bbW\rtimes \Theta_\rho$. Les deux induisent le m\^eme
homomorphisme sur le sous-groupe $W_{\tilde a}\cap \bbW\subset
\bbW_\HH$ et induisent l'identit\'e au niveau du quotient $\Theta_\rho$. Il
s'ensuit que les deux homomorphismes \`a comparer sont les m\^emes. Par
cons\'equent, on a
$$W_{\tilde a}\subset \bbW_\HH\rtimes \Theta_\rho$$
ce qui implique que $\tilde a\in \calA_H$.

Enfin le diagramme commutatif \ref{triangle commutatif} montre que ${\bf
o}$ est compl\`etement d\'etermin\'e par ${\bf o}_{\tilde a}$. Ceci implique
que les ferm\'es $\tilde\calA_H$ et $\tilde\calA_{H'}$ pour les groupes
endoscopiques $H$ et $H'$ associ\'es \`a $(\kappa,\rho_\kappa^\bullet)$ et
$(\kappa,\rho'_\kappa)$ diff\'erents, sont des ferm\'es disjoints de
$\tilde\calA\otimes_k \bar k$.
\end{proof}

\subsection{Stabilisation g\'eom\'etrique : formulation sur $\tilde\calA^\ani$}
\label{subsection : stabilisation sur tilde A}

Supposons maintenant qu'on a une donn\'ee endoscopique point\'ee
$(\kappa,\rho_\kappa^\bullet)$ d\'efini sur $X$. Soit $H$ le groupe
endoscopique associ\'e. Consid\'erons la fibration de Hitchin
$f_H:\calM_H\rightarrow\calA_H$ pour le groupe $H$. Consid\'erons le
rev\^etement \'etale $\tilde\calA_H$ de $\calA_H$. Le morphisme
$$\tilde\nu:\tilde\calA_H \rightarrow \tilde\calA_G$$
est alors une immersion ferm\'ee car il en est ainsi apr\`es le changement
de base de $k$ \`a $\bar k$ \cf \ref{immersion fermee}.

\begin{numero}
Consid\'erons l'image r\'eciproque de $^p\rmH^*(f^\ani_*\Ql)$ \`a
$\tilde\calA^\ani$. Au-dessus de $\tilde\calA$, on a un homomorphisme
surjectif $\bbX_*\rightarrow \pi_0(\calP)|_{\tilde\calA}$ \cf \ref{surjectif}
de sorte qu'on a une action de $\bbX_*$ sur
$^p\rmH^*(f^\ani_*\Ql)|_{\tilde\calA}$ qui se factorise par un quotient fini
au-dessus de l'ouvert $\tilde\calA^\infty$. Consid\'erons le facteur direct
$$^p\rmH^*(\tilde f^\ani_*\Ql)_\kappa$$
o\`u $\bbX_*$ agit par le caract\`ere $\kappa:\bbX_* \rightarrow
\Ql^\times$. D'apr\`es le lemme \ref{tilde A_kappa}, le support de ce facteur direct est
contenu dans $\tilde\calA_\kappa$ qui est la r\'eunion disjointe
des images des immersions ferm\'ees $\tilde\calA_{H'}\rightarrow
\tilde\calA$ sur l'ensemble des donn\'ees endoscopiques point\'ees
$(\rho_\kappa^\bullet)':\pi_1(\ovl X,\infty)\rightarrow \pi_0(\kappa)$ et
o\`u $H'$ est le groupe endoscopique associ\'e. Ainsi $^p\rmH^*(\tilde
f^\ani_*\Ql)_\kappa$ se d\'ecompose en une somme directe de termes
$^p\rmH^*(\tilde f^\ani_*\Ql)_{\kappa}|_{\calA_{H'}}$ support\'es par
$\calA_{H'}$. Le facteur direct $^p\rmH^*(\tilde
f^\ani_*\Ql)_{\kappa}|_{\calA_H}$ est d\'efini sur $k$.
\end{numero}

\begin{numero}
Consid\'erons la fibration de Hitchin $f_H:\calM_H\rightarrow\calA_H$ pour
le groupe $H$. En prenant la partie $\kappa=1$ dans les faisceaux pervers
de cohomologie restreints \`a $\tilde\calA_H$, on obtient ainsi un faisceau
pervers gradu\'e
$$^p\rmH^*(\tilde f^\ani_{H,*}\Ql))_{\rm st}.$$
Son image directe par $\tilde\nu_*$ d\'efinit un faisceau pervers gradu\'e
sur $\tilde\calA$.
\end{numero}

\begin{theoreme}\label{stabilisation sur tilde A}
Il existe un isomorphisme entre les
semi-simplifi\-cations des faisceaux pervers gradu\'es
$$\tilde \nu^*\bigoplus_n \,^p\rmH^n(\tilde f^\ani_*\Ql)_{\kappa} [2 r_H^G(D)](r_H^G(D))
\simeq\bigoplus_n \,^p\rmH^n(\tilde
f^\ani_{H,*}\Ql)_{\rm st}
$$
o\`u
$$r_H^G(D)=(\sharp\, \Phi-\sharp\, \Phi_\HH)\deg(D)/ 2$$
est la codimension de $\tilde\calA_H$ dans $\tilde\calA$.
\end{theoreme}

Notons que comme il s'agit des faisceaux pervers purs, au-dessus de
$\tilde\calA^\ani\otimes_k \bar k$, ils sont semi-simples d'apr\`es les
th\'eor\`eme de d\'ecomposition. Il s'en suit que les faisceaux pervers
gradu\'es ci-dessus sont isomorphes au-dessus de $\tilde\calA^\ani\otimes_k
\bar k$.

La d\'emonstration de ce th\'eor\`eme occupera la grande partie du reste
de l'article. On en d\'eduira la conjecture de Langlands-Shelstad \ref{LS}
en cours de sa d\'emonstration.
\subsection{Le compl\'ement de $\calA^\ani$ dans $\calA$}
\label{complement de A ani}

Pour terminer ce chapitre, notons que le compl\'ement de $\calA^\ani$ dans
$\calA^\heartsuit$ est un sous-sch\'ema ferm\'e de grande codimension.

\begin{proposition}
La codimension de $\calA^\heartsuit-\calA^\ani$ dans $\calA^\heartsuit$ est
plus grande ou \'egale \`a $\deg(D)$.
\end{proposition}

\begin{proof}
Le m\^eme argument que dans le lemme \ref{Galois} montre que le compl\'ement
$\calA^\heartsuit-\calA^\ani$ est contenu dans la r\'eunion des images de
$$\calA_M \rightarrow \calA$$
pour les sous-groupes de Levi de $G$ contenant le tore maximal $T$.
D'apr\`es la formule \ref{dim A}, on a
$$\dim(\calA)-\dim(\calA_M)=(\sharp\, \Phi-\sharp\, \Phi_M) \deg(D)/ 2 \geq \deg(D)$$
d'o\`u la proposition.
\end{proof}


\section{Th\'eor\`eme du support} \label{support}

Dans ce chapitre, nous d\'emontrons un th\'eor\`eme du support qui est l'\'enonc\'e technique central de l'article. Dans le premier paragraphe, nous donnons l'\'enonc\'e \cf \ref{support general kappa} de ce th\'eor\`eme du support et discutons de ses variantes utiles. Les hypoth\`eses du th\'eor\`eme sont bien entendues model\'ees sur la fibration de Hitchin mais restent suffisamment raisonnables pour qu'on puisse esp\'erer  l'appliquer \`a d'autres situations.

La d\'emonstration de \ref{support general kappa} occupe les quatres paragraphes suivants. La notion cl\'e d'amplitude d'un support est introduite dans \ref{subsection : amplitude}. Nous y \'enon\c cons une estimation de cette amplitude. En admettant cette estimation, nous d\'emontrons \ref{support general kappa} dans \ref{subsection : demo support} en suivant l'argument de Goresky et MacPherson rappel\'e dans l'appendice \ref{appendice GM}. La d\'emonstration de l'estimation d'amplitude est fond\'ee sur une notion de libert\'e de l'homologie d'un sch\'ema vu comme un module sur l'homogie de la partie ab\'elienne du groupe agissant. Cette propri\'et\'e est bien connue dans une situation absolue. Dans une situation relative, il est assez difficile de d\'egager cet \'enonc\'e de libert\'e et a fortiori de le d\'emontrer. C'est le contenu du long paragraphe \ref{subsection : cap-produit}. Dans le paragraphe suivant, nous expliquons comment cette propri\'et\'e de libert\'e implique l'estimation d'amplitude d\'esir\'ee.

Dans le dernier paragraphe, nous appliquons le th\'eor\`eme du support \`a la partie anisotrope de la fibration de Hitchin.

\subsection{L'\'enonc\'e du th\'eor\`eme du support}\label{subsection : support general}

Soient $S$ et $M$ deux $k$-sch\'emas de type fini lisses et $f:M \rightarrow S$ un morphisme propre, plat et \`a fibres g\'eom\'etriques r\'eduites. D'apr\`es le th\'eor\`eme de puret\'e de Deligne, le complexe $f_*\Ql$ est pur. En appliquant le th\'eor\`eme de d\'ecomposition \cite{BBD}, on sait que $f_*\Ql$ est isomorphe \`a la somme directe de ses faisceaux pervers de cohomologie avec des d\'ecalages \'evidents
$$f_*\Ql \simeq \bigoplus_n ^p\rmH^n(f_*\Ql)[-n].$$
De plus, les faisceaux pervers $K^n=\,^p\rmH^n(f_*\Ql)$ sont g\'eom\'etriquement semi-simples. D'apr\`es \cite{BBD}, pour tout faisceau pervers g\'eom\'etriquement simple $K$ sur $S\otimes_k \bar k$, il existe un sous-sch\'ema ferm\'e r\'eduit irr\'eductible $i:Z\hookrightarrow S\otimes_k \bar k$, un ouvert dense $U\hookrightarrow Z$ et un syst\`eme local irr\'eductible $L$ sur $U$ tel que
$$K=i_* j_{!*} L[\dim(Z)].$$
Le sous-sch\'ema ferm\'e $Z$ est compl\`etement d\'etermin\'e par $K$ et
sera appel\'e le {\em support} de $K$. En g\'en\'eral, le probl\`eme de d\'eterminer les supports des faisceaux pervers simples pr\'esents dans la d\'ecomposition de $f_*\Ql$ est un probl\`eme tr\`es difficile. On peut cependant le r\'esoudre dans le cas o\`u $M$ est munie d'une grosse action d'un sch\'ema en groupes lisse commutatif moyennant quelques hypoth\`eses que nous allons \'enum\'erer.

\begin{numero}\label{Irr}

La r\'eponse au probl\`eme de d\'eterminer les supports sera donn\'ee en termes du faisceau ${\rm Irr}(M/S)$ des composantes irr\'eductibles des fibres de $f$. Ce faisceau existe sous l'hypoth\`ese $f$ plat \`a fibres g\'eom\'etriques r\'eduites. En effet, notons $U$ l'ouvert de $M$ o\`u $f$ est un morphisme lisse. Cet ouvert est alors dense dans chaque fibre g\'eom\'etrique de $f$. En effet, comme un morphisme plat de fibres g\'eom\'etriques lisses est lisse \cite[19.7.1]{EGA4}, l'intersection de $U$ avec une fibre g\'eom\'etrique $M_s$ est exactement la partie lisse de $M_s$ qui est dense dans $M_s$ car celle-ci est g\'eom\'etriquement r\'eduite. Pour tout point g\'eom\'etrique $s$ de $S$, l'ensemble des composantes connexes de $U_s$ s'identifie donc canoniquement \`a l'ensemble des composantes irr\'eductibles de $M_s$. Il suffit donc de construire le faisceau $\pi_0(U/S)$ des composantes connexes des fibres de $U/S$.

Rappelons cette construction dont le lecteur peut se r\'ef\'erer \`a la  proposition 6.2 de \cite{N} pour plus de d\'etails. Pour toute section locale
$u$ de $U$ au-dessus d'un ouvert \'etale $S'$ de $S$, d'apr\`es un
th\'eor\`eme de Grothendieck \cite[15.6.4]{EGA4} , il existe un unique
ouvert de Zariski $M_u$ de $U'=U\times_S S'$ tel que pour tout $s\in S'$, la
fibre $U_{u,s}$ est la composante connexe de $U_s$ contenant le point $u(s)$.
Deux sections locales $u$ et $u'$ sont dites \'equivalentes si les ouverts
correspondants $U_u$ et $U_{u'}$ sont \'egaux. Le faisceau des sections locales de
$U/ S$ modulo cette relation d'\'equivalence d\'efinit un faisceau
constructible $\pi_0(U/ S)$. Ce faisceau v\'erifie en plus la propri\'et\'e que
pour tout point g\'eom\'etrique $s$ de $S$, la fibre de $\pi_0(U/ S)$ en $s$
est l'ensemble des composantes connexes de $U_s$.
\end{numero}

\begin{numero}\label{stabilisateur affine}

Soit $P$ un $S$-sch\'ema en groupes commutatifs lisse \`a fibres connexes. Supposons qu'on a une action de $P$ sur $M$
$${\rm act}:P\times_S M \rightarrow M.$$
Pour tout point g\'eom\'etrique $s$ de $S$, la fibre $P_s$ admet un d\'evissage canonique de Chevalley
$$1\rightarrow R_s \rightarrow P_s \rightarrow A_s \rightarrow 1$$
o\`u $A_s$ est une vari\'et\'e ab\'elienne et o\`u $R_s$ est un groupes alg\'ebrique affine commutatif connexe. Nous dirons que que $P$ agit $M$ avec stabilisateurs affines si pour tout point g\'eom\'etrique $m\in M$ au-dessus de $s\in S$, le stabilisateur de $m$ dans $P_s$ est un sous-groupe affine. Autrement dit,
la composante neutre du stabilisateur de $m$ dans $P_s$ est contenue dans le sous-groupe affine connexe maximal $R_s$ de $P_s$.
\end{numero}

\begin{numero}\label{delta regulier}
Pour tout point g\'eom\'etrique $s\in S$, notons $\delta_s=\dim(R_s)$ la dimension de la partie affine de $P_s$. Si $s\in S$ un point quelconque, le d\'evissage de Chevalley de $P_s$ existe et est unique apr\`es un changement  de base radiciel de sorte que l'entier $\delta_s$ est aussi bien d\'efini. La fonction $\delta$ d\'efini sur l'espace topologique sous-jacent \`a $S$ \`a valeurs dans les entiers naturels est semi-continue \cf \cite{SGA7}. Il existe donc une stratification de $S$ en des sous-sch\'emas localement ferm\'es $S_\delta$ tels que pour tout point g\'eom\'etrique $s\in S_\delta$, on a $\delta_s=\delta$.

On dira que le $S$-sch\'ema en groupes commutatif lisse $P$ est
{\em $\delta$-r\'egulier} si pour tout $\delta\in\mathbb{N}$, on a
$${\rm codim}_S(S_\delta)\geq \delta.$$
Cette hypoth\`ese implique en particulier que ${\rm codim}(S_1)\geq 1$ de sorte que $S_0\not=\emptyset$.
Il en r\'esulte que $P$ est g\'en\'eriquement un sch\'ema ab\'elien.

Soit $Z$ un sous-sch\'ema ferm\'e irr\'eductible de $S$. Soit $\delta_Z$ la valeur minimale que prend la fonction $\delta$ sur $Z$ qui est atteinte sur un ouvert dense de $Z$ qui est alors n\'ecessairement contenu dans l'adh\'erence de la strate $S_{\delta_Z}$. Si $P$ est $\delta$-r\'egulier, alors ${\rm codim}(S_{\delta_Z})\geq \delta_Z$ et a fortiori
$${\rm codim}(Z)\geq \delta_Z.$$
Il est clair que $P$ est $\delta$-r\'egulier si et seulement si pour tout sous-sch\'ema ferm\'e irr\'eductible $Z$ de $S$, on a ${\rm codim}(Z)\geq \delta_Z$.
\end{numero}

\begin{numero}\label{module de Tate polarisable}
Consid\'erons le faisceau des modules de Tate
$$\rmT_\Ql(P)=\rmH^{2d-1}(g_! \Ql)(d)$$
dont la fibre au-dessus de chaque point g\'eom\'etrique $s$ de $S$ est le $\Ql$-module de Tate $\rmT_\Ql(P_s)$ de la fibre en $s$ de $P^0$. Pour tout point g\'eom\'etrique $s$ de $S$, le d\'evissage canonique de Chevalley de $P_s$ induit un d\'evissage de son module de Tate
$$0\rightarrow \rmT_\Ql(R_s)\rightarrow \rmT_\Ql(P_s) \rightarrow \rmT_\Ql(A_s)\rightarrow 0.$$
On a
$$\dim(\rmT_\Ql(A_s))=2\dim(A_s)=2(d-\delta_s)$$
alors que la dimension de $\rmT_\Ql(R_s)$ est \'egale \`a la dimension de la partie torique de $R_s$.

On dira que $\rmT_\Ql(P)$ est {\em polarisable} si localelement pour la topologie \'etale de $S$, il existe une forme bilin\'eaire altern\'ee
$$ \rmT_\Ql(P) \otimes \rmT_\Ql(P) \rightarrow \Ql(1)$$
dont la fibre en chaque point g\'eom\'etrique $s$ a comme noyau $\rmT_\Ql(R_s)$ c'est-\`a-dire qu'elle s'annule sur $\rmT_\Ql(R_s)$ et induit un accouplement parfait sur $\rmT_\Ql(A_s)$.
\end{numero}

\begin{theoreme} \label{support general}
Soient $M$ et $S$ des $k$-sch\'emas lisses et $f:M\rightarrow S$ un morphisme propre plat de fibres g\'eom\'etriques r\'eduites de dimension $d$. Soit $P$ un $S$-sch\'ema en groupes commutatifs lisse de fibre connexe de dimension $d$ qui agit sur $M$ avec stabilisateurs affines. Supposons que le module de Tate $\rmT_\Ql(P)$ est polarisable.

Consid\'erons une stratification
$$S\otimes_k \bar k=\bigsqcup_{\sigma\in\Sigma} S_\sigma$$
telle que pour tout $\sigma\in\Sigma$, $S_\sigma$ est irr\'eductible et telle que
la restriction du faisceau ${\rm Irr}(M/S)$ \`a chaque strate $S_\sigma$ est localement constante.

Soit $K$ un faisceau perfers g\'eom\'etriquement simple pr\'esent dans $f_*\Ql$ dont le support est un sous-sch\'ema ferm\'e irr\'eductible $Z$ de $S\otimes_k \bar k$. Supposons que $\delta_Z\geq {\rm codim}
(Z)$. Alors $Z$ est l'adh\'erence de l'une des strates $S_\sigma$.
\end{theoreme}

Compte tenu de la d\'efinition de $\delta$-r\'egularit\'e, on a un corollaire \'evident.

\begin{corollaire} En plus des hypoth\`eses du th\'eor\`eme, supposons que $P$ est $\delta$-r\'egulier. Alors les supports des faisceaux pervers g\'eom\'etrique\-ment semi-simples pr\'esents dans $f_*\Ql$ sont parmi les adh\'erences $\bar S_\sigma$ des strates $S_\sigma$.
\end{corollaire}

\begin{numero}
Avant de travailler avec les faisceaux pervers de cohomologie $K^n=^p\rmH^n(f_*\Ql)$, envisageons
la m\^eme question de d\'eterminer le support des facteurs du faisceau de cohomologie ordinaire
en degr\'e maximal $\rmH^{2d}(f_*\Ql)$. L'\'enonc\'e \ref{support general} avec le faisceau ordinaire $\rmH^{2d}(f_*\Ql)$ \`a la place des faisceaux pervrers $K^n$, est quasiment imm\'ediat car $\rmH^{2d}(f_*\Ql)$ s'identifie au faisceau $\ell$-adique engendr\'e par le faisceau d'ensembles ${\rm Irr}(M/S)$ \`a l'aide du morphisme trace. L'analogue de \ref{support general} pour le faisceau de cohomologie ordinaire de degr\'e maximal de $f_*\Ql$ r\'esulte de deux lemmes suivants qui sont bien connus.\end{numero}

\begin{lemme}\label{top}
Soit $S$ un $k$-sch\'ema de type fini. Soit $f:M\rightarrow S$ un morphisme
plat de fibres g\'eom\'etriques r\'eduites de dimension $d$. Soit ${\rm Irr}(M/S)$ le faisceau des composantes irr\'eductibles des fibres de $M/S$ d\'efini dans \ref{Irr}. Alors, il existe un isomorphisme
canonique
$$\Ql^{{\rm Irr}(M/ S)}(-d) \isom \rmR^{2d}f_!\Ql(d)$$
o\`u $(-d)$ d\'esigne un twist \`a la Tate.
\end{lemme}

\begin{proof}
Comme dans \ref{Irr}, puisque $f$ est un morphisme plat \`a fibres g\'eom\'etriques r\'eduites, il existe un sous-sch\'ema ouvert $U$ de $M$ dont la trace sur chaque fibre g\'eom\'etrique de $f$ est la partie lisse de cette fibre. On a alors $\pi_0(U/S)={\rm Irr}(M/S)$. De plus, le morphisme $M-U\rightarrow S$ est de dimension relative strictement plus petite que $d$. Soit $h:U\rightarrow S$ la restriction de $f$ \`a $U$. La fl\`eche dans la suite exacte d'excision
$$\rmR^{2d} h_! \Ql \longrightarrow \rmR^{2d} f_! \Ql$$
est alors un isomorphisme. Il suffit donc de construire un isomorphisme
$$\Ql^{\pi_0(U/S)}(-d)\longrightarrow\rmR^{2d} h_!\Ql.$$

Soit $u$ une section locale de $U$ au-dessus d'un ouvert \'etale $S'$ de
$S$. D'apr\`es un th\'eor\`eme de Grothendieck \cite[15.6.4]{EGA4}, il
existe un ouvert il existe un unique ouvert de Zariski $U_u$ de $U'=U\times_S
S'$ tel que pour tout $s\in S$, la fibre $U_{u,s}$ est la composante connexe
de $U_s$ contenant le point $u(s)$. Notons $h_u:U_u\rightarrow S'$ la
restriction de $h$ \`a $U_u$. L'inclusion $U_u\subset U'$ induit une
fl\`eche entre faisceaux de cohomologie \`a support de degr\'e maximal
$$\rmR^{2d} h_{u!} \Ql \longrightarrow \rmR^{2d} h_! \Ql |_{S'}.$$
Par ailleurs, on a un morphisme trace
$$\rmR^{2d} h_{u!} \Ql \longrightarrow \Ql(-d)$$
qui est un isomorphisme puisque $U_u/ S'$ a les fibres g\'eom\'etriquement
connexes. On en d\'eduit un morphisme
$$\Ql(-d)\longrightarrow\rmR^{2d} h_! \Ql |_{Y'}.$$
Il est clair que ce morphisme ne d\'epend pas de la section locale $u$ mais
seulement de sa classe d'\'equivalence c'est-\`a-dire son image dans
$\pi_0(U/ S)$. On en d\'eduit un morphisme canonique
$$\Ql^{\pi_0(U/ S)}(-d) \longrightarrow \rmR^{2d}h_!\Ql.$$
Pour v\'erifier que c'est un isomorphisme, il suffit de le faire fibre par fibre ce qui est \'evident.
\end{proof}

\begin{lemme} \label{strat}
Soit $\mathcal F$ un faisceau $\ell$-adique sur $S$.
Soit $S=\bigsqcup_{\sigma\in\Sigma} S_\sigma$ une stratification de $S$ de strates irr\'eductibles telle que la restriction de $\mathcal F$ \`a chaque strate est localement constante. Soient $i:Z\rightarrow S$ un sous-sch\'ema ferm\'e irr\'eductible de $S$ et $L$ un syst\`eme local sur $Z$. Supposons que $i_*L$ soit un sous-quotient de $\mathcal F$. Alors $S$ est l'adh\'erence de l'une des strates $S_\sigma$.
\end{lemme}

\begin{proof}
La restriction de $i_*L$ \`a chaque strate $S_\sigma$, \'etant un sous-quotient du syst\`eme local $\mathcal F|_{S_\sigma}$, est aussi un syst\`eme local. Par cons\'equent, pour tout $\sigma\in\Sigma$, ou bien $S_\sigma\subset Z$ ou bien $S_\sigma\cap Z=\emptyset$. Puisque $Z$ est irr\'eductible, il existe un unique $\sigma\in\Sigma$ tel que $Z\cap S_\sigma$ est un ouvert dense de $Z$. Alors $S_\sigma$ est un ouvert dense de $Z$.
\end{proof}

\begin{numero}
Consid\'erons une variante plus compliqu\'ee du th\'eor\`eme \ref{support general} o\`u on ne suppose plus les fibres de $P$ sont connexes. Soit $P$ un $S$-sch\'ema en groupes commutatifs lisse de type fini. Il existe un sous-sch\'ema en groupes ouvert $P^0$ de $P$ dont la  fibre $P^0_s$ au-dessus de chaque point $s$ de $S$ est la composante neutre de $P_s$. En rempla\c cant $P$ par $P^0$, on peut d\'efinir la notion de $\delta$-r\'egularit\'e et celle de module de Tate ${\rm T}_\Ql(P^0)$ polarisable.
\end{numero}

\begin{numero}
Il existe un faisceau en groupes ab\'eliens finis $\pi_0(P)$ pour la topologie \'etale de $S$ qui interpole les groupes de composantes connexes $\pi_0(P_s)$ des fibres de $P$   \cf \cite[6.2]{N}.

L'action de $P$ sur $M$ induit une action du faisceau en groupes $\pi_0(P)$ sur le faisceau d'ensembles ${\rm Irr}(M/S)$. Donnons-nous un homomorphisme
$$\Pi_0\rightarrow \pi_0(P)$$
o\`u $\Pi_0$ est le faisceau constant de valeur d'un certain groupe ab\'elien fini $\Pi_0$. Localement pour la topologie \'etale de $S$, il existe de tels homomorphismes qui ne sont pas triviaux.

On a alors une action du groupe fini $\Pi_0$ sur le faisceau d'ensembles ${\rm Irr}(M/S)$. Consid\'erons le faisceau $\ell$-adique $\Ql^{{\rm Irr}(M/S)}$ qui associe \`a tout ouvert \'etale $S'$ de $S$ l'espace vectoriel
$\Ql^{{\rm Irr}(M/S)(U')}$ et qui lui aussi est muni d'une action du groupe fini $\Pi_0$. Pour tout caract\`ere $\kappa:\Pi_0\rightarrow\Ql^\times$, on peut donc d\'efinir le facteur direct
$$(\Ql^{{\rm Irr}(M/S)})_\kappa$$
o\`u $\Pi_0$ agit \`a travers le caract\`ere $\kappa$. L'ensemble des points g\'eom\'etriques $s$ de $S$ tels que $\kappa:\Pi_0\rightarrow \Ql^\times$ se factorise \`a travers $\Pi_0\rightarrow \pi_0(P_s)$ forme un sous-sch\'ema ferm\'e de $S$ que nous allons noter $S_\kappa$. Il est clair que le faisceau $(\Ql^{{\rm Irr}(M/S)})_\kappa$ est support\'e par $S_\kappa$.
\end{numero}

\begin{numero}\label{kappa decomposition} D'apr\`es le lemme d'homotopie, le faisceau $P$ agit sur les faisceaux pervers de cohomologie $K^n=\,^p\rmH^n(f_*\Ql)$ \`a travers le quotient $\pi_0(P)$. Par cons\'equent, le groupe fini $\Pi_0$ agit sur le faisceau pervers $K^n$ qui induit une d\'ecomposition en somme directe
$$K^n=\bigoplus_{\kappa\in \Pi_0^*} K^n_\kappa$$
o\`u $\Pi_0^*=\Hom(\Pi_0,\Ql^\times)$.
\end{numero}

Voici une g\'en\'eralisation de \ref{support general}.

\begin{theoreme} \label{support general kappa}
Soient $M$ et $S$ des $k$-sch\'emas lisses et $f:M\rightarrow S$ un morphisme propre plat de fibres g\'eom\'etriques r\'eduites de dimension $d$. Soit $P$ un $S$-sch\'ema en groupes commutatifs lisse de type fini et de dimension relative $d$ qui agit sur $M$ avec stabilisateurs affines. Supposons que le module de Tate $\rmT_\Ql(P^0)$ est polarisable.

Soit $\Pi_0\rightarrow \pi_0(P)$ un homomorphisme d'un groupe fini $\Pi_0$ dans le faisceau $\pi_0(P)$ des composantes connexes des fibres de $P$. Soit $\kappa\in \Pi_0^*$ un caract\`ere de $\Pi_0$. Soit $S_\kappa$ le ferm\'e de $S$ des points $s\in S$ tel que $\kappa:\Pi_0\rightarrow \Ql^\times$ se factorise \`a travers $\pi_0(P_s)$.

Consid\'erons une stratification
$$S\otimes_k \bar k=\bigsqcup_{\sigma\in\Sigma} S_\sigma$$
telle que pour tout $\sigma\in\Sigma$, $S_\sigma$ est irr\'eductible et telle que
la restriction du faisceau $(\Ql^{{\rm Irr}(M/S)})_\kappa$ \`a chaque strate $S_\sigma$ est localement constante. Le ferm\'e $S_\kappa$ est alors n\'ecessairement une r\'eunion des strates $S_\sigma$ pour $\sigma$ dans un sous-ensemble de $\Sigma$.

Soit $K$ un faisceau perfers g\'eom\'etriquement simple pr\'esent dans $K^n_\kappa$ dont le support est un sous-sch\'ema ferm\'e irr\'eductible $Z$ de $S\otimes_k \bar k$. Supposons que $\delta_Z\geq {\rm codim}
(Z)$. Alors $Z$ est l'adh\'erence de l'une des strates $S_\sigma$ qui sont contenues dans $S_\kappa$.

\end{theoreme}

Signalons un corollaire imm\'ediat de ce th\'eor\`eme.

\begin{corollaire} \label{kappa partie delta regulier}
Si nous supposons en plus que $P$ est $\delta$-r\'egulier alors tous les supports des faisceaux pervers g\'eom\'etriquement simples pr\'esents dans $K^n_\kappa$ sont des adh\'erences de strates $S_\sigma$ contenues dans $S_\kappa$.
\end{corollaire}

\begin{numero}
De nouveau, l'analogue de \ref{support general kappa} pour le faisceau de cohomologie ordinaire de degr\'e maximal de $f_*\Ql$ \`a la place des faisceaux pervers de cohomolofie $K^n$ r\'sulte imm\'ediatement des lemmes \ref{top} et \ref{strat}.
\end{numero}

\begin{numero}
En g\'en\'eral, le probl\`eme de trouver une stratification de $S$  telle que la restriction du faisceau $(\Ql^{{\rm Irr}(M/S)})_\kappa$ \`a chaque strate  est localement constante, est accessibile car il s'agit d'\'etudier la variation en fonction de $s$ de la repr\'esentation du groupe fini $\pi_0(P_s)$ sur l'espace vectoriel de dimension finies ${\rm Irr}(M_s)$ qui sont compl\`etement explicites. Sur l'ouvert anisotrope de la fibration de Hitchin la situation est encore plus favorable comme nous allons voir.

Supposons qu'il existe un ouvert $M^{\reg}$ de $M$ qui est dense dans chaque fibre de $f$ et tel que $P$ agit simplement transitivement sur $M^\reg$. Supposons aussi qu'il existe une section $S\rightarrow M^\reg$. On a alors des isomorphismes
$${\rm Irr}(M/S)= \pi_0(M^\reg/S)=\pi_0(P)$$
des faisceaux d'ensembles munis d'action de $\pi_0(P)$. Ici $\pi_0(P)$ agit sur lui-m\^eme par translation.

Si $S$ est un trait hens\'elien de point ferm\'e $s_0$, la fl\`eche de sp\'ecialisaiton ${\rm Irr}(M_{s_0})\rightarrow {\rm Irr}(M_s)$ est surjectif car $M$ est propre sur $S$. Cette surjectivit\'e de la fl\`eche de sp\'ecialisation est donc aussi v\'erifi\'e pour les faisceau de groupes ab\'eliens $\pi_0(P)$. Localement pour le topologie \'etale de $S$, il existe donc un homomorphisme surjectif $\Pi_0\rightarrow \pi_0(P)$ d'un groupe constant $\Pi_0$ sur $\pi_0(P)$. Pour simplifier, nous allons supposer l'existence de l'homomorphisme surjectif $\Pi_0\rightarrow \pi_0(P)$ sur $S$.
\end{numero}

\begin{lemme} \label{i kappa}
Supposons que ${\rm Irr}(M/S)= \pi_0(M^\reg/S)=\pi_0(P)$  et qu'il existe un homomorphisme surjectif $\Pi_0\rightarrow \pi_0(P)$ comme ci-dessus. Soient $\kappa:\Pi_0\rightarrow \Ql^\times$ un caract\`ere de $\Pi_0$. Soit $S_\kappa$ le ferm\'e de $S$ des points $s\in S$ tel que $\kappa$ se factorise \`a travers $\Pi_0 \rightarrow \pi_0(P_s)$ et $i_\kappa:S_\kappa\rightarrow S$ l'immersion ferm\'ee. On a alors un isomorphisme
$$(\Ql^{{\rm Irr}(M/S)})_\kappa=i_{\kappa,*}\Ql.$$
\end{lemme}

\begin{proof}
On a une d\'ecomposition du faisceau constant $\Ql^{\Pi_0}$ en somme directe
$$\Ql^{\Pi_0}=\bigoplus_{\kappa\in \Pi_0^*} (\Ql)_\kappa$$
o\`u $\Pi_0$ agit sur $(\Ql)_\kappa$ via le caract\`ere $\kappa$. L'homomorphisme surjectif $\Pi_0\rightarrow \pi_0(P)$ induit alors un homomorphisme surjectif de faisceaux $\ell$-adiques
$$(\Ql)_\kappa \rightarrow (\Ql^{{\rm Irr}(M/S)})_\kappa$$
dont la fibre est nulle en dehors de $S_\kappa$ et non nulle sur $S_\kappa$. Le lemme s'en d\'eduit.
\end{proof}

En combinant le corollaire \ref{kappa partie delta regulier} et le lemme \ref{i kappa}, on obtient une description compl\`ete des supports des faisceaux pervers simples sous des hypoth\`eses favorables.

\begin{corollaire}\label{support final}
Soient $M$ et $S$ des $k$-sch\'emas lisses et $f:M\rightarrow S$ un morphisme propre plat de fibres g\'eom\'etriques r\'eduites de dimension $d$. Soit $P$ un $S$-sch\'ema en groupes commutatifs lisse de type fini et de dimension relative $d$ qui agit sur $M$ avec stabilisateurs affines. Supposons que les hypoth\`eses suivantes soient v\'erifi\'ees :
\begin{enumerate}
\item le module de Tate $\rmT_\Ql(P^0)$ est polarisable,
\item ${\rm Irr}(M/S)= \pi_0(M^\reg/S)=\pi_0(P)$,
\item il existe un homomorphisme surjectif d'un faisceau constant $\Pi_0$ sur le faisceau  $\pi_0(P)$.
\end{enumerate}
Alors, pour tout $\kappa\in \Pi_0^*$, soit $K$ un faisceau pervers g\'eom\'etriquement simple $K$ pr\'esent dans $^p\rmH^n(f_*\Ql)_\kappa$ de support $Z$ v\'erifiant ${\rm codim}(Z) \geq \delta_Z$. Alors $Z$  est \'egal \`a une des composantes irr\'eductibles de $S_\kappa$.

Si on suppose de plus que $P$ est $\delta$-r\'egulier, alors le support de n'importe quel faisceau pervers g\'eom\'etriquement simple pr\'esent dans $^p\rmH^n(f_*\Ql)_\kappa$ est \'egal \`a une des composantes irr\'eductibles de $S_\kappa$.
\end{corollaire}

\subsection{Amplitude}\label{subsection : amplitude}
Nous allons introduire la notion de l'amplitude de chaque support qui joue un r\^ole cl\'e dans la d\'emonstration de \ref{support general kappa}. Nous allons garder les notations de \ref{support general kappa}. Pour tout $\kappa\in \Pi_0^*$, pour tout entier $n$, le faisceau pervers de cohomologie $K^n_\kappa$ est un g\'eom\'etriquement semi-simple. En regroupant ses facteurs simples ayant le m\^eme support, on obtient une d\'ecomposition canonique
\begin{equation}\label{decomposition par le support}
\bigoplus_{n\in \ZZ}K^n_\kappa=\bigoplus_{\alpha\in {\mathfrak A}_\kappa} K^n_\alpha
\end{equation}
o\`u  ${\mathfrak A}_\kappa$ est une collection finie de sous-sch\'emas ferm\'es irr\'eductibles $Z_\alpha$ de $S\otimes_k \bar k$ et o\`u $K^n_\alpha$ est la somme directe des facteurs simples de $K^n_\kappa$ de support $Z_\alpha$. En supposant que pour tout $\alpha\in{\mathfrak A}_\kappa$, $K^n_\alpha$ est non nul pour au moins un entier $n$, l'ensemble ${\mathfrak A}_\kappa$ est uniquement d\'etermin\'e. Pour tout $\alpha\in {\mathfrak A}_\kappa$, $K^n_\alpha$ est un facteur direct canonique de $K^n_\kappa$.

\begin{numero}
On va maintenant introduire la notion d'amplitude de $\alpha\in {\mathfrak A}_\kappa$. Pour tout
$\alpha$, notons
$${\rm occ}(\alpha)=\{n\in \ZZ \mid K_\alpha^n \not=0\}.$$
Notons aussi $n_+(\alpha)$ l'\'el\'ement maximal de ${\rm occ}(\alpha)$ et
$n_-(\alpha)$ l'\'el\'ement minimal. On d\'efinit l'{\em amplitude} de $\alpha$ par la formule
$${\rm amp}(\alpha)=n_+(\alpha)-n_-(\alpha).$$
Voici une estimation cruciale de l'amplitude dont on reporte la d\'emonstration \`a \ref{demo amplitude}.

\end{numero}

\begin{proposition}\label{inegalite amplitude}
Gardons les notations de \ref{subsection : support general}. Supposons que $\rmT_\Ql(P^0)$ est polarisable \cf \ref{module de Tate polarisable}.

Pour tout $\kappa\in \Pi_0^*$, pour tout $\alpha\in{\mathfrak A}_\kappa$, soit $\delta_\alpha$ la valeur minimale de l'invariant $\delta:S(\bar k)\rightarrow \mathbb{N}$ sur le sous-sch\'ema ferm\'e irr\'eductible $Z_\alpha$. Alors on a l'in\'egalit\'e
$${\rm amp}(\alpha)\geq 2(d-\delta_\alpha)$$
o\`u $d$ est la dimension relative de $g:P\rightarrow S$.
\end{proposition}

\subsection{D\'emonstration du th\'eor\`eme du support}
\label{subsection : demo support} Dans ce paragraphe, nous allons d\'emontrer
le th\'eor\`eme du support \ref{support general kappa} en admettant l'in\'egalit\'e \ref{inegalite
amplitude}. La d\'emonstration est fond\'ee sur la dualit\'e de Poincar\'e et
un argument de comptage de dimension du \`a Goresky et MacPherson.
Le lecteur consultera l'annexe \cf \ref{appendice GM} pour voir comment
cet argument marche dans un contexte plus g\'en\'eral.

\begin{proof} Puisque $M$ et $S$ sont lisses sur $k$, la dualit\'e de
Poincar\'e-Verdier fournit isomorphisme de complexes
$$f_* \Ql = \underline\RHom(f_* \Ql, \Ql [-2d](-d))$$
o\`u $d$ est la dimension relative de $f$. En prenant les faisceaux pervers de cohomologie, on obtient un
isomorphisme
$$K^n=K^{2\dim(M)-n,\vee}(\dim(M))$$
o\`u $^\vee$ d\'esigne la dualit\'e de Verdier de la cat\'egorie des complexes $\ell$-adiques sur $S$ qui pr\'eserve la sous-cat\'egorie des faisceaux pervers. Pour tout $\kappa\in \Pi_0^*$, on en d\'eduit un isomorphisme entre les $\kappa$-parties
$$K^n_\kappa=K^{2\dim(M)-n,\vee}_\kappa(\dim(M)).$$
Comme cet isomorphisme respecte la d\'ecomposition par le support
\ref{decomposition par le support}, pour tout $\alpha\in {\mathfrak A}_\kappa$,
l'ensemble d'entiers ${\rm occ}(\alpha)$ est sym\'etrique par rapport \`a
$\dim(M)$.

De plus, en admettant l'estimation de l'amplitude \ref{inegalite amplitude}
$${\rm amp}(\alpha)\geq 2(d-\delta_\alpha)$$
on constate qu'il existe un entier $n\in {\rm occ}(\alpha)$ tel que
$$n\geq \dim(M)+ d-\delta_\alpha.$$
Il existe donc un faisceau pervers simple non nul $K$ de support $Z_\alpha$
tel que $K$ soit un facteur direct de $K^n_\kappa$ avec $n\geq \dim(M)+ d-\delta_\alpha$.

Localement pour la topologie \'etale de $S$, il existe des rel\`evements $\pi_0(P)\rightarrow P$ de l'homomorphisme surjectif $P\rightarrow \pi_0(P)$. En d\'eduit une action de $\Pi_0$ sur sur le complexe $f_*\Ql$. Puisque $f_*\Ql$ est un complexe born\'e constructible, il existe une
d\'ecomposition en somme directe
$$f_*\Ql=\bigoplus_{\kappa\in \Pi_0^*} (f_*\Ql)_\kappa$$
tel qu'il existe un entier $N$ tel que pour tout $\lambda\in \Pi_0$,
$(\lambda-\kappa(\lambda){\rm id})^N$ agit trivialement sur $(f_*\Ql)_\kappa$ apr\`es changement de base \`a un recouverement \'etale de $S$. Cette d\'ecomposition est ind\'ependante du choix du rel\`evement $\pi_0(P)\rightarrow P$ de sorte qu'on a en fait une d\'ecomposition canonique au-dessus de $S$. De plus, on a la compatibilit\'e \'evidente
$$^p \rmH^n((f_*\Ql)_\kappa)=K^n_\kappa.$$
Ainsi $K[-n]$ est un facteur direct du complexe $(f_*\Ql)_\kappa$.

Soit $U_\alpha$ un ouvert dense de $Z_\alpha$ o\`u $K$ est de la forme $L[\dim(Z_\alpha)]$ o\`u $L$
est un syst\`eme local irr\'eductible sur $U_\alpha$. Soit $V_\alpha$ un ouvert de $S$ qui contient $U_\alpha$ comme un ferm\'e. Notons $i_\alpha:U_\alpha\rightarrow V_\alpha$. Quitte \`a remplacer $S$ par $V_\alpha$, on peut donc supposer que $U_\alpha=Z_\alpha$.

Soit $u_\alpha$ un point g\'eom\'etrique de $U_\alpha$. Puisque $L[-n]$ est un facteur direct de $(f_*\Ql)_\kappa$, la fibre en $u_\alpha$ de
$L[-n]$ est un facteur direct de la  fibre en $u_\alpha$ de $(f_{*}\Ql)_\kappa$. La fibre en $u_\alpha$ de $(
f_*\Ql)_\kappa$ est le complexe $\rmR\Gamma(M_{u_\alpha},\Ql)_\kappa$
d'apr\`es le th\'eor\`eme de changement de base pour un morphisme
propre alors que la fibre de $L[-n]$ en $u_\alpha$ est un $\Ql$-espace vectoriel
non nul plac\'e en degr\'e $n-\dim(Z_\alpha)$.

Il s'ensuit que
$$\rmH^{n-\dim(Z_\alpha)}(M_{u_\alpha},\Ql)\not=0.$$
Or, l'entier
$$n-\dim(Z_\alpha)=n-\dim(S)+{\rm codim}(Z_\alpha)$$
est sup\'erieur ou \'egal \`a
$$\dim(M)+ d-\delta_\alpha -\dim(S)+{\rm codim}(Z_\alpha).$$
Par hypoth\`ese, on a l'in\'egalit\'e
$${\rm codim}(Z_\alpha)\geq \delta_\alpha.$$
En la combinant avec l'\'egalit\'e \'evidente $\dim(M)-\dim(A)=d$, on en d\'eduit
l'in\'egalit\'e
$$n-\dim(Z_\alpha) \geq 2d.$$
Comme la dimension de la fibre $M_{u_\alpha}$ est \'egale \`a $d$,  la
non-annulation $\rmH^{n-\dim(Z)}(M_{u_\alpha},\Ql)\not=0$ implique que
$n-\dim(Z_\alpha)=2d$.

En tronquant par l'op\'erateur $\tau^{\geq 2d}$, on trouve que $i_{\alpha *} L[-n+\dim(U_\alpha)]$ est un facteur direct de $\rmH^{2d}(f_*\Ql)_\kappa[-2d]$ donc $i_{\alpha *} L$ est un facteur direct de $\rmH^{2d}(f_*\Ql)_\kappa$. Le th\'eor\`eme r\'esulte donc de \ref{strat}.
\end{proof}

\subsection{Cap-produit}\label{subsection : cap-produit}

Nous allons d\'evelopper dans ce paragraphe certaines propri\'et\'es
g\'en\'erales du cap-produit afin de d\'emontrer l'in\'egalit\'e d'amplitude
\ref{inegalite amplitude}.

\begin{numero}
Nous nous pla{\c c}ons dans la situation g\'en\'erale suivante. Soit $S$ un
sch\'ema quelconque. Soit $g:P\rightarrow S$ un $S$-sch\'ema en groupes
de type lisse commutatif de fibres connexes de dimension $d$.
Consid\'erons le complexe d'homologie de $P$ sur $S$ d\'efini par la formule
$$\Lambda_P= g_! \Ql [2d](d).$$
Ce complexe est concentr\'e en degr\'es n\'egatifs. En degr\'e $0$, on a
$\rmH^0(\Lambda_P)=\Ql$. La partie la plus importante de $\Lambda_P$ est
$$\rmT_\Ql(P):=\rmH^{-1}(\Lambda_P)$$
o\`u $\rmT_\Ql(P)$ est un faisceau dont la fibre en chaque point
g\'eom\'etrique $s\in S$ est le $\Ql$-module de Tate ${\rm T}_\Ql(P_s)$ de
la fibre de $P$ en $s$. Plus g\'en\'eralement, le th\'eor\`eme de
changement de base nous fournit un isomorphisme
$$\rmH^{-i}(\Lambda_P)_s =\rmH^{2d-i}_c(P_s)(d)$$
entre la fibre en  $s$ de $\rmH^{-i}(\Lambda_P)$ est le $i$-\`eme groupe
d'homologie
$$\rmH_i(P_s)=\rmH^{2d-i}_c(P_s)(d)$$
de la fibre de $P$ en $s$.
\end{numero}

\begin{numero}\label{cap}
Soit $f:M\rightarrow S$ un morphisme de type fini muni d'une action de $P$
relativement \`a la base $S$
$${\rm act}: P\times_S M \rightarrow M.$$
Puisque $P$ est lisse de dimension relative $d$ sur $S$, le morphisme
${\rm act}$ est aussi lisse de m\^eme dimension relative. On a donc un
morphisme trace
$${\rm act}_! \Ql[2d](d) \rightarrow \Ql$$
au-dessus de $M$. En poussant ce morphisme de trace par $f_!$, on obtient
un morphisme
$$(g\times_S f)_!\Ql[2d](d)  \rightarrow f_! \Ql.$$
En utilisant maintenant l'isomorphisme de Kunneth, on obtient un
morphisme de cap-produit
$$ \Lambda_P \otimes f_! \Ql
\rightarrow f_!\Ql.$$
\end{numero}

\begin{numero}
Cette construction s'applique en particulier \`a $f=g$. Elle d\'efinit alors un
morphisme de complexes
$$\Lambda_P\otimes \Lambda_P \rightarrow \Lambda_P.$$
On en d\'eduit une structure d'alg\`ebres gradu\'ees sur les faisceaux de
cohomologie de $\Lambda_P$
$$\rmH^{-i}(\Lambda_P)\otimes \rmH^{-j}(\Lambda_P)\rightarrow \rmH^{-i-j}(\Lambda_P)$$
qui est commutative au sens gradu\'e. On en d\'eduit en particulier un
morphisme de faisceaux
$$\wedge^i \rmT_\Ql(P)\rightarrow \rmH^{-i}(\Lambda_P)$$
qui est en fait un isomorphisme. Pour le v\'erifier, il suffit de le faire fibre par
fibre o\`u on retrouve les groupes d'homologie de $P_s$ munis du produit
de Pontryagin.
\end{numero}

\begin{numero}
Comme m'a fait remarquer Deligne, la multiplication par entier $N\not=0$
induit la multiplication par $N$ sur le module de Tate $\rmT_\Ql(P)$ de $P$
et induit donc la multiplication par $N^i$ sur
$\rmH^{-i}(\Lambda_P)=\wedge^i \rmT_\Ql(P)$. On en d\'eduit une
d\'ecomposition canonique du complexe $\Lambda_P$
$$\Lambda_P=\bigoplus_{i\geq 0} \wedge^i \rmT_\Ql(P)[i]$$
compatible avec la structure de cap-produit.
\end{numero}

\begin{proposition}\label{action variete abelienne}
Soit $\bar k$ un corps alg\'ebriquement clos. Soit $M$ un $\bar k$-sch\'ema
de type fini muni d'une action d'une $\bar k$-vari\'et\'e ab\'elienne
ab\'elienne $A$ avec stabilisateurs finis. Alors $\bigoplus_n
\rmH^n_c(M)[-n]$ est un module libre sur l'alg\`ebre gradu\'ee
$\Lambda_A=\bigoplus_{i} \wedge^i \rmT_\Ql(A)[i]$.
\end{proposition}

\begin{proof}
Puisque le stabilisateur dans $A$ de tout point de $M$ est un sous-groupe
fini, le quotient $N=[M/ A]$ est hom\'eomorphe \`a un champ de
Deligne-Mumford. Le morphisme $m:M\rightarrow N$ \'etant un morphisme
prorpe et lisse, il existe un isomorphisme non-canonique
$$m_*\Ql \simeq \bigoplus_i \rmR^i m_*\Ql[-i]$$
o\`u les $\rmR^i(m_*\Ql)$ sont des syst\`emes locaux sur $N$. Puisque
$M\times_N M=A\times M$, l'image inverse $m^*\rmR^i(m_*\Ql)$ est le
faisceau constant de valeur $\rmH^i(A)$ sur $M$. En utilisant \cf
\cite[4.2.5]{BBD}, on obtient un isomorphisme canonique entre
$\rmR^i(m_*\Ql)$ est le faisceau constant sur $N$ de valeur $\rmH^i(A)$.

La d\'ecomposition en somme directe ci-dessus implique la
d\'eg\'en\'ere\-scence de la suite spectrale
$$\rmH^j_c(N,\rmR^i m_*\Ql) \Rightarrow \rmH^{i+j}_c(M).$$
On en d\'eduit dans $\bigoplus_n \rmH^n_c (M)$ une filtration
$\Lambda_A$-stable dont le $j$-i\`eme gradu\'e est
$$\bigoplus_i \rmH^j_c(N,\rmR^i m_*\Ql)=\rmH^j_c(N)\otimes \bigoplus_i \rmH^i(A_s).$$
Ces gradu\'es sont des $\Lambda_A$-modules libres. Il en r\'esulte que
la somme directe $\bigoplus_n \rmH^n_c (M)[-n]$ est un
$\Lambda_A$-module libre.
\end{proof}

\begin{numero}
Soit $P$ un groupe alg\'ebrique commutatif lisse et connexe sur un corps alg\'ebriquement clos
$\bar k$. D'apr\`es le th\'eor\`eme de Chevalley \cf \cite{Ros}, $P$ admet
un d\'evissage comme une extension d'une vari\'et\'e ab\'elienne par un
groupe affine
$$1\rightarrow R\rightarrow P \rightarrow A \rightarrow 1$$
o\`u $A$ est une vari\'et\'e ab\'elienne et o\`u $R$ est un groupe affine.
Notons au passage que si $P$ est d\'efini sur un corps parfait, cette suite
exacte existe au-dessus de ce m\^eme corps.

On en d\'eduit de la suite exacte ci-dessus une suite exacte de modules de
Tate
$$0\rightarrow \rmT_\Ql(R)\rightarrow \rmT_\Ql(P) \rightarrow
\rmT_\Ql(A) \rightarrow 0.
$$
On appelle {\em rel\`evement holomogique} une application lin\'eaire
$$\lambda:\rmT_\Ql(A_s)\rightarrow \rmT_\Ql(P_s)$$
qui scinde cette suite exacte. Un rel\`evement homologique induit un
homomorphisme d'alg\`ebres $\lambda:\Lambda_{A_s} \rightarrow
\Lambda_s$ o\`u $\Lambda_{A_s}=\bigoplus_{-i} \wedge^i \rmT_\Ql(A_s)$
est l'alg\`ebre d'homologie de la vari\'et\'e ab\'elienne $A_s$. L'ensemble
des rel\`eve\-ments homologiques de la partie ab\'elienne forment un espace
affine, torseur sous le $\Ql$-espace vectoriel
$\Hom(\rmT_\Ql(A_s),\rmT_\Ql(R_s))$.
Si $P$ est d\'efini sur un corps fini, il existe un rel\`evement homologique
canonique.
\end{numero}

\begin{proposition}\label{quasi-relevement}
Soit $P$ un groupe alg\'ebrique commutatif lisse et connexe d\'efini sur un
corps fini $k$. Soit
$$1\rightarrow R \rightarrow P \rightarrow A \rightarrow 1$$
la suite exacte canonique qui r\'ealise $P$ comme une extension d'une
vari\'et\'e ab\'elienne $A$ par un groupe affine $R$.  Il existe alors un
entier $N>0$ et un homomorphisme $a:A\rightarrow P$ tel qu'en le
composant avec $P\rightarrow A$, on obtient la multiplication par $N$ dans
$A$. On dira alors que $a$ est un quasi-rel\`evement.

De plus, $\rmT_\Ql(a):\rmT_\Ql(A)\rightarrow\rmT_\Ql(P)$ est l'application
induite sur les modules de Tate, alors $N^{-1}\rmT_\Ql(a)$ est l'unique
rel\`evement homologique compatible \`a l'action de $\Gal(\bar k/ k)$ sur
$T_\ell(P)$ et $T_\ell(A)$. On dira que $N^{-1}\rmT_\Ql(a)$ est le
rel\`evement canonique.
\end{proposition}

\begin{proof}
D'apr\`es \cite[p. 184]{Serre-GACC}, le groupe des extensions d'une
vari\'et\'e ab\'elienne $A$ par $\GG_m$ s'identifie au groupe des $k$-points
de la vari\'et\'e ab\'elienne duale. Le groupe des extensions d'une vari\'et\'e
ab\'elienne par $\GG_a$ est le $k$-espace vectoriel de dimension finie
$\rmH^1(A,\calO_A)$. Un groupe affine commutatif lisse connexe est
isomorphe \`a un produit de $\GG_m$ et de $\GG_a$ apr\`es passer \`a une
extension finie du corps de base, il s'ensuit que le groupe des extensions
d'une vari\'et\'e ab\'elienne $A$ par un groupe affine commutatif lisse $R$
d\'efini sur un corps fini, est un groupe fini. La premi\`ere assertion s'en
d\'eduit.

Le rel\`evement homologique $N^{-1}\rmT_\Ql(a)$ est invariant sous
l'action de $\Gal(\bar k/ k)$ puisque le rel\`evement $a$ est d\'efini sur
$k$. Par ailleurs, l'\'el\'ement de Frobenius $\sigma\in\Gal(\bar k/ k)$ agit
sur $T_\Ql(A)$ par la multiplication par $q^{-1/ 2}$ et sur $\rmT_\Ql(R)$ par
la multiplication par $q^{-1}$ o\`u $q$ est le cardinal du corps fini $k$.
Ainsi $N^{-1}\rmT_\Ql(a)$ est le seul rel\`evement homologique compatible
avec l'action de $\Gal(\bar k/ k)$.
\end{proof}

\begin{corollaire}\label{liberte fibre par fibre}
Soit $\bar k$ un corps alg\'ebriquement clos. Soit $M$ un $\bar k$-sch\'ema
de type fini. Soit $P$ un $\bar k$-groupe  lisse commutatif et connexe
agissant sur $M$ avec stabilisateur affine. Soit $A$ le quotient ab\'elien
maximal de $P$. Supposons que $P$ est d\'efini sur un corps fini de sorte
qu'on a un quasi-rel\`evement $a:A\rightarrow P$. Alors le
$\Lambda_A$-module qui se d\'eduit du $\Lambda_P$-module $\bigoplus_n
\rmH_c^n(M)$  est un module libre.
\end{corollaire}

\begin{proof}
Le quasi-rel\`evement $a:A\rightarrow P$ d\'efinit une action de $A$ sur
$M$ avec stabilisateurs finis. On se ram\`ene donc imm\'ediatement \`a la
proposition \ref{action variete abelienne}.
\end{proof}

L'hypoth\`ese que $P$ soit d\'efini sur un corps fini dans l'\'enonc\'e
ci-dessus est bien curieuse. Voici un \'enonc\'e plus g\'en\'eral et bien plus satisfaisant.
Je reproduis la d\'emonstration que j'ai apprise de Deligne \cf \cite{Deligne lettre}.
Cet \'enonc\'e n'est pas indispensable pour la suite de l'argument mais il est
beaucoup plus agr\'eable d'en disposer.

\begin{proposition}\label{relevement arbitraire}
Soit $\bar k$ un corps alg\'ebriquement clos. Soit $M$ un $\bar k$-sch\'ema
de type fini. Soit $P$ un $\bar k$-groupe  lisse commutatif et connexe
agissant sur $M$ avec stabilisateur affine. Soit $A$ le quotient ab\'elien
maximal de $P$ et soit $\lambda:\rmT_\Ql(A)\rightarrow \rmT_\Ql(P)$
n'importe quel rel\`evement homologique. Alors le $\Lambda_A$-module
qui se d\'eduit du $\Lambda_P$-module $\bigoplus_n \rmH_c^n(M)$ par
$\lambda$ est un module libre.
\end{proposition}

\begin{proof}
Par le proc\'ed\'e g\'en\'eral de passage \`a la limite inductive comme dans
\cite[6.1.7]{BBD}, on se ram\`ene au cas o\`u $M, P$ et l'action de $P$ sur
$M$ est d\'efini sur un corps fini $k$. Ecrivons $\bar k$ comme une limite
inductive $(A_i)_{i\in I}$ des anneaux de type fini sur $\ZZ$. La cat\'egorie
des $\bar k$-sch\'emas de type fini est une 2-limite inductive des
$A_i$-sch\'emas de type fini c'est-\`a-dire pour tout $X/ \bar k$, il existe
$i\in I$ et $X_i/ A_i$ de type fini tel que $X=X_i\otimes_{A_i} \bar k$ \cf
\cite[8.9.1]{EGA4} et pour tous $X_i, Y_i / A_i$, on a \cf \cite[8.8.2]{EGA4}
$$\Hom_{\bar k}(X_i\otimes_{A_i}\bar k, Y_i\otimes_{A_i}\bar k)=
\lim_{\genfrac{}{}{0pt}{}{\longrightarrow}{j\geq i}}
\Hom_{A_j}(X_i\otimes_{A_i} A_j, Y_i\otimes_{A_i} A_j).
$$
Le m\^eme \'enonc\'e vaut donc pour la cat\'egorie des groupes alg\'ebriques
et celle des triplets form\'es d'un groupe alg\'ebrique $G$, d'un sch\'ema de
type fini $X$ et d'une action de $G$ sur $X$. Soit $f:X\rightarrow Y$ un
morphisme dans la cat\'egorie des $\bar k$-sch\'emas de type fini provenant
d'un morphisme $f_i:X_i \rightarrow Y_i$ dans la cat\'egorie des
$A_i$-sch\'emas. Pour que $f$ ait l'une des propri\'et\'es suivantes : plat,
lisse, affine ou propre, il faut et il suffit que $f_i$ acqui\`ere cette
propri\'et\'e apr\`es une extension de scalaires $A_i \rightarrow A_j$ \cf
\cite[11.2.6 et 8.10.5]{EGA4}. Il existe donc $A_i$ de type fini sur $\ZZ$
contenu dans $\bar k$, un $A_i$-sch\'ema de type fini $M_i$, un
$A_i$-sch\'ema en groupes lisse $P_i$ extension d'un sch\'ema ab\'elien par
un sch\'ema en groupes affine qui agit sur $M_i$ avec stabilisateur affine tel
que apr\`es l'extension des scalaires $A_i\rightarrow\bar k$, on retrouve la
situation de d\'epart. On peut \'egalement supposer que $\rmH^n(f_{i!}\Ql)$
sont des syst\`emes locaux qui permet de ramener l'\'enonc\'e \`a
d\'emontrer \`a un point d\'efini sur un corps fini.

On peut donc supposer que $M, P$ et l'action de $P$ sur $M$ est d\'efini sur
un corps fini $k$. Soit $A$ le quotient ab\'elien de $P$ et $a:A\rightarrow
P$ le quasi-rel\`evement qui d\'efinit un rel\`evement canonique
$\lambda_0:\rmT_\Ql(A)\rightarrow \rmT_\Ql(M)$ de modules de Tate. On
sait d\'ej\`a que  le $\Lambda_A$-module qui se d\'eduit du
$\Lambda_P$-module $\bigoplus_n \rmH_c^n(M)$ par $\lambda_0$ est un
module libre d'apr\`es le corollaire pr\'ec\'edent \ref{liberte fibre par fibre}.

D\'emontrons maintenant le m\^eme \'enonc\'e pour tout rel\`evement
homologique $\lambda$. L'espace de tous les rel\`evements homologiques,
point\'e par $\lambda_0$, s'identifie au $\Ql$-espace vectoriel
$\Hom(\rmT_\Ql(A_s),\rmT_\Ql(R_s))$ sur lequel l'\'el\'ement de
Frobenius $\sigma\in\Gal(\bar k/ k)$ agit avec les valeurs propres ayant pour valeur absolue
$|k|^{1/ 2}$. Sur cet espace, l'ensemble des $\lambda$ pour lequel la
structure de $\Lambda_{A_s}$-module sur $\bigoplus_n \rmH^n_c (M)$
d\'eduite de $\lambda$ est libre forme un ouvert Zariski de $\Hom(\rmT_\Ql(A_s),\rmT_\Ql(R_s))$. Cet
ouvert est stable sous $\sigma$ et contient l'origine $\lambda_0$. Son compl\'ement
est un ferm\'e stable sous $\sigma$ qui ne contient pas $\lambda_{0}$.
Puisque $\sigma$ agit avec des valeurs propres de valeur absolue $|k|^{1/ 2}$, un ferm\'e de $\Hom(\rmT_\Ql(A_s),\rmT_\Ql(R_s))$ stable sous $\sigma$ est n\'ecessairement stable sous l'action de $\GG_m$. Le fait que ce ferm\'e ne contient pas l'origine $\lambda_0$ implique qu'il est vide.
\end{proof}

\begin{numero}\label{base henselienne}
Soit maintenant $S$ un sch\'ema strictement hens\'elien avec un morphisme
$\epsilon:S\rightarrow \Spec(\bar k)$ o\`u $\bar k$ est le corps r\'esiduel de
$S$. Notons $s$ le point ferm\'e de $S$. La fibre $\Lambda_{P,s}$
s'identifie alors avec $\epsilon_* \Lambda_P$. On en d\'eduit par adjonction
un morphisme
$$\epsilon^* \Lambda_{P,s}\rightarrow \Lambda.$$
En restreignant le cap-produit \ref{cap} \`a $\epsilon^* \Lambda_{P,s}$, on
obtient une fl\`eche
$$\Lambda_{P,s} \boxtimes f_!\Ql \rightarrow f_!\Ql$$
qui d\'efinit une action de l'alg\`ebre gradu\'ee $\Lambda_{P,s}=\bigoplus_i
\wedge^i \rmT_\Ql(P_s)[i]$ sur le complexe $f_!\Ql$.

En particulier, on a un morphisme
$$\rmT_\Ql(P_s) \otimes f_!\Ql \rightarrow f_! \Ql[-1].$$
On en d\'eduit morphisme entre les tronqu\'es pour la $t$-structure
perverse :
$$\rmT_\Ql(P_s) \otimes \,^p\tau^{\leq n}(f_!\Ql)\rightarrow \,^p\tau^{\leq n-1}(f_!\Ql)$$
pour tout $n\in\ZZ$. Ceci induit un morphisme entre les faisceaux pervers
de cohomologie en degr\'es $n$ et $n-1$
$$\rmT_\Ql(P_s) \otimes \,^p \rmH^n(f_!\Ql)\rightarrow \,^p \rmH^{n-1}(f_!\Ql).$$
On en d\'eduit une action gradu\'ee de l'alg\`ebre gradu\'ee
$\Lambda_{P,s}$ sur la somme directe $\bigoplus_n\,^p
\rmH^n(f_!\Ql)[-n]$.
\end{numero}

\begin{numero}\label{suite spectrale}
 La filtration $^p\tau^{\leq n}(f_!\Ql)$ de $f_!\Ql$ d\'efinit une suite
spectrale
$$E^{m,n}_2=\rmH^m(\,^p \rmH^n(f_!\Ql)_s)\Rightarrow \rmH^{m+n}_c(M_s)$$
\'equivariante par rapport \`a l'action de $\Lambda_{P,s}$. Cette suite
spectrale est convergente car le complexe $f_!\Ql$ est born\'e. On dispose
donc sur la somme directe
$$\rmH_c(M_s)=\bigoplus_N \rmH^{N}_c(M_s)[-N]$$
d'une filtration d\'ecroissante $F^m \rmH_c(M_s)$ telle que
$$F^m \rmH_c(M_s)/ F^{m+1} \rmH_c(M_s)= \bigoplus_n E^{m,n}_\infty[-m-n].$$
L'action de $\Lambda_{P,s}$ sur $\rmH_c(M_s)$ respecte cette filtration.
L'action induite sur les gradu\'es associ\'es \`a la filtration $F^m
\rmH_c(M_s)$
$$\bigoplus_n E^{m,n}_\infty[-m-n]$$
se d\'eduit de son action sur les $E^{m,n}_2$ qui \`a son tour se d\'eduit de
l'action gradu\'ee de $\Lambda_{P,s}$ sur $\bigoplus_n\,^p
\rmH^n(f_!\Ql)[-n]$.
\end{numero}

\begin{numero}\label{degenerescence de suite spectrale}
Soient maintenant $S$ un sch\'ema de type fini sur un corps fini, $P$ un
$S$-sch\'ema en groupes lisse commutatif de fibres connexes de dimension
$d$ et $f:M\rightarrow S$ un morphisme propre dont la source $M$ est un
$k$-sch\'ema lisse. Dans cette situation d'apr\`es Deligne, il existe un
isomorphisme non-canonique
$$f_! \Ql \simeq \,^p \rmH^n(f_!\Ql)[-n].$$
Cet isomorphisme subsiste quand on restreint \`a l'hens\'elis\'e strict $S_s$
d'un point g\'eom\'etrique $s$ de $S$ de sorte que les suites spectrales
\ref{suite spectrale} d\'eg\'en\`e\-rent en $E_2$ c'est-\`a-dire
$E^{m,n}_\infty=E^{m,n}_2$.
\end{numero}

\begin{numero}
Toujours sous l'hypoth\`ese du num\'ero pr\'ec\'edent, d'apr\`es le
th\'eor\`eme de puret\'e de Deligne, le faisceau pervers $^p
\rmH^n(f_!\Ql)$ est pur. D'apr\`es le th\'eor\`eme de d\'ecomposition, sur
$S\otimes_k\bar k$ il est isomorphe \`a une somme directe de faisceaux
pervers simples. On a une d\'ecomposition canonique de $^p
\rmH^n(f_!\Ql)$ par les supports : il existe une famille finie de ferm\'es
irr\'eductibles $Z_\alpha$ de $S\otimes_k\bar k$ index\'e par $\alpha\in
{\mathfrak A}$ tels que
$$^p \rmH^n(f_! \Ql) =\bigoplus_{\alpha\in{\mathfrak A}} K_\alpha^n.$$
o\`u $K_\alpha^n$ est la somme directe des facteurs simples de $^p
\rmH^n(f_! \Ql)$ ayant pour support le ferm\'e irr\'eductible $Z_\alpha$.
Nous notons
$$K_\alpha=\bigoplus K_\alpha^n[-n]$$
et supposons que $K_\alpha\not= 0$ pour tout $\alpha\in {\mathfrak A}$.
\end{numero}

\begin{lemme}\label{U alpha} Pour tout $\alpha\in {\mathfrak A}$,
il existe un ouvert dense $U_\alpha$ de $Z_\alpha$ tel que les propri\'et\'es
suivantes soient satisfaites :
\begin{itemize}
    \item $U_\alpha$ est lisse,
  \item pour tout $\alpha'\in {\mathfrak A}$ diff\'erent de $\alpha$ avec
      $\dim(Z_{\alpha'})\leq \dim(Z_\alpha)$, on a $U_\alpha \cap
      Z_{\alpha'}=\emptyset$,
  \item la restriction de $K_\alpha^n$ \`a $U_\alpha$ est de la forme
      $L_\alpha^n[\dim(U_\alpha)]$ o\`u $L_\alpha^n$ est un syst\`eme
      local  pur de poids $n$. Notons $L_\alpha=\bigoplus_n
      L_\alpha^n$,
  \item la restriction $P_\alpha$ de $P$ \`a $U_\alpha$ admet un
      d\'evissage apr\`es un changement de base fini radiciel
        $$1\rightarrow R_\alpha \rightarrow P_\alpha \rightarrow A_\alpha \rightarrow 1$$
      comme une extension d'une vari\'et\'e ab\'elienne par un sch\'ema en
      groupes affine qui induit une suite exacte de module de Tate sur
      $U_\alpha$
      $$0\rightarrow \rmT_\Ql(R_\alpha)\rightarrow \rmT_\Ql(P_\alpha)\rightarrow
      \rmT_\Ql(A_\alpha)\rightarrow 0,$$
  \item le module de Tate $\rmT_\Ql(R_\alpha)$ est un syst\`eme local.
\end{itemize}
\end{lemme}

\begin{proof}
$Z_\alpha$ \'etant un $\bar k$-sch\'ema irr\'eductible r\'eduit, il contient un
ouvert dense lisse. Si $\dim(Z_{\alpha'})< \dim(Z_\alpha)$ ou
$\dim(Z_{\alpha'})=\dim(Z_\alpha)$ mais $\alpha'\not=\alpha$, l'intersection
$Z_\alpha \cap Z_{\alpha'}$ est de dimension plus petite que celle de
$Z_\alpha$. En retranchant ces intersections, il reste un ouvert dense de
$Z_\alpha$. Il n'y a qu'un nombre fini de faisceaux pervers simples
pr\'esents dans $K_\alpha$, on peut
supposer que tous ces faisceaux pervers sont des syst\`emes
locaux avec d\'ecalage  en r\'etr\'ecissant encore l'ouvert.

Le point g\'en\'erique $\eta_\alpha$ de $Z_\alpha$ n'\'etant pas
n\'ecessairement parfait. On peut passer \`a une extension finie ins\'eparable
$\eta'_\alpha$ pour que $P_{\eta_\alpha}$ admette un d\'evissage comme une
extension d'une vari\'et\'e ab\'elienne par un groupe affine. Consid\'erons la
normalisation $Z'_\alpha$ de $Z_\alpha$ dans $\eta'_\alpha$. Prenons
l'image r\'eciproque $P_{Z'_\alpha}$ de $P$ \`a $Z'_\alpha$, consid\'erons
l'adh\'erence de la partie affine de $P_{\eta'_\alpha}$ dans $P_{Z'_\alpha}$.
Dans un ouvert de $Z'_\alpha$, on peut r\'ealiser ainsi $P_{Z'_\alpha}$
comme une extension d'un sch\'ema ab\'elien par un sch\'ema en groupes
affines. En r\'etr\'ecissant encore l'ouvert, on peut s'assurer que le module
de Tate de la partie affine soit un syst\`eme local. \end{proof}

\begin{numero}\label{K alpha}
Consid\'erons le cap-produit par le module de Tate \ref{cap}
$$\rmT_\Ql(P)\otimes f_!\Ql \rightarrow f_! \Ql[-1].$$
Pour tout $n$, on dispose d'un morphisme
$$\rmT_\Ql(P)\otimes \,^p \tau^{\leq n}(f_!\Ql) \rightarrow f_! \Ql[-1].$$
Puisque le produit tensoriel
$\rmT_\Ql(P)\otimes \,^p \tau^{\leq n}(f_!\Ql)$ appartient \`a
$$\rmT_\Ql(P)\otimes \,^p \tau^{\leq n}(f_!\Ql)\in\,^p \rmD^{\leq n}_c(S,\Ql)$$
le morphisme ci-dessus se factorise par
$$\rmT_\Ql(P)\otimes \,^p \tau^{\leq n}(f_!\Ql) \rightarrow \,^p \tau^{\leq n}((f_!\Ql) [-1])
$$
On en d\'eduit une fl\`eche
$$\rmT_\Ql(P)\otimes \,^p \tau^{\leq n}(f_!\Ql) \rightarrow \,^{p}\rmH^{n-1}(f_{!\Ql})[-n].$$
De nouveau parce que $\rmT_\Ql(P)\otimes \,^p \tau^{\leq n-1}(f_!\Ql)\in\,^p \rmD^{\leq n-1}_c(S,\Ql)$,
la fl\`eche ci-dessus se factorise par
$$\rmT_\Ql(P)\otimes \,^p \rmH^{n}(f_!\Ql)[-n] \rightarrow \,^{p}\rmH^{n-1}(f_{!\Ql})[-n].$$
Consid\'erons la d\'ecomposition par les supports
$$\bigoplus_{\alpha\in {\mathfrak A}} \rmT_\Ql(P) \otimes K_\alpha^n
\rightarrow \bigoplus_{\alpha\in {\mathfrak A}} K^{n-1}_\alpha.
$$
Pour tout $\alpha\in \mathfrak A$, on a donc une fl\`eche canonique
$$\rmT_\Ql(P) \otimes K_\alpha^n \rightarrow K^{n-1}_\alpha.$$
\end{numero}

\begin{numero}\label{V alpha}
Soit $U_\alpha$ l'ouvert dense du ferm\'e $Z_\alpha$ comme dans \ref{U
alpha}. Pour tout $\alpha\in{\mathfrak A}$, choisissons un ouvert de Zariski
$V_\alpha$ de $S\otimes_k \bar k$ contenant $U_\alpha$ comme un
sous-sch\'ema ferm\'e. Notons $i_\alpha:U_\alpha \rightarrow V_\alpha$
l'immersion ferm\'ee. En restreignant la fl\`eche ci-dessus \`a l'ouvert
$V_\alpha$, on obtient un fl\`eche
$$\rmT_\Ql(P) \otimes i_{\alpha*}L_\alpha^n[\dim(U_\alpha)] \rightarrow
i_{\alpha*}L_\alpha^{n-1}[\dim(U_\alpha)].
$$
Par la formule de projection, on a
$$\rmT_\Ql(P) \otimes i_{\alpha*}L_\alpha^n=i_{\alpha*}(\rmT_\Ql(P_\alpha)
\otimes L_\alpha)
$$
o\`u $\rmT_\Ql(P_\alpha)$ est la restriction de $\rmT_\Ql(P)$ \`a
$U_\alpha$.  En appliquant le foncteur $i_{\alpha}^{*}$ \`a
$$i_{\alpha*}(\rmT_\Ql(P_\alpha) \otimes L_\alpha^n) \rightarrow
i_{\alpha*}L_\alpha^{n-1}
$$
on obtient un morphisme
$$\rmT_\Ql(P_\alpha) \otimes L_\alpha^n \rightarrow  L_\alpha^{n-1}$$
sur $U_\alpha$. C'est un morphisme entre syst\`emes locaux sur $U_{\alpha}$ d'apr\`es
l'hypoth\`ese sur $U_{\alpha}$ \cf \ref{U alpha}.
\end{numero}

\begin{numero}\label{A alpha}
D'apr\`es \ref{U alpha}, on a un d\'evissage de $\rmT_\Ql(P_\alpha)$
$$0\rightarrow \rmT_\Ql(R_\alpha)\rightarrow \rmT_\Ql(P_\alpha)
\rightarrow \rmT_\Ql(A_\alpha) \rightarrow 0
$$
o\`u $\rmT_\Ql(A_\alpha)$ est un syst\`eme local pur de poids $-1$ et
$\rmT_\Ql(R_\alpha)$ est un syst\`eme local pur de poids $-2$. Puisque
$L_\alpha^n$ est pur de poids $n$ et $L_\alpha^{n-1}$ est pur de poids
$n-1$, l'action de $\rmT_{\Ql}(P_{\alpha})$ se factorise par $\rmT_{\Ql}(A_{\alpha})$
$$\rmT_\Ql(A_\alpha) \otimes L_\alpha^n \rightarrow  L_\alpha^{n-1}.$$
On a donc muni \`a la somme directe de syst\`emes locaux
$L_\alpha=\bigoplus_n L^n_\alpha[-n]$ d'une structure de module gradu\'e
sur l'alg\`ebre gradu\'ee $ \Lambda_{A_\alpha}$
$$\Lambda_{A_\alpha}\otimes L_\alpha \rightarrow L_\alpha.$$
En particulier, pour tout point g\'eom\'etrique $u_\alpha$ de $U_\alpha$, la
fibre $L_{\alpha,u_\alpha}$ est un module gradu\'e sur l'alg\`ebre gradu\'ee
de $\Lambda_{A_\alpha,u_\alpha}$.
\end{numero}

\begin{lemme}\label{Nakayama}
Soit $U$ un $k$-sch\'ema connexe. Soit $\Lambda$ un syst\`eme local
gradu\'e en un nombre fini de degr\'es n\'egatifs avec $\Lambda^0=\Ql$ et
qui est muni d'une structure d'alg\`ebres gradu\'ees $\Lambda\otimes
\Lambda \rightarrow \Lambda$. Soit $L$ un syst\`eme local gradu\'e muni
d'une structure de module gradu\'e
$$\Lambda \otimes L\rightarrow L.$$
Supposons qu'il existe un point g\'eom\'etrique $u$ de $U$ tel que la fibre
$L_u$ de $L$ en $u$ est un module libre sur la fibre $\Lambda_u$ de
$\Lambda$ en $u$. Supposons que $L$ est semi-simple comme syst\`eme
local gradu\'e. Alors, il existe un syst\`eme local gradu\'e $E$ sur $U$ et un
isomorphisme
$$L= \Lambda\otimes E$$
compatible avec la structure de $\Lambda$-modules.
\end{lemme}

\begin{proof} Consid\'erons l'id\'eal d'augmentation de $\Lambda$
$$\Lambda^+=\bigoplus_{i>0} \Lambda^{-i}[i].
$$
Consid\'erons le conoyau $E$ de la fl\`eche
$$\Lambda^+\otimes L \rightarrow L.$$
Puisque $L$ est semi-simple, l'homomorphisme surjectif $L\rightarrow E$
admet un rel\`evement $E\rightarrow L$. On va montrer que la fl\`eche
induite
$$\Lambda\otimes E\rightarrow L$$
est un isomorphisme. Puisqu'il s'agit d'une fl\`eche entre syst\`emes locaux,
pour v\'erifier que c'est un isomorphisme, il suffit de v\'erifier qu'il est un
isomorphisme dans la fibre en $u$. Dans l'espace vectoriel $L_u$, $E_u$ est
un sous-espace vectoriel compl\'ementaire de $\Lambda^+ L_u$. La
fl\`eche
$$\Lambda_u\otimes E_u\rightarrow L_u$$
est donc surjective d'apr\`es le lemme de Nakayama. Elle est bijective car
$L_u$ \'etant un $\Lambda_u$-module libre, on a l'\'egalit\'e de
dimension
$$\dim_\Ql(\Lambda_u)\times \dim_\Ql(E_u)=\dim_\Ql(L_u)$$
d'o\`u le lemme.
\end{proof}

\begin{proposition}\label{liberte}
Soit $S$ un sch\'ema de type fini sur un corps fini $k$. Soit $f:X\rightarrow S$ un morphisme
propre de source un $k$-sch\'ema lisse $X$. Soit $g:P\rightarrow S$ un sch\'ema
en groupes commutatif lisse de type fini ayant les fibres connexes de
dimension $d$. Supposons que le stabilisateur dans $P$ de n'importe quel
point de $M$ est affine. Supposons que le module de Tate $\rmT_\Ql(P)$ est polarisable \cf \ref{module de Tate polarisable}.

Consid\'erons la d\'ecomposition canonique par les supports
$$K^n= \bigoplus_{\alpha\in {\mathfrak A}} K^n_\alpha$$
o\`u ${\mathfrak A}$ est un ensemble fini de sous-sch\'emas ferm\'es irr\'eductibles $Z_\alpha$ de $S\otimes_k \bar k$ et o\`u $K^n_\alpha$ est la somme directe des facteurs g\'eom\'etriquement simples ayant pour support le sous-sch\'ema irr\'eductible $Z_\alpha$.

Si $U_\alpha$ est un ouvert dense de $Z_\alpha$ comme dans \ref{U alpha},
en particulier tel que la restriction de $K^n_\alpha$ \`a $U_\alpha$ est un
syst\`eme local $L_\alpha^n$ d\'ecal\'e de $\dim(Z_\alpha)$ et que la restriction
$P_\alpha$ de $P$ \`a $U_\alpha$ admet un d\'evissage apr\`es un
changement de base fini radiciel
$$1\rightarrow R_\alpha \rightarrow P_\alpha \rightarrow A_\alpha
\rightarrow 1.
$$
Alors, pour tout point g\'eom\'etrique $u_\alpha$ de $U_\alpha$, l'espace
vectoriel gradu\'e $L_{\alpha,u_\alpha}=\bigoplus_{n}
L_{\alpha,u_\alpha}^n$ est un module libre sur l'alg\`ebre gradu\'ee
$\Lambda_{A_\alpha,u_\alpha}$.
\end{proposition}

\begin{proof}
D\'emontrons la proposition par  une r\'ecurrence descendante sur la
dimension de $Z_\alpha$. Soit $\alpha_0\in {\mathfrak A}$ l'\'el\'ement maximal
tel que $Z_{\alpha_0}$ soit $S\otimes_k \bar k$ tout entier. Soit
$U_{\alpha_0}$ un ouvert dense de $S\otimes_k \bar k$ assez petit au sens
de \ref{U alpha}. On a alors
$$^p \rmH^n(f_*\Ql)|_{U_{\alpha_0}}=L^n_{\alpha_0}[\dim(S)]$$
o\`u $L^n_{\alpha_0}$ est un syst\`eme local semi-simple sur
$U_{\alpha_0}$. Cette somme directe est munie d'une action canonique de
$\Lambda_{A_\alpha}$ d'apr\`es \ref{A alpha}.

Soit $u_{\alpha_0}$ un point g\'eom\'etrique quelconque de
$U_{\alpha_0}$. On a alors un isomorphisme
$$\bigoplus_n L^n_{\alpha_0,u_{\alpha_0}}= \bigoplus_n \rmH^n_c(M_{u_{\alpha_0}})
$$
compatible avec l'action de $\Lambda_{P_{u_{\alpha_0}}}$. Par la raison de
poids comme dans \ref{V alpha}, l'action de $\Lambda_{P_{u_{\alpha_0}}}$
se factorise par $\Lambda_{A_{u_{\alpha_0}}}$.

En prenant un point g\'eom\'etrique $u_{\alpha_0}$ au-dessus d'un point \`a
valeur dans un corps fini, on dispose alors d'un rel\`evement canonique
$\lambda_0:\rmT_\Ql(A_{u_{\alpha_0}})\rightarrow
\rmT_\Ql(P_{u_{\alpha_0}})$. D'apr\`es \ref{liberte fibre par fibre}, avec ce
rel\`evement $\bigoplus_n \rmH^n_c(M_{u_{\alpha_0}})$ est un
$\Lambda_{A_{u_{\alpha_0}}}$-module libre. Cette propri\'et\'e est alors
vraie pour n'importe quel point g\'eom\'etrique de $U_{\alpha_0}$ et pour
n'importe quel rel\`evement.

Le cas g\'en\'eral est d\'emontr\'e essentiellement par la m\^eme m\'ethode.
On utilise la propri\'et\'e de libert\'e de la cohomologie de la fibre
$M_{u_\alpha}$ pour d\'eduire la libert\'e du syst\`eme local gradu\'e
$L_\alpha$ comme $\Lambda_{A_\alpha}$-module. La difficult\'e est de
contr\^oler le bruit caus\'e par les $K_{\alpha'}$ avec $\dim(Z_{\alpha'})>
\dim(Z_\alpha)$. Notons que par r\'ecurrence, on peut supposer que pour
ces $\alpha'$,
 $L_{\alpha'}$ est un module libre sur
$\Lambda_{A_{\alpha'}}$.

Prenons un point g\'eom\'etrique $u_\alpha$ de $U_\alpha$ au-dessus d'un
point \`a valeur dans un corps fini. Soit $S_{u_\alpha}$ l'hens\'elis\'e strict de
$S$ en $u_\alpha$. La construction de \ref{base henselienne} s'applique \`a
$S_{u_\alpha}$. On dispose donc d'une action de $\Lambda_{P,u_\alpha}$
sur la restriction de $f_!\Ql$ \`a
$S_{u_\alpha}$
\begin{equation}\label{action Lambda sur S u alpha}
\Lambda_{P,u_\alpha}\boxtimes (f_!\Ql |_{S_\alpha}) \rightarrow (f_!\Ql |_{S_\alpha}).
\end{equation}
Comme dans \ref{base henselienne}, celle-ci induit une action gradu\'ee de
de $\Lambda_{P,u_\alpha}$ sur la somme directe de faisceaux pervers de
cohomologie dont la partie de degr\'e $-1$ s'\'ecrit
$$T_\Ql(P_{u_\alpha})\otimes \,^p \rmH^n(f_!\Ql)|_{S_\alpha} \rightarrow
\,^p \rmH^{n-1}(f_!\Ql)|_{S_\alpha}.
$$
Cette fl\`eche se d\'ecompose en somme directe suivant la d\'ecomposition
canonique de $\,^p \rmH^n(f_!\Ql)$ et $\,^p \rmH^{n-1}(f_!\Ql)$ par le support
\begin{equation}\label{action Lambda sur S u alpha faisceaux pervers}
\bigoplus_{\alpha'\in{\mathfrak A}} T_\Ql(P_{u_\alpha})\otimes K^n_{\alpha'}|_{S_\alpha} \rightarrow
\bigoplus_{\alpha'\in{\mathfrak A}} K^{n-1}_{\alpha'}|_{S_\alpha}.
\end{equation}
Si $\alpha'\not=\alpha''$, la fl\`eche induite
$$T_\Ql(P_{u_\alpha})\otimes K^n_{\alpha'}|_{S_\alpha} \rightarrow K^{n-1}_{\alpha''}|_{S_\alpha}$$
est nulle car $T_\Ql(P_{u_\alpha})\otimes K^n_{\alpha'}|_{S_\alpha}$ est
une extension successive de faisceaux pervers simples de support
$Z_{\alpha'}\cap S_\alpha$ alors que $K^{n-1}_{\alpha''}|_{S_\alpha}$ est
une extension successive de faisceaux pervers simples de support
$Z_{\alpha''}\cap S_\alpha$. Ainsi, la fl\`eche \ref{action Lambda sur S u
alpha faisceaux pervers} est une somme directe des fl\`eches
\begin{equation}\label{action Lambda u alpha sur K alpha'}
T_\Ql(P_{u_\alpha})\otimes K^n_{\alpha'}|_{S_\alpha} \rightarrow
 K^{n-1}_{\alpha'}|_{S_\alpha}
\end{equation}
la somme \'etant \'etendue tous les $\alpha'\in {\mathfrak A}$. L'action gradu\'ee
de $\Lambda_{P,u_\alpha}$ sur $\bigoplus_{\alpha'\in {\mathfrak A}}
K_{\alpha'}|_{S_{\alpha}}$ se d\'ecompose donc en une somme directe de
des actions gradu\'ees de $\Lambda_{P,u_\alpha}$ sur chaque
$K_{\alpha'}|_{S_{\alpha}}$.

Le point g\'eom\'etrique $u_\alpha$ est au-dessus d'un point $u_\alpha^0$
\`a valeur dans un corps fini. D'apr\`es \ref{quasi-relevement}, on a une
d\'ecomposition canonique en somme directe
$$\rmT_\Ql(P_{u_\alpha})=\rmT_\Ql(R_{u_\alpha})\oplus \rmT_\Ql(A_{u_\alpha})$$
gr\^ace \`a l'action de $\Gal(u_\alpha/ u_\alpha^0)$. On en d\'eduit donc
une action de $\rmT_\Ql(A_{u_\alpha})$ sur $f_!\Ql|_{S_\alpha}$
$$\rmT_\Ql(A_{u_\alpha}) \boxtimes (f_!\Ql|_{S_\alpha})\rightarrow f_!\Ql|_{S_\alpha}[-1],$$
puis une action  sur la somme directe des
faisceaux pervers de cohomologie
$$\rmT_\Ql(A_{u_\alpha}) \otimes \,^p \rmH^n(f_!\Ql)|_{S_\alpha} \rightarrow
\,^p \rmH^{n-1}(f_!\Ql)|_{S_\alpha}
$$
laquelle se d\'ecompose en une somme directe des fl\`eches
$$\rmT_\Ql(A_{u_\alpha})\otimes K^n_{\alpha'}|_{S_\alpha} \rightarrow
 K^{n-1}_{\alpha'}|_{S_\alpha}$$
pour $\alpha'\in {\mathfrak A}$.

\begin{proposition}\label{alpha' libre}
Pour tout $\alpha'\not=\alpha$, pour tout entier $m$, le $\Ql$-espace
vectoriel gradu\'e $\rmH^m(K_{\alpha',u_\alpha})$ est un
$\Lambda_{A_{u_\alpha}}$-module libre.
\end{proposition}

\begin{proof}
Si $\dim(Z_{\alpha'})\leq \dim(Z_\alpha)$ avec $\alpha'\not=\alpha$, alors
$K_{\alpha'}|_{S_\alpha}$ est nul de sorte qu'il n'y a rien \`a d\'emontrer. On
peut donc supposer que $\dim(Z_{\alpha'})> \dim(Z_\alpha)$.

Au-dessus de $U_{\alpha'}$, on a une fl\`eche canonique entre syst\`emes
locaux gradu\'es
$$\Lambda_{A_{\alpha'}}\otimes L_{\alpha'}\rightarrow L_{\alpha'}$$
d\'efinie dans \ref{A alpha}. D'apr\`es l'hypoth\`ese de r\'ecurrence, pour
tout point g\'eom\'etrique $u_{\alpha'}$ de $U_{\alpha'}$ d\'efini sur un
corps fini, la fibre de $L_{\alpha'}$ en $u_{\alpha'}$ est
$\Lambda_{A_{\alpha'},u_{\alpha'}}$-module libre. D'apr\`es le th\'eor\`eme
de d\'ecomposition, $L_{\alpha'}$ est un syst\`eme local gradu\'e
semi-simple. En appliquant le lemme \ref{Nakayama}, on sait qu'il existe un
syst\`eme local gradu\'e $E_{\alpha'}$ sur $U_{\alpha'}$ et un isomorphisme
$$L_{\alpha'}\simeq\Lambda_{A_{\alpha'}} \otimes E_{\alpha'}$$
en tant que $\Lambda_{A_{\alpha'}}$-modules gradu\'es.

Sur l'intersection $U_{\alpha'}\cap S_\alpha$ on a encore cette
factorisation. Notons $y_{\alpha'}$ le point g\'en\'erique de $U_{\alpha'}\cap
S_\alpha$ et $\bar y_{\alpha'}$ un point g\'eom\'etrique au-dessus de
$y_{\alpha'}$. On a alors un isomorphisme de repr\'esentations de $\Gal(\bar
y_{\alpha'}/ y_{\alpha'})$
$$L_{\alpha',\bar y_{\alpha'}}=\Lambda_{A_{\alpha',\bar y_{\alpha'}}} \otimes
E_{\alpha',\bar y_{\alpha'}}.
$$

La fl\`eche de sp\'ecialisation
$$\rmT_\Ql(P_{u_\alpha})\rightarrow \rmT_\Ql(P_{\bar y_{\alpha'}})$$
est injective et identifie $\rmT_\Ql(P_{u_\alpha})$ et le sous-espace vectoriel de $\rmT_\Ql(P_{\bar y_{\alpha'}})$ des vecteurs $\Gal(\bar y_{\alpha'}/ y_{\alpha'})$-invariants.
Nous avons besoin de l'hypoth\`ese $\rmT_Ql(P)$ est polarisable dans le lemme interm\'ediaire suivant.

\begin{lemme} \label{split}
Pour tout rel\`evement homologique
$$\rmT_\Ql(A_{u_\alpha})\rightarrow \rmT_\Ql(P_{u_\alpha}),$$ l'application
$$\rmT_\Ql(A_{u_\alpha}) \rightarrow \rmT_\Ql(A_{\bar y_{\alpha'}})$$
qui s'en d\'eduit est injective. De plus, il existe un compl\'ement de $\rmT_\Ql(A_{u_\alpha})$
dans $\rmT_\Ql(A_{\bar y_{\alpha'}})$ qui est $\Gal(\bar y_{\alpha'}/ y_{\alpha'})$-stable.
\end{lemme}

\begin{proof}
La fl\`eche de sp\'ecialisation $\rmT_\Ql(P_{u_\alpha})\rightarrow \rmT_\Ql(P_{\bar y_{\alpha'}})$ est compatible avec la form altern\'ee de polarisation. N'importe quel rel\`evement homologique $\rmT_\Ql(A_{u_\alpha})\rightarrow \rmT_\Ql(P_{u_\alpha})$ est compatible avec la forme altern\'ee de sorte que l'application  $\rmT_\Ql(A_{u_\alpha})\rightarrow\rmT_\Ql(P_{\bar y_{\alpha'}})$ qui s'en d\'eduit l'est aussi. Il s'ensuit que $\rmT_\Ql(A_{u_\alpha}) \rightarrow \rmT_\Ql(A_{\bar y_{\alpha'}})$ est injective et que l'orthogonal de $\rmT_\Ql(A_{u_\alpha})$ dans $\rmT_\Ql(A_{\bar y_{\alpha'}})$ est un compl\'ement $\Gal(\bar y_{\alpha'}/ y_{\alpha'})$-stable.
\end{proof}

Continuons la d\'emonstration de \ref{alpha' libre}. La d\'ecomposition en somme directe \ref{split}
$$\rmT_\Ql(A_{\bar y_{\alpha'}})=\rmT_\Ql(A_{u_\alpha}) \oplus {\rm U}$$
de repr\'esentations de $\Gal(\bar y_{\alpha'}/ y_{\alpha'})$. induit un
isomorphisme de repr\'e\-sentations de $\Gal(\bar y_{\alpha'}/ y_{\alpha'})$
$$\Lambda_{A_{\bar y_{\alpha'}}}= \Lambda_{A_{u_\alpha}}\otimes \Lambda(U)$$
o\`u $\Lambda(U)=\bigoplus_i \wedge^i(U)[i]$. Ceci implique une
factorisation en produit tensoriel de repr\'e\-sentations de $\Gal(\bar y_{\alpha'}/ y_{\alpha'})$
$$L_{\alpha',\bar y_{\alpha'}}=\Lambda_{A_{u_\alpha}}\otimes \Lambda(U) \otimes
E_{\alpha',\bar y_{\alpha'}}.
$$
Il existe donc un isomorphisme
$$L_{\alpha'}|_{U_{\alpha'}\cap S_\alpha}=\Lambda_{A_{u_\alpha}} \boxtimes E'_{\alpha'}$$
o\`u $E'_{\alpha'}$ est un syst\`eme local sur $U_{\alpha'}\cap S_\alpha$.
Puisque le produit tensoriel externe avec $\Lambda_{A_{u_\alpha}}$
commute avec le prolongement interm\'ediaire de $U_{\alpha'}\cap
S_\alpha$ \`a $Z_{\alpha'}\cap S_\alpha$ et avec le foncteur fibre en $u_{\alpha_0}$,
la proposition \ref{alpha' libre} s'en d\'eduit.
\end{proof}

Consid\'erons maintenant la suite spectrale \ref{suite spectrale}
$$E_2^{m,n}=\rmH^m(\,^p \rmH^n (f_* \Ql)_{u_\alpha}) \Rightarrow
\rmH^{m+n}(M_{u_\alpha})$$ qui d\'eg\'en\`ere en $E_2$ d'apr\`es
\ref{degenerescence de suite spectrale}.
On obtient ainsi une filtration de
$$H=\bigoplus_j \rmH^j(M_{u_\alpha})$$
dont le $m$-\`eme gradu\'e est
$$\bigoplus_n \rmH^m(\, ^p\rmH^n (f_*\Ql)_{s_0}).$$
Cette filtration est stable sous l'action de $\Lambda_{A_{u_\alpha}}$. Son
action sur le $m$-\`eme gradu\'e se d\'eduit de l'action de
$\Lambda_{A_{u_\alpha}}$ sur la somme directe
$$\bigoplus_n \, ^p\rmH^n (f_*\Ql)|_{S_\alpha}$$
et donc de celle sur les $K_{\alpha'}|_{S_\alpha}$. Ce
$m$-\`eme gradu\'e $\bigoplus_n \rmH^m(\, ^p\rmH^n (f_*\Ql)_{s_0})$  se
d\'ecompose donc en une somme directe de
$\Lambda_{A_{u_\alpha}}$-modules gradu\'es
$$\bigoplus_{\alpha'} \rmH^m(K_{\alpha',u_\alpha})$$

Pour $\alpha'\not=\alpha$, on sait d\'ej\`a que
$\rmH^m(K_{\alpha',u_\alpha})$ est un $\Lambda_{A_{u_\alpha}}$-module
libre \cf \ref{alpha' libre}. Pour $\alpha'=\alpha$, on a
$\rmH^m(K_{\alpha,u_\alpha})=0$ sauf pour $m=-\dim(Z_\alpha)$. On
obtient ainsi une filtration de $H$ par des sous
$\Lambda_{A_{u_\alpha}}$-modules
$$0\subset H' \subset H''\subset H=\bigoplus_j \rmH^j(M_{u_\alpha})$$
tels que $H'$ et $H/H''$ sont des
$\Lambda_{A_{u_\alpha}}$-modules libres et tel que
$$H''/ H'= L_{\alpha,u_\alpha}.$$
D'apr\`es \ref{liberte fibre par fibre}, on sait que $H$ est aussi un
$\Lambda_{A_{u_\alpha}}$-module libre. On va en d\'eduire que
$L_{\alpha,u_\alpha}$ est aussi un $\Lambda_{A_{u_\alpha}}$-module libre
par une propri\'et\'e particuli\`ere de l'anneau $\Lambda_{A_{u_\alpha}}$.

Puisque $\Lambda_{A_{u_\alpha}}$ est une alg\`ebre locale,
tout $\Lambda_{A_{u_\alpha}}$-module projectif est libre. La suite exacte
$$0\rightarrow H'' \rightarrow H \rightarrow H/ H'' \rightarrow 0$$
avec $H$ et $H/H''$ libres, implique que $H''$ est aussi libre.

Notons aussi que  $\Lambda_{A_{u_\alpha}}$ est une $\Ql$-alg\`ebre de
locale dimension finie ayant un socle de dimension un. On en d\'eduit que le
dual $(\Lambda_{A_{u_\alpha}})^*$ de $\Lambda_{A_{u_\alpha}}$ en tant
que $\Ql$-espaces vectoriels est un $\Lambda_{A_{u_\alpha}}$-module
libre. Consid\'erons la suite exacte duale
$$0\rightarrow (H''/ H')^* \rightarrow (H'')^* \rightarrow (H')^* \rightarrow 0$$
o\`u $(H'')^*$ et $(H')^*$ sont des $\Lambda_{A_{u_\alpha}}$-modules
libres. Il s'ensuit que $(H''/ H')^*$ est un $\Lambda_{A_{u_\alpha}}$-module
libre de sorte que $H''/ H'$ l'est aussi.
\end{proof}

\begin{remarque} \label{liberte avec pi 0}
La discussion de ce paragraphe s'\'etend mot pour mot au cas o\`u $P$ a \'eventuellement des fibres non connexes. Soit $g:P\rightarrow S$ un $S$-sch\'ema en groupe lisse de type fini agissant sur un $S$-sch\'ema propre $f:M\rightarrow S$. Supposons que $M$ est lisse au-dessus du corps de base $k$. Soit $\pi_0(P)$ le faisceau des composantes connexes des fibres de $P$. Comme dans \ref{kappa decomposition}, supposons qu'il existe un groupe fini $\Pi_0$ est un homomorphisme surjectif $\Pi_0\rightarrow \pi_0(P)$. On a alors une d\'ecomposition caninque en somme directe
$$f_* \Ql=\bigoplus_{\kappa\in \Pi_0^*} (f_* \Ql)_\kappa$$
telle que pour tout $x\in \Pi_0$, $(x-\kappa(x))^N$ agit trivialement sur $(f_* \Ql)_\kappa$ pour un certain entier $N$.

Pour tout $\kappa\in \Pi_0^*$, il existe un ensemble fini $\mathfrak A_\kappa$ de sous-sch\'emas ferm\'es $Z_\alpha$ de $S\otimes_k \bar k$ tel qu'on a des d\'ecompositions en somme directe canonique
$$K^n_\kappa=\bigoplus_{\alpha\in \mathfrak A_\kappa} K^n_\alpha$$
o\`u $K^n_\kappa=\, ^p\rmH^n((f_* \Ql)_\kappa)$ et $K^n_\kappa$ est la somme directe des facteurs simples de $K^n_\kappa$ de support $Z_\alpha$. Le lemme \ref{U alpha} s'applique \`a $Z_\alpha$ ; en particulier, il existe un ouvert dense $U_\alpha$ de $Z_\alpha$ au-dessus du quel $K^n_\alpha$ est un syst\`eme local $L^n_\alpha$ avec un d\'ecalage et la composante neutre $P^0_\alpha$ de $P|_{U_\alpha}$ admet un quotient ab\'elien $A_\alpha$ apr\`es un changement de base radiciel qui fibre par fibre est le d\'evissage de Chevalley. La proposition \ref{liberte} s'applique \`a $L^n_\alpha$ c'est-\`a-dire $\bigoplus_n L^n_\alpha$ est un module libre sur l'alg\`ebre d'homologie de $A_\alpha$.
\end{remarque}

\subsection{D\'emonstration de l'in\'egalit\'e d'amplitude}\label{demo amplitude}

\ref{inegalite amplitude} est une cons\'equence
imm\'ediate de \ref{liberte} et de la remarque \ref{liberte avec pi 0}.

\begin{proof}
Soit $\kappa\in \Pi_0^*$ et $\alpha\in \mathfrak A_\kappa$. Soit $Z_\alpha$ un sous-sch\'ema ferm\'e irr\'eductible de $S\otimes_k \bar k$ correspondant. Comme dans \ref{U alpha}, il existe un ouvert dense de $U_\alpha$ de $Z_\alpha$ tel que la restriction de $K^n_\alpha$ \`a $U_\alpha$ est un syst\`eme local $L^n_\alpha$ avec un d\'ecalage et la restriction de $P^0$ \`a $U_\alpha$ admet un quotient ab\'elien $A_\alpha$ apr\`es un changment de base radiciel tel que fibre par fibre on trouve le d\'evissage de Chevalley. D'apr\`es \ref{liberte} et en tenant compte de la remarque \ref{liberte avec pi 0}, $\bigoplus_n L^n_\alpha$ est un module libre sur l'ag\`ebre $\Lambda_{A_\alpha}$ des homologies de $A_\alpha$. Puisque c'est un module non nul, son amplitude est sup\'erieure ou \'egale \`a celle de $\Lambda_{A_\alpha}$ \'egale \`a $2(d-\delta_\alpha)$. On en d\'eduit
l'in\'egalit\'e
$${\rm amp}(\alpha) \geq 2(d-\delta_\alpha)$$
qu'on voulait.
\end{proof}

\begin{remarque}
Notons les r\'esultats de ce paragraphe restent inchang\'es si au lieu de
$k$-sch\'emas $P$ et $M$ on a des champs de Deligne-Mumford.
\end{remarque}

\subsection{Le cas de la fibration de Hitchin}
\label{subsection : description support}

Nous allons maintenant appliquer le th\'eor\`eme g\'en\'eral du support au cas particulier de la fibration de Hitchin ou plus pr\'ecis\'ement \`a sa partie anisotrope. Il s'agit de v\'erifier que les hypoth`\`eses de \ref{support final} est bien v\'erifi\'ees et d'en tirer les cons\'equences.

Rappelons qu'on a un morphisme propre plat de fibres r\'eduites $f^\ani:\calM^\ani\rightarrow\calA^\ani$
de source d'un champ de Deligne-Mumford lisse et de but un ouvert d'un espace affine standard. On a aussi d\'efini un champs de Picard de Deligne-Mumford $\calP^\ani \rightarrow \calA^\ani$ qui agit sur $\calM^\ani$. D'apr\`es \ref{densite globale}, il existe un ouvert $\calM^{\reg,\ani}$ de $\calM^\ani$ au-dessus duquel $\calP^\ani$ agit simplement transitivement et une section $\calA^\ani\rightarrow \calM^{\reg,\ani}$ de sorte qu'on a
$${\rm Irr}(\calM^\ani/\calA^\ani)=\pi_0(\calM^{\reg,\ani}/\calA^\ani)=\pi_0(\calP^\ani/\calA^\ani).$$
En passant \`a l'ouvert \'etale $\tilde\calA$, on a vu qu'il existe un homomorphisme surjectif
$$\Pi_0\longrightarrow \pi_0(\calP^\ani/\calA^\ani)|_{\tilde\calA^\ani}$$
d'un faisceau constant fini $\Pi_0$ sur la restriction du  faisceau $\pi_0(\calP^\ani/\calA^\ani)$ \`a $\tilde\calA^\ani$. Pour appliquer \ref{support final} au morphisme
$$\tilde f^\ani :\tilde\calM^\ani \longrightarrow \tilde\calA^\ani$$
il ne reste qu'\`a v\'erifier l'assertion suivante.

\begin{proposition}\label{T ell P est polarisable}
Le module de Tate $\rmT_\Ql(\calP^\ani)$ est polarisable au sens de \ref{module de Tate polarisable}.
\end{proposition}

La d\'emonstration de cette proposition est fond\'ee sur th\'eorie de l'accouplement de Weil que nous allons rappeler pour la commodit\'e du lecteur. Soit $S$  un sch\'ema local strictement hens\'elien.  Soit $c:C\rightarrow S$ un morphisme propre plat de fibres g\'eom\'etriquement r\'eduites de dimension un.

Supponsons que $C$ est connexe. Consid\'erons la factorisation de Stein $C\rightarrow S'\rightarrow S$ o\`u $C\rightarrow S'$ est un morphisme propre de fibres connexes non vide et $S'\rightarrow S$ est un morphisme fini. Puisque la fibre sp\'eciale de $c$ est r\'eduite, le morphisme $S'\rightarrow S$ est fini et \'etale. Puisque $C$ est connexe, $S'$ l'est aussi. Puisqu'on a suppos\'e que $S$ est stricement hens\'elien, on a alors $S'=S$. Autrement dit les fibres de $c:C\rightarrow S$ sont connexes.

Consid\'erons le $S$-champ d'Artin $\Pic_{C/S}$ qui associe \`a tout $S$-sch\'ema $Y$ le groupoide des fibr\'es inversibles sur $C\times_S Y$. Il est lisse au-dessus de $S$. Consid\'erons sa composante neutre $\Pic^0_{C/S}$. Pour tout point $L\in \Pic_{X/S}(Y)$, pour tout $y\in Y$, on peut d\'efinir caract\'eristique d'Euler-Poincar\'e $\chi_y(L)$ de la restriction de $L$ \`a $C_y$. Si $Y$ est connexe, cet entier est ind\'ependant de $y$ et nous le notons $\chi(L)$. Si $L\in \Pic^0_{C/S}$, on a $\chi(L)=\chi(\calO_C)$.

Pour tout couple de $L,L'\in \Pic^0_{C/S}$, nous d\'efinissions leur accouplement de Weil par la formule
\begin{eqnarray*}
\langle L,L' \rangle_{C/S} = {\rm det}(\rmR c_*(L\otimes L'))\otimes {\rm det}(\rmR c_* L)^{\otimes-1} \\ \otimes {\rm det}(\rmR c_* L')^{\otimes-1} \otimes {\rm det}(\rmR c_* \calO_C)
\end{eqnarray*}
en utilisant le d\'eterminant de cohomologie. Si $t$ est un automorphisme de $L$ qui est alors un scalaire, $t$ agit sur $\det(Rc_* L)$ par le scalaire $t^{\chi(L)}$. En utilisant les \'egalit\'es
$$\chi(L\otimes L')=\chi(L)=\chi(L')=\chi(\calO_C)$$
pour $L,L'\in \Pic^0_{X/S}$, on v\'erifie que pour tout couple de scalaires $(t,t')$ l'action de $t$ sur $L$ et l'action de $t'$ sur $L'$ induisent l'identit\'e sur $\langle L,L' \rangle_{C/S}$.

Si $N$ est un entier inversible sur $S$ et $L$ est un fibr\'e inversible muni d'un isomorphisme $\iota_L:L^{\otimes N}\rightarrow\calO_C$, on a un isomorphisme
$$\langle L,L' \rangle_{C/S}^{\otimes N}=\calO_S.$$
Si en plus $L'$ est aussi muni d'un isomorphisme $\iota_{L'}:L'^{\otimes N}\rightarrow \calO_C$, on a autre isomorphisme $\langle L,L' \rangle_{C/S}^{\otimes N}=\calO_S$. La diff\'erence de ces deux isomorphismes d\'efinit une $N$-\`eme racine d'unit\'e. Cette racine ne d\'epend que des classes d'isomorphisme de $L$ et de $L'$ en vertu de la discussion sur l'effet des scalaires.

Supposons maintenant que $C$ est un courbe projective connexe sur un corps alg\'ebriquement clos $k$. Le champ ${\rm Pic}_C^0$ est alors le quotient d'un $k$-groupe alg\'ebrique commutative connexe ${\rm Jac}_C$ par $\GG_m$ agissant trivialement. La construction ci-dessus d\'efinit une forme altern\'ee
$$\rmT_\Ql({\rm Jac}_C)\times \rmT_\Ql({\rm Jac}_C) \longrightarrow \Ql(-1)$$
qui est non-d\'eg\'en\'er\'ee lorsque $C$ est lisse. Pour terminer cette digression, il reste \`a consid\'erer le comportement de l'accouplement de Weil vis-\`a-vis de la normalisation.

\begin{lemme}\label{normalisation Weil pairing}
Soit $C$ une courbe projective r\'eduite sur un corps alg\'ebriquement clos $k$. Soit $C^\flat$ sa normalisation et $\xi:C^\flat \rightarrow C$ le morphisme de normalisation. Soient $L,L'$ deux fibr\'es inversibles sur $C$. On a alors un isomorphisme canonique de $k$-espaces vectoriels de dimension un
$$\langle L,L'\rangle_{C}=\langle \xi^*L,\xi^*L' \rangle_{C^\flat}.$$
\end{lemme}

\begin{proof} On a une suite exacte
$$0\longrightarrow \calO_C \longrightarrow \xi_* \calO_{C^\flat} \longrightarrow \mathcal{D} \longrightarrow 0$$
o\`u $\mathcal{D}$ est un $\calO_C$-modules fini support\'e par une collection finie de points $\{c_1,\ldots,c_n\}$ de $C$. Notons $\mathcal{D}_i$ le facteur direct de $\mathcal{D}$ support\'e par $c_i$ et $d_i$ sa longueur. La multiplicativit\'e du d\'eterminant nous fournit alors un isomorphisme
$${\rm det}(c^\flat_*\calO_{C^\flat})={\rm det}(c_*\xi_*\calO_{C^\flat})={\rm det}(c_*\calO_{C})\otimes\bigotimes_{i=1}^r \wedge^{d_i} \mathcal{D}_i$$
o\`u on a not\'e $c:C\rightarrow\Spec(k)$ et $c^\flat:C^\flat\rightarrow \Spec(k)$ les morphismes structuraux.

Soit $L$ un fibr\'e inversible sur $C$. En utilisant la formule de produit $\xi_* \xi^* L=(\xi_*\calO)\otimes L$, on obtient la formule
$${\rm det}(c^\flat_*\xi^*L)={\rm det}(c_*L)\otimes\bigotimes_{i=1}^r (L_{c_i}^{\otimes d_i}\otimes  \wedge^{d_i} \mathcal{D}_i)$$
o\`u $L_{c_i}$ est la fibre de $L$ en $c_i$. En appliquant cette formule \`a $L\otimes L'$, $L$ et \`a $L'$, on obtient le lemme.
\end{proof}

On est maintenant en mesure de d\'emontrer la proposition \ref{T ell P est polarisable}.

\begin{proof}
Rappelons que pour tout point $a\in \calA^\heartsuit$, $\calP_a$ est le champ de Picard des $J_a$-torseurs sur $X$. Soit $\pi_a:\tilde X_a\rightarrow X$ le rev\^etement cam\'erale associ\'ee \`a $a$. D'apr\`es la description galoisienne du centralisateur r\'egulier, on a un homomorphisme de faisceaux en groupes \cf \ref{J J1}
$$J_a \rightarrow \pi_{a,*}(T\times_X \tilde X_a)$$
o\`u $T$ est le tore maximal de $G$ dans l'\'epinglage fix\'e. En consid\'erant un d\'eploiement $X_\rho\rightarrow X$ de $G$, on obtient un rev\^etement fini \'etale $\tilde X_{\rho,a}\rightarrow \tilde X_a$ par changement de base. Notons $\pi_{\rho,a}:\tilde X_{\rho,a}\rightarrow X$ le morphisme compos\'e. On a donc un homomorphisme
$$J_a \rightarrow \pi_{\rho,a,*}(\bbT\times \tilde X_{\rho,a}).$$
o\`u $\bbT$ est un tor d\'eploy\'e. La donn\'ee d'un point de $\ell^n$-torsion de $\calP_a$ induit donc un point de $\ell^n$ torsion de ${\rm Pic}_{\tilde X_{\rho,a}}\otimes \bbX_*(\bbT)$. En choisissant une forme sym\'etrique invariante sur $\bbX_*(\bbT)\otimes\Ql$ et en utilisant l'accouplement de Weil sur ${\rm Pic}^0_{\tilde X_{\rho,a}}$, on obtient une forme altern\'ee
$$\rmT_\Ql(\calP_a^0)\otimes \rmT_\Ql(\calP_a^0)\longrightarrow \Ql(-1).$$
Il reste \`a d\'emontrer qu'en chaque point g\'eom\'etrique $a$, cette forme symplectique est nulle sur la partie affine  de $\rmT_\ell(\calP_a)$ et induit un accouplement non d\'eg\'en\'er\'ee sur sa partie ab\'elienne. Ceci se d\'eduit du lemme \ref{normalisation Weil pairing}\end{proof}

\begin{numero}
On peut maintenant appliquer \ref{support final} car toutes ses trois hypoth\`eses ont \'et\'e v\'erif\'ees. On ne peut pas toutefois tirer imm\'ediatement une conclusion tangible car il rne s'applique qu'aux sous-sch\'emas ferm\'es irr\'eductibles $Z$ v\'erifiant l'in\'egalit\'e
$${\rm codim}(Z)\geq \delta_Z.$$
Nous conjecturons que $\calP^\ani$ est $\delta$-r\'egulier au sens de \ref{delta regulier} ce qui revient \`a dire que l'in\'egalit\'e ci-dessus est satisfaite pour n'importe quel sous-sch\'ema ferm\'e irr\'eductible $Z$ de $\calA^\ani$. Cette conjecture est connue en caract\'eristique $0$. En caract\'eristique $p$, on devra se contenter de l'\'enonc\'e plus faible \ref{codimension}.
\end{numero}

\begin{numero}\label{delta bad}
Soit $\delta_{H}^{\rm bad}(D)$ le plus petit entier tel que
$${\rm codim}_{\calA_H}(\calA_{H,\delta_H^{\rm bad}(D)})<\delta_{H}^{\rm bad}(D).$$
Si cet entier n'existe pas, on pose $\delta_{H}^{\rm bad}(D)=\infty$. Notons
$\calA_H^{\rm bad}$ la r\'eunion de toutes les strates $\calA_{H,\delta_H}$
\`a $\delta_H$ constant avec $\delta_H \geq \delta_H^{\rm bad}(D)$. C'est
un sous-sch\'ema ferm\'e de $\calA_H^\heartsuit$ qui est \'eventuellement vide.
D'apr\`es \ref{codimension}, l'entier $\delta_{H}^{\rm bad}(D)$  tend vers
l'infini quand $\deg(D)$ tend vers l'infini.
\end{numero}

\begin{lemme}\label{codimension dans tilde calA}
Soit $Z$ un sous-sch\'ema ferm\'e irr\'eductible de $\tilde\calA_H$ qui n'est pas contenu dans $\tilde\calA_H^{\rm bad}$. Alors on a
$${\rm codim}_{\tilde\calA_H}(Z) \geq \delta_{H,Z}\ { et}\
{\rm codim}_{\tilde\calA}(Z) \geq \delta_{Z}$$
o\`u
$\delta_{H,Z}$ est la valeur minimale de la fonction $\delta_H:Z(\bar
k)\rightarrow \NN$ et o\`u $\delta_Z$ est la valeur minimale de la fonction
$\delta:Z(\bar k)\rightarrow \NN$. \end{lemme}

\begin{proof}
Puisque $Z$ est irr\'eductible, il est contenu dans l'adh\'e\-rence de la
strate \`a $\delta_H$-constant $\tilde\calA_{H,\delta_{H,Z}}$. Comme $Z$ n'est pas contenu dans $\tilde\calA_H^{\rm bad}$, on a $\delta_{H,Z}< \delta_{H}^{\rm bad}(D)$. Il r\'esulte de
la d\'efinition de $\delta_H^{\rm bad}(D)$ que
$${\rm codim}_{\tilde\calA_H}(\tilde \calA_{H,\delta_{H,Z}}) \geq \delta_{H,Z}.$$
On en d\'eduit la m\^eme in\'egalit\'e pour ${\rm
codim}_{\tilde\calA_H}(Z)$. L'in\'egalit\'e portant sur ${\rm
codim}_{\tilde\calA}(Z)$ s'en d\'eduit car d'apr\`es \ref{delta-dellta_H} et
\ref{codim A_H}, la fonction $$\delta-\delta_H:Z(\bar k)\rightarrow \NN$$ est
constante de valeur \'egale \`a $\dim(\calA)-\dim(\calA_H)$.
\end{proof}

Maintenant, \ref{support final} implique l'\'enonc\'e suivant.

\begin{theoreme}\label{support faible}
Soit $K$ un facteur g\'eom\'etriquement simple de
$$\bigoplus_n\, ^p \rmH^n(\tilde f^\ani_* \Ql)_\kappa.$$
Supposons que le support de $K$ est contenu dans $\tilde\calA_H$ mais
n'est pas contenu dans $\tilde\calA_H^{\rm bad}$. Alors il est \'egal \`a
$\tilde\calA_H$.
\end{theoreme}

Cet \'enonc\'e n'est pas aussi joli qu'on voudrait \`a cause de la pr\'esence du fantomatique $\tilde\calA_H^{\rm bad}$. Il suffit n\'eanmoins pour d\'emontrer le lemme fondamental avec l'aide de \ref{codimension}.

\section{Comptage de points}
\label{section : comptage}

Dans ce dernier chapitre, nous compl\'etons la d\'emonstration du th\'eo\-r\`eme de stabilisation g\'eom\'etrique \ref{stabilisation sur tilde A} ainsi que celle des conjectures de Langlands-Shelstad \ref{LS} et de Waldspuger \ref{Waldspurger} en s'appuyant le th\'eor\`eme \ref{support faible}.

La d\'emonstration est fond\'ee sur le principe suivant. Soit $K_1,K_2$ deux complexes purs sur un $k$-sch\'ema de type fini $S$ irr\'eductible. Supposons que tout faisceau pervers g\'eom\'etriquement simples pr\'esent dans $K_1$ ou $K_2$ a pour support $S$ tout entier, alors pour d\'emontrer que $K_1$ et $K_2$ ont la m\^eme classe dans le groupe de Grothendieck, il suffit de d\'emontrer qu'il existe un ouvert dense $U$ de $S$ aussi petit que l'on veut, tel que pour toute extension finie $k'$ de $k$, pour tout $u\in U(k')$, les traces du Frobenius $\sigma_{k'}$ de $k'$ sur $K_{1,u}$ et $K_{2,u}$ sont \'egales. L'\'egalit\'e dans le groupe de Grothendieck implique qu'on ait la m\^eme \'egalit\'e mais pour tout point $s\in S(k')$. Le comptage de points  dans une fibre au-dessus d'un point d'un petit ouvert $U$ devrait \^etre beaucoup plus agr\'eable que dans une fibre quelconque ce qui fait la force du th\'eor\`eme du support.

On va donc compter le nombre $\sharp \calM_a(k)$ des $k$-points dans une fibre de Hitchin anisotrope.
Plus pr\'ecis\'ement, on veut une formule pour la partie stable $\sharp \calM_a(k)_{\rm st}$.
La formule de produit \ref{produit} permet d'exprimer ce nombre comme un produit du nombre $\sharp\calP_a^0(k)$ des $k$-points de la composante neutre de $\calP_a$ et des nombres de points dans des quotients de fibres de Springer affine \cf \ref{subsection : comptage Hitchin}. Pour ces quotients des fibres de Springer affines, un comptage plus ou moins direct donne comme r\'esultat une int\'egale orbitale stable \cf \ref{subsection : comptage Springer}. On a aussi la variante du comptage avec une $\kappa$-pond\'eration qui donne lieu aux $\kappa$-int\'egrales orbitales. A chaque fois, il s'agit du comptage de points d'un champ d'Artin de la forme $[M/P]$ o\`u $M$ est un $k$-sch\'ema et o\`u $P$ est un $k$-groupe alg\'ebrique agissant sur $M$. Ce formalisme est rappel\'e dans \ref{subsection : comptage} en m\^eme temps qu'une formule des points fixes ad hoc qui a \'et\'e d\'emontr\'ee dans l'appendice A.3 de \cite{LN}.

Pour $a\in \calA^\diamondsuit(k)$, les quotients des fibres de Springer affines sont tous triviaux de sorte que $\sharp \calM_a(k)_{\rm st}$ est \'egal au nombre $\sharp \calP_a^0(k)$ o\`u $\calP_a^0$ est essentiellement une vari\'et\'e ab\'elienne. Ceci permet de d\'emontrer l'\'egalit\'e les nombres $\sharp \calM_{1,a}(k)_{\rm st}$ et $\sharp \calM_{2,a}(k)_{\rm st}$ pour $a\in\calA^\diamondsuit(k)$ associ\'es \`a deux groupes $G_1$ et $G_2$ ayant des donn\'ees radicielles isog\`enes. Le th\'eor\`eme de support \ref{support final} permet de prolonger l'identit\'e \`a $a\in (\calA^{\rm ani}-\calA^{\rm bad})(k)$. On obtient ainsi assez de points globaux pour obtenir toutes les int\'egrales orbitales stables locales en faisant tendre $\deg(D)$ vers l'infini. On d\'emontre ainsi le lemme fondamental non standard conjectur\'e par Waldspurger \cf \ref{subsection : demo Waldspurger}. En renversant l'argument global-local, on peut prolonger l'identit\'e $\sharp \calM_{1,a}(k)_{\rm st}=\sharp \calM_{2,a}(k)_{\rm st}$ \`a tout $a\in\calA^\ani(k)$.

La d\'emonstration de la conjecture de Langlands-Shelstad suit essentiellement la m\^eme strat\'egie avec un peu plus de difficult\'es techiniques. Pour $a_H\in\tilde\calA_H^\diamondsuit(k)$, les int\'egrales orbitales stables locales de $a_H$ sont triviales mais les $\kappa$-int\'egrales orbitales locales dans $G$ du point $a$ correspondant ne le sont pas n\'ecessairement. Toutefois, en r\'etr\'ecissant encore plus $\tilde\calA_H^\diamondsuit$, on peut supposer que ces $\kappa$-int\'egrales orbitales locales non triviales sont aussi simples que possibles. Le calcul de ces int\'egrales locales simples est bien connu et se ram\`ene au cas ${\rm SL}(2)$ trait\'e par Labesse et Langlands. On le reprend dans \ref{subsection : fibre simple}. En combinant ce calcul avec le th\'eor\`eme du support \ref{support final}, on obtient lla partie du th\'eor\`eme de stabilsation g\'eom\'etrique sur l'ouvert $\tilde\calA^\ani_H-\tilde\calA^{\rm bad}_H$ \cf \ref{subsection : A good}. De nouveau, en faisant tendre $\deg(D)$ vers l'infini, on obtient toutes les int\'egrales orbitales locales et d\'emontre ainsi le lemme fondamental \cf \ref{lemme fondamental}. En renversant l'argument global-local, on en d\'eduit tout le th\'eor\`eme de stabilisation g\'eom\'etrique.

\subsection{Remarques g\'en\'erales sur le comptage}\label{subsection : comptage}

Nous allons fixer dans ce paragraphe  un cadre pour les diff\'erents
probl\`emes de comptage que nous devrons r\'esoudre dans la suite. Nous
allons aussi pr\'eciser les abus de notations que nous allons
pratiquer syst\'ematiquement dans ce chapitre. Dans ce
paragraphe et contrairement au reste de l'article, la lettre $X$ ne d\'esigne
pas n\'ecessairement une courbe.

\begin{numero}
Si $M$ est un $k$-sch\'ema de type fini. D'apr\`es la formule de traces de
Grothendieck-Lefschetz, le nombre de $k$-points de $M$ peut \^etre
calcul\'e comme une somme altern\'ee de traces de l'\'el\'ement de
Frobenius $\sigma\in\Gal(\bar k/ k)$
$${\sharp\,} M(k)=\sum_n (-1)^n \tr(\sigma,\rmH^n_c(M))$$
o\`u $\rmH^n_c(M)$ d\'esigne le $n$-\`eme groupe de cohomologie \`a
support compact $\rmH^n_c(M\otimes_k \bar k,\Ql)$ de $M\otimes_k \bar
k$.
\end{numero}

Nous aurons besoin d'une variante de cette formule des traces dans le
contexte suivant. Soit $M$ un $k$-sch\'ema de type fini ou plus
g\'en\'eralement un champ de Deligne-Mumford de type fini, muni d'une
action d'un $k$-groupe alg\'ebrique commutatif de type fini $P$. Nous
voulons relier la trace de $\sigma$ sur une partie de la cohomologie \`a
support de $M$ avec le nombre de points du quotient $X=[M/ P]$ et le
nombre de points de la composante neutre de $P$. C'est le contenu de
l'appendice A.3 de \cite{LN} que nous allons maintenant rappeler et
g\'en\'eraliser l\'eg\`erement.
\begin{numero}
Soit donc $X=[M/ P]$ comme ci-dessus. On va  \'ecrire
$X$ pour le groupo\"ide $X(\bar k)$, $M$ pour l'ensemble $M(\bar k)$ et $P$
pour le groupe $P(\bar k)$. Par d\'efinition l'ensemble des objets de $X=[M/
P]$ est l'ensemble $M$. Soit $m_1, m_2\in M$. Alors l'ensemble des
fl\`eches $\Hom_X(m_1,m_2)$ est le transporteur
$$\Hom_X(m_1,m_2)=\{p\in P| pm_1=m_2\}.$$
La r\`egle de composition des fl\`eches se d\'eduit de la multiplication dans
le groupe $P$.
\end{numero}

\begin{numero}
L'action de l'\'el\'ement de Frobenius $\sigma\in \Gal(\bar k/ k)$ sur $M$ et
$P$ induit une action sur le groupo\"ide $X=[M/ P]$. Le groupo\"ide $X(k)$ des
points fixes sous l'action de $\sigma$ est par d\'efinition la cat\'egorie dont
\begin{itemize}
\item les objets sont les couples $(m,p)$ o\`u $m\in M$ et $p\in P$ tels
    que $p\sigma(m)=m$ ;
\item une fl\`eche $h:(m,p)\rightarrow (m',p')$ dans $X(k)$ est un
    \'el\'ement $h\in P$ tel que $hm=m'$ et $p'=hp\sigma(h)^{-1}$.
\end{itemize}
La cat\'egorie $X(k)$ a un nombre fini de classes d'isomorphisme d'objets et
chaque objet a un nombre fini d'automorphismes. On s'int\'eresse \`a la
somme
$${\sharp\,} X(k)=\sum_{x\in X(k)/\sim} \frac{1}{{\sharp\,} \Aut_{X(k)}(x)}$$
o\`u $x$ parcourt un ensemble de repr\'esentants des classes
d'isomorphisme des objets de $X(k)$.
\end{numero}

\begin{numero}
Soit $X=[M/ P]$ comme ci-dessus et soit $x=(m,p)$ un objet de $X(k)$. Par
d\'efinition des fl\`eches dans $X(k)$, la classe de $\sigma$-conjugaison de
$p$ ne d\'epend que de la classe d'isomorphisme de $x$. Puisque $P$ est un
$k$-groupe de type fini, le groupe des classes de $\sigma$-conjugaison de
$P$ s'identifie canoniquement \`a $\rmH^1(k,P)$. Notons $\cl(x)\in
\rmH^1(k,P)$ la classe de $\sigma$-conjugaison de $p$ qui ne d\'epend que
de la classe d'isomorphisme de $x$. D'apr\`es un th\'eor\`eme de Lang,
$\rmH^1(k,P)$ s'identifie \`a $\rmH^1(k,\pi_0(P))$ o\`u $\pi_0(P)$ d\'esigne
le groupe des composantes connexes de $P\otimes_k
\bar k$ qui est donc un groupe fini commutatif muni d'une action de
$\sigma$. Ainsi, pour tout caract\`ere $\sigma$-invariant $\kappa:
\pi_0(P)\rightarrow \Ql^\times$, pour tout objet $x\in X(k)$, on peut
d\'efinir un accouplement
$$\langle \cl(x), \kappa \rangle =\kappa(\cl(x)) \in \Ql^\times$$
qui ne d\'epend que de la classe d'isomorphisme de $x$. On peut donc
d\'efinir le nombre de points de $X(k)$ avec la $\kappa$-pond\'eration
$${\sharp\,} X(k)_\kappa= \sum_{x\in X(k)/\sim}
\frac {\langle \cl(x), \kappa \rangle}{{\sharp\,} \Aut_{X(k)}(x)}$$
o\`u $x$ parcourt un ensemble de repr\'esentants des classes
d'isomorphisme des objets de $X(k)$.
\end{numero}

\begin{numero}
Le groupe $P$ agit sur les groupes de cohomologie $\rmH^n_c(M)$ \`a
travers son groupe des composantes connexes $\pi_0(P)$ d'apr\`es le
lemme d'homotopie. On peut former le sous-espace propre
$\rmH^n_c(M)_\kappa$ de $\rmH^n_c(M)$ de valeur propre $\kappa$.
Puisque $\kappa$ est $\sigma$-invariant, $\sigma$ agit sur
$\rmH^n_c(M)_\kappa$.
\end{numero}

On a la variante suivante de la formule des traces de
Grothendieck-Lefschetz d\'emontr\'ee dans l'appendice A.3 de \cite{LN}.

\begin{proposition}\label{formule des points fixes}
Soient $M$ un $k$-sch\'ema de type fini, $P$ un $k$-groupe de type fini
agissant sur $M$ et $X=[M/ P]$. Alors la cat\'egorie $X(k)$ a un nombre fini
de classes d'isomorphisme d'objets et chaque objet a un nombre fini
d'automorphismes. Pour tout $\kappa:\pi_0(P)\rightarrow \Ql^\times$ un
caract\`ere $\sigma$-invariant, le nombre ${\sharp\,} X(k)_\kappa$ a
l'interpr\'etation cohomologique suivante :
$${\sharp\,} P^0(k) {\sharp\,} X(k)_\kappa= \sum_n (-1)^n\tr(\sigma,\rmH^n_c(M)_\kappa)$$
o\`u $P^0$ la composante neutre de $P$.
\end{proposition}

Nous envoyons \`a \cite[A.3.1]{LN} pour la d\'emonstration. On peut
consid\'erer deux exemples instructifs qui expliquent pourquoi il faut
s\'eparer le r\^ole de $P^0$ de $\pi_0(P)$.

\begin{exemple}\label{exemple : classifiant fini}
Supposons que $M=\Spec(k)$ et que $P$ est un groupe fini sur lequel
$\sigma$ agit. Soit $X=[M/ P]$ le classifiant de $P$. Alors les objets de la
cat\'egorie $X(k)$ sont les \'el\'ements $p\in P$ alors que les fl\`eches
$p\rightarrow p'$ sont les \'el\'ements $h\in P$ tels que
$p'=hp\sigma(h)^{-1}$. On a alors
$${\sharp\,} X(k)=\frac{{\sharp\,} P} {{\sharp\,} P}=1.$$
Par ailleurs, pour tout caract\`ere  $\sigma$-invariant $\kappa:
P\rightarrow \Ql^\times$ non trivial, on a ${\sharp\,} X(k)_\kappa=0$.
\end{exemple}

\begin{exemple}\label{exemple : classifiant G_m}
Si $M=\Spec(k)$ et $P=\GG_m$. Soit $X=[M/ P]$ le classifiant de $\GG_m$.
Les objets de $X(k)$ sont de nouveau les \'el\'ements $p\in P$ alors que les
fl\`eches $p\rightarrow p'$ sont les \'el\'ements $h\in P$ tels que
$p'=hp\sigma(h)^{-1}$. D'apr\`es Lang, tous les \'el\'ements de $\bar
k^\times$ sont $\sigma$-conjugu\'es de sorte qu'il n'y a qu'une seule classe
d'isomorphisme d'objets dans $X(k)$. Le groupe des automorphisme de
chaque objet de $X(k)$ est $k^\times$ et a donc $q-1$ \'el\'ements. On a
donc ${\sharp\,} X(k)=(q-1)^{-1}$.
\end{exemple}

Pour \'etudier le comptage des points dans les fibres de Springer affines,
nous avons besoin d'une variante de la discussion pr\'ec\'edente pour le cas o\`u
$M$ et $P$ sont localement de type fini et o\`u $[M/ P]$ a une propri\'et\'e
de finitude raisonnable.

Soient $M$ un $k$-sch\'ema localement de type fini et $P$ un $k$-groupe
commutatif localement de type fini agissant sur $M$. Pour que le quotient
$[M/ P]$ soit raisonnablement fini, nous faisons en plus les hypoth\`eses
suivantes.

\begin{hypothese}\label{hypothese : finitude}
\begin{enumerate}
  \item Le groupe des composantes connexes $\pi_0(P)$ est un groupe
      ab\'elien de type fini.
  \item Le stabilisateur dans $P$ de chaque point de $M$ est un
      sous-groupe de type fini.
  \item Il existe un sous-groupe discret sans torsion $\Lambda \subset P$
      tel que $P/ \Lambda$ et $M/ \Lambda$ sont de type fini.
\end{enumerate}
\end{hypothese}

La derni\`ere hypoth\`ese n\'ecessite d'\^etre comment\'ee. Puisque le
groupe $\Lambda$ est sans torsion, il agit sans points fixes sur $M$ puisque
les stabilisateurs dans $P$ de n'importe quel point de $M$ est de type fini
et en particulier ont une intersection triviale avec $\Lambda$. De plus, si
l'hypoth\`ese est v\'erifi\'ee, elle est v\'erifi\'ee pour n'importe quel
sous-group $\Lambda'\subset \Lambda$ d'indice finie et $\sigma$-invariant.

\begin{numero}
Il faut aussi expliquer comment construire un sous-groupe discret sans
torsion $\Lambda$ de $P$ tel que $P/ \Lambda$ soit de type fini. Le groupe
$P$ admet un d\'evissage canonique
$$1\rightarrow P^{\rm tf} \rightarrow P \rightarrow \pi_0(P)^{\rm lib} \rightarrow 0$$
o\`u $\pi_0(P)^{\rm lib}$ est le quotient libre maximal de $\pi_0(P)$ et
o\`u $P^{\rm tf}$ est le sous-groupe de type fini maximal de $P$. Puisque
$\pi_0(P)^{\rm lib}$ est un groupe ab\'elien libre, il existe un rel\`evement
$$\gamma:\pi_0(P)^{\rm lib} \rightarrow P$$
qui n'est pas n\'ecessairement $\sigma$-\'equivariant. Puisque $P^{\rm tf}$
est un $k$-groupe de type fini, la restriction de $\gamma$ \`a
$\Lambda=N\pi_0(P)^{\rm lib}$ pour $N$ assez divisible est
$\sigma$-\'equivariant. On obtient ainsi un sous-groupe discret sans torsion
$P$. L'hypoth\`ese $(3)$ est \'equivalente \`a ce que le quotient de $M$
par ce sous-groupe discret sans torsion est un $k$-sch\'ema de type fini.
\end{numero}

\begin{numero}
Consid\'erons le quotient $X=[M/ P]$. Le groupo\"ide $X(k)$ des points fixes
sous $\sigma$ a pour objet les couples $x=(m,p)$ avec $m\in M$ et $p\in P$
tels que $p\sigma(m)=m$. Une fl\`eche $h:(m,p)\rightarrow (m',p')$ est un
\'el\'ement $h\in P$ tel que $hm=m'$ et $p'=hp\sigma(h)^{-1}$. Soit
$P_\sigma$ le groupe des classes de $\sigma$-conjugaison dans $P$. La
classe de $\sigma$-conjugaison de $p$ d\'efinit un \'el\'ement $\cl(x)\in
P_\sigma$ qui ne d\'epend que de la classe d'isomorphisme de $x$. Puisque
$m$ est d\'efini sur une extension finie de $k$, $\cl(x)$ appartient au
sous-groupe
$$\cl(x)\in \rmH^1(k,P)$$
qui est la partie torsion de $P_\sigma$.
\end{numero}

\begin{lemme}\label{tilde kappa}
Tout caract\`ere $\kappa:\rmH^1(k,P)\rightarrow \Ql^\times$ peut
s'\'etendre en un caract\`ere d'ordre fini $\tilde \kappa:P_\sigma
\rightarrow \Ql^\times$.
\end{lemme}

\begin{proof}
Soit $P^0$ le groupe des composantes neutres de $P$. D'apr\`es le
th\'eor\`eme de Lang, tout \'el\'ement de $P^0$ est $\sigma$-conjugu\'e \`a
l'unit\'e. Il s'en suit que l'application $P_\sigma \rightarrow \pi_0(P)_\sigma$
de $P_\sigma$ sur le groupe des classes de $\sigma$-conjugaison de
$\pi_0(P)$ est un isomorphisme. Les caract\`eres $\tilde\kappa: P_\sigma
\rightarrow \Ql^\times$ sont donc les $\Ql$-points du $\Ql$-groupe
diagonalisable de type fini
$(\pi_0(P)_\sigma)^*=\Spec(\Ql[\pi_0(P)_\sigma])$. Les caract\`eres
$\kappa:\rmH^1(k,P)\rightarrow \Ql^\times$ forment le groupe
$\pi_0((\pi_0(P)_\sigma)^*)$ des composantes connexes de
$(\pi_0(P)_\sigma)^*$. Tout \'el\'ement $\kappa\in
\pi_0((\pi_0(P)_\sigma)^*)$ peut se relever en un \'el\'ement de torsion
$\tilde \kappa\in (\pi_0(P)_\sigma)^*$.
\end{proof}

Le quotient $X=[M/ P]$ est clairement \'equivalent \`a $[(M/ \Lambda)/ (P/
\Lambda)]$ o\`u $M/ \Lambda$ et $P/ \Lambda$ sont de type fini. En
particulier, l'ensemble des classes d'isomorphisme des objets de $X(k)$
est fini et le groupe des automorphismes de chaque objet est fini. Pour tout
caract\`ere $\kappa:\rmH^1(k,P)\rightarrow \Ql^\times$, on peut donc
former la somme finie
$${\sharp\,} X(k)_\kappa=\sum_{x\in X(k)/\sim} \frac{\langle \cl(x), \kappa \rangle}
{{\sharp\,} \Aut(x)}.
$$
Soit $\tilde\kappa:P \rightarrow \Ql^\times$ un caract\`ere
$\sigma$-invariant d'ordre fini qui repr\'esente $\kappa$. On peut alors
choisir un sous-groupe discret $\Lambda\subset P$ tel que la condition
$(3)$ de \ref{hypothese : finitude} soit v\'erifi\'ee et telle que la restriction
$\kappa$ \`a $\Lambda$ est triviale. Notons encore par $\tilde\kappa$ le
caract\`ere de $P/ \Lambda$ qui s'en d\'eduit. Notons $\rmH^n_c(M/
\Lambda)_{\tilde\kappa}$ le sous-espace propre de $\rmH^n_c(M/
\Lambda)$ de valeur propre $\tilde\kappa$. L'\'enonc\'e suivant est un
corollaire imm\'ediat de \ref{formule des points fixes}.

\begin{proposition}\label{formule des points fixes 2}
On a la formule
$$ {\sharp\,} P^0(k) {\sharp\,} X(k)_{\kappa} =\sum_n (-1)^n \tr(\sigma, \rmH^n_c(M/
\Lambda)_{\tilde\kappa}).
$$
De plus, si $\Lambda'\subset \Lambda$ est un sous-groupe de rang maximal
et $\sigma$-invariant alors on a un isomorphisme canonique
$$\rmH^n_c(M/\Lambda')_{\tilde\kappa} \rightarrow \rmH^n_c(M/\Lambda)_{\tilde\kappa}$$
pour chaque entier $n$.
\end{proposition}

\begin{proof}
Notons $\Lambda$ est un sous-groupe sans torsion de $P$ de sorte que
l'intersection de $\Lambda$ avec la composante neutre $P^0$ est triviale.
Par cons\'equent, l'homomorphisme $P^0 \rightarrow P/ \Lambda$ induit un
isomorphisme de $P^0$ sur la composante neutre de $P/ \Lambda$. En
comparant avec \ref{formule des points fixes}, on obtient la formule qu'on
voulait.
\end{proof}

En pratique, cette proposition est utile pour comparer et pour contr\^oler
la variation de la somme ${\sharp\,} X(k')_\kappa$ quand on prend les points
\`a valeurs dans une extension finie $k'$ de $k$ variable. Soit $m=\deg(k'/
k)$. Pour toute classe d'isomorphisme $x'$ de $X(k')$, on a une classe de
$\sigma^m$-conjugaison dans $P$. Puisque $\kappa:P\rightarrow
\Ql^\times$ est $\sigma$-invariant, il est \`a plus forte raison
$\sigma^m$-invariant de sorte qu'on peut d\'efinir l'accouplement $\langle
\cl(x),\kappa \rangle \in \Ql^\times$. On a alors la formule
$$ {\sharp\,} P^0(k') {\sharp\,} X(k')_\kappa =\sum_n (-1)^n \tr(\sigma^m, \rmH^n_c(M/
\Lambda)_{\tilde\kappa}).
$$

\begin{corollaire}\label{k'->k}
Soit $X=[M/ P]$ et $X'=[M'/ P']$ comme ci-dessus. Soient $\kappa:P
\rightarrow \Ql^\times$ et $\kappa':P' \rightarrow \Ql^\times$ deux
caract\`eres $\sigma$-invariants d'ordre fini. Supposons qu'il existe un
entier $m$ tel que pour toute extension $k'/ k$ de degr\'e $m'\geq m$, on a
$${\sharp\,} P^0(k') {\sharp\,} X(k')_\kappa= {\sharp\,} {P'}^0(k') {\sharp\,}
X'(k')_{\kappa'}.$$ Alors, on a l'\'egalit\'e
$${\sharp\,} P^0(k) {\sharp\,} X(k)_\kappa= {\sharp\,} {P'}^0(k) {\sharp\,} X'(k)_{\kappa'}.$$
\end{corollaire}

Nous allons maintenant \'etudier deux exemples jumeaux qui sont parmi les
fibres de Springer affines les plus simples. Dans ces cas, on peut calculer
directement les nombres ${\sharp\,} X(k)_\kappa$.

\begin{exemple}\label{exemple : chaine infinie}
Soient $A$ un tore d\'eploy\'e de dimension un sur $k$ et $\bbX_*(A)$ son
groupe des cocaract\`eres. Consid\'erons une extension
$$1\rightarrow A\rightarrow P \rightarrow \bbX_*(A) \rightarrow 1.$$
Le faisceau $\underline \RHom(\bbX_*(A), A)$ est concentr\'e en degr\'e $0$
avec
$$\Hom(\bbX_*(A),A)=\GG_m.$$
Puisque $\rmH^1(k,\GG_m)=0$, on peut scinder la suite exacte ci-dessus et
en particulier, il existe un isomorphisme
$$P\simeq A\times \bbX_*(A).$$
Dans les deux cas, on a un isomorphisme $P=\GG_m\times \ZZ$ sur $\bar k$.

Consid\'erons la cha{\^\i}ne infinie de $\bbP^1$ obtenue \`a partir de la
r\'eunion disjoints $\bigsqcup_i \bbP^n_i$ o\`u $\bbP^n_i$ d\'esigne la
$i$-\`eme copie de $\bbP^1$, en recollant le point infini $\infty_i$ de
$\bbP^1_i$ avec le point z\'ero $0_{i+1}$ de $\bbP^1_{i+1}$. Le groupe
$\GG_m\times \ZZ$ agit sur $\bigsqcup_i \bbP^n_i$ de fa{\c c}on compatible
avec le recollement de sorte qu'il agit encore sur la cha{\^\i}ne.

Consid\'erons un $k$-sch\'ema $M$ muni d'une action de $P$ tel qu'apr\`es
le changement de base \`a $\bar k$, $M$ est isomorphe \`a la cha{\^\i}ne infinie de $\bbP^1$
muni de l'action de $\GG_m\times \ZZ$ comme ci-dessus. On a en particulier
une stratification $M=M_1\sqcup M_0$ o\`u $M_1$ est un torseur sous $P$
et o\`u $M_0$ est un torseur sous le groupe discret $\bbX_*(A)$. Le
groupo\"ide $X=[M/ P]$ a essentiellement deux objets $x_1$ et $x_0$ avec
$\Aut(x_1)=A$ et $\Aut(x_0)$ trivial. On a aussi deux classes de
cohomologie
$$\cl_1,\cl_0 \in \rmH^1(k,P)=\rmH^1(k,\bbX_*(A)).$$
Il y a exactement trois possibilit\'es pour ces classes.

\begin{enumerate}
  \item Si $A=\GG_m$ alors $\bbX_*(A)=\ZZ$. Dans ce cas, on a
      l'annulation de $\rmH^1(k,\bbX_*(A))$. On a alors la formule
      $${\sharp\,} X(k)=1+\frac{1}{q-1}=\frac{q}{q-1}
      $$
      qu'il est plus agr\'eable de retenir sous la forme
      $${\sharp\,} A(k) {\sharp\,} X(k)=q.$$
  \item Si $A$ est le tore de dimension un non d\'eploy\'e sur $k$ ayant donc
      $q+1$ points \`a valeurs dans $k$. Dans ce cas $\bbX_*(A)$ est le groupe $\ZZ$
      muni de l'action de $\sigma$ donn\'ee par $m\mapsto -m$. Dans ce
      cas, le groupe $\rmH^1(k,\bbX_*(A))$ a deux \'el\'ements et ses
      \'el\'ements peuvent d\'ecrits explicitement comme suit. Un
      $\bbX_*(A)$-torseur $E$ est un espace principal homog\`ene $E$
      sous $\ZZ$ muni d'une application $\sigma:E\rightarrow E$
      compatible avec l'action de $\sigma$ sur $\ZZ$ donn\'ee ci-dessus.
      Pour tout $e\in E$, la diff\'erence $\sigma(e)-e$ est un entier dont la
      parit\'e est ind\'ependante du choix de $e$. Si $\sigma(e)-e$ est pair,
      $\cl(E)$ est l'\'el\'ement trivial de $\rmH^1(k,\bbX_*(A))$. Si
      $\sigma(e)-e$ est impair, $\cl(E)$ est l'\'el\'ement non trivial de
      $\rmH^1(k,\bbX_*(A))$.

      Supposons maintenant que $\cl_1=0$. Alors, il existe exactement
      une copie $\bbP^1_i$ dans la cha{\^\i}ne infinie stable sous $\sigma$.
      Puisque $A$ est le tore non-d\'eploy\'e, $\sigma$ \'echange
      n\'ecessairment $0_i$ et $\infty_i$. Il s'ensuit que $\cl_0\not=0$.

      Si $\kappa:\rmH^1(k,\bbX_*(A))\rightarrow \Ql^\times$ est le
      caract\`ere non-trivial, on a alors la formule
      $${\sharp\,} X(k)_\kappa=1-\frac{1}{q+1}=\frac{q}{q+1}
      $$
      qu'il est plus agr\'eable de retenir sous la forme
      $${\sharp\,} A(k) {\sharp\,} X(k)_\kappa=q.$$

  \item Supposons toujours que $A$ est le tore non-d\'eploy\'e mais
      consid\'erons le cas o\`u $\cl_1\not =0$. Il existe alors deux copies
      $\bbP^1_i$ et $\bbP^1_{i+1}$ \'echang\'ees par $\sigma$ de sorte
      que dans la cha{\^\i}ne infinie $M$, le point commun $\infty_i=
      0_{i+1}$ est $\sigma$-invariant. Il en r\'esulte que $\cl_1=0$. On a
      alors la formule
      $${\sharp\,} A(k) {\sharp\,} X(k)_\kappa=-q$$
      o\`u $\kappa:\rmH^1(k,\bbX_*(A))\rightarrow \Ql^\times$ est le
      caract\`ere non-trivial.
\end{enumerate}

En supposant que $M_1$ un $k$-point, on a $\cl_1=0$  de sorte que la
troisi\`eme possibilit\'e est exclue. Avec cette hypoth\`ese, on constate
que la formule
$${\sharp\,} A(k) {\sharp\,} X(k)_\kappa=q$$
est alors valide.

\end{exemple}

\subsection{Comptage dans une fibre de Springer affine}\label{subsection : comptage Springer}

Dans ce paragraphe, nous \'etudions le comptage de points dans une fibre de
Springer en suivant essentiellement \cite[\S 15]{GKM} affine mais avec le
langage d\'evelopp\'e dans le paragraphe pr\'ec\'edent. Comme dans le
paragraphe pr\'ec\'edent, la notation $x\in X$ signifiera $x\in X(\bar k)$.
Lorsqu'il s'agit de coefficients autres que $\bar k$, on le pr\'ecisera.

Soit $v\in |X|$ un point ferm\'e de $X$. Soient $F_v$ la compl\'etion du
corps des fractions rationnelles $F$ par la topologie $v$-adique et
$\calO_v$ l'anneau des entiers de $F_v$, $k_v$ le corps r\'esiduel. On note
$X_v=\Spec(\calO_v)$ et $X_v^\bullet=\Spec(F_v)$. Soit $\bar F_v$ la
compl\'etion $v$-adique de $\bar F=F\otimes_k\bar k$. Soit $\bar\calO_v$
l'anneau des entiers de $\bar F_v$.

\begin{numero}
Soit $a\in\frakc^\heartsuit(\calO_v)$ un $X_v$-point de $\frakc$ dont la
fibre g\'en\'erique est semi-simple et r\'eguli\`ere. La fibre de Springer
affine r\'eduite $\calM_v^{\rm red}(a)$ est un $k$-sch\'ema localement de
type fini dont l'ensemble des $\bar k$-points est
$$\calM_v^{\rm red}(a)=\{g\in G(\bar F_v)/G(\bar \calO_v)| \ad(g)^{-1}\gamma_0 \in
\frakg(\bar \calO_v)\}
$$
o\`u $\gamma_0=\epsilon(a)$ est la section de Kostant au point $a$.
Puisque que les nilpotents ne jouent aucun r\^ole dans la discussion qui va
suivre, on va \'ecrire simplement $\calM_v(a)$ \`a la place de
$\calM_v^{\rm red}(a)$.
\end{numero}

\begin{numero}
Soit $J_a=a^* J$ l'image inverse du centralisateur r\'egulier. Soit $J'_a$ un
$X_v$-sch\'ema en groupes lisse de fibres connexes muni d'un
homomorphisme $J'_a\rightarrow J_a$ qui induit un isomorphisme sur la
fibre g\'en\'erique. En particulier, $J'_a$ peut \^etre le sch\'ema en groupes
$J_a^0$ des composantes neutres de $J_a$ mais il sera plus souple pour les
applications de consid\'erer $J'_a$ g\'en\'eral.
\end{numero}

\begin{numero}
Consid\'erons le $k$-sch\'ema en groupes localement de type fini
$\calP_v^{\rm red}(J'_a)$ dont l'ensemble des $\bar k$-points est
$$\calP_v^{\rm red}(J'_a)= J_a(\bar F_v)/J'_a(\bar \calO_v).$$
De nouveau, on va \'ecrire simplement $\calP_v(J'_a)$ pour $\calP_v^{\rm
red}(J'_a)$ car les nilpotents ne jouent pas de r\^ole dans la discussion qui
suit. L'homomorphisme $J'_a\rightarrow J_a$ induit un homomorphisme
$$\calP_v(J'_a)\rightarrow \calP_v(J_a)$$
qui induit une action de $\calP_v(J'_a)$ sur $\calM_v(a)$.  D'apr\`es
\ref{quotient projectif}, cette action v\'erifie l'hypoth\`ese \ref{hypothese :
finitude} de sorte qu'on peut envisager le nombre de $k$-points du quotient
$[\calM_v(a)/\calP_v(J'_a)]$ ainsi que la variante avec une
$\kappa$-pond\'eration. D'apr\`es \ref{formule des points fixes 2}, on sait
que ce sont des nombres finis qui ont une interpr\'etation cohomologique.
Dans ce paragraphe, nous nous int\'eressons \`a la question d'exprimer ces
nombres en termes d'int\'egrales orbitales stables et de
$\kappa$-int\'egrales orbitales.
\end{numero}

Commen{\c c}ons par comparer les $\kappa$ qui apparaissent dans le
probl\`eme de comptage \ref{subsection : comptage} et ceux qui
apparaissent dans la d\'efinition des $\kappa$-int\'egrales orbitales
\ref{subsection : integrales orbitales}.

\begin{lemme} Supposons que la fibre sp\'eciale  de $J'_a$ est connexe.
Il existe un isomorphisme canonique
$$\rmH^1(F_v,J_a)=\rmH^1(k, \calP_v(J'_a)).$$
\end{lemme}

\begin{proof}
D'apr\`es un th\'eor\`eme de Steinberg, $\rmH^1(\bar F_v, J_a)=0$. Il
s'ensuit que
$$
\rmH^1(F_v,J_a)=\rmH^1(\Gal(\bar k/ k), J_a(\bar F_v)).
$$
D'apr\`es un th\'eor\`eme de Lang, on a l'annulation
$\rmH^1(k,J'_a(\bar\calO_v))=0$ car $J'_a$ est un sch\'ema en groupes
lisse de fibres connexes. Le lemme s'en d\'eduit.
\end{proof}

\begin{proposition}\label{comptage avec P_v(J'_a)}
Supposons la fibre sp\'eciale  de $J'_a$ connexe. Soit
$\kappa:\rmH^1(F_v,J_a) \rightarrow \Ql^\times$ un caract\`ere et
$\kappa:\rmH^1(k, \calP_v(J'_a)) \rightarrow \Ql^\times$ le caract\`ere qui
s'en d\'eduit. Le nombre de $k$-points avec la $\kappa$-pond\'eration de
$[\calM_v(a)/\calP_v(J'_a)]$ peut alors s'exprimer comme suit
$${\sharp\,} [\calM_v(a)/\calP_v(J'_a)](k)_\kappa=
{\rm vol}(J'_a(\calO_v),{\rm d}t_v){\bf O}^\kappa_a(1_{\frakg_v},{\rm d}t_v).
$$
o\`u $1_{\frakg_v}$ est la fonction caract\'eristique de $\frakg(\calO_v)$ et
o\`u ${\rm d}t_v$ est n'importe quelle mesure de Haar de $J_a(F_v)$.

\end{proposition}

\begin{proof}
Rappelons la description de la cat\'egorie
$$[\calM_v(a)/\calP_v(J'_a)](k).$$
Un objet de cette cat\'egorie consiste en un couple $x=(m,p)$ form\'e d'un
\'el\'ement $m\in \calM_v(a)$ et d'un \'el\'ement $p\in \calP_v(J'_a)$ tels
qu'on ait l'\'egalit\'e $p\sigma(m)=m$. Une fl\`eche de $(m,p)$ dans
$(m',p')$ dans cette cat\'egorie consiste en un \'el\'ement $h\in
\calP_v(J'_a)$ tel que $m'=hm$ et $p'=hp \sigma(h)^{-1}$.

Notons $\calP_v(J'_a)_\sigma$ le groupe des classes de
$\sigma$-conjugaison de $\calP_v(J'_a)$ et $\rmH^1(k,\calP_v(J'_a))$ le
sous-groupe des cocycles continus. Pour chaque objet $x=(m,p)$ comme
ci-dessus, on d\'efinit $\cl(x)\in \calP_v(J'_a)_\sigma$ la classe de
$\sigma$-conjugaison de $p$ qui appartient en fait au sous-groupe
$$\cl(x)\in \rmH^1(k,\calP_v(J'_a))=\rmH^1(F_v,J_a).$$
Soit $\gamma_0$ l'image de $a$ dans $\frakg$ par la section de Kostant. On
a un isomorphisme canonique $J_a=I_{\gamma_0}$ \cf \ref{J a} si bien que
$\cl(x)$ peut aussi \^etre vu comme un \'el\'ement de
$\rmH^1(F_v,I_{\gamma_0})$.

\begin{lemme}
Pour tout objet $x$ de $[\calM_v(a)/\calP_v(J'_a)](k)$, la classe $\cl(x)$
appartient au noyau de l'homomorphisme
$$\rmH^1(F_v,I_{\gamma_0})\rightarrow \rmH^1(F_v,G).$$
\end{lemme}

\begin{proof}
Soit $x=(p,m)$ comme ci-dessus. Soient $g\in G(\bar F_v)$ un repr\'esentant
de $m\in G(\bar F_v)/ G(\bar\calO_v)$ et $j\in I_{\gamma_0}(\bar F_v)$ un
repr\'esentant de $p$. L'\'egalit\'e $p\sigma(m)=m$ implique que
$$g^{-1} j \sigma(g) \in G(\bar\calO_v)$$
si bien que cet \'el\'ement est $\sigma$-conjugu\'e \`a l'\'el\'ement neutre
dans $G(\bar F_v)$. Ainsi $\cl(x)\in \rmH^1(F_v,I_{\gamma_0})$ a une image
triviale dans $\rmH^1(F_v,G)$.
\end{proof}

Continuons la d\'emonstration de \ref{comptage avec P_v(J'_a)}. Pour tout
\'el\'ement $\xi$ dans le noyau de $\rmH^1(F_v,I_{\gamma_0})\rightarrow
\rmH^1(F_v,G)$, consid\'e\-rons la sous-cat\'egorie
$$[\calM_v(a)/\calP_v(J'_a)]_\xi(k)$$
des objets de $[\calM_v(a)/\calP_v(J'_a)](k)$ tels que $\cl(x)=\xi$. Nous
nous proposons d'\'ecrire le nombre de classes d'isomorphisme de
$[\calM_v(a)/\calP_v(J'_a)]_\xi(k)$ comme une int\'egrale orbitale.

Fixons $j_\xi \in I_{\gamma_0}(\bar F_v)$ dans la classe de
$\sigma$-conjugaison $\xi$. Soit $(m,p)$ un objet de
$[\calM_v(a)/\calP_v(J'_a)]_\xi(k)$. Puisque $p$ est $\sigma$-conjugu\'e \`a
$j_\xi$ dans $\calP_v(J'_a)$, il existe $h\in \calP_v(J'_a)$ tel que
$j_\xi=h^{-1}p \sigma(h)$. Alors $(m,p)$ est isomorphe \`a
$(h^{-1}m,j_\xi)$. Par cons\'equent, la cat\'egorie
$[\calM_v(a)/\calP_v(J'_a)]_\xi(k)$ est \'equivalente \`a la sous-cat\'egorie
pleine form\'ee des objets de la forme $(m,j_\xi)$. Une fl\`eche $(m,j_\xi)
\rightarrow (m',j_\xi)$ est un \'el\'ement $h\in \calP_v(J'_a)(k)$ tel que
$hm=m'$. Puisque $J'_a$ a des fibres connexes, on a
$$\calP_v(J'_a)(k)=J_a(F_v)/J'_a(\calO_v).$$

Soit $(m,j_\xi)$ comme ci-dessus. Choisissons un repr\'esentant $g\in G(\bar
F_v)$ de $m$. On a alors $g^{-1} j_\xi \sigma(g) \in G(\bar\calO_v)$. Comme
tous les \'el\'ements de $G(\bar\calO_v)$ sont $\sigma$-conjugu\'es \`a
l'\'el\'ement neutre, on peut choisir $g$ avec $m=g G(\bar\calO_v)$ tel que
$$g^{-1} j_\xi \sigma(g)=1.$$
Soient $g,g' \in G(\bar F_v)$ deux \'el\'ements  v\'erifiant l'\'equation
ci-dessus et repr\'e\-sentant la m\^eme classe $m\in G(\bar F_v)/
G(\bar\calO_v)$. Alors $g'=g k$ avec $k\in G(\calO_v)$ si bien que l'image
de $g$ dans $G(\bar F_v)/ G(\calO_v)$ est bien d\'etermin\'ee par
$(m,j_\xi)$.

Ainsi la cat\'egorie $[\calM_v(a)/\calP_v(J'_a)]_\xi(k)$ est \'equivalente \`a
la cat\'egorie $O_\xi$ dont les objets sont les \'el\'ements $g\in G(\bar F_v)/
G(\calO_v)$ v\'erifiant deux \'equations
\begin{enumerate}
  \item $g^{-1} j_\xi \sigma(g)=1$
  \item $\ad(g)^{-1} \gamma_0 \in \frakg(\bar\calO_v)$
\end{enumerate}
et dont les fl\`eches $g \rightarrow g_1$ sont les \'el\'ements $h\in
J_a(F_v)/J'_a(\calO_v)$ tel que $g=hg_1$. Ici, pour faire agir $h$ \`a
gauche, on a utilis\'e l'isomorphisme canonique $J_a=I_{\gamma_0}$.

Puisque l'image de $\xi$ dans $\rmH^1(F_v,G)$ est triviale, il existe
$g_\xi\in G(\bar F_v)$ tel que $g_\xi^{-1}j_\xi \sigma(g_\xi)=1$. Fixons un
tel $G_\xi$ et posons
$$\gamma_\xi=\ad(g_\xi)^{-1} \gamma_0.$$
L'\'equation $g_\xi^{-1}j_\xi \sigma(g_\xi)=1$ implique que $\gamma_\xi \in
G(F_v)$. La classe de $G(F_v)$-conjugaison de $\gamma_\xi$ ne d\'epend
pas des choix de $j_\xi$ et $g_\xi$ mais seulement de la classe $\xi\in
\rmH^1(F_v,I_{\gamma_0})$.

En posant $g'=g_\xi^{-1} g$, la cat\'egorie $O_\xi$ peut \^etre d\'ecrite
comme suit. Ses objets sont les \'el\'ements $g'\in G(F_v)/ G(\calO_v)$
v\'erifiant
$$\ad(g')^{-1}(\gamma_\xi)\in \frakg(\calO_v).$$
Ses fl\`eches $g'\rightarrow g'_1$ sont les \'el\'ements $h\in
J_a(F_v)/J'_a(\calO_v)$ tel que $g'=hg'_1$. Ici, pour faire agir $h$ \`a
gauche, on a utilis\'e l'isomorphisme canonique $J_a=I_{\gamma_\xi}$ o\`u
$I_{\gamma_\xi}$ est le centralisateur de $\gamma_\xi$.

L'ensemble des classes d'isomorphisme de $O_\xi$ est donc l'ensemble des
double-classes
$$g'\in I_{\gamma_\xi}(F_v)\backslash G(F_v)/G(\calO_v)$$
telles que $\ad(g')^{-1} \gamma_\xi\in\frakg(\calO_v)$. Le groupe des
automorphismes de $g'$ est le groupe
$$(I_{\gamma_\xi}(F_v)\cap g'G(\calO_v){g'}^{-1}) / J'_a(\calO_v)$$
dont le cardinal peut \^etre exprim\'e en termes de volumes comme suit
$$
\frac{{\rm vol}(I_{\gamma_\xi}(F_v)\cap g'G(\calO_v){g'}^{-1},{\rm d}t_v)}
{{\rm vol}(J'_a(\calO_v),{\rm d}t_v)}.
$$
On a donc
$${\sharp\,} O_\xi= \sum \frac{{\rm vol}(J'_a(\calO_v),{\rm d}t_v)} {{\rm
vol}(J_a(F_v)\cap g'G(\calO_v){g'}^{-1},{\rm d}t_v) }$$ la sommation \'etant
\'etendue sur l'ensemble des doubles classes
$$g'\in I_{\gamma_\xi}(F_v)\backslash G(F_v)/G(\calO_v)$$
telles que $\ad(g')^{-1} \gamma_\xi\in\frakg(\calO_v)$. On obtient donc la
formule
$${\sharp\,} [\calM_v(a)/\calP_v(J_a)]_\xi(k)= {\sharp\,} O_\xi=
{\rm vol}(J'_a(\calO_v),{\rm d}t){\bf O}_{\gamma_\xi}(1_{\frakg_v},{\rm
d}t_v).
$$
En sommant sur le noyau de $\rmH^1(F,I_{\gamma_0}) \rightarrow
\rmH^1(F,G)$, on obtient la formule
$${\sharp\,} [\calM_v(a)/\calP_v(J'_a)](k)_\kappa=
{\rm vol}(J'_a(\calO_v),{\rm d}t_v){\bf O}^\kappa_a(1_{\frakg_v},{\rm d}t_v)
$$
qu'on voulait.
\end{proof}

On voudra \'eventuellement le m\^eme type d'\'enonc\'e pour les sch\'emas
en groupes $J'_a$ n'ayant pas n\'ecessairement une fibre sp\'eciale connexe
et en particulier pour $J_a$ lui-m\^eme. Le comptage direct du
nombre de $k$-points du quotient $[\calM_v(a)/\calP_v(J_a)]$ s'av\`ere
tr\`es p\'enible. Il est plus \'economique de passer par la
comparaison avec le cas connexe. Nous allons \'enoncer le r\'esultat
seulement dans le cas de $J_a$ bien qu'il est valide en g\'en\'eral.

Soit $\kappa:\rmH^1(k,\calP_v(J_a))\rightarrow \Ql^\times$. En utilisant
l'homomorphisme
$$\rmH^1(k,P_v(J_a^0))\rightarrow
\rmH^1(k,\calP_v(J_a)),
$$
on obtient un caract\`ere de $\rmH^1(k,\calP_v(J_a^0))$ et par cons\'equent un
caract\`ere de $\rmH^1(F_v,J_a)$ que nous allons noter tous $\kappa$.

\begin{proposition}\label{comptage avec P_v(J_a)}
Consid\'erons $\kappa$ un caract\`ere de $\rmH^1(k,\calP_v(J_a))$ et
notons $\kappa:\rmH^1(F_v,J_a) \rightarrow \Ql^\times$ le caract\`ere de
$\rmH^1(F_v,J_a)$ qui s'en d\'eduit. Alors, le nombre de $k$-points avec
la $\kappa$-pond\'eration de $[\calM_v(a)/\calP_v(J_a)]$ peut s'exprimer
comme suit
$$
{\sharp\,} [\calM_v(a)/\calP_v(J_a)](k)_\kappa={\rm vol}(J_a^0(\calO_v),{\rm d}t_v){\bf O}^\kappa_a(1_{\frakg_v},{\rm d}t_v).
$$
o\`u $1_{\frakg_v}$ est la fonction caract\'eristique de $\frakg(\calO_v)$ et
o\`u ${\rm d}t_v$ est n'importe quelle mesure de Haar de $J_a(F_v)$.

De plus, si $\kappa:\rmH^1(F_v,J_a) \rightarrow \Ql^\times$ est un
caract\`ere qui ne provient pas d'un caract\`ere de
$\rmH^1(k,\calP_v(J_a))$ alors la $\kappa$-int\'egrale orbitale ${\bf
O}^\kappa_a(1_{\frakg_v},{\rm d}t_v)$ est nulle.
\end{proposition}

\begin{proof}
Il s'agit d'une comparaison entre $\calP_v(J_a)$ et le cas connu
$\calP_v(J_a^0)$. Il suffit de d\'emontrer l'\'egalit\'e
$$
{\sharp\,} [\calM_v(a)/\calP_v(J_a)](k)_\kappa= {\sharp\,}
[\calM_v(a)/\calP_v(J_a^0)](k)_\kappa
$$
dans le premier cas et l'annulation de
$[\calM_v(a)/\calP_v(J_a^0)](k)_\kappa$ dans le second cas. La
d\'emonstration sera fond\'ee sur le m\^eme principe que \ref{exemple :
classifiant fini} bien que les d\'etails sont plus compliqu\'es dans le cas
pr\'esent.

On a une suite exacte
$$
1 \rightarrow \pi_0(J_{a,v}) \rightarrow \calP_v(J_a^0) \rightarrow \calP_v(J_a) \rightarrow 1
$$
o\`u $\pi_0(J_{a,v})$ est le groupe des composantes connexes de la fibre
de $J_a$ en $v$. On en d\'eduit une suite exacte longue
\begin{eqnarray}
1 \rightarrow \pi_0(J_{a,v})^\sigma \rightarrow \calP_v(J_a^0)^\sigma \rightarrow
\calP_v(J_a)^\sigma \rightarrow \\
\pi_0(J_{a,v})_\sigma \rightarrow \calP_v(J_a^0)_\sigma \rightarrow
\calP_v(J_a)_\sigma \rightarrow 1
\end{eqnarray}
o\`u l'exposant $\sigma$ d\'esigne un groupe des $\sigma$-invariants et
l'indice $\sigma$ d\'esigne un groupe des $\sigma$-coinvariants.

Consid\'erons le foncteur
$$\pi:[\calM_v(J_a)/\calP_v(J_a^0)](k) \rightarrow  [\calM_v(J_a)/\calP_v(J_a)](k).$$
Il envoie un objet $x'=(m,p')$ avec $m\in\calM_v(a)$ et $p'\in
\calP_v(J_a^0)$ tel que $p' \sigma(m)=m$ sur l'objet $x=(m,p)$ o\`u $p$
est l'image de $p'$ dans $\calP_v(J_a)$. Choisissons un ensemble de
repr\'esentants $\{x_\psi\mid \psi\in \Psi\}$ des classes d'isomorphisme de
$[\calM_v(J_a)/\calP_v(J_a)](k)$.

Puisque l'homomorphisme $\calP_v(J_a^0)\rightarrow \calP_v(J_a)$ est
surjectif, cette sous-cat\'egorie de $[\calM_v(J_a)/\calP_v(J_a^0)](k)$ est
\'equivalente \`a $[\calM_v(J_a)/\calP_v(J_a^0)](k)$ si bien que pour
calculer les nombres ${\sharp\,} [\calM_v(J_a)/\calP_v(J_a^0)](k)_\kappa$,
il est loisible de se restreindre \`a cette sous-cat\'egorie.

Pour tout $\psi\in \Psi$, consid\'erons la sous-cat\'egorie pleine
$$[\calM_v(J_a)/\calP_v(J_a^0)](k)_{x_\psi}$$
de $[\calM_v(J_a)/\calP_v(J_a^0)](k)$ form\'ee des objets de la forme
$x'_\psi=(m_\psi,p'_\psi)$ avec $p'_\psi\mapsto p_\psi$. Si $\psi\not=\psi'$,
deux objets $x'_\psi\in [\calM_v(J_a)/\calP_v(J_a^0)](k)_{x_\psi}$  et
$x'_{\psi'}\in [\calM_v(J_a)/\calP_v(J_a^0)](k)_{x_{\psi'}}$ ne sont pas
isomorphes. De plus pour tout $x'\in [\calM_v(J_a)/\calP_v(J_a^0)](k)$ il
existe $\psi\in \Psi$ et $x'_{\psi}\in [\calM_v(J_a)/\calP_v(J_a^0)]
(k)_{x_\psi}$ tels que $x'$ et $x'_\psi$ sont isomorphes. Ainsi, la r\'eunion
disjointe des cat\'egories
$$\bigsqcup_\psi [\calM_v(J_a)/\calP_v(J_a^0)](k)_{x_\psi}$$
forment une sous-cat\'egorie pleine de $[\calM_v(J_a)/\calP_v(J_a^0)](k)$
qui lui est \'equi\-valente. Par cons\'equent, pour compter les objets de
$[\calM_v(J_a)/\calP_v(J_a^0)](k)$, on peut compter dans chaque
$[\calM_v(J_a)/\calP_v(J_a^0)](k)_{x_\psi}$ et puis faire la somme sur les
$\psi\in \Psi$.

Fixons un $x_\psi$ et notons le simplement $x=(m,p)$. La cat\'egorie
$$[\calM_v(J_a)/\calP_v(J_a^0)](k)_{x}$$
n'est pas difficile \`a d\'ecrire. En fixant un objet $x_1=(m,p_1)$ avec $p_1
\mapsto p$, on identifie l'ensemble des objets de cette cat\'egorie avec le
groupe $\pi_0(J_{a,v})$ qui est le noyau de $\calP_v(J_a^0)\rightarrow
\calP_v(J_a)$. Une fl\`eche dans cette cat\'egorie est donn\'ee par un
\'el\'ement du groupe
$$H_1=\{h_1\in  \calP_v(J_a^0) \mid hm=m \ {\rm et}\ h_1 \sigma(h_1)^{-1} \in \pi_0(J_{a,v})\}.$$
La cat\'egorie $[\calM_v(J_a)/\calP_v(J_a^0)]_x$ est \'equivalente \`a la
cat\'egorie quotient de l'ensemble $\pi_0(J_{a,v})$ par l'action de $H_1$
avec $H_1$ agissant \`a travers l'homomorphisme $\alpha:H_1\rightarrow
\pi_0(J_{a,v})$ d\'efini par $\alpha(h_1)=h_1 \sigma(h_1)^{-1}$. L'ensemble
des classes d'isomorphisme de cette cat\'egorie s'identifie avec le conoyau
${\rm cok}(\alpha)$ de $\alpha$ et le groupe des automorphismes de chaque
objet s'identifie avec le noyau ${\rm ker}(\alpha)$. Ceci montre en
particulier que $H_1$ est un groupe fini.

Soit $\kappa$ un caract\`ere de $\rmH^1(k,\calP_v(J_a^0))$. Ce groupe
s'identifie canoniquement \`a la partie torsion de $\calP_v(J_a^0)_\sigma$.
Consid\'erons la restriction de $\kappa$ \`a $\pi_0(J_{a,v})_\sigma$ et
notons aussi $\kappa$ le caract\`ere $\pi_0(J_{a,v})\rightarrow \Ql^\times$
qui s'en d\'eduit. Ce caract\`ere est certainement triviale sur l'image de
$\alpha:H_1 \rightarrow \pi_0(J_{a,v})_\sigma$ si bien qu'il d\'efinit un
cocaract\`ere sur le conoyau $\kappa: {\rm cok}(\alpha)\rightarrow
\Ql^\times$. La somme
$$\sum_{x'} \frac{\langle \cl(x'),\kappa \rangle}
{{\sharp\,} \Aut(x')}=0
$$
sur l'ensemble des classes d'isomorphisme de la sous-cat\'egorie pleine
$[\calM_v(J_a)/\calP_v(J_a^0)]_x$ est alors \'egale \`a
$$
\sum_{z\in {\rm cok}(\alpha)}\frac{\langle z,\kappa \rangle}{{\sharp\,} {\rm
ker}(\alpha)}\langle \cl(x_1),\kappa \rangle.
$$
Si la restriction de $\kappa$ \`a $\pi_0(J_{a,v})$ est non triviale alors cette
somme est nulle ce qui entra{\^\i}ne l'annulation
$$[\calM_v(a)/\calP_v(J_a^0)](k)_\kappa=0.$$

Supposons maintenant que la restriction de $\kappa$ \`a $\pi_0(J_{a,v})$
est triviale. Dans ce cas, $\kappa$ se factorise par $\rmH^1(k,\calP_v(J_a))$
et la somme ci-dessus est \'egale \`a
$$
\frac{{\sharp\,} {\rm cok}(\alpha)}{{\sharp\,} {\rm ker}(\alpha)} \langle \cl(x),\kappa
\rangle.
$$
Il reste \`a calculer ${{\sharp\,} {\rm cok}(\alpha) / {\sharp\,} {\rm
ker}(\alpha)}$. La suite exacte
$$
1\rightarrow {\rm ker}(\alpha) \rightarrow H_1 \rightarrow
\pi_0(J_{a,v}) \rightarrow {\rm coker}(\alpha) \rightarrow  1
$$
implique l'\'egalit\'e
$$\frac{{\sharp\,} {\rm cok}(\alpha)}{{\sharp\,} {\rm ker}(\alpha)} =
\frac{{\sharp\,}\pi_0(J_{a,v})}{{\sharp\,} H_1}.$$

Soient $h_1\in H_1$ et $h$ son image dans $\calP_v(J_a)$. On a alors $h
\sigma(h)^{-1}=1$ de sorte que $h\in \calP_v(J_a)^\sigma$. Soit ${\rm
stab}(m)$ le sous-groupe des \'el\'ements de $\calP_v(J_a)$ qui stabilisent
$m$. On a alors la suite exacte
$$1\rightarrow \pi_0(J_{a,v})\rightarrow H_1 \rightarrow \calP_v(J_a)^\sigma\cap {\rm
stab}(m)\rightarrow 1
$$
o\`u $\calP_v(J_a)^\sigma\cap {\rm stab}(m)$ est exactement le groupe des
automorphismes de ${\rm Aur}(x)$ dans la cat\'egorie
$[\calM_v(J_a)/\calP_v(J_a)]$. On a donc l'\'egalit\'e
$$\frac{{\sharp\,}\pi_0(J_{a,v})}{{\sharp\,} H_1}=\frac{1}{{\sharp\,}\Aut(x)}$$
qui implique l'\'egalit\'e
$${\sharp\,} [\calM_v(a)/\calP_v(J_a)](k)_\kappa= {\sharp\,} [\calM_v(a)/\calP_v(J_a^0)](k)_\kappa$$
qu'on voulait.
\end{proof}

En appliquant les r\'esultats g\'en\'eraux de \ref{subsection : comptage} \`a
la situation pr\'esente, on obtient une interpr\'etation cohomologique des
int\'egrales orbitales stables et des $\kappa$-int\'egrales orbitales. Prenons
un sous-groupe sans torsion $\sigma$-stable $\Lambda$ de $\calP_v(J_a^0)$
comme dans \ref{hypothese : finitude}. Ce sous-groupe existe en vertu de
la discussions qui suit \ref{hypothese : finitude} et du r\'esultat de
Kazhdan-Lusztig \cf \ref{quotient projectif}. On obtient le corollaire suivant
de \ref{formule des points fixes 2} et \ref{comptage avec P_v(J'_a)}.

\begin{corollaire}
Soit $\kappa$ un caract\`ere de $\rmH^1(F_v,J_a)$. Soit $J_a^{\flat,0}$ la
composante neutre du mod\`ele de N\'eron de $J_a$. Pour tout $\Lambda$
comme ci-dessus, on a l'\'egalit\'e
$$
\sum_n (-1)^n \tr(\sigma, \rmH^n([\calM_v(a)/\Lambda])_\kappa)
={\rm vol}(J_a^{\flat,0}(\calO_v) ,{\rm d}t_v) {\bf O}^\kappa_a(1_{\frakg_v},{\rm d}t_v).
$$
\end{corollaire}

\begin{proof}
En mettant ensemble  \ref{formule des points fixes 2} et \ref{comptage
avec P_v(J'_a)} on obtient la formule
\begin{eqnarray*}
&\sum_n (-1)^n \tr(\sigma, \rmH^n([\calM_v(a)/\Lambda])_\kappa)\\
=&({\sharp\,}\calP_v^0(J_a^0)(k)) {\rm vol}(J_a^0(\calO_v) ,{\rm d}t_v) {\bf
O}^\kappa_a(1_{\frakg_v},{\rm d}t_v).
\end{eqnarray*}
o\`u ${\sharp\,} \calP_v^0(J_a^0)(k)$ est le nombre de $k$-points de la
composante neutre de $\calP_v(J_a^0)$. Si $J_a^{\flat,0}$ est la
composante neutre du mod\`ele de N\'eron, on a un homomorphisme
$J_a^0 \rightarrow J_a^{\flat,0}$ qui induit une suite exacte
$$1 \rightarrow J_a^{\flat,0}(\bar\calO_v)/J_a^0(\bar\calO_v) \rightarrow \calP_v(J_a^0)
\rightarrow \calP_v(J_a^{\flat,0}) \rightarrow 1
$$
qui permet d'identifer la composante neutre $\calP_v^0(J_a^0)$ de
$\calP_v(J_a^0)$ avec le $k$-groupe affine connexe dont les $\bar
k$-points sont $ J_a^{\flat,0}(\bar\calO_v)/J_a^0(\bar\calO_v)$. Puisque
$J_a^0$ et $J_a^{\flat,0}$ sont des sch\'emas en groupes de fibres
connexes, on en d\'eduit une suite exacte
$$1 \rightarrow J_a^0(\calO_v) \rightarrow J_a^{\flat,0}(\calO_v) \rightarrow
\calP_v^0(J_a^0)(k) \rightarrow 1$$ d'o\`u l'\'egalit\'e
$$({\sharp\,} \calP_v^0(J_a^0)(k))
{\rm vol}(J_a^0(\calO_v) ,{\rm d}t_v)  = {\rm vol}(J_a^{\flat,0}(\calO_v)
,{\rm d}t_v) $$ pour n'importe quelle mesure de Haar ${\rm d}t_v$ sur
$J(F_v)$. Le corollaire s'en d\'eduit.
\end{proof}

\subsection{Un cas tr\`es simple}
\label{subsection : fibre simple}

Soit $v$ un point ferm\'e de $X$. Notons $k_v$ le corps r\'esiduel de $v$.
Choisissons un point g\'eom\'etrique $\bar v$ au-dessus de $v$. Notons
$X_v=\Spec(\calO_v)$ la compl\'etion de $X$ en $v$ et
$X_v^\bullet=\Spec(F_v)$ o\`u $F_v$ est le corps des fractions de
$\calO_v$. Choisissons un uniformisant $\epsilon_v$. Dans ce paragraphe,
nous notons $G$ la restriction de $G$ \`a $X_v$ et $G^\bullet$ la
restriction de $G$ \`a $X_v^\bullet$.

\begin{numero}
Soit $a\in \frakc(\calO_v)$ dont l'image dans $\frakc(F_v)$ est r\'eguli\`ere
et semi-simple. Soit $d_v(a)={\rm deg}_v (a^* \discrim_G)\in \NN$. Nous
allons supposer que
$$d_v(a) =2\ {\rm et}\ c_v(a)=0.$$
D'apr\`es la formule de Bezrukavnikov \cf \ref{dimension Bezru}, la
dimension de la fibre de Springer affine $\calM_v(a)$ est alors \'egale \`a
un. On va montrer qu'au-dessus de $\bar k$, cette fibre de Springer affine
est une r\'eunion disjointe des copies de la cha{\^\i}ne infinie des droites
projectives.
\end{numero}

\begin{numero}
Soit $\gamma_0=\epsilon(a)\in \frakg(F_v)$  la section de Kostant
appliqu\'ee \`a $a$. Son centralisateur $T^\bullet=I_{\gamma_0}$ est un
sous-tore maximal de $G^\bullet$. L'hypoth\`ese $c_v(a)=0$ implique que
$T^\bullet$ est un sous-tore non ramifi\'e c'est-\`a-dire qu'il s'\'etend en un
sous-tore maximal $T$ de $G$. Soit $\Phi$ l'ensemble des  racines de $\bar
T=T\otimes_{\calO_v} {\bar\calO_{\bar v}}$. On a alors la fonction de
valuation radicielles de Goresky, Kottwitz et MacPherson \cf
\cite{GKM-codim}
$$r_a:\Phi\rightarrow \NN$$
d\'efini par $r_a(\alpha)={\rm val}(\alpha(\gamma_0))$ o\`u on a pris la
valuation sur $\bar F_{\bar v}$ qui \'etend la valuation sur $F_v$. On a alors
$$d_v(a)=\sum_{\alpha\in\Phi} r_\alpha(\gamma_0).$$
Il existe donc une unique paire de racines $\pm \alpha$ tel que $r_{\pm
\alpha}(\gamma_0)=1$ et $r_{\alpha'}(\gamma_0)=0$ pour toute racine
$\alpha'\notin \{\pm \alpha\}$.
\end{numero}

\begin{numero}
Le groupe de Galois $\Gal(\bar k/ k_v)$ agit sur $\Phi$ en laissant invariant
la fonction $r_a(\alpha)$ de sorte qu'il laisse stable le couple $\{\pm
\alpha\}$. Soit $\bar G_{\pm \alpha}$ le sous-groupe de $\bar
G=G\otimes_{\calO_v}{\bar \calO_{\bar v}}$ engendr\'e par $\bar T$ et par
les sous-groupes radiciels $U_\alpha$ et $U_{-\alpha}$. L'action du groupe
de Galois de $\Gal(\bar k/ k_v)$ sur $\bar T$ laissant stable $\{\pm\alpha\}$
permet de descendre $\bar G_{\pm \alpha}$ en un sous-sch\'ema en groupes
r\'eductifs $G_{\pm \alpha}$ de $G_{\calO_v}$. Le centre $Z_{\pm \alpha}$
de $G_{\pm \alpha}$ qui s'identifie au noyau de
$\alpha:T\rightarrow\GG_m$ est aussi d\'efini sur $\calO_v$. Soit $A_{\pm
\alpha}=T/Z_{\pm \alpha}$ ; c'est un tore de dimension un sur $X_v$.
\end{numero}

\begin{proposition}
On a un homomorphisme canonique $J_a\rightarrow T$ dont l'image au
niveau des $\bar\calO_{\bar v}$-points est le noyau de l'homomorphisme
compos\'e de la r\'eduction modulo l'id\'eal maximal $T(\bar\calO_{\bar
v})\rightarrow T(\bar k)$ et la racine $\alpha:T(\bar k)\rightarrow
\GG_m(\bar k)$.
\end{proposition}

On en d\'eduit la description suivante de $\calP_v(J_a)$.

\begin{corollaire}
On a une suite exacte de groupes ab\'eliens avec l'action d'endomorphisme
de Frobenius  $\sigma$
$$1 \rightarrow A_{\pm \alpha}(\bar k)\rightarrow \calP_{\bar v}(J_a)(\bar k) \rightarrow
\bbX_*(T)\rightarrow 1
$$
o\`u $\bbX_*(T)$ est le groupe des cocaract\`eres du tore $T$ au-dessus de
$\bar \calO_{\bar v}$.
\end{corollaire}

\begin{proof}
On d\'eduit de la  suite exacte
$$1\rightarrow J_a(\bar\calO_{\bar v}) \rightarrow T(\bar\calO_{\bar v})\rightarrow
A_{\pm \alpha}(\bar k)\rightarrow 1
$$
la suite exacte
$$1 \rightarrow A_{\pm \alpha}(\bar k) \rightarrow J_a(\bar F_{\bar v})/J_a(\bar\calO_{\bar v})
\rightarrow T(\bar F_v)/T(\bar \calO_v) \rightarrow 1.
$$
On a par ailleurs un isomorphisme
$$T(\bar F_v)/T(\bar \calO_v)=\bbX_*(T)$$
qui se d\'eduit de l'homomorphisme $\bbX_*(T) \rightarrow T(\bar F_v)$
d\'efini par $\lambda\mapsto \epsilon_v^\lambda$ d'o\`u le corollaire.
\end{proof}

\begin{lemme}
Les $\bar k$-points de la fibre de Springer affine
$$\calM_{\bar v}(a)(\bar k)=\{g\in G(\bar F_{\bar v})/ G(\bar\calO_{\bar v})
| \ad(g)^{-1}(\gamma_0)\in \frakg(\bar\calO_{\bar v})\}
$$
s'\'ecrivent de fa{\c c}on unique sous la forme
$$g=\epsilon_v^\lambda U_\alpha(x \epsilon_v^{-1}) $$
avec $\lambda\in \bbX_*(T)$ et $x\in \bar k$.
\end{lemme}

\begin{proof}
Avec la d\'ecomposition d'Iwasawa, pour tout \'el\'ement $g\in G(\bar
F_{\bar v})/ G(\bar\calO_{\bar v})$, il existe uniques $\lambda\in \bbX_*(T)$
et $u\in U(\bar F_{\bar v})/ U(\bar\calO_{\bar v})$ tels que
$$g=\epsilon_v^\lambda u .$$
Comme $T_{\bar v}$ commute avec $\gamma_0$, le plongement
$$\calM_{\bar v}(a) \rightarrow \calG_{\bar v}=G(\bar F_{\bar
v})/ G(\bar\calO_{\bar v})
$$
est $T_{\bar v}$-\'equivariant. L'action de $T_{\bar v}$ sur $\calM_v(a)$ se
factorise par
$$T_{\bar v} \rightarrow T(\bar F_{\bar v}) \rightarrow \calP_v(J_a)
=T(\bar F_{\bar v})/ J_a(\bar \calO_{\bar v})$$
de sorte qu'il se factorise par le tore de dimension un
$$T_{\bar v}\rightarrow A_{\pm \alpha, \bar v}.$$
Il s'ensuit que si on \'ecrit $g\in \calM_v(a)(\bar k)$ sous la forme
$g=\epsilon_v^\lambda u $, $u$ doit \^etre de la forme $u=U_\alpha(y)$
avec $y\in \bar F_{\bar v}/ \bar \calO_{\bar v}$ uniquement d\'etermin\'e.
Un calcul dans $\SL_2$ \cf \cite[lemme 6.2]{GKM} montre alors qu'avec
l'hypoth\`ese $r_{\pm \alpha}(a)=1$, $y$ s'\'ecrit uniquement sous la forme
$y=x \epsilon_v^{-1}$ avec $x\in \bar k$ et inversement les \'el\'ements $g$
de a forme $g=\epsilon_v^\lambda U_\alpha(x \epsilon_v^{-1})$
appartiennent \`a $\calM_{\bar v}(a)(\bar k)$.
\end{proof}

\begin{lemme}
Les points fixes du tore de dimension un $A_{\pm \alpha,\bar v}$ dans
$\calM_{\bar v}$ sont $\epsilon_v^\lambda$. Pour $\lambda\in \bbX_*(T)$
fix\'e, $A_{\pm \alpha,\bar v}$ agit simplement transitivement sur
$$O_\lambda=\{\epsilon_v^\lambda U_\alpha(x \epsilon_v^{-1}) \in \calM_v(a)(\bar k)
| x\in \bar k^\times\}.
$$
De plus le bord de l'adh\'erence de cette orbite est constitu\'e de
$\epsilon_v^\lambda$ et $\epsilon_v^{\lambda-\alpha^\vee}$ o\`u
$\alpha^\vee$ est la coracine associ\'ee \`a la racine $\alpha$.
\end{lemme}

\begin{proof}
Un calcul direct montre que les $\epsilon_v^\lambda$ sont fixes sous
l'action de $A_{\pm \alpha,\bar v}$ et que $A_{\pm \alpha,\bar v}$ agit
simplement transitivement sur $O_\lambda$. Ceci montre que l'ensemble
des points fixes de $A_{\pm \alpha,\bar v}$ est exactement
$\{\epsilon_v^\lambda| \lambda\in \bbX_*(T)\}$. Quand $x\rightarrow 0$,
$\epsilon_v^\lambda U_\alpha(x \epsilon_v^{-1})$ tend vers
$\epsilon_v^\lambda$ de sorte que $\epsilon_v^\lambda$ appartient \`a
l'adh\'erence de $O_\lambda$.

Il reste \`a d\'emontrer que quand $x\rightarrow \infty$,
$\epsilon_v^\lambda U_\alpha(x \epsilon_v^{-1})$ tend vers
$\epsilon_v^{\lambda-\alpha^\vee}$ Il revient au m\^eme de d\'emontrer
que $U_\alpha(x \epsilon_v^{-1})$ tend vers $\epsilon_v^{-\alpha^\vee}$
quand $x\rightarrow \infty$. Il s'agit d'un calcul bien connu dans la
Grassmannienne affine qui d\'ecoule de la relation de Steinberg dans
$G(\bar F_{\bar v})$ \cf \cite[chap. 3, lemme 19]{St}
$$y^{-\alpha^\vee} w_\alpha= U_{-\alpha}(y) U_{\alpha}(-y^{-1}) U_\alpha(y)$$
qui vaut pour tout $y\in \bar F_{\bar v}$, pour toute racine $\alpha$ et
pour un repr\'esentant $w_\alpha$ de la r\'eflexion $s_\alpha\in W$
attach\'ee \`a la racine $\alpha$ ind\'ependant de $y$ et qui en particulier
appartient \`a $G(\bar k)$. En prenant $y=-x^{-1} \epsilon_v$ avec $x\in
\bar k^\times$, on obtient la relation suivante dans $G(\bar F_{\bar v})/
G(\bar \calO_{\bar v})$
$$U_\alpha(x \epsilon_v^{-1})=\epsilon_v^{-\alpha^\vee} U_{-\alpha}(-x^{-1}\epsilon_v^{-1}).$$
En faisant tendre $x$ vers $\infty$, on constate que $U_\alpha(x
\epsilon_v^{-1})$ tend vers $\epsilon_v^{-\alpha^\vee}$.
\end{proof}

\begin{proposition}\label{comptage fibre simple}
Soit $\kappa\in \hat T^\sigma$ un \'el\'ement de torsion tel que
$$\kappa(\alpha^\vee)\not =1.$$
Alors, on a la formule
$$[\calM_v(a)/\calP_v(J_a)](k)_\kappa=  \sharp\, A_{\pm\alpha}(k_v)^{-1} q^{\deg(v)}.$$
\end{proposition}

\begin{proof}
La fibre de Springer affine $\calM_v(a)$ s'obtient comme la restriction des
scalaires $k_v/ k$ d'une fibre de Springer affine d\'efinie sur $k_v$. Il en est
de m\^eme de $\calP_v(J_a)$ et du quotient $[\calM_v(a)/\calP_v(J_a)]$.
On peut donc supposer que $k_v=k$.

Notons $A=A_{\pm \alpha,v}$ qui est un tore de dimension un sur $k$.
Consid\'erons le sous-groupe $\sigma$-invariant de $\bbX_*(T)$ engendr\'e
par $\alpha^\vee$ et consid\'erons l'image r\'eciproque de la suite exacte
$$1\rightarrow A\rightarrow \calP_v(J_a) \rightarrow \bbX_*(T) \rightarrow 1$$
par l'homomorphisme $\ZZ  \alpha^\vee \rightarrow \bbX_*(T)$. C'est un
groupe alg\'ebrique $P$ d\'efini sur $k$ muni d'une suite exacte
$$1\rightarrow A \rightarrow P \rightarrow \ZZ  \alpha^\vee \rightarrow 1$$
et un homomorphisme injectif $P\rightarrow\calP_v(J_a)$ qui induit
l'identit\'e sur la composante neutre $A$.

Dans la fibre de Springer affine $\calM_v(a)$, on dispose d'un $k$-point $m$
donn\'e par la section de Kostant. D'apr\`es la description ci-dessus de
$\calM_v(a)\otimes_k \bar k$, la composante connexe $M$ de
$\calM_v(a)\otimes_k \bar k$ est une cha{\^\i}ne infinie de droites
projectives munie d'une action de $P\otimes_k \bar k$ qui est simplement
transitive sur la $M^\reg$. Comme $M$ contient le $k$-point $m$, il est
d\'efini sur $k$. On retrouve $\calM_v(J_a)$ \`a partir de $M$ par
l'induction de $P$ dans $\calP_v(J_a)$
$$\calM_v(a)=M \wedge^P \calP_v(J_a).$$
En particulier, on a une \'equivalence de cat\'egories
$$[\calM_v(a)/\calP_v(J_a)]=[M/ P].$$
On se ram\`ene donc \`a l'exercice de comptage de points d\'ej\`a r\'esolu
dans l'exemple \ref{exemple : chaine infinie}.
\end{proof}

\subsection{Comptage dans une fibre de Hitchin anisotrope}
\label{subsection : comptage Hitchin}

Rappelons le comptage de points dans une fibre de Hitchin $\calM_a$ avec
$a\in\calM_a^\ani(k)$ en suivant le paragraphe 9 de \cite{N}. Dans {\em loc.
cit.}, nous avons consid\'erer le quotient de $\calM_a$ par $\calP_a$. Il est
en fait plus commode et pas plus on\'ereux de consid\'erer un quotient plus
g\'en\'eral.

\begin{numero}\label{choix J'_a}
Soit $J'_a$ un $X$-sch\'ema en groupes lisse commutatif de type fini muni
d'un homomorphisme $J'_a\rightarrow J_a$ qui un isomorphisme sur un
ouvert non-vide $U$ de $X$. La donn\'ee de $J'_a$ est \'equivalente \`a la
donn\'ee des sous-groupes ouverts compacts $J'_a(\calO_v)\subset
J_a(\calO_v)$ pour les points $v\in |X-U|$. Notons $\calP'_a=\calP(J'_a)$ le
classifiant des $J'_a$-torseurs sur $X$. On a alors un homomorphisme
$\calP'_a\rightarrow \calP_a$ qui induit une action de $\calP'_a$ sur
$\calM_a$. Pour $J'_a$ assez petit au sens o\`u les sous-groupes ouverts
compacts $J'_a(\calO_v)$ soient assez petit, $\rmH^0(\bar X,J'_a)$ est
trivial et de sorte que $\calP'_a$ est repr\'esentable par un groupe
alg\'ebrique localement de type fini sur $k$. Pour le comptage, il sera
commode de supposer que les fibres de $J'_a$ sont toutes connexes.
Consid\'erons la cat\'egorie quotient $[\calM_a /\calP'_a]$.
\end{numero}

Le comptage des $k$-points de $\calM_a$ est fond\'e sur la forme suivante
de la formule de produit \cf \ref{produit} et \cite[th\'eor\`eme 4.6]{N}.

\begin{proposition}\label{produit quotient}
Soit $U=a^{-1}(\frakc_D^\rs)$ l'image inverse du lieu semi-simple r\'egulier
de $\frakc_D$. On a une \'equiva\-lence de cat\'egories
$$[\calM_a/ \calP'_a]=\prod_{v\in X-U} [\calM_v(a)/ \calP_v(J'_a)]$$
compatible \`a l'action de $\sigma\in\Gal(\bar k/ k)$. En particulier, on a
une \'equiva\-lence entre les cat\'egories des $k$-points
$$[\calM_a/ \calP'_a](k)=\prod_{v \in |X-U|} [\calM_v(a)/ \calP_v(J'_a)](k).$$
\end{proposition}

Soit $(\calP'_a)^0$ la composante neutre de $\calP'_a$. Puisque
$a\in\calA^\ani(k)$, le groupe des composantes connexes $\pi_0(\calP_a)$
est un groupe fini muni d'une action de l'\'el\'ement de Frobenius
$\sigma\in\Gal(\bar k/ k)$.

\begin{proposition}
Supposons que les fibres de $J'_a$ sont connexes. Pour tout caract\`ere
$\sigma$-invariant de $\pi_0(\calP_a)$
$$\kappa:\pi_0(\calP_a)_\sigma\rightarrow \Ql^\times,$$
on a
$$\sharp\,[\calM_a/ \calP'_a](k)_\kappa=\prod_{v \in |X-U|} \sharp\,[\calM_v(a)/ \calP_v(J'_a)](k)_\kappa
.$$
\end{proposition}

\begin{proof}
Pour tout point ferm\'e $v$ de $X$ dans le compl\'ementaire de l'ouvert
$U=a^{-a}(\frakc_D^\rs)$, l'homomorphisme compos\'e
$$\calP_v(J'_a)\rightarrow \calP'_a \rightarrow \pi_0(\calP'_a)$$
se factorise par $\pi_0(\calP_v(J'_a))$. L'homomorphisme
$\kappa:\pi_0(\calP'_a)_\sigma \rightarrow \Ql^\times$ d\'efinit donc un
homomorphisme $\kappa:\pi_0(\calP_v(J'_a))_\sigma \rightarrow
\Ql^\times$.

Dans \ref{produit quotient}, si un point $y\in [\calM_a/ \calP'_a](k)$
correspond \`a une collection de points $y_v\in [\calM_v(a)/
\calP_v(J'_a)](k)$ pour $v\in |X-U|$, alors on a la formule
$$\langle \cl(y),\kappa \rangle =\prod_{v\in |X-U|}  \langle \cl(y_v),\kappa \rangle.$$
Par cons\'equent, on a la factorisation
$$\sharp\,[\calM_a/ \calP'_a](k)_\kappa=\prod_{v\in |X-U|}
\sharp\,[\calM_v(a)/\calP_v(J'_a)](k)_\kappa$$
d'o\`u la proposition.
\end{proof}

La conjonction avec \ref{formule des points fixes} donne le corollaire
suivant.

\begin{corollaire}\label{formule des points fixes Hitchin}
Pour tout caract\`ere $\kappa:\pi_0(\calP_a)_\sigma\rightarrow
\Ql^\times$, on a
$$\sum_n (-1)^n\tr(\sigma, \rmH^n(\calM_a)_\kappa)
= {\sharp\,}
(\calP'_a)^0(k)\prod_{v\in |X-U|}
\sharp\,[\calM_v(a)/\calP_v(J'_a)](k)_\kappa
$$
o\`u $\rmH^n(\calM_a)_\kappa$ est le sous-espace propre de
$\rmH^n(\calM_a)$ o\`u le groupe $\pi_0(\calP_a)$ agit \`a travers le
caract\`ere $\kappa$.
\end{corollaire}

\begin{numero}\label{trivialisation de D'}
Pour exprimer les nombres $\sharp\,[\calM_v(a)/ \calP_v(J'_a)](k)_\kappa$
en termes d'int\'egrales orbitales locales, faisons les choix suivants :
\begin{itemize}
  \item en toute place $v\in |X-U|$, choisissons une trivialisation du
      fibr\'e inversible $D'|_{X_v}$,
  \item en toute place $v\in |X-U|$, choisissons une mesure de Haar
      ${\rm d}t_v$ du tore $J_a(F_v)$.
\end{itemize}
Ces choix nous permettent :
\begin{itemize}
  \item d'identifier la restriction de $a$ \`a $X_v$ avec un \'el\'ement
      $a_v\in\frakc(\calO_v)\cap \frakc^\rs(F_v)$,
  \item d'identifier les sch\'emas en groupes $J_a$ et $J_{a_v}$ ce qui
      munit une mesure de Haar ${\rm d}t_v$ \`a $J_{a_v}(F_v)$,
  \item d'identifier la fibre de Springer affine $\calM_v(a)$ munie de
      l'action de $\calP_v(J_a)$ avec la fibre de Springer affine
      $\calM_v(a_v)$ munie de l'action de $\calP_v(J_{a_v})$.
\end{itemize}
\end{numero}
Avec ces choix faits, on peut exprimer les nombres $\sharp\,[\calM_v(a)/
\calP_v(J'_a)](k)_\kappa$ en termes de $\kappa$-int\'egrales orbitales \cf
\ref{comptage avec P_v(J'_a)}
$$\sharp\,[\calM_v(a)/ \calP_v(J'_a)](k)_\kappa=
{\rm vol}(J'_a(\calO_v),{\rm d}t_v){\bf O}^\kappa_{a_v}(1_{\frakg_v},{\rm
d}t_v).
$$
On obtient alors la formule
$$\sum_n (-1)^n\tr(\sigma, \rmH^n(\calM_a)_\kappa)=
(\calP'_a)^0(k) \prod_{v\in |X-U|} {\rm vol} (J'_a(\calO_v),{\rm d}t_v){\bf
O}^\kappa_{a_v}(1_{\frakg_v},{\rm d}t_v).$$

\subsection{Stabilisation sur $\tilde\calA_H^{\rm ani} -\tilde\calA_H^{\rm bad}$}
\label{subsection : A good}

On est en position de d\'emontrer le th\'eor\`eme \ref{stabilisation sur tilde
A} sur l'ouvert $\tilde\calA_H^{\rm good}=\tilde\calA_H^{\rm ani}
-\tilde\calA_H^{\rm bad}$. Ici $\tilde\calA_H^{\rm bad}$ est le sous-sch\'ema ferm\'e de $\tilde\calA_H^\ani$ d\'efini dans \ref{delta bad}.

\begin{proof}
D'apr\`es le th\'eor\`eme du support \ref{support faible}, sur
$\tilde\calA_H^{\rm good}$, les faisceaux pervers purs
$$K^n_\kappa=\tilde\nu^*\, ^p\rmH^{n}(\tilde f^\ani_*\Ql)_{\kappa}
\  {\rm et}\  K_{H,{\rm st}}^{n}=\,^p\rmH^{n+2r_H^G(D)}(\tilde
f^\ani_{H,*}\Ql)_{\rm st}(-r_H^G(D))
$$
sont les prolongement de leurs restrictions \`a n'importe quel ouvert
non-vide $\tilde\calU$ de $\tilde\calA_H^{\rm good}$. Pour d\'emontrer que
$K^n$ et $K^n_H$ sont isomorphes sur $\tilde\calA_H^{\rm good}\otimes_k
\bar k$ et sont isomorphes apr\`es semi-simplification sur
$\tilde\calA_H^{\rm good}$, il suffit de le faire sur n'importe quel ouvert
dense $\tilde\calU$ de $\tilde\calA_H^{\rm good}$. On va construire un bon
ouvert $\tilde\calU$ comme suit.

Rappelons qu'on a une \'egalit\'e de diviseurs \cf \ref{discrimnant resultant}
$$ \nu^* \discrim_{G,D}= \discrim_{H,D}+2\resultant_{H,D}^G$$
sur $\frakc_{H,D}$. De plus, on sait que $\discrim_{H,D}$ et
$\resultant_{H,D}^G$ sont des diviseurs r\'eduits de $\frakc_{H,D}$
\'etrangers de sorte que la r\'eunion
$$\discrim_{H,D}+\resultant_{H,D}^G$$
est aussi un diviseur r\'eduit.

\begin{lemme} Supposons que $\deg(D)> 2g$. Alors,
l'ensemble des $a_H\in \calA_H(\bar k)$ tel que $a_H(\bar X)$ coupe
transversalement $\discrim_{H,D}+\resultant_{H,D}^G$ forme un ouvert non
vide $\calU$ de $\calA_H$.
\end{lemme}

\begin{proof}
La d\'emonstration est identique \`a la d\'emonstration de \ref{A diamond
non vide}.
\end{proof}

\begin{numero}
Consid\'erons l'image r\'eciproque $\tilde\calU$ de $\calU$ dans
$\tilde\calA$. En rapetissant cet ouvert si n\'ecessaire, on peut supposer
que $\tilde\calU\subset \tilde\calA^{\rm good}$. En le rapetissant encore si
n\'ecessaire, on peut supposer que pour tout $n\in\ZZ$, les restrictions
$$L^n_\kappa=K^n_\kappa|_{\tilde\calU}\  {\rm et}\  L_{H,{\rm st}}^n=K_{H,{\rm st}}^n|_{\tilde\calU}
$$
sont des syst\`emes locaux purs de poids $n$ sur $\tilde\calU$.
\end{numero}

\begin{numero}
Il suffit en fait de d\'emontrer que pour toute extension finie $k'$ de $k$,
pour tout $k'$-point $\tilde a_H\in \tilde\calU(k')$ d'image $\tilde
a\in\tilde\calA(k')$, on a l'\'egalit\'e des traces
\begin{equation}\label{egalite de traces}
\sum_{n} (-1)^{n}\tr(\sigma_{k'},(L^n_\kappa)_{\tilde a}) \\
=\sum_{n} (-1)^{n} \tr(\sigma_{k'},(L^n_{H,{\rm st}})_{\tilde a_H}).
\end{equation}
En effet, on a alors l'\'egalit\'e
$$
\sum_{n} (-1)^{n}\tr(\sigma_{k'}^j,(L^n_\kappa)_{\tilde a}) \\
=\sum_{n} (-1)^{n} \tr(\sigma_{k'}^j,(L^n_{H,{\rm st}})_{\tilde a_H}).
$$
pour tout entier naturel $j$ en consid\'erant $\tilde a_H$ comme un point
\`a valeurs dans l'extension de degr\'e $j$ de $k'$. De plus, comme les
valeurs propres de $\sigma_{k'}$ dans $(L^n_\kappa)_{\tilde a}$ et
$(L^n_{H,{\rm st}})_{\tilde a_H}$ sont toutes de valeur absolue
$q^{n\deg(k'/ k)/ 2}$, ces \'egalit\'es impliquent que pour toute extension
finie $k'$ de $k$, pour tout $\tilde a_H\in \tilde\calU(k')$ d'image $\tilde
a\in \tilde\calA$ et pour tout $n$, on a
$$\tr(\sigma_{k'},(L^n_\kappa)_{\tilde a})=\tr(\sigma_{k'},(L^n_{H,{\rm
st}})_{\tilde a_H}).$$ D'apr\`es le th\'eor\`eme de Chebotarev, ceci
implique que $L^n_\kappa$ et $L^n_{H,{\rm st}}$ sont isomorphes apr\`es
la semi-simplification.
\end{numero}

\begin{numero}
D\'emontrons maintenant la formule (\ref{egalite de traces}). Pour tout
point $\tilde a_H \in\tilde\calU(k')$, on peut calculer directement les deux
membres de cette formule et ensuite les comparer. En rempla{\c c}ant $k$
par $k'$ et $X$ par $X\otimes_k k'$, on peut supposer que $\tilde a_H$ est
un $k$-point de $\tilde\calU$.
\end{numero}

\begin{numero}
Soit $a_H\in \calA_H(k)$ l'image de $\tilde a_H$. Soit $a$ l'image de $a_H$
dans $\calA(k)$. D'apr\`es \ref{J JH}, on a un homomorphisme canonique
$$J_a\rightarrow J_{H,a_H}$$
qui est un isomorphisme au-dessus de l'ouvert $U=a^{-1}(\frakc_D^\rs)$.
Choisissons un sch\'ema en groupes $J'_a$ lisse commutatif de fibres
connexes muni d'un homomorphisme $J'_a\rightarrow J_a$ qui est
g\'en\'eriquement un isomorphisme. Soit $\calP'_a$ le champ classifiant des
$J'_a$-torseurs sur $X$. En rapetissant $J'_a$ si n\'ecessaire, on peut
supposer que $\rmH^0(\bar X,J'_a)$ trivial dans quel cas $\calP'_a$ est un
groupe alg\'ebrique de type fini. On a des homomorphismes naturels
$\calP'_a \rightarrow \calP_a$ et $\calP'_a \rightarrow \calP_{H,a_H}$ qui
induit une action de $\calP'_a$ sur $\calM_a$ et $\calM_{H,a_H}$. On peut
calculer les deux membres dans la formule \ref{egalite de traces} en
consid\'erant les quotients $[\calM_a/\calP'_a]$ et
$[\calM_{H,a_H}/\calP'_a]$. D'apr\`es \ref{formule des points fixes Hitchin},
le membre de gauche de \ref{egalite de traces} vaut
$$(\calP'_a)^0(k) \prod_{v\in |X-U|}  \sharp\,
[\calM_v(a)/\calP_v(J'_a)](k)_\kappa
$$
alors que le membre de droite vaut
$$q^{r_H^G(D)}(\calP'_a)^0(k) \prod_{v\in |X-U|}  \sharp\,
[\calM_{H,v}(a)/\calP_v(J'_a)](k).
$$
\end{numero}

Comme
$$r_H^G(D)=\sum_v  \deg(v) r_{H,v}^G(a_H)$$
o\`u $r_{H,v}^G(a_H)$ est le degr\'e du diviseur $a_H^*
\resultant_{H,D}^G$ en $v$, il suffit de d\'emontrer l'\'enonc\'e suivant qui
est un cas particuli\`erement simple du lemme fondamental de
Langlands-Shelstad \ref{LS}. Notons que ce cas particulier du lemme
fondamental est essentiellement contenu dans l'article \cite{LL} de Labesse
et Langlands sur $\SL(2)$ et a \'et\'e repris d'un point de vue plus g\'eom\'etrique dans \cite{GKM} par
Goresky, Kottwitz et MacPherson.
\end{proof}

\begin{lemme}\label{LF simple}
Soit $\tilde a_H \in \tilde\calA_H(k)$. Soit $v$ un place de $X$ au-dessus
duquel $a_H(\bar X_v)$ ne coupe pas ou coupe transversalement le diviseur
$\discrim_{H,D}+\resultant_{H,D}^G$. On a l'\'egalit\'e entre les
nombres
$$\sharp\,[\calM_v(a)/\calP_v(J'_a)](k)_\kappa=
q^{\deg(v)r_{H,v}^G(a_H) } \sharp\,[\calM_{H,v}(a)/\calP_v(J'_a)](k)
$$
qui sont des nombres non nuls.
\end{lemme}

\begin{proof}
Soit ${\bar v}$ un point g\'eom\'etrique au-dessus de $v$. Comme $a_H(\bar
X_v)$ coupe transversalement $\discrim_{H,D}+\resultant_{H,D}^G$, il y a
trois possibilit\'es pour les entiers $d_{H,\bar v}(a_H), d_{\bar v}(a)$ et
$r_{H,v}^G(a_H)$ :
\begin{enumerate}
  \item si $a_H(\bar v)\notin \discrim_{H,D}\cup \resultant_{H,D}^G$,
      alors
$$d_{H,\bar v}(a_H)=0, d_{\bar v}(a)=0\ {\rm et}\ r_{H,v}^G(a_H)=0,$$
  \item si $a_H(\bar v)\in \discrim_{H,D}$ alors $a_H(\bar v)\notin
      \resultant_{H,D}^G$ et on a
$$d_{H,\bar v}(a_H)=1, d_{\bar v}(a)=1 \ {\rm et}\ r_{H,v}^G(a_H)=0,$$
  \item si $a_H(\bar v)\in \resultant_{H,D}^G$ alors $a_H(\bar v)\notin
      \discrim_{H,D}$ et on a
$$d_{H,\bar v}(a_H)=0, d_{\bar v}(a)=2\ {\rm et}\ r_{H,v}^G(a_H)=1.$$
\end{enumerate}

\begin{numero}
Dans les deux premiers cas, la formule de Bezrukavnikov \ref{dimension
Bezru} montre que $\delta_{H,\bar v}(a_H)=\delta_{\bar v}(a)=0$ de sorte
que les fibres de Springer affines $\calM_{H,v}(a_H)$ et $\calM_v(a)$ sont
toutes les deux de dimension z\'ero. Il s'ensuit que $\calP_v(J_a)$ agit
simplement transitivement sur $\calM_v(a)$ et $\calP_v(J_{H,a_H})$ agit
simplement transitivement sur $\calM_{H,v}(a_H)$ \cf \ref{dim zero}.
Choisissons une trivialisation de $D'$ au-dessus de $X_v$ comme dans
\ref{trivialisation de D'} pour pouvoir \'ecrire agr\'eablement les int\'egrales
orbitales. En particulier $a_H$ et $a$ d\'efinissent des \'el\'ements
$a_{H,v}\in \frakc_H(\calO_v)$ et $a_v\in\frakc(\calO_v)$. En conjonction
avec \ref{comptage avec P_v(J_a)}, on en d\'eduit la formule
$$1=[\calM_{H,v}(a_H)/\calP_v(J_{H,a_H})](k)
={\rm vol}(J^0_{H,a_H}(\calO_v),{\rm d}t_v){\bf
SO}_{a_{H,v}}(1_{\frakh_v},{\rm d}t_v)
$$
En comparant avec la formule \ref{comptage avec P_v(J'_a)}
$$\sharp\,[\calM_{H,v}(a)/\calP_v(J'_a)](k)=
{\rm vol}(J'_a(\calO_v),{\rm d}t_v){\bf SO}_{a_{H,v}}(1_{\frakh_v},{\rm
d}t_v)
$$
on obtient
$$\sharp\,[\calM_{H,v}(a)/\calP_v(J'_a)](k)=\frac{{\rm vol} (J'_a(\calO_v),{\rm
d}t_v)}{{\rm vol}(J^0_{H,a_H}(\calO_v),{\rm d}t_v)}.$$

En remarquant que $\calM_v(a)$ est donn\'e avec un $k$-point par la
section de Kostant, on a aussi
$$\sharp\, [\calM_v(a)/\calP_v(J_a)](k)_\kappa=1.$$
Le m\^eme raisonnement comme ci-dessus implique alors
$$\sharp\,[\calM_{v}(a)/\calP_v(J'_a)](k)_\kappa=\frac{{\rm vol} (J'_a(\calO_v),{\rm
d}t_v)}{{\rm vol}(J^0_{a}(\calO_v),{\rm d}t_v)}.
$$
Il reste \`a remarquer que dans les deux premiers cas, l'homomorphisme
$J_a\rightarrow J_{H,a_H}$ induit un isomorphisme sur les composantes
neutres $J_a^0\rightarrow J_{H,a_H}^0$ et on obtient l'\'egalit\'e
$$\sharp\,[\calM_v(a)/\calP_v(J'_a)](k)_\kappa=
\sharp\,[\calM_{H,v}(a)/\calP_v(J'_a)](k)
$$
qu'on voulait.
\end{numero}

\begin{numero}
Consid\'erons maintenant le troisi\`eme cas o\`u $a_H(X_v)$ coupe
transversalement $\resultant_{H,D}^G$ et n'intersecte pas $\frakc_{H,D}$.
On a dans ce cas $d_{H,\bar v}(a_H)$ est nul de sorte aue
$$\delta_{H,\bar v}(a_H)=0\ {\rm et}\ c_{H,\bar v}(a_H)=0.$$
Comme $\delta_{H,\bar v}(a_H)=0$, la fibre de Springer affine
$\calM_{H,v}(a_H)$ est de dimension z\'ero. D'apr\`es \ref{dim zero},
$\calP_v(J_{H,a_H})$ agit simplement transitivement sur
$\calM_{H,v}(a_H)$ de sorte qu'on a encore
$$\sharp\, [\calM_{H,v}(a_H)/\calP_v(J_{H,a_H})](k)=1.$$
Le m\^eme raisonnement comme ci-dessus montre que
$$\sharp\,[\calM_{H,v}(a)/\calP_v(J'_a)](k)=\frac{{\rm vol} (J'_a(\calO_v),{\rm
d}t_v) }{ {\rm vol}(J^0_{H,a_H}(\calO_v),{\rm d}t_v)}.
$$
Notons que comme $a_H(X_v)$ n'intersecte pas $\frakc_{H,D}$,
$J_{H,a_H}$ est un tore et en particulier $J_{H,a_H}=J_{H,a_H}^0$.

Comme l'invariant $c_{\bar v}(a)$ ne d\'epend que de la fibre g\'en\'erique
de $J_a|_{X_v}$, on a
$$c_{\bar v}(a)=c_{H,\bar v}(a_H)=0.$$
Ceci implique que la fibre de Springer affine $\calM_{\bar v}(a)$ est de
dimension
$$\delta_{\bar v}(a)=\frac{d_{\bar v}(a)-c_{\bar v}(a) }{ 2}=1.$$
On est donc exactement dans la situation de \ref{subsection : fibre simple}.
En appliquant la formule \ref{comptage fibre simple} tout en notant que
l'hypoth\`ese $\kappa(\alpha^\vee)\not=1$ est bien v\'erifi\'ee ici, on
obtient la formule
$$[\calM_v(a)/\calP_v(J_a)](k)_\kappa=  \sharp\, A_{\pm\alpha}(k_v) q^{\deg(v)}$$
o\`u $A_{\pm \alpha}$ est le tore de dimension un sur $k_v$ d\'efini par
$$A_{\pm \alpha}(\bar k)=J_{H,a_H}(\calO_{\bar v})/ J_a(\calO_{\bar v}).$$
En appliquant la formule \ref{comptage avec P_v(J_a)}, on trouve
$${\bf O}_{a}^\kappa(1_{\frakg_v},{\rm d}t_v)= \frac{q^{\deg(v)}
}{ \sharp\, A_{\pm\alpha}(k_v) {\rm vol}(J^0_{a}(\calO_v),{\rm d}t_v)}.
$$
On peut aussi v\'erifier la formule
$$\sharp\, A_{\pm\alpha}(k_v){\rm vol}(J^0_{a}(\calO_v),{\rm d}t_v)=
{\rm vol}(J^0_{H,a_H}(\calO_v),{\rm d}t_v)$$ qui implique
$$[\calM_v(a)/\calP_v(J'_a)](k)_\kappa= \frac{ q^{\deg(v)} {\rm vol}(J'_a(\calO_v, {\rm d}t_v))
}{ {\rm vol}(J^0_{H,a_H}(\calO_v),{\rm d}t_v)}.
$$
On obtient donc l'\'egalit\'e
$$\sharp\,[\calM_v(a)/\calP_v(J'_a)](k)_\kappa=
q^{\deg(v)}\sharp\,[\calM_{H,v}(a)/\calP_v(J'_a)](k)
$$
qu'on voulait.
\end{numero}

Ceci termine la d\'emonstration de \ref{LF simple}.
\end{proof}

\subsection{Le lemme fondamental de Langlands-Shelstad}
\label{lemme fondamental}

On est maintenant en mesure de d\'emontrer le lemme fondamental de
Langlands et Shelstad \ref{LS}.

\begin{proof}
Soient $X_v=\Spec(\calO_v)$ avec $\calO_v=k[[\epsilon_v]]$ et
$X_v^\bullet=\Spec(F_v)$ avec $F_v=k((\epsilon_v))$. Soit $G_v$ un
groupe r\'eductif sur $X_v$ forme quasi-d\'eploy\'ee de $\GG$ donn\'ee par
un $\Out(\GG)$-torseur $\rho_{G,v}$ sur $X_v$. Soit
$(\kappa,\rho_{\kappa,v})$ une donn\'ee endoscopique  elliptique \cf
\ref{subsection : Groupes endoscopiques}. Soit $H_v$ le groupe
endoscopique associ\'e. Soit $a_H\in \frakc_{H}(\calO_v)$ d'image $a\in
\frakc(\calO_v)\cap \frakc^\rs(F_v)$.

\begin{numero}
Si le centre de $G_v$ contient un tore d\'eploy\'e $C$, $C$ est aussi
contenu dans le centre de $H_v$. En rempla{\c c}ant $G_v$ par $G_v/C$ et
$H_v$ par $H_v/C$, la $\kappa$-int\'egrale orbitale et l'int\'egrale orbitale
stable envisag\'ees ne changent pas de sorte qu'on peut supposer que le
centre de $G_v$ ne contient pas de tore d\'eploy\'e.

Si le centre de $H_v$ contient un tore d\'eploy\'e $C$, $C$ est
naturellement inclus dans le tore $I_{\gamma_0}$ o\`u
$\gamma_0=\epsilon(a)$. Le centralisateur $M_v$ de $C$ dans $G_v$ est un
sous-groupe de Levi de $G_v$. Par la formule de descente, on peut
remplacer la $\kappa$-int\'egrales orbitales dans $G_v$ par une
$\kappa$-int\'egrale orbitale dans $M_v$. Le centre de $M_v$ contient
maintenant un tore d\'eploy\'e et on se ram\`ene \`a la situation discut\'ee
ci-dessus. Apr\`es un nombre fini de ces r\'eductions, on peut supposer que
les centres de $G_v$ et $H_v$ ne contiennent pas de tores d\'eploy\'es.
\end{numero}

\begin{numero}
Soit $\calM_{H,v}(a_H)$ et $\calM_v(a)$ les fibres de Springer affines
associ\'ees. D'apr\`es \ref{subsection : approximation}, il existe un entier
naturel $N$ tel que pour toute extension finie $k'$ de $k$, pour tout $a'_H
\in \frakc_H(\calO_v\otimes_k k')$ tel que
$$a_H \equiv a'_H \mod \epsilon_v^N$$
alors
\begin{itemize}
  \item $a'_H$ a une image $a'\in \frakc(\calO_v\otimes_k k')\cap
      \frakc^\rs(F_v\otimes_k k')$ ;
  \item les fibres de Springer affines $\calM_{H,v}(a_H)\otimes_k k'$ et
      $\calM_{H,v}(a'_H)$ munies de l'action de
      $\calP_{v}(J_{H,a_H})\otimes_k k'$ et $\calP_v(J_{H,a'_H})$
      respectivement,  sont isomorphes ;
  \item les fibres de Springer affines $\calM_v(a)\otimes_k k'$ et
      $\calM_v(a')$ munies de l'action de $\calP_v(J_a)$ et
      $\calP_v(J_{a'})$ respectivement, sont isomorphes.
\end{itemize}
Notons $\delta_H(a_H)$ la dimension de la fibre de Springer affine
$\calM_{H,v}(a_H)$.
\end{numero}

\begin{numero}
Il existe une courbe projective lisse g\'eom\'etriquement connexe $X$ sur
$k$ muni des donn\'ees suivantes :
\begin{itemize}
  \item deux $k$-points distincts not\'es $v$ et $\infty$,
  \item un isomorphisme entre la compl\'etion de $X$ en $v$ avec le
      sch\'ema $X_v$ ci-dessus,
  \item un $\pi_0(\kappa)$-torseur $\rho_\kappa$ sur $X$ muni d'une
      trivialisation $\alpha_\infty$ au-dessus de $\infty$ et d'un
      isomorphisme $\alpha_v$ avec $\rho_{\kappa,v}$ au-dessus de
      $X_v$.
\end{itemize}
Soient $G$ et $H$ les $X$-sch\'emas en groupes associ\'es comme dans
\ref{subsection : Groupes endoscopiques}. Avec la donn\'ee de $\alpha_v$,
les restrictions de $G$ et $H$ \`a $X_v$ sont canoniquement isomorphes
\`a $G_v$ et $H_v$. Les centres de $G_v$ et $H_v$ ne contenant
par de tores d\'eploy\'es, il en est de m\^eme de $G$ et $H$. Avec la
donn\'ee de $\alpha_\infty$, on a un homomorphisme
$$\rho_\kappa^\bullet:\pi_1(X,\infty)=\pi_1(\ovl X,\infty)\rtimes \Gal(\bar k/ k)
\rightarrow \pi_0(\kappa)
$$
trivial sur le facteur $\Gal(\bar k/ k)$. Par cons\'equent, le sous-groupe
$\pi_1(\ovl X,\infty)$ a la m\^eme image dans $\pi_0(\kappa)$ que
$\pi_1(X,\infty)$. Il s'ensuit que sur $\ovl X$, les centres de $G$ et $H$ ne
contiennent pas de tores d\'eploy\'es.
\end{numero}

\begin{numero}
On choisit maintenant un fibr\'e inversible $D$ sur $X$ v\'erifiant les
hypoth\`eses suivantes
\begin{itemize}
 \item il existe un fibr\'e inversible $D'$ sur $X$ avec $D={D'}^{\otimes
     2}$,
 \item une section globale de $D'$ non nulle en $v$,
 \item $\deg(D) > rN+2g$ o\`u $r$ est le rang de $G$ et $g$ est le genre
     de $X$,
 \item $\delta_H(a_H)$ est plus petit que l'entier $\delta_H^{\rm bad}(D)$
     d\'efini dans \ref{delta bad}.
\end{itemize}
D'apr\`es \ref{codimension}, l'entier $\delta_H^{\rm bad}$ tend vers
$\infty$ lors que $\deg(D)\rightarrow \infty$ de sorte que pour $\deg(D)$
assez grand la derni\`ere hypoth\`ese est r\'ealis\'ee.
\end{numero}

\begin{numero}
Consid\'erons les fibrations de Hitchin associ\'ees \`a la courbe $X$, au
fibr\'e inversible $D$ et aux groupes $G$ et $H$ respectivement
$$f:\calM\rightarrow \calA \ {\rm et}\ f_H:\calM_H\rightarrow \calA_H.$$
L'hypoth\`ese que le centre $G$ et $H$ ne contient pas de tores
d\'eploy\'es sur $\ovl X$ et l'hypoth\`ese $\deg(D)>2g$ assurent que les
ouverts anisotropes $\calA^\ani$ et $\calA_H^\ani$ ne sont pas vides.
\end{numero}

\begin{numero}
Avec l'hypoth\`ese $\deg(D) > rN+2g$, l'application de restriction des
section globales \`a $\Spec(\calO_v/ \epsilon_v^N)$
$$\rmH^0(X,\frakc_{H,D}) \rightarrow \frakc_H(\calO_v/ \epsilon_v^N)$$
est une application lin\'eaire surjective. Soit $Z$ le sous-espace affine de
$\rmH^0(X,\frakc_{H,D})$ des \'el\'ements ayant la m\^eme image dans
$\frakc_H(\calO_v/ \epsilon_v^N)$ que $a_H$. C'est un sous-sch\'ema
ferm\'e de codimension $rN$ de $\calA_H$. Par le m\^eme raisonnement
que dans \ref{A diamond non vide}, on peut montrer que l'ouvert $Z'$ de
$Z$ constitu\'e des points $a'_H\in Z'(\bar k)$ tels que la courbe $a'_H(\bar
X-\{v\})$ coupe le diviseur $\discrim_{H,D}+\resultant_{H,D}^G$
transversalement, est un ouvert non vide. Comme le compl\'ement de
$\calA_H^\ani$ dans $\calA_H$ est un ferm\'e de codimension plus grande
ou \'egale \`a $\deg(D)$ \cf \ref{complement de A ani}, $Z'\cap
\calA_H^\ani\not= \emptyset$. En rempla{\c c}ant $Z'$ par  $Z'\cap
\calA_H^\ani$, on peut supposer que $Z'\subset \calA_H^\ani$. Pour tout
$a'_H\in Z'(\bar k)$, l'invariant $\delta_H(a')=\delta_H(a_H)$ de sorte qu'on
a $Z'\subset \calA_H^{\rm good}$ o\`u on dispose du th\'eor\`eme de
stabilisation \cf \ref{subsection : A good}.

Soit $\tilde Z'$ l'image r\'eciproque de $Z'$ dans $\tilde\calA_H^{\rm
good}$. Comme $\tilde Z'$ est un sch\'ema de dimension positive, il existe
un entier $m$ tel que pour toute extension $k'/ k$ de degr\'e plus grand ou
\'egal \`a $m$, l'ensemble $\tilde Z'(k')\not= \emptyset$. Soit $\tilde a'_H
\in Z'(k')$ au-dessus de $a'_H\in Z(k')$. Soient $\tilde a'$ l'image de $\tilde
a'_H$ dans $\tilde\calA$ et $a'$ l'image de $a'_H$ dans $\calA$.
\end{numero}

\begin{numero}
La conjonction de la partie du th\'eor\`eme de stabilisation \ref{stabilisation sur tilde
A} d\'emontr\'ee sur l'ouvert $\tilde\calA_H^{\rm good}$ avec la
formule \ref{formule des points fixes Hitchin} nous donne l'\'egalit\'e
\begin{eqnarray*}
&{\sharp\,}(\calP'_{a'})^0(k') \prod_{v'\in |(X-U)\otimes_k k'|} {\sharp\,}
[\calM_{H,v'}(a'_H)/\calP_{v'}(J'_{a'})](k')_\kappa
\\
=&q^{r(a'_H)\deg(k'/ k)}{\sharp\,}
(\calP'_{a'})^0(k') \prod_{v'\in |(X-U)\otimes_k k'|}
{\sharp\,} [\calM_{v'}(a'_H)/\calP_{v'}(J'_{a'})](k')
\end{eqnarray*}
avec un choix d'un sch\'ema en groupes $J'_{a'}$ lisse commutatif de fibre
connexe muni d'un homomorphisme
$$J'_{a'}\rightarrow J_{a'}\rightarrow J_{H,a'_H}$$
comme dans \ref{choix J'_a}. Il sera commode de choisir $J'_{a'}=J_{a'}^0$
au-dessus de $X_v$.
\end{numero}

\begin{numero}
Puisque $\tilde a'_H\in \tilde\calU(k')$, on peut appliquer \ref{LF simple}
\`a toutes les places $v'\not=v$. En retranchant les \'egalit\'es en ces
places, on trouve une \'egalit\'e en la place $v$ du d\'epart
$${\sharp\,}[\calM_{H,v}(a'_H)/\calP_{v}(J^0_{a'})](k')_\kappa
=q^{r_v(a'_H)\deg(k'/ k)} {\sharp\,} [\calM_{v}(a'_H)/\calP_{v}(J^0_{a'})](k').
$$
Comme on sait que la fibre de Springer affine $\calM_v(a)\otimes_k k'$ muni
de l'action de $\calP_v(J_a)\otimes_k k'$ est isomorphe \`a $\calM_v(a')$
muni de l'action de $\calP_v(J_{a'})$ et la m\^eme chose pour le groupe
$H$, on en d\'eduit l'\'egalit\'e
$${\sharp\,}[\calM_{H,v}(a_H)/\calP_{v}(J^0_a)](k')_\kappa
=q^{r_{H,v}^G(a_H)\deg(k'/ k)} {\sharp\,} [\calM_{v}(a_H)/\calP_{v}(J^0_a)](k').
$$
pour toute extension $k'$ de $k$ de degr\'e plus grand que $m$. En
appliquant \ref{k'->k}, on obtient l'\'egalit\'e
$${\sharp\,}[\calM_{H,v}(a_H)/\calP_{v}(J^0_a)](k)_\kappa
=q^{r_{H,v}^G(a_H)} {\sharp\,} [\calM_{v}(a_H)/\calP_{v}(J^0_a)](k).
$$
En appliquant la formule \ref{comptage avec P_v(J'_a)}, on obtient
maintenant l'\'egalit\'e
$${\bf O}_{a}^\kappa(1_{\frakg_v},{\rm d}t_v)=q^{r_{H,v}^G(a_H)}
{\bf SO}_{a_H}(1_{\frakh_v},{\rm d}t_v)
$$
qu'on voulait.
\end{numero}

Ceci termine la d\'emonstration de la conjecture de Langlands-Sheldstad
\ref{LS}.
\end{proof}

\subsection{Stabilisation sur $\tilde\calA^\ani$}
\label{subsection : renversement}

En renversant une nouvelle fois le processus local-global, on peut
maintenant compl\'eter la d\'emonstration de \ref{stabilisation sur tilde A}.

\begin{proof}
Gardons les notations de \ref{subsection : A good}. Comme dans
\ref{subsection : A good}, il suffit de d\'emontrer l'\'egalit\'e de traces
\ref{egalite de traces}
\begin{equation}\label{egalite de traces 2}
\sum_{n} (-1)^{n}\tr(\sigma_{k'},(L^n_\kappa)_{\tilde a}) \\
=\sum_{n} (-1)^{n} \tr(\sigma_{k'},(L^n_{H,{\rm st}})_{\tilde a_H}).
\end{equation}
pour tout $\tilde a_H\in\tilde\calA_H^\ani(k)$ d'image $\tilde a\in
\tilde\calA^\ani(k)$. D'apr\`es \ref{formule des points fixes Hitchin}, le
membre de gauche de \ref{egalite de traces} vaut
$$(\calP'_a)^0(k) \prod_{v\in |X-U|}  \sharp\,
[\calM_v(a)/\calP_v(J'_a)](k)_\kappa
$$
alors que le membre de droite vaut
$$q^{r_H^G(D)}(\calP'_a)^0(k) \prod_{v\in |X-U|}  \sharp\,
[\calM_{H,v}(a)/\calP_v(J'_a)](k).
$$
En combinant \ref{LS}, \ref{comptage avec P_v(J'_a)}, on a l'\'egalit\'e
$$[\calM_v(a)/\calP_v(J'_a)](k)_\kappa=q^{r_{H,v}^G(D)}[\calM_{H,v}(a)/\calP_v(J'_a)](k)$$
d'o\`u \ref{egalite de traces 2}.
\end{proof}

\subsection{Conjecture de Waldspurger}
\label{subsection : demo Waldspurger}

Soient maintenant $G_1$ et $G_2$ deux $X$-sch\'emas en groupes
appari\'es au sens de \ref{apparies}. Supposons que les centres de $G_1$ et
$G_2$ ne contiennent pas de tores d\'eploy\'es au-dessus de $\ovl
X=X\otimes_k \bar k$.

\begin{numero}
Soient $f_1:\calM_1 \rightarrow \calA_1$ et $f_2:\calM_2 \rightarrow
\calA_2$ les fibrations de Hitchin associ\'ees. Comme dans \ref{Hitchin
apparies}, on a
$$\calA=\calA_1=\calA_2.$$
Soient $\calP_1$ et $\calP_2$ les $\calA$-champs de Picard associ\'es \`a
$G_1$ et $G_2$. D'apr\`es \ref{isogenie P1 P2}, il existe un
homomorphisme $\calP_1\rightarrow \calP_2$ qui induit une isog\'enie
entre leurs composantes neutres.
\end{numero}

\begin{theoreme}\label{pair stable}
Il existe un isomorphisme entre les simplifications des faisceaux pervers
gradu\'es sur $\calA^\ani$
$$K_1=\bigoplus_n\,^p\rmH^n (f_{1,*}^\ani\Ql)_{\rm st} \ {\rm et}\
K_2=\bigoplus_n\,^p\rmH^n (f_{2,*}^\ani\Ql)_{\rm st}.$$
\end{theoreme}

La d\'emonstration de ce th\'eor\`eme suit essentiellement les m\^emes
\'etapes que celle de \ref{stabilisation sur tilde A}. En particulier, on
d\'emontrera en cours de route le lemme fondamental non standard
conjectur\'e par Waldspurger \ref{Waldspurger}.

\begin{numero}
On d\'emontre d'abord qu'il existe un tel isomorphisme au-dessus de l'ouvert
$\calA^\diamondsuit$. Au-dessus de cet ouvert, les morphismes $f_1$ et
$f_2$ sont propres et lisses de sorte que les restrictions de $K_1$ et $K_2$
\`a $\calA^\diamondsuit$ sont des syst\`emes locaux gradu\'es. Pour
d\'emontrer qu'il existe un isomorphisme entre leurs simplifications, il suffit
d'apr\`es le th\'eor\`eme de Chebotarev de d\'emontrer l'\'egalit\'e de
traces
\begin{equation}\label{egalite des traces 3}
\tr(\sigma_{k'},K_{1,a})=\tr(\sigma_{k'},K_{2,a})
\end{equation}
pour toute extension finie $k'$ de $k$ et pour tout point $a\in
\calA^\diamondsuit(k')$. En rempla\c cant $X$ par $X\otimes_k k'$, on peut
supposer que $k=k'$.
\end{numero}

\begin{numero}
Soit $a\in\calA^\diamondsuit(k)$. D'apr\`es \ref{A lozenge}, on sait :
\begin{itemize}
  \item $\calP_{i,a}$ agit simplement transitivement sur $\calM_{i,a}$
  \item $\calP_{i,a}$ est de la forme $[P_i/ A_i]$ o\`u $P_i$ est une
      extension d'un groupe fini par une vari\'et\'e ab\'elienne et o\`u
      $A_i$ est un groupe fini agissant trivialement sur $P_i$.
\end{itemize}
Par cons\'equent, le quotient $[\calM_{i,a}/ P_i]$ est isomorphe au
classifiant du groupe fini $A_i$. D'apr\`es la formule des points fixes
\ref{formule des points fixes} et \ref{exemple : classifiant fini}, on a
$$\tr(\sigma,K_{i,a})=\sharp P_i^0(k).$$
Puisque $P_1^0$ et $P_2^0$ sont des vari\'et\'es ab\'eliennes isog\`enes sur
$k$, ils ont le m\^eme nombre de $k$-points d'o\`u l'\'egalit\'e des traces
\ref{egalite des traces 3}. Il existe donc un isomorphisme entre les
simplifications des restrictions de $K_1$ et $K_2$ \`a
$\calA^\diamondsuit$.
\end{numero}

\begin{numero}
D'apr\`es le th\'eor\`eme du support \ref{support faible}, il existe un
isomorphisme entre les semi-simplifications des restrictions de $K_1$ et $K_2$
\`a
$$\calA^{\rm good}=\calA^{\rm ani}-\calA^{\rm bad}.$$
\end{numero}

\begin{numero}
En proc\'edant comme dans \ref{lemme fondamental}, on en d\'eduit le
lemme fondamental non standard conjectur\'e par Waldspurger
\ref{Waldspurger}.
\end{numero}

\begin{numero}
En renversant de nouveau le processus local-global comme dans
\ref{subsection : renversement}, on en d\'eduit l'\'egalit\'e des traces
\ref{egalite des traces 3} pour tout $a\in \calA^\ani(k')$ pour toute
extension finie $k'$ de $k$. On en d\'eduit le th\'eor\`eme \ref{pair stable}.
\end{numero}

\appendix

\section[Dualit\'e de Pincar\'e et comptage de dimension]
{Dualit\'e de Poincar\'e et comptage de dimension d'apr\`es Goresky et MacPherson}
\label{appendice GM}

Goresky et MacPherson ont observ\'e que la dualit\'e de Poincar\'e impose
une contrainte sur la codimension de support des faisceaux pervers simples
pr\'esents dans le th\'eor\`eme de d\'ecomposition. Cette observation a
jou\'e un r\^ole crucial dans la d\'emonstration du th\'eor\`eme du support.
Il nous semble qu'elle devrait avoir d'autres applications \'egalement.
Avec leur permission, nous rappelons cet argument de comptage de
dimension et dualit\'e de Poincar\'e un contexte g\'en\'eral et sous la forme
la plus simple possible. En toute g\'en\'eralit\'e, cette contrainte est
relativement faible mais en pratique, elle fournit une amorce pr\'ecieuse
\`a l'aide de laquelle on peut faire jouer d'autres arguments plus
sp\'ecifiques.

Soient $X$ et $Y$ des sch\'emas de type fini sur un corps fini $k=\FF_q$.  Soit
$f:X\rightarrow Y$ un morphisme propre. Supposons de plus que $X$ est un
$k$-sch\'ema lisse. Alors, d'apr\`es Deligne \cite{Weil2} l'image directe
$f_*\Ql$ est un complexe pur. D'apr\`es \cite{BBD}, il d\'ecompose
g\'eom\'etriquement comme une somme directe de faisceaux pervers
g\'eom\'etriquement irr\'eductibles avec d\'ecalage. Il existe au-dessus de
$Y\otimes_k \bar k$ un isomorphisme
$$f_*\Ql=\bigoplus_{(K,n)} K[-n]^{m_{K,n}}$$
o\`u la somme directe est \'etendue sur l'ensemble  des classes
d'\'equivalence des couples $(K,n)$ constitu\'e d'un  faisceau pervers
irr\'eductible $K$ sur $Y\otimes_k \bar k$ et d'un entier $n$ et o\`u
$$m_{K,n}:=\dim\Hom(K,\,^p\rmH^n(f_*\Ql))$$
est un entier naturel nul sauf pour un nombre fini de couples $(K,n)$. Un
faisceau pervers irr\'eductible $K$ est dit {\em pr\'esent} dans $f_*\Ql$ s'il
existe un entier $n$ tel que $m_{K,n}\not=0$.

La dualit\'e de Poincar\'e implique une sym\'etrie pour ces entiers
$m_{K,n}$.

\begin{appproposition}\label{symetrie Lefschetz}
Pour tout faisceau pervers g\'eom\'etriquement irr\'e\-ductible $K$ sur $Y$,
alors on a
$$m_{K,n+\dim(X)}=m_{DK,-n+\dim(X)}$$
o\`u $DK$ est le dual de Verdier de $K$.
\end{appproposition}

Goresky et MacPherson ont observ\'e que cette sym\'etrie impose une
contrainte sur la codimension des supports des faisceaux pervers
g\'eom\'e\-triquement irr\'eductibles $K$ pr\'esents dans $f_*\Ql$.

\begin{apptheoreme}\label{inegalite de codim}
Soient $X$ et $Y$ des sch\'emas de type fini sur un corps $k$, $X$ est lisse
sur $k$. Soit $f:X\rta Y$ un morphisme propre. Supposons que le morphisme
$f$ est de dimension relative $d$. Soit $K$ un faisceau pervers
irr\'eductible sur $Y\otimes_k \bar k$ pr\'esent dans $f_*\Ql$. Soit $Z$ le
support de $K$. Alors on a l'in\'egalit\'e
$${\rm codim}(Z)\leq d.$$
\end{apptheoreme}

\begin{proof}
Supposons au contraire que ${\rm codim}(Z)>d$. D'apr\`es la dualit\'e de
Poincar\'e et quitte \`a \'echanger $K$ et $DK$ qui ont le m\^eme support,
on peut supposer qu'il existe un entier $n$
$$n\geq\dim(X)$$
tel que $m_{K,n}\not=0$. D'apr\`es \cite{BBD}, il existe un ouvert $Z'$ de
$Z$, un syst\`eme local $K'$ sur $S$ tel que $K=j_{! *} K' [\dim(Z)]$ o\`u
$j$ est l'inclusion de l'ouvert $Z'$ dans $Z$. Soit $y$ un point
g\'eom\'etrique de $Z'$. La fibre de $K$ en $y$ est alors un $\Ql$-espace
vectoriel plac\'e en degr\'e
$$-\dim(Z)=-\dim(Y)+{\rm codim}(Z).$$
La fibre de $K[-n]$ en $y$ \'etant un facteur direct
$\rmR\Gamma(X_y,\Ql)$, ceci implique que
$$\rmH^{n-\dim(Y)+{\rm codim}(Z)}(X_y)\not=0.$$
Or, on a l'in\'egalit\'e
$$n-\dim(Y)+{\rm codim}(Z) > 2d$$
parce que $n\geq\dim(X)$ et ${\rm codim}(Z) >d$. Cette non annulation est
en contradiction avec l'hypoth\`ese que la fibre $X_y$ est de dimension
inf\'erieure ou \'egale \`a $d$.
\end{proof}

\begin{apptheoreme}\label{inegalite stricte de codim}
Mettons-nous sous les hypoth\`eses du th\'eor\`eme pr\'e\-c\'edent.
Supposons en plus que les fibres de $f$ sont g\'eom\'etriquement
irr\'e\-ductibles de dimension $d>0$. Soit $K$ un faisceau pervers
irr\'eductible sur $Y\otimes_k \bar k$ pr\'esent dans $f_*\Ql$. Soit $Z$ le
support de $K$. Alors on a l'in\'egalit\'e stricte
$${\rm codim}(Z)<d.$$
\end{apptheoreme}

\begin{proof}
Sous l'hypoth\`ese que les fibres de $f$ sont irr\'eductibles, le faisceau de
cohomologie de degr\'e maximal $2d$ est le syst\`eme local
$$\rmH^{2d}(f_*\Ql)=\Ql(-d).$$

Consid\'erons la d\'ecomposition de $f_*\Ql$ sur $Y\otimes_k \bar k$
$$f_*\Ql=\bigoplus_{(L,n)} L[-n]^{m_{L,n}}$$
o\`u $L$ parcourt l'ensemble des classes d'isomorphismes de faisceaux
pervers irr\'eductibles sur $Y\otimes_k \bar k$ et $n$ l'ensemble des
entiers. Si $m_{L,n}\not=0$, alors on a
$\rmH^i(L[-n])=0$ pour tout $i> 2d$ et m\^eme $\rmH^{2d}(L[-n])=0$ si $L$ n'est pas
isomorphe \`a $\Ql[\dim(Y)]$.

Le m\^eme argument que dans le th\'eor\`eme pr\'ec\'edent montre que
pour tout faisceau pervers irr\'eductible $K$ sur $Y\otimes_k \bar k$ qui
n'est pas isomorphe \`a $\Ql[\dim(Y)]$ alors le support $Z$ de $K$ doit
v\'erifier l'in\'egalit\'e stricte ${\rm codim}(Z)<d$. Bien entendu, si $K$ est
isomorphe \`a $\Ql[-\dim(Y)]$, son support est
$Y$ tout entier et l'in\'egalit\'e est trivialement satisfaite.
\end{proof}

\newpage

\setcounter{tocdepth}{2}
\begin{small}
\tableofcontents
\end{small}

\end{document}